\documentclass[11pt, a4paper]{amsart}
\usepackage{amscd,amssymb,graphics,color,hyperref,verbatim}
\usepackage[a4paper, includeheadfoot, margin=2 cm, top=0.8 cm, footskip=0.8 cm]{geometry} 
\usepackage{multicol} 
\usepackage{tikz-cd}

\usepackage{graphicx}
\usepackage{psfrag}
\usepackage{marvosym}
\usepackage{amsmath}
\usepackage{amsfonts}
\usepackage{amssymb}
\usepackage{amsthm}
\usepackage{mathrsfs}
\usepackage{enumitem}

\renewcommand\theenumi{\arabic{enumi}}

\usepackage{ytableau}

\input xy
\xyoption{all}

\usepackage{calligra}
\usepackage{mathrsfs}
\DeclareMathAlphabet{\mathcalligra}{T1}{calligra}{m}{n}
\DeclareFontShape{T1}{calligra}{m}{n}{<->s*[1.5]callig15}{}

\newtheorem{theorem}{Theorem}[section]
\newtheorem*{theoremstar}{Theorem}
\newtheorem{lemma}[theorem]{Lemma}
\newtheorem{proposition}[theorem]{Proposition}
\newtheorem{corollary}[theorem]{Corollary}
\newtheorem{conjecture}[theorem]{Conjecture}
\newtheorem*{conjecturestar}{Conjecture}

\theoremstyle{definition}
\newtheorem{definition}[theorem]{Definition}

\newtheorem{example}[theorem]{Example}

\newtheorem{remark}[theorem]{Remark}
\newtheorem{theorem-definition}[theorem]{Theorem-Definition}
\newtheorem{lemma-definition}[theorem]{Lemma-Definition}

\numberwithin{equation}{section}

\makeatletter
\renewcommand\part{%
   \if@noskipsec \leavevmode \fi
   \par
   \addvspace{4ex}%
   \@afterindentfalse
   \secdef\@part\@spart}

\def\@part[#1]#2{%
    \ifnum \c@secnumdepth >\m@ne
      \refstepcounter{part}%
      \addcontentsline{toc}{part}{Part \thepart.\hspace{1em}#1}%
    \else
      \addcontentsline{toc}{part}{#1}%
    \fi
    {\parindent \z@ \raggedright
     \interlinepenalty \@M
     \normalfont
     \ifnum \c@secnumdepth >\m@ne
     \centering 
       \large\bfseries \partname\nobreakspace\thepart
       \nobreak. 
     \fi
     \large \bfseries { #2}%
     \par}%
    \nobreak
    \vskip 3ex
    \@afterheading}
\def\@spart#1{%
    {\parindent \z@ \raggedright
     \interlinepenalty \@M
     \normalfont
     \huge \bfseries #1\par}%
     \nobreak
     \vskip 3ex
     \@afterheading}
\makeatother

\renewcommand{\thepart}{\Roman{part}}

\renewcommand{\AA} {\mathbb{A}}

\newcommand{\CC} {\mathbb{C}}
\newcommand{\DD} {\mathbb{D}}

\newcommand{\FF} {\mathbb{F}}
\newcommand{\GG} {\mathbb{G}}

\newcommand{\LL} {\mathbb{L}}

\newcommand{\NN} {\mathbb{N}}

\newcommand{\PP} {\mathbb{P}}
\newcommand{\QQ} {\mathbb{Q}}
\newcommand{\RR} {\mathbb{R}}
\renewcommand{\SS} {\mathbb{S}}

\newcommand{\VV} {\mathbb{V}}

\newcommand{\ZZ} {\mathbb{Z}}

\newcommand {\shA} {\mathcal{A}}
\newcommand {\shB} {\mathcal{B}}
\newcommand {\shC} {\mathcal{C}}
\newcommand {\shD} {\mathcal{D}}
\newcommand {\shE} {\mathcal{E}}
\newcommand {\shF} {\mathcal{F}}

\newcommand {\shK} {\mathcal{K}}
\newcommand {\shL} {\mathcal{L}}

\newcommand {\shN} {\mathcal{N}}

\newcommand {\shQ} {\mathcal{Q}}
\newcommand {\shR} {\mathcal{R}}
\newcommand {\shS} {\mathcal{S}}
\newcommand {\shT} {\mathcal{T}}

\newcommand {\shU} {\mathcal{U}}
\newcommand {\shV} {\mathcal{V}}

\newcommand {\shZ} {\mathcal{Z}}

\newcommand {\sE} {\mathscr{E}}
\newcommand {\sF} {\mathscr{F}}
\newcommand {\sG} {\mathscr{G}}
\newcommand {\sH} {\mathscr{H}}
\newcommand {\sI} {\mathscr{I}}

\newcommand {\sK} {\mathscr{K}}
\newcommand {\sL} {\mathscr{L}}

\newcommand {\sN} {\mathscr{N}}
\newcommand {\sO} {\mathscr{O}}
\newcommand {\sP} {\mathscr{P}}
\newcommand {\sQ} {\mathscr{Q}}

\newcommand {\sT} {\mathscr{T}}

\newcommand {\sV} {\mathscr{V}}
\newcommand {\sW} {\mathscr{W}}

\newcommand {\foh}  {\mathfrak{h}}

\newcommand {\fom}  {\mathfrak{m}}
\newcommand {\fon}  {\mathfrak{n}}

\newcommand {\foL} {\mathfrak{L}}

\newcommand {\foR} {\mathfrak{R}}

\newcommand {\Ann} {\operatorname{Ann}}

\newcommand{\blank}{\underline{\hphantom{A}}}

\newcommand {\codim} {\operatorname{codim}}
\newcommand {\Coh} {\operatorname{Coh}}
\newcommand {\coh} {\operatorname{coh}}
\newcommand {\Coker} {\operatorname{Coker}}

\newcommand {\cone} {\operatorname{cone}}

\newcommand {\D} {\operatorname{D}}

\newcommand {\Ext} {\operatorname{Ext}}
\newcommand{\sExt}{\mathscr{E} \kern -1pt xt}

\newcommand {\Gr} {\operatorname{Gr}}

\newcommand{\Hilb}{\mathrm{Hilb}}

\newcommand {\Hom} {\operatorname{Hom}}
\newcommand {\sHom}{\mathscr{H}\kern-5pt\mathcalligra{om}}

\newcommand {\id} {\operatorname{id}}
\newcommand {\Id} {\operatorname{Id}}

\renewcommand {\Im} {\operatorname{Im}}

\newcommand {\Jac} {\operatorname{Jac}}

\newcommand {\kk} {\Bbbk}
\renewcommand {\ker } {\operatorname{Ker}}
\newcommand {\Ker} {\operatorname{Ker}}

\newcommand {\Pic} {\operatorname{Pic}}

\newcommand {\Proj} {\operatorname{Proj}}

\newcommand {\pr} {\operatorname{pr}}

\newcommand {\rank} {\operatorname{rank}}

\newcommand {\Spec} {\operatorname{Spec}}

\newcommand {\Supp} {\operatorname{Supp}}
\newcommand {\Sym} {\operatorname{Sym}}

\newcommand {\Tor} {\operatorname{Tor}}
\newcommand{\sTor}{\mathscr{T} \kern -3pt or}
\newcommand {\Tot} {\operatorname{Tot}}

\newcommand {\Bl} {\operatorname{Bl}}
\newcommand {\Quot} {\operatorname{Quot}}
\newcommand {\foQuot} {\mathfrak{Quot}}

\newcommand {\bR} {\mathbf{R}}

\newcommand {\Qcoh} {\operatorname{Qcoh}}
\newcommand{\Mod}{\operatorname{Mod}}

\renewcommand {\b} {\mathrm{b}}

\newcommand {\qc} {\mathrm{qc}}

\newcommand{\Perf}{\mathrm{Perf}}
\newcommand{\Db}{\mathrm{D}^{\mathrm{b}}}
\newcommand{\Dqc}{\mathrm{D}_{\mathrm{qc}}}
\newcommand{\Dpc}{\mathrm{D}_{\mathrm{pc}}}

\renewcommand {\SS}{\mathrm{S}}

\newcommand{\bQ}{\mathbf{Q}}  % coherator
\newcommand{\RsHom}{\mathbf{R} \sHom}
\newcommand{\RHom}{\mathbf{R} \Hom}
\newcommand{\ihom}{{\underline{\hom}}} %\newcommand{\ihom}{{\mathsf{hom}}} % internel Homs, or {{\underline{\hom}}}

\newcommand{\shom}{\sHom}

\newcommand{\Fitt}{\operatorname{Fitt}}
\newcommand{\depth}{\operatorname{depth}}

\newcommand{\pdepth}{\operatorname{p.depth}}

\title[]{Derived categories of Quot schemes of locally free quotients}%: decomposition, duality and autoequivalence}
\author[Q.Y.\ JIANG]{Qingyuan Jiang}
\address{School of Mathematics, University of Edinburgh, James Clerk Maxwell Building, Peter Guthrie Tait Road, Edinburgh EH9 3FD, United Kingdom.}
\email{qingyuan.jiang@ed.ac.uk}

\begin{document}

\begin{abstract} 
%\emph{{\bf N.B.} (Date: Feb, 2023). This paper will be divided into two parts. We'll keep this long version so people can see how everything fits together. However, for results reference, please refer to the new versions. We apologize for any inconvenience this has caused.} \medskip

This paper studies the derived category of the Quot scheme of rank $d$ locally free quotients of a sheaf $\sG$ of homological dimension $\le 1$ over a scheme $X$. In particular, we propose a conjecture about the structure of its derived category and verify the conjecture in various cases. This framework allows us to relax certain regularity conditions on various known formulae -- such as the ones for blowups (along Koszul-regular centers), Cayley's trick, standard flips, projectivizations, and Grassmannain-flips -- and supplement these formulae with the results on mutations and relative Serre functors. This framework also leads us to many new phenomena such as virtual flips, and structural results for the derived categories of (i) $\Quot_2$ schemes, (ii) flips from partial desingularizations of $\rank \le 2$ degeneracy loci, and (iii) blowups along determinantal subschemes of codimension $\le 4$.
\vspace{-2mm} 
\end{abstract}
\maketitle
\tableofcontents

\section{Introduction} 

This series of papers studies the derived category of the Quot scheme $\Quot_{X,d}(\sG)$ of locally free quotients of an $\sO_X$-module $\sG$ of homological dimension $\le 1$. In this paper:
	\begin{enumerate}[leftmargin=*]
		\item We set up the relevant foundations for this work and sequels, including:
			\begin{enumerate}[leftmargin=*]
				\item A study of the fundamental properties of Quot schemes and sheaves of finite homological dimensions, after Grothendieck;
				\item An investigation of the theory of relative Fourier--Mukai transforms for quasi-compact, quasi-separated schemes, which is of independent interest on its own; 
				\item Lascoux-type resolutions for the images of generators under the correspondences of Quot schemes;
			\end{enumerate}
		\item We propose a semiorthogonal decomposition of the derived category of $\Quot_{X,d}(\sG)$ in terms of the derived categories of the Quot schemes $\Quot_{X, d_-}(\sK)$ of its ``dual" $\sK = \sExt^1(\sG, \sO_X)$, and verify this proposal in various cases. 
	\end{enumerate}

In particular, this paper focuses on the cases when the Fourier--Mukai kernels are given by {\em vector bundles}; The sequels study the general case based on the results of this paper.

\subsection{Quot schemes of locally free quotients}
Let $X$ be a scheme, let $\sG$ be a quasi-coherent $\sO_X$-module, and let $d\ge0$ be an integer. The {\em (relative) Quot scheme $\Quot_{X,d}(\sG)$ of locally free quotients of $\sG$ over $X$}, introduced by Grothendieck (see \cite{Gro, EGAI, AK, Nit}), parametrizes rank $d$ locally free quotients of $\sG$; see Def. \ref{def:Quot} for the precise definition. We will simply call them {\em Quot schemes} in this paper \footnote{This paper reserves the name ``Grassmannian" and the notation ``$\Gr_d(\sE)$" for the (usual) Grassmannian bundles of rank-$d$ {\em subbundles} of a {\em locally free sheaf} $\sE$. In particular, if $\sE$ is locally free, $\Gr_d(\sE)=\Quot_d(\sE^\vee)$.}. 
All the fibers of $\pi \colon \Quot_{X,d}(\sG) \to X$ are (usual) Grassmannian varieties (see Rmk. \ref{rmk:Quot:fiber}) but in general of different dimensions. 

 We will investigate the fundamental properties of Quot schemes and their relations with degeneracy loci in \S \ref{sec:Quot_deg}, Part \ref{part:FM}, after Grothendieck. Examples of Quot schemes include projectivizations $\PP(\sG)$, (usual) Grassmannian bundles, blowups along Koszul-regularly immersed centers (Lem. \ref{lem:blowup:univ}), blowups along determinantal ideals (Lem. \ref{lem:Quot=Bldet} and Lem. \ref{lem:Quotfib=Bl}). 

%therefore one could view $\pi \colon \Quot_{X,d}(\sG) \to X$ as a ``Grassmannian bundle with fiber dimension jumping" \footnote{A similar description was suggested to me by an anonymous referee of my  previous paper \cite{J19}. }.
%The typical feature of these Quot schemes is that all the fibers of $\pi \colon \Quot_{X,d}(\sG) \to X$ are Grassmannian varieties, see Rmk. \ref{rmk:Quot:fiber}. Hence one could view $\pi \Quot_{X,d}(\sG) \to X$ as "Grassmannian bundles with fiber dimension jumping". Examples of Quot schemes include projectivizations $\PP(\sG) = \Quot_{X,1}(\sG)$, (usual) Grassmannian bundles, blowups along Koszul-regularly immersed centers Lem. \ref{lem:blowup:univ}, blowups along determinantal ideals Lem. \ref{lem:Quot=Bldet} and Lem. \ref{lem:Quotfib=Bl}, as well as correspondence spaces of Quot schemes Lem \ref{lem:hdK:iterated}. 

Assume $\sG$ has homological dimension $\le 1$ and rank $\delta$, then the ``interesting part" of information of the derived dual $\sG^\vee =\RsHom(\sG,\sO_X)$ of $\sG$ is encoded by the extension sheaf 
	$$\sK : = \sExt^1_{\sO_X}(\sG, \sO_X),$$
which, under mild assumptions, also has finite homological dimension; see Lem. \ref{lem:hdK}.

Motivated by the philosophy of Orlov \cite{O05}, Kuznetsov and Shinder \cite{KS}, we make the following conjecture based on computations in the Grothendieck ring of varieties in \S \ref{sec:K0}:

\begin{comment}
	\begin{align*}
	[ \Quot_{X,d}(\sG) ] = \sum_{i=0}^{\min\{d, \delta\}} \LL^{(d-i)(\delta-i)} \cdot [\Gr_{i}(\delta)] \cdot [\Quot_{X,d-i}(\sK)] \in K_0(\mathrm{Var}/\kk).
	\end{align*}  
Inspired by the philosophy of Orlov \cite{O05}, Kuznetsov and Shinder \cite{KS}, it is natural to conjecture that there is {\em categorification} of the above formula; see Conj. \ref{conj:cat}.  This framework unifies various known formulae such as the formulae for Grassmannian bundles, blowups, Cayley's tricks and standard flips; see \S \ref{sec:intro:global} for more examples. 
\end{comment}

\begin{conjecturestar}[The ``Quot formula", Conj. \ref{conj:cat}] Assume that certain Tor-independent condition \eqref{assumption:intro:Tor} holds. Then for each $i \in [0,\min\{d,\delta\}]$, there are Fourier--Mukai kernels
	$$E_{i,\alpha} \in \Perf(\Quot_{X,d-i}(\sK) \times_X \Quot_{X,d}(\sG))$$
parametrised by Young diagrams $\alpha$ inscribed in a box of size $i \times(\delta-i)$ such that the corresponding relative Fourier--Mukai functors $\Phi_{\shE_{i,\alpha}}\colon \D(\Quot_{X,d-i}(\sK) \to  \D(\Quot_{X,d}(\sG))$ are fully faithful. Moreover, these functors induce a semiorthogonal decomposition
	$$ \D(\Quot_{X,d}(\sG)) = \left \langle  \text{$\binom{\delta}{i}$-copies of~~} \D(\Quot_{X,d-i}(\sK) \right \rangle_{i \in [0,\min\{d,\delta\}]}.$$
\end{conjecturestar}
We will discuss the Tor-independent condition \eqref{assumption:intro:Tor} in \S \ref{sec:intro:Tor} below. Here and for the rest of the introduction, for a scheme $Y$, the derived category $\D(Y)$ stands for any one of the following categories: (i) the category of perfect complexes $\Perf(Y)$; (ii) the bounded (pseudo)coherent category $\Db(Y)$; (iii) the unbounded quasi-coherent category $\Dqc(Y)$; see \S \ref{sec:generalities:derived}.

The ``Quot formula" unifies various known formulae such as the formulae for Grassmannian bundles, blowups, Cayley's tricks and standard flips; see \S \ref{sec:intro:global:known} for more examples. 

The conjecture in the case  $d=1$ is equivalent to the {\em projectivization formula} 
	$$\D(\PP(\sG)) = \big \langle \D(\PP(\sK)),  ~\text{$\delta$-copies of} ~ \D(X)\big \rangle,$$
proved by the author and Leung in \cite{JL18}; see \cite{JL18} or Thm. \ref{thm:projectivization} for more details.

This paper proves the conjecture in the case $d=2$ (see \S \ref{sec:intro:global} or Thm. \ref{thm:d=2}):
		\begin{align*}
			\D(\Quot_{X,2}(\sG)) = \Big \langle \D(\Quot_{X,2}(\sK)), ~\text{$\delta$-copies of} ~ \D(\PP(\sK)), ~  \text{$\binom{\delta}{2}$-copies of} ~ \D(X) \Big \rangle.
		\end{align*}
		
If $\sG$ is {\em locally free}, then Beilinson and Orlov's theorem (resp. Kapranov's theorem) implies that $\D(\PP(\sG))$ (resp. $\D(\Quot_{X,2}(\sG)) =\D(\Gr_2(\sG^\vee))$) is generated by a relative exceptional sequence of length $\delta$ (resp. $\binom{\delta}{2}$). Therefore the above two formulae can be viewed as ``correcting" Beilinson and Orlov's and Kapranov's theorems over the locus where $\sG$ is {\em not locally free}; The ``corrections" are precisely given by the derived categories of Quot schemes of the ``dual" $\sK=\sExt^1_{\sO_X}(\sG, \sO_X)$ who is supported on the non-locally-free locus of $\sG$.

In addition to the case (i) $d=2$ mentioned above, we also verify the conjecture in various other cases, including (ii) $\rank \sG \le 3$ and (iii) $\ell \le 2$, where $\ell = n - d$ if $\sG$ is $n$-generated. See subsection \S \ref{sec:intro:global} below for more results and details.

\begin{remark} In the above conjecture, we do not specify the Fourier--Mukai kernels $E_{i, \alpha}$. In fact, there are at least two natural choices for the kernels:
	\begin{enumerate}[leftmargin=*]
		\item (Vector bundle choices) $E_{i,\alpha}$ are given by natural vector bundles on the fiber products;
		\item (Universal choices) $E_{i, \alpha}$ are given by derived Schur perfect complexes associated to tautological two-term complexes on the fiber products;
	\end{enumerate}
For example, in the projectivization formula \cite{JL18} mentioned above, the last component ``$\text{$\delta$-copies of} ~ \D(X)$" could be induced by either of the following relative exceptional sequences:
	\begin{enumerate}[leftmargin=*]
		\item (Vector bundle choices) $\{\sO_{\PP(\sG)}, \sO_{\PP(\sG)(1)}, \ldots, \sO_{\PP(\sG)(\delta-1)}\}$; Or
		\item (Universal choices) $\{\mathbf{\Lambda}^{\delta-1}L_{\PP(\sG)/X} (\delta-1), \ldots, L_{\PP(\sG)/X}(1), \sO_{\PP(\sG)}\}$, where $L_{\PP(\sG)/X}$ is the relative {\em cotangent complex}, and $\mathbf{\Lambda}^i$ denotes the {\em derived} $i$th exterior power over $\PP(\sG)$.
	\end{enumerate}
This paper focuses on the ``vector bundle choices" for kernels; The sequels focus on the ``universal choices". These two choices are in general different but mutation-equivalent. This means that for all the cases covered in this paper, in particular for all the examples of \S \ref{sec:intro:global} below, the presentations of Fourier--Mukai kernels here would be simpler. 
\end{remark}

\subsection{Local situation and Lascoux-type resolutions} 
In Part \ref{part:local}, we prove our main results in the universal local situation $X= |\Hom_\kk(W,V)|$, where $W$ and $V$ are vector spaces over a field $\kk$ of ranks $m$ and $n$, $m \le n$. Let $\sG$ be the cokernel of the tautological map $\tau \colon W \otimes \sO_X \to V \otimes \sO_X$, then for a given pair of integers $(d_+, d_-)$, the Quot schemes 
	$$\shZ_+ =\Quot_{X, d_+}(\sG) \to \DD_{\ell_+} \subseteq X, \qquad \shZ_- =\Quot_{X, d_-}(\sK) \to \DD_{\ell_-} \subseteq X$$ 
are in general resolutions of the degeneracy loci $\DD_{\ell_\pm} \subset X$ (which are the loci where $\rank \tau \le \ell_{\pm}$, see \S \ref{sec:deg}), where $\ell_+ = n - d_+$, $\ell_- = m - d_-$. These schemes play crucial roles in studying determinantal varieties, see \cite{ACGH, Ful, FP, Wey, Laz04} and the references therein. 

The key technical result for the local case, Lem. \ref{lem:key}, is a {\em Lascoux-type resolution} for the image of each generator of $\Db(\shZ_-)$ in $\Db(\shZ_+)$ via the correspondence 
 	\begin{equation*}
	\begin{tikzcd}[row sep=1 em, column sep = 2.6 em]
			 &\widehat{\shZ} := \shZ_- \times_X \shZ_+ \ar{ld}[swap]{r_-} \ar{rd}{r_+} & \\
			 \shZ_- & & \shZ_+.
	\end{tikzcd}
	\end{equation*}
More precisely, if $\ell_+ \ge \ell_-$, i.e., $\DD_{\ell_+} \supseteq \DD_{\ell_-}$, then for each generator $\Sigma^{\alpha} \shQ_{-}^\vee$ of $\Db(\shZ_-)$, where $\alpha \in B_{\ell_-,d_-}$ is a Young diagram \S \ref{sec:Young_Grassmannian}, there is a resolution:
	$$r_{+\,*} \, r_{-}^*(\Sigma^{\alpha} \shQ_{-}^\vee) = \{0 \to F^{-(\ell_+-\ell_-)d_-} \to \cdots \to F^{-1} \to  F^0  \to 0\} \in \Db(\shZ_+),$$
where $F^p$'s are explicitly given in terms of generators of $\Db(\shZ_+)$ in Lem. \ref{lem:key}.

The above resolution is a far-reaching generalization of the {\em Lascoux resolutions} of determinantal ideals \cite{Lasc}, \cite[\S 6.1]{Wey}. In fact, in the special case when $d_-=0$, $\alpha = (0)$, our complex $F^\bullet$ of Lem. \ref{lem:key} coincide with the Lascoux complex for the ideal sheaf $\sI_{\DD_{\ell_+}}$; see Ex. \ref{ex:Lascoux}. In general, if the Young diagram $\alpha \ne (0)$, the complex $F^\bullet$ still enjoys nice patterns similar to the Lascoux complexes but complicated and ``twisted" by the contributions of $\alpha$. The patterns of these complexes $F^\bullet$ allow us to prove our main results in the local situation.

\begin{remark}
The local Part \ref{part:local} is the only place in this paper where we use the {\em characteristic zero} assumption. It is well-known that the Lascoux complexes \cite{Lasc} depend on characteristics; see Hashimoto's example \cite{Has} for the different behaviours of these complexes when $\kk = \QQ$ and  $\kk = \FF_3$. Hence, as a generalization of Lascoux resolutions, the complex $F^\bullet$ also depends on characteristics. However, when the complex $F^\bullet$ is characteristic-free, for example when it is a Koszul complex, our results on derived categories are also characteristic-free. Furthermore, the combinatorial patterns of the nonzero-terms $F^\bullet$ appear not to depend on characteristics. Hence we do expect analogous statements of Lem. \ref{lem:key} and the main results of this paper to hold in arbitrary characteristics.
\end{remark}

\subsection{From local to global: the theory of relative Fourier--Mukai transforms}
The key to globalizing our local results is the theory of relative Fourier--Mukai transforms, especially Tor-independent base-change \cite{Kuz06, Kuz11} and descent theory \cite{Ela, Shi, BS, BOR, AE}. 

In Part \ref{part:FM}, \S \ref{sec:FM}, we investigate the theory of relative Fourier-Mukai transforms for general {\em quasi-compact, quasi-separated schemes}. Our definition of a relative Fourier-Mukai transform follows the definition of Bergh and Schn{\"u}rer \cite{BS}. Such a framework is vital for our study of Quot schemes in this paper and sequels, where we use the powerful tools of mutation theory, relative Serre functor, base-changes, descent for {\em non-flat} families over general bases.

The possible novelty of our investigation of this part might be the use of $\Hom$-spaces  modified by {\em coherators}. More precisely, for a morphism $f \colon X \to S$ between of quasi-compact, quasi-separated schemes, and $F, G \in \Dqc(X)$, we define the $\Dqc(S)$-valued Hom-space by:
	$$\shom_S(F,G) : = f_* \,  \bQ_X \circ \RsHom_X(F,G) = \bQ_S \circ f_*  \RsHom_X(F,G) \in \Dqc(S),$$	
where $\bQ_X$ (resp. $\bQ_S$) is the {\em coherator} \cite{SGA, TT}, i.e., the right adjoint of $j_{\qc} \colon \Dqc(X) \hookrightarrow D(\sO_X)$ (resp. $j_{\qc} \colon \Dqc(S) \hookrightarrow D(\sO_S)$); see Def. \ref{def:shom}. Then $\Dqc(X)$ is a closed symmetric monoidal category (or an {\em unital algebraic stable homotopy category} in the sense of \cite{HPS}), enriched over $\Dqc(S)$ by $\shom_S(\blank, \blank)$; see Thm. \ref{thm:Lipman}. Here we build on the work of Lipman \cite{Lip}; The idea of using coherators also appeared in \cite{TLRG} in their study of $DQcoh(X)$.

One benefit of using Hom-spaces modified by coherators is that the Grothendieck--Serre duality takes neater forms than usual, see Thm. \ref{thm:Lipman} \eqref{thm:Lipman-2} \eqref{thm:Lipman-3} \eqref{thm:Lipman-4}, while these formulae for the usual sheafified Homs typically require certain boundedness assumptions on $F, G$ or on $f$.
 These features allow us to investigate in the same general framework  the theory of relative Serre duality \S \ref{sec:Serre}, linear categories and base-change \S \ref{sec:bc}, relative Fourier--Mukai transforms, descent theory and convolutions \S \ref{sec:relFM}, \S \ref{sec:relFM:cov}, relative exceptional collections and their mutations \S \ref{sec:relexc}. We also include the discussions on projective bundles \S \ref{sec:proj.bundle} and Grassmannian bundles \S \ref{sec:Grass_bundles}  for completeness and the reader's convenience.

\subsection{Global situation and main results} \label{sec:intro:global} Part \ref{part:global}, which builds on the first two parts \ref{part:FM} and \ref{part:local}, contains our main results in the global situation.

\subsubsection{Tor-independent condition} \label{sec:intro:Tor} We say {\em \eqref{assumption:intro:Tor} holds for a pair of integers $(d_+,d_-)$} if:
\begin{enumerate}[label=$(\dagger)$, ref=$\dagger$]
	\item \label{assumption:intro:Tor} The pair $(\Quot_{X,d_+}(\sG), \Quot_{X,d_-}(\sK))$ is a {\em Tor-independent} base-change from the ``universal situation" (see \S \ref{sec:Tor-ind:Quot:bc}, Def. \ref{def:Tor-ind:quot} for more details). 
\end{enumerate}
This condition is independent of a local presentation (Lem. \ref{lem:Tor-ind-ind}). If $X$ is Cohen--Macaulay (for example, if $X$ is smooth), then the Tor-independent condition \eqref{assumption:intro:Tor} is equivalent to certain {\em expected dimension conditions}; see Lem. \ref{lem:Tor-ind:quot:CM}. 

In the situation of proving a case of Quot formula, we say {\em \eqref{assumption:intro:Tor} holds} if above condition holds for any pair of integers $(d_+,d_-)$ that appears in the desired semiorthogonal decomposition.

\subsubsection{} \label{sec:intro:global:known} On the one hand, the framework of this paper provides a {\em unified} treatment of various known formulae; In these cases, our approach usually leads to a different proof from the existing ones, allows us to relax certain regularity conditions, and  supplements these formulae with the results about mutations and relative Serre functors. These cases include:
	\begin{enumerate}[leftmargin=*]
		\item Orlov's formula \cite[Thm. 4.3]{Orlov92} (see also \cite[Thm. 6.9]{BS} for the stacky case) for blowups along Koszul-regularly immersed centers; see Thm. \ref{thm:blow-up}; 
		\item Orlov's formula {\cite[Prop. 2.10]{Orlov06}, \cite{RT}} for Cayley's trick; see Thm. \ref{thm:Cayley};
		\item The projectivization formula \cite{JL18} of the author and Leung; see Thm. \ref{thm:projectivization};
		\item Pirozhkov's formula \cite{Pi20} for generalized Cayley's trick; see Thm. \ref{thm:gen:Cayley};
		\item The embedding of derived categories for Grassmannian flips \cite{BLV2, BLV3, DS, LX, BCF+}; see Thm.\ref{thm:Grassflips};
		\item Bondal--Orlov's formula for standard flips \cite{BO, Tod2, BLT}; see Thm. \ref{thm:standardflip}.
	\end{enumerate}
Notice that one common feature of these formulae (apart from the one for Grassmannian flips) is that, they provide a complete description of the structure of the derived category of $\Quot_{X,d}(\sG)$ in terms of that of {\em two different} schemes of the form $\Quot_{X,d_-}(\sK)$. In the following, we will prove formulae involving {\em three} (or {\em four}) different schemes.

\subsubsection{} \label{sec:intro:global:new} On the other hand, this approach also leads to various new phenomena. Let us assume for simplicity that $X$ is Cohen--Macaulay over a field $\kk$ of characteristic zero, and Tor-independent condition \eqref{assumption:intro:Tor}   holds in all the following corresponding situations. 

First, we have the following general results (let $\D$ stand for $\Perf$, $\Db$ or $\Dqc$):
	\begin{enumerate}[leftmargin=*]
		\item If $d \le \rank \sG=:\delta$, then  $\{\Sigma^{\alpha^t} \sQ_d \}_{\alpha \in B_{\delta-d,d}^{\preceq}}$ forms a relative exceptional collection of $\D(\Quot_{X,d}(\sG))$, see Prop. \ref {prop:top:bottom} (which also contains a ``dual" description about contributions from bottom strata);
		\item If $\Quot_{X,d}(\sK) \neq \emptyset$, then there exists the {\em virtual flip} phenomenon Thm. \ref{thm:virtualflips}, i.e., a fully faithful embedding $\D(\Quot_{X,d}(\sK)) \hookrightarrow \D(\Quot_{X,d}(\sG))$ for the ``virtual flip $\Quot_{X,d}(\sG) \dashrightarrow \Quot_{X,d}(\sK)$". The virtual flip is closely related to Toda's d-critical flips \cite{Tod2};
	\end{enumerate}	

Secondly, we have the following structural results. 

\subsubsection{$\Quot_2$--formula}  If $d = 2$, and $\delta := \rank \sG \ge 2$, then $\pi \colon \Quot_{X,2}(\sG) \to X$ is a generic Grassmannian $\Gr_2(\delta)$-bundle. For the structure of  $\D(\Quot_{X,2}(\sG))$, we have: 

\begin{theoremstar}[$\Quot_2$--formula, {Thm. \ref{thm:d=2}}] There is a semiorthogonal decomposition:
	\begin{align*}
			\D(\Quot_{X,2}(\sG)) = \Big \langle \D(\Quot_{X,2}(\sK)), ~\text{$\delta$-copies of} ~ \D(\PP(\sK)), ~  \text{$\binom{\delta}{2}$-copies of} ~ \D(X) \Big \rangle
		\end{align*}
(provided that the Tor-independent condition \eqref{assumption:intro:Tor} holds). Here, the last component is induced by a relative exceptional sequence over $X$, and $\D$ stands for $\Perf$, $\Db$ or $\Dqc$. 
\end{theoremstar}

In the above formula, one can regard the last component ``$\text{$\binom{\delta}{2}$-copies of} ~\D(X)$" as the ``Grassmannian part" for the generic $\Gr_2(\delta)$-bundle structure of $\pi \colon \Quot_2(\sG) \to X$; the middle component as the ``first-order correction" for the difference between $\pi$ and a genuine $\Gr_2(\delta)$-bundle; and first component $\D(\Quot_{X,2}(\sK))$ as the ``virtual flip part" for the  ``virtual flip $\Quot_{X,2}(\sG) \dashrightarrow \Quot_{X,2}(\sK)$".

\subsubsection{Flips from partial desingularizations of $\rank \le 2$ loci} \label{sec:intro:fliprk2}  Let $\sW$ and $\sV$ be vector bundles over $X$ of rank $m$ and $n$, and let $\sigma \colon \sW \to \sV$ be an $\sO_X$-module map (such that certain Tor-independent condition is verified; see Def. \ref{def:Tor-ind:quot}). Denote $Y_\ell : = D_{\ell}(\sigma)$ the {\em degeneracy locus} \S \ref{sec:deg} of $\sigma$ of $\rank \le \ell$. Set $\sG = \Coker (\sigma)$ and $\sK = \Coker(\sigma^\vee)$. Then the Quot schemes
	$$Y_\ell^+ : = \Quot_{X,n-\ell}(\sG) \to Y_\ell \subseteq X \quad  \text{and} \quad Y_\ell^- : = \Quot_{X, m -\ell}(\sK) \to Y_\ell \subseteq X$$
give rise to two different partial desingularizations of $Y_\ell$. The birational map $Y_1^+ \dashrightarrow Y_1^-$ is a {\em standard flip} considered in \S \ref{sec:standardflip}, and the previous Thm.\ref{thm:standardflip} describes the structure of $\D(Y_1^+)$ in terms of $\D(Y_1^-)$ and $\D(Y_0)$. Regarding the flip $Y_2^+ \dashrightarrow Y_2^-$:

\begin{theoremstar}[Thm. \ref{thm:rk<=2flip}]
There is a semiorthogonal decomposition:
	\begin{align*}
	\D(Y_2^+) = \big \langle  \text{$\binom{\delta}{2}$-copies of} ~ \D(Y_0),  ~\text{$\delta$-copies of} ~ \D(Y_1^-), ~\D(Y_2^-) \big \rangle
	\end{align*}
(provided that \eqref{assumption:intro:Tor} holds; $\D$ stands for $\Perf$, $\Db$, or $\Dqc$ as usual). 
\end{theoremstar}
The flip $Y_2^+ \dashrightarrow Y_2^-$ is {\em stratified} in the sense that the flipping center has two strata $Y_0$ and $Y_1\backslash Y_0$. Hence in the above formula, while the last component $\D(Y_2^-)$ comes from the flip $Y_2^+ \dashrightarrow Y_2^-$, the first and middle components can be regarded as the ``correction terms" of the flip functor by the respective contributions of the two strata of the flipping center.

\subsubsection{Blowups along determinantal subschemes of codimension $\le 4$} Let $X$ be a Cohen--Macaulay scheme, and let $\sG$ be a quasi-coherent $\sO_X$-module of homological dimension $\le 1$ and $\rank \sG = \delta$. Consider the following sequence of degeneracy loci \S \ref{sec:deg}:
	$$Z_3: = X^{\ge \delta+3}(\sG) \quad \subset \quad Z_2 := X^{\ge \delta+2}(\sG) \quad \subset \quad Z = Z_1 := X^{\ge \delta+1}(\sG) \quad \subset \quad X = X^{\ge \delta}(\sG).$$
Then (under Tor-independent condition \eqref{assumption:intro:Tor}) the determinantal subscheme $Z \subset X$ is Cohen--Macaulay and has codimension $\delta+1$. The Quot scheme $\pi \colon \Quot_{X,\delta}(\sG) = \Bl_Z X \to X$ is the blowup of $X$ along $Z$; see Lem. \ref{lem:Quot=Bldet}. $\widetilde{Z}_3 = \Quot_{X,3}(\sK) \to Z_3$, $\widetilde{Z}_2 := \Quot_{X,2}(\sK) \to Z_2$ and $\widetilde{Z} := \PP(\sK) \to Z$ are partial desingularizations of $Z_3$, $Z_2$ and $Z_1=Z$.

In the case $\delta=1$, the author and Leung  \cite{JL18} prove the following blowup formula for blowups along Cohen--Macaulay subschemes of codimension two: 
	\begin{align*}\D(\Bl_Z X) = \langle \D(\widetilde{Z}), ~\D(X) \otimes \sO_{\Bl_Z X}(1)\rangle =  \langle \D(X) , ~\D(\widetilde{Z}) \rangle.
	\end{align*} 
See \cite{JL18} or Cor. \ref{cor:rk=1} for more details. For the cases $\delta=2,3$, we have:

%For the cases $\delta=3$ and $\delta=4$, we have:
\begin{theoremstar}[Cor. \ref{cor:rk=2} and Cor. \ref{cor:rk=3}] In the above situation (assume \eqref{assumption:intro:Tor} holds, $\D$ stands for $\Perf$, $\Db$, or $\Dqc$), if $\delta=2$, i.e., $\codim_Z(X) = 3$, then there is semiorthogonal decomposition:
	\begin{align*}\D(\Bl_Z X) = \langle  \D(X),  ~\text{2-copies of} ~ \D(\widetilde{Z}), ~\D(\widetilde{Z}_2) \rangle;
	\end{align*}
If $\delta=3$, i.e., $\codim_Z(X) = 4$, then there is semiorthogonal decomposition:
	\begin{align*}
	\D(\Bl_Z X) = \big \langle \D(X), ~\text{3-copies of} ~ \D(\widetilde{Z}), ~\text{3-copies of} ~ \D(\widetilde{Z}_2), ~\D(\widetilde{Z}_3)\big \rangle.
	\end{align*}
\end{theoremstar}

In fact, we obtain these formulae more generally for all $d$, see \S \ref{sec:rk<=3}, Thms. \ref{thm:rk=1}, \ref{thm:rk=2}, \ref{thm:rk=3}.)

Notice that if we base-change the above formulae to the Zariski open dense subset $X \backslash Z_2$, we recover Orlov's blowup formula Thm. \ref{thm:blow-up} for the blowup of $X \backslash Z_2$ along the local complete intersection subscheme $Z\backslash Z_2$. Therefore these formulae could be viewed as extending Orlov's blowup formula {\em beyond} the local complete intersection locus.

\subsection{Applications}
The results of this paper are immediately applicable to various moduli spaces of objects on curves, surfaces or $K3$ categories.
 
\subsubsection{Linear series on curves}
Let $C$ be a smooth projective curve of genus $g$ over a field $\kk$ of characteristic zero, and let $X = \Pic^{g-1+\delta}(C)$ be the Picard variety, $\delta \ge 0$. Then \cite{JL18} shows that there is a tautological sheaf $\sG$ on $X$ which ``sheafifies" $H^0(C,\sL)^\vee$, $[\sL] \in X$ and has homological dimension $\le 1$; Then $\sK: = \sExt^1(\sG,\sO_X)$ ``sheafifies" $H^1(C,\sL)$ for $[\sL] \in X$, and $\sK \ne 0$ if and only if $0 \le \delta \le g-1$. Furthermore, $\PP(\sG) = \Sym^{g-1+\delta}(C)$ and $\PP(\sK) = \Sym^{g-1-\delta}(C)$ are the symmetric products. In general, we have $\Quot_{X, r+1}(\sG) = G_{g-1+\delta}^r(C)$ and $\Quot_{X,r+1}(\sK) = G_{g-1-\delta}^r(C)$, where $G_k^{r}(C)= \{\text{$g_k^r$'s  on $C$}\}$ is the variety of linear series on $C$ of degree $k$ and dimension $r$ (see \cite[Ch. IV, \S 3]{ACGH}).

In \cite[Cor. 5.11]{Tod2}, Toda shows that for all $\delta \ge 0$, the following holds:
	\begin{align*}
	\D(\Sym^{g-1+\delta} (C))   = \big\langle \D(\Sym^{g-1-\delta} (C)), ~ 
	 \D(\Jac(C))(1), \ldots, \D(\Jac(C))(\delta)\big\rangle,
	\end{align*}
where $\Jac(C)$ denotes the Jacobian variety of $C$. We also provide a different proof of this result using the {\em projectivization formula} in \cite{JL18}; See also \cite{BK19}.

The main results of this paper are directly applicable to this situation; For example, Prop. \ref {prop:top:bottom} and $\Quot_2$-formula Thm. \ref{thm:d=2} immediately imply:

\begin{corollary} Let $C$ be a general curve, $\D$ stand for $\Perf, \Db$ or $\Dqc$. For any $\delta \ge 0$,
	 \begin{enumerate}
	 	\item $\D(G_{g-1-\delta}^{r-i}(C)) \hookrightarrow \D(G_{g-1+\delta}^r(C))$ for $i=0,1,\dots, \min\{r+1, \delta\}$.
		\item If $-1 \le r \le \delta-1$, $\D(G_{g-1+\delta}^r(C))$ contains $\binom{\delta}{r+1}$-copies of $\D(Jac \,C)$, induced by a relative exceptional sequence of $\D(G_{g-1+\delta}^r(C))$ of length $\binom{\delta}{r+1}$ over $\Pic^{g-1+\delta}(C)$;
		\item There is a semiorthogonal decomposition:	
	\begin{align*}
	\D(G_{g-1+\delta}^1(C))  = \Big \langle \D(G_{g-1-\delta}^1(C)), ~\text{$\delta$-copies of}~~\D(\Sym^{g-1-\delta}(C)),~ 
	  \text{$\binom{\delta}{2}$-copies of} ~~\D({\rm Jac}(C)) \Big \rangle.
	\end{align*}
 	\end{enumerate}
 \end{corollary}
The phenomena of standard flips Thm.\ref{thm:standardflip}, two-step stratified flips, Thm. \ref{thm:rk<=2flip}, virtual flips Thm. \ref{thm:virtualflips}, blowups along determinantal ideals Cors. \ref{cor:rk=1}, \ref{cor:rk=2}, \ref{cor:rk=3} also occur among these moduli spaces $\{G_{k}^r(C)\}$; we omit the details here.	

\subsubsection{(Nested) Hilbert schemes of points on surface.} Let $S$ be a smooth complex algebraic surface, for any $n \ge 0$, denote $\Hilb_n(S)$ the Hilbert scheme of ideals of $S$ of colength $n$. For any $d \ge 1$, consider the following scheme:
 	$$\Hilb_{n,n+d}^{\dagger}(S) : = \{ (I_n \supset I_{n+d}) \mid I_{n}/I_{n+d} \simeq \kappa(p)^{\oplus d} \text{~ for some $p \in S$}\} \subset \Hilb_n \times \Hilb_{n+d}.$$
(By convention, $\Hilb_{n,n}^{\dagger}(S) = \Hilb_n(S) \times S$, and $\Hilb_{n,n+d}^{\dagger}(S) = \emptyset$ if $d<0$.) If $d=1$, $\Hilb_{n,n+1}^{\dagger}(S) = \Hilb_{n,n+1}(S)$ is the usual (two-step) nested Hilbert scheme, hence $\Hilb_{n,n+d}^{\dagger}(S) $ could be viewed as the {\em $d$-fold (two-step) nested Hilbert scheme}. 

By \cite[\S 5.3]{J20}, the family of schemes $\{ \Hilb_{n,n+d}^{\dagger}(S)\}$ fits into the framework of this paper. In particular, Thm. \ref{thm:rk=1} implies: 
 \begin{corollary}For integers $n,d,k \ge 1$, there is a semiorthogonal decomposition:
 	$$\D(\Hilb_{n,n+d}^{\dagger}(S)) =  \langle \D(\Hilb_{n-d,n}^{\dagger}(S)), \D(\Hilb_{n-d+1,n}^{\dagger}(S)) \rangle.$$
 \end{corollary}
 
If $d=1$, this recovers the formula for usual (two-step) nested Hilbert scheme \cite{BK19, JL18}.

\subsubsection{Other applications} \label{sec:intro:otherapp} The following situations also fit into the framework of ``Quot geometry" considered in this paper: Moduli of stable sheaves on surfaces and their Hecke correspondences \cite{Neg1, Neg, MN}; Brill--Noether theory of moduli of stable sheaves on K3 surfaces \cite{Markman, AT20}, or more generally, Brill--Noether theory of moduli of stable objects in K3 categories \cite{B,BCJ2}; The pair of $\rank$ two Thaddeus moduli spaces \cite{Tha} (when the parameters are large) with their maps to the moduli of rank $2$ vector bundles on curves, considered by \cite{KT}. The results of this paper could be applied verbatim to these situations; we omit the details here in this already long paper.

\subsection{Further directions} The current paper is the first in the series that study the structures of derived categories of Quot schemes. This paper focuses on the cases when essentially no more than three different schemes of the type $\Quot_{X, d_-}(\sK)$ are involved in the formula (except in one case where four are involved), and when the Fourier--Mukai kernels can be given by vector bundles. In the sequels,  we will use the framework of derived algebraic geometry and the theory of derived Schur functor to prove the Quot formula in general.
%we provide two approaches to handling the general situations based on the framework set up in this paper. In the first approach, we use Schur Koszul complexes associated to the tautological two-term complexes on the fibre product as kernels; this approach is explicit and combinatorial. In the second approach, we use derived algebraic geometry to ``categorify" our proof \cite{J20} of the Quot formula for Chow groups.

Next, there are rich algebraic structures among these Quot formulae for different $d$. The Quot formula describes $\D(\Quot_d(\sG))$ in terms of $\D(\Quot_{d-j}(\sK))$, $j=0,\ldots,d$; On the other hand, the categories $\D(\Quot_d(\sG))$ (resp. $\D(\Quot_{d-j}(\sK))$) with different $d$ are connected via flag correspondence schemes. In the flop case $\delta=0$, the algebraic structures among the Quot schemes from the above correspondences are closed related to``flop--flop=twist" phenomena \cite{JL18, ADM, AT, DS}; the case $d=1$ is studied in \cite{JL18}.  In the flip case $\delta \ne 0$, this question is also (or even more) intriguing. For example, in the situation of subsection \S \ref{sec:intro:fliprk2}, Thm. \ref{thm:standardflip} describes $\D(Y_1^+)$ by $\D(Y_1^-)$ and $\D(Y_0)$, and Thm. \ref{thm:rk<=2flip} describes $\D(Y_2^+)$ by $\D(Y_2^-)$, $\D(Y_1^-)$ and  $\D(Y_0)$. Then the question in this case reduces to how Thm. \ref{thm:standardflip} and Thm. \ref{thm:rk<=2flip} intertwine with each other under flag correspondences.
  
Thirdly, the study of degeneracy loci for generic matrices has its counterparts for symmetric and anti-symmetric matrices, see for example \cite{Ful, FP, Wey}.  Similarly, there are parallel theories of the Quot formula of this paper in symmetric and anti-symmetric situations.

Finally, besides the applications mentioned in \S \ref{sec:intro:otherapp}, we also expect close connections between our work and the categorification of Donaldson--Thomas theory studied by \cite{Tod1, Tod2, GT1, GT2} and the geometric categorification studied by \cite{CKL1, CKL2}.

\subsection{Related works} This paper extends the work \cite{JL18} of the author and Leung. The Chow-theoretical counterpart of the current work is studied by the author in \cite{J19,J20}. 

The Quot formula is related to the following previous works: Beilinson's work on projective spaces \cite{Be}; Orlov's formulae for blowups \cite[Thm. 4.3]{Orlov92}  \cite[Thm. 6.9]{BS} and Cayley's trick {\cite[Prop. 2.10]{Orlov06}, \cite{RT}}; Buchweitz, Leuschke and Van den Bergh's work on Grassmannians \cite{BLV} and Grassmannian flips \cite{BLV2,BLV3}; Efimov's work on Grassmannians \cite{Ef}; Bondal--Orlov's formula for standard flops and flips \cite{BO, ADM, Tod2, BLT}. 

Pirozhkov's theorem \cite{Pi20} corresponds to case $m=1$ of Quot formula. The local-to-global strategy of this paper is similar to Pirozhkov's in \cite{Pi20}. 

The embedding of the derived categories under Grassmannian flips is studied in \cite{BLV2, BLV3, DS, LX, BCF+}; In particular, Ballard et al.'s work \cite{BCF+} describes the orthogonal component of the image in terms of matrix factorizations. 

The virtual flip phenomena are closely related the Toda's d-critical flips \cite{Tod1, Tod2, Tod3, Tod4}.

The geometry of Quot schemes is closely related to correspondence spaces for various moduli spaces: the moduli of stable sheaves on surfaces and their Hecke correspondences studied by Maulik and Negu{\c{t}} \cite{Neg1, Neg, MN}; The pair of Thaddeus moduli spaces studied by Thaddeus \cite{Tha}, Koseki and Toda \cite{KT}; The moduli spaces from the Brill--Noether theory of moduli of stable sheaves on K3 surfaces studied by Markman \cite{Markman}, Addington and Takahashi \cite{AT}; The nested Hilbert schemes of points studied by Gholampour and  Thomas \cite{GT1, GT2}; The correspondences in the theory of geometric categorification and Hecke correspondences studied by Cautis, Kamnitzer and Licata \cite{CKL1, CKL2}.

{\em Updates.} 
The Quot formula conjecture of this paper has been proved by Yukinobu Toda \cite{Tod6} (in the case where $X$ is a smooth quasi-projective variety over $\CC$) and by the author \cite{J23} (for general $X$ over $\QQ$). Toda's method \cite{Tod6} is based on techniques of categorical wall-crossing for Donaldson--Thomas theory and categorified Hall products \cite{Tod5}. Our method is based on the theory of derived Grassmannians, flags and derived Schur functors, developed in \cite{J22,J22b}.

%The current paper seems to be the beginning of many exciting projects. Since the appearance of the preprint of this paper, there has been various progress on this topic: 
%First, Toda \cite{Tod6} has proved the Conj. \ref{conj:cat} on Quot formula (in the case where $X$ is a complex smooth and quasi-projective variety), based on techniques of categorical wall-crossing for Donaldson--Thomas theory and categorified Hall products \cite{Tod5}. Second, in \cite{J22}, we studied the theory of derived symmetric algebras and derived projectivizations, which lay the foundations for the theory of derived Schur functors. 
%In forthcoming works, we will use the combination of techniques of \cite{J22} and this paper to study the counterparts of the theory studied in this paper in the context of derived algebraic geometry. This includes the theory of derived Grassmannians and Quot schemes, derived degeneracy loci, and derived Schur functors. In particular, we expect that the generalization of the results on Serre's theorem, Beilinson's relations, and semiorthogonal decompositions in \cite{J22} to the situation of Quot schemes will not only lead to a solution of the conjecture  \ref{conj:cat} on Quot formula, but also a systematical descriptions of the Fourier--Mukai kernels and morphisms algebras in terms of derived Schur functors.

\subsection{Organization of the paper} Part \ref{part:FM} is the preliminary part contains two sections. \S \ref{sec:Quot_deg} studies the basic properties of Quot schemes, degeneracy loci and sheaves of finite homological dimensions.  \S \ref{sec:FM}. investigates the theory of relative Fourier--Mukai transforms.

Part \ref{part:local} contains the main results of this paper in the local situation. In \S \ref{sec:Young_Grassmannian}, we review the theory of Grassmannians and study their mutation theory that will be used later. In \S \ref{sec:local}, we prove our main results in the universal local situation.

Part \ref{part:global} contains the main results in the global situation. In \S \ref{sec:univHom}, we discuss the Tor-independent condition and the general procedure to pass from local to global siutations. The rest of \S \ref{sec:global} contains the majority of cases of the main results of this paper. We single out the results on flips and virtual flips in \S \ref{sec:flips}.

Appendix \S \ref{sec:K0} contains the computations in Grothendieck rings of varieties and the conjectures.  Appendix \S \ref{sec:proj.bundle} includes characteristic-free results on projective bundles.

\subsection*{Acknowledgement}  The author would like to thank Arend Bayer for many helpful discussions throughout this work, J{\'a}nos Koll{\'a}r and Mikhail Kapranov for inspiring discussions on Quot schemes, and Richard Thomas for valuable discussions on nested Hilbert schemes, degeneracy loci and many helpful suggestions, Yukinobu Toda for many helpful communications on the derived categories of Quot schemes, and his solution to the conjecture on Quot formula via categorical wall-crossing and Hall algebras. This work was supported by the Engineering and Physical Sciences Research Council [EP/R034826/1].

 %%%% Part 1: Preliminary 
 \newpage 
\part{Preliminaries on Quot schemes and Fourier--Mukai transforms}  \label{part:FM}
 %\addtocontents{toc}{\smallskip}	
 
%%%% Quot schemes
\addtocontents{toc}{\vspace{0.5\normalbaselineskip}}	
%\addtocontents{toc}{\smallskip}	
\section{Quot schemes and degeneracy loci}  \label{sec:Quot_deg}
\noindent \textbf{Notations.} Let $X$ be a scheme, and $\sE$ a quasi-coherent sheaf on $X$. For a point $x \in X$, we denote $\sE_x$ the stalk of $\sE$ at $x$, and $\sE|_x : = \sE_x/\fom_{x} \sE_x = \sE_x \otimes_{\sO_{X,x}} \kappa(x)$ the fiber of $\sE$ over $x$, where $\fom_x$ is the maximal ideal of $\sO_{X,x}$, $\kappa(x)$ denote the residue field of $x$. % If $X$ is integral, then the rank of $\sE$ is understood as the dimension of the fiber of $\sE$ at the generic point. 
In this section, for a morphism $\phi \colon T \to X$, $\sE_T: = \phi^* \sE$ denotes the {\em underived} pullback of $\sE$.

\subsubsection*{Polynomial depth} For an ideal $I \subset R$ and a $R$-module $M$ such that $I M \ne M$, $\depth_R(I, M)$ is defined as the supremum of the lengths of all $M$-regular sequences contained in $I$. The {\em polynomial depth of $M$ with respect to $I$} is defined by (see e.g. \cite{No,  Gl, IR, AT}): 
	$$\pdepth_R(I; M) : = \lim_{t \to \infty} \depth_{R[x_1,\ldots,x_t]} (I \cdot R[x_1,\ldots, x_t]; M[x_1, \ldots, x_t]).$$
It is often convenient to set $\pdepth_R(I;M) = + \infty$ if $I M = M$.
In the case when $(R,\fom)$ is a local ring, and $I = \fom$, we denote $\pdepth_R(M) = \pdepth_R(\fom; M)$. 

Let $X$ be a scheme, let $\sI \subset \sO_X$ be a finite type quasi-coherent ideal, and let $0 \ne \sE$ be a quasi-coherent $\sO_X$-module. The polynomial depth of $\sE$ with respect to $\sI$ is defined by
	$$\pdepth_X(\sI; \sE) : = \inf_{x \in X}  \pdepth_{\sO_{X,x}} (\sI_x; \sE_x) .$$
If $X=\Spec R$, $\sI = \tilde{I}$ for a finitely generated ideal $I \subset R$, and $\sE = \tilde{M}$ for a $R$-module $M$. Then by \cite[Thm. 7.1.11]{Gl}, $\pdepth_X(\sI; \sE)  = \pdepth_R(I; M)$. This justifies our definition. 

In the case when $\sI_Z$ is the ideal of a closed subscheme $Z \subset X$, and $\sF = \sO_X$, 
	$$\pdepth(Z,X) : = \pdepth_{X}(\sI_Z; \sO_X)$$
is called the polynomial depth of the closed subscheme $Z$ on $X$. Notice that if $\pdepth_X(\sI; \sE)  \ge r$, where $r\ge0$ is an integer, then by definition $ \pdepth_{\sO_{X,x}} (\sI_x; \sE_x) \ge r$ for all $x \in X$.

\subsubsection*{Weak associated points} A point $x$ of a scheme $X$ is called a {\em weak associated point}, denoted by $x \in {\rm WeakAss}(X)$, if there is an element $a \in \sO_{X,x}$ whose annihilator $\Ann_{\sO_{X,x}}(a)$ has radical equal to $\fom_x$ (the maximal ideal of the local ring $\sO_{X,x}$). In particular, any generic point of an irreducible component of $X$ is a weak associated point of $X$. If $X$ is locally noetherian, then weak associated points coincide with associated points: ${\rm WeakAss}(X) = {\rm Ass}(X)$. See \cite{IR}, or \cite[\href{https://stacks.math.columbia.edu/tag/056K}{Tag 056K}]{stacks-project} for more details about weak associated points.

\subsection{Quot schemes of locally free quotients} \label{sec:Quot}

The Quot schemes, introduced by Grothendieck \cite{Gro}, further developed by Mumford and by Altman--Kleiman \cite{AK}, plays an important role in modern algebraic geometry, especially for deformation theory and moduli problems. See Nitsure \cite{Nit} for a nice survey of the construction, and Ciocan-Fontanine and Kapranov's \cite{C-FK} for a generalisation in the setting of derived algebraic geometry. 

In this paper, we will be mainly concerned with Quot schems of {\em locally free quotients}. (Our to-be-defined $\foQuot_{X,d}(\sE)$ is the functor $\foQuot_{\sE/X/X}^{d, \sO_X}$ of \cite{Nit}.)

\begin{definition}\label{def:Quot} Let $X$ be scheme, $\sE$ a quasi-coherent sheaf on $X$, and let $d>0$ be an integer. The contravariant functor $\foQuot_{X, d}(\sE) \colon (\mathbf{Sch}/X)^{\rm op} \to \mathbf{Sets}$ is defined as follows: 
	\begin{itemize}[leftmargin=1 em]
		\item For any $X$-scheme $\phi \colon T \to X$, $\foQuot_{X,d}(\sE)(T)$ is the set of equivalence classes of quotients $q' \colon \sE_T \twoheadrightarrow \sP$, where $\sP$ is locally free of rank $d$; %  and $\sE_T := \phi^* \sE$ is {\em underived} pullback of $\sE$. 
		Two quotients $q \colon \sE_T  \twoheadrightarrow \sP$ and $q \colon \sE_T  \twoheadrightarrow \sP'$ are said to be equivalent if $\ker(q) = \ker(q')$.
		\item For any morphism $g \colon T' \to T$ over $X$, $\foQuot_{X,d}(\sE)(g) \colon \foQuot_{X,d}(\sE)(T) \to \foQuot_{X,d}(\sE)(T')$ is the map which sends an epimorphism $\sE_{T} \twoheadrightarrow \sP$ to the pullback $\sE_{T'} \twoheadrightarrow g^*\sP$, where $g^*$ is the usual {\em underived} pullback. (This is well defined since the pullback $g^*$ is right exact.)
	\end{itemize}
\end{definition}
We will write $\Quot_{d}(\sE) = \Quot_{X,d}(\sE)$ if the base scheme $X$ is clear from context. We now summarise Grothendieck's related results in the following theorem:

\begin{theorem}[Grothendieck] \label{thm:Quot} For any scheme $X$ and any quasi-coherent $\sO_X$-module $\sE$, the functor $\foQuot_{X, d}(\sE)$ is representable by a separated $X$-scheme $\pi \colon \Quot_{X,d}(\sE) \to X$. There is a {\em tautological quotient bundle} $\sQ =\sQ_d$ which is locally free of rank $d$, and a {\em tautological quotient} map $\pi^* \sE \twoheadrightarrow \sQ_d$, such that for any $X$-scheme $T \to X$ and any rank $d$ locally free quotient $p_T \colon \sE_T \twoheadrightarrow \sP$, there is a unique map $f \colon T \to \Quot_{X,d}(\sE)$ such that $p_T$ is the pullback of the tautological quotient. If $\sE$ is of finite type (resp. of finite presentation), then $\pi$ is of finite type (resp. of finite presentation). Furthermore, the following holds:
\begin{enumerate}[leftmargin=*]
	\item \label{thm:Quot-1} 
	{\em (The formation of Quot schemes commutes with base change.)} For any map of schemes $g \colon X' \to X$, $\Quot_{X',d}(g^*\sE)$ is canonically identified with $\Quot_{X,d}(\sE)\times_X X'$, with tautological bundle given by the inverse image of $\sQ$.
	\item \label{thm:Quot-2}
	For any epimorphism $\psi \colon \sE \twoheadrightarrow \sE'$ of quasi-coherent $\sO_X$-modules, there is a caonical closed immersion $i_{\psi} \colon \Quot_{X,d}(\sE') \hookrightarrow \Quot_{X,d}(\sE)$ defined by sending the epimorphism $\sE_T \twoheadrightarrow \sP$ to the composition $\sE_T' \twoheadrightarrow \sE_T \twoheadrightarrow \sP$ for all $T \to X$. The tautological bundle of $\Quot_{X,d}(\sE')$ is canonically identified with $i_{\psi} ^* \sQ$.
	\item \label{thm:Quot-3}
	There is a canonical closed immersion, called {\em Pl{\"u}cker embedding}, $\varpi_{\sE} \colon \Quot_{X,d}(\sE) \hookrightarrow \PP(\bigwedge^d \sE)$ defined by sending $\sE \twoheadrightarrow \sQ_d$ to $\wedge^d \sE \twoheadrightarrow \wedge^d \sQ_d$, such that $\varpi_{\sE}^* \sO_{\PP(\bigwedge^d \sE)}(1) \simeq \wedge^d \sQ_d$. We denote the line bundle $\wedge^d \sQ_d$ by $\sO_{\Quot_{X,d}}(1)$. If $\sE$ is of finite presentation, then $\varpi_{\sE}$ is finite presented. For any epimorphism $\psi \colon \sE \twoheadrightarrow \sE'$, the Pl{\"u}cker embeddings $\varpi_{\sE}$, $\varpi_{\sE'}$ commute with the closed immersions defined by $\psi$ in \eqref{thm:Quot-2}.
\end{enumerate} 
\end{theorem}
\begin{proof} The statement before ``furthermore" part is a combination of \cite[Thm. 9.7.4 \& Prop. 9.7.7]{EGAI}; For the ``furthermore" part, \eqref{thm:Quot-1} is \cite[Prop. 9.7.6]{EGAI}, \eqref{thm:Quot-2} is \cite[Prop. 9.7.8]{EGAI}, \eqref{thm:Quot-3} is \cite[Prop 9.8.3 \& Prop. 9.8.4]{EGAI}. \end{proof}

 \begin{example}[Projectivization] \label{ex:proj}  The {\em projectivization} of $\sE$, denoted by $\PP(\sE) = \PP_X(\sE) :=\Proj \Sym_{\sO_X}^\bullet \sE$, is the Quot scheme of rank $d=1$ locally free quotients: $\PP(\sE) = \Quot_{X,1}(\sE)$. The reason is that, for any $X$-scheme $\phi \colon T \to X$, to give a $X$-morphism $\phi \colon T \to \PP_X(\sE)$ is equivalent to give a line bundle $\sL$ over $T$ together with an epimorphic $\sO_T$-module map $\phi^* \sE \twoheadrightarrow \sL$. If $\sE$ is locally free, we also use the notation $\PP_{\rm sub}(\sE) : = \PP(\sE^\vee)$.
\end{example}

\begin{example}[Grassmannian bundles]\label{ex:Grass} If $\sE$ is locally free of rank $r$ over $X$, then for any integer $1 \le d \le r$, the {\em rank $d$ Grassmannian bundle of $\sE$ over $X$} is defined to be 
	$$\pi \colon \Gr_d(\sE) : = \Gr(\sE^\vee, d) : =  \Quot_{d}(\sE^\vee) \to X,$$
where $\sE^\vee: = \sHom_{\sO_X}(\sE,\sO_X).$ Notice in the case $d=1$, $\Gr_1(\sE) = \PP_{\rm sub}(\sE) := \PP(\sE^\vee)$. The scheme $\Gr_d(\sE)$ by definition parametrises rank $d$ sub-bundles of $\sE$, or equivalently rank $d$ locally free quotients of $\sE^\vee$. Denote by $\pi^* \sE^\vee \twoheadrightarrow \sQ$ the universal quotient for $\Gr_d(\sE) = \Quot_d(\sE^\vee)$, then the vector subbundle $\shU := \sQ^\vee \subseteq \pi^* \sE$ is called the {\em universal subbundle} of $\sE$ of rank $d$, and $\shQ: = \pi^* \sE/ \shU$ is called the {\em universal quotient bundle} of $\sE$ of rank $r - d$. There is thus a {\em tautological  exact sequence} of vector bundles over $\Gr_d(\sE)$:
	$$0 \to \shU \to \pi^* \sE \to \shQ \to 0.$$
Notice that if $X = \Spec \kk$ for a field $\kk$, and $\sE = V$ is a $\kk$-vector space of rank $r$, then the notation $\Gr_d(\sE)$ agrees with the notation $\Gr_d(V) = \Gr_d(r)$ for the {\em (usual) Grassmannian} varieties of $d$-dimension $\kk$-linear subspaces of $V$ of  \cite{Ful}.
\end{example}

\begin{remark}[Fibers of $\Quot_{d}$ are Grassmannians]\label{rmk:Quot:fiber}
Let $K$ be any field, and consider a $K$-point $u \colon \Spec K \to X$ which corresponds to a point $x \in X$ and a field extension $\kappa(x) \to K$ such that $x = u(\Spec K)$. Then it follows from \eqref{thm:Quot-2} that the fiber of $\Quot_{X,d}(\sE) \to X$ over $\Spec K$ is the (usual) Grassmannian variety of $d$-dimension $K$-linear subspaces. More precisely, $\Quot_{X,d}(\sE) \times_X \Spec  K = \Gr_d((\sE|_x \otimes_{\kappa(x)} K)^*)$, where $\sE|_x = \sE_x / \fom_{x} \sE_x$ is the fiber of $\sE$ over $x$, and $(\blank)^* = \Hom_K(\blank, K)$ is taking the dual of $K$-vector spaces. 
\end{remark}

\subsection{Degeneracy loci} \label{sec:deg} The theory of degeneracy loci has been  studied in the context of manifolds and algebraic varieties; see \cite{FP, Ful, Laz04, ACGH}. This theory extends to schemes as follows. Let $R$ be a commutative ring, $\varphi \colon F \to E$ a map of free modules. For $j \ge 0$, define $I_{j}(\varphi)$ to be image of the map $\wedge^j F\otimes \wedge^j E^\vee \to R$ induced by $\wedge^j \varphi \colon \wedge^j F \to \wedge^j E$. If we choose bases for $F$ and $E$, then $\varphi$ is a matrix with coefficients in $R$, and $I_{j}(\varphi)$ is the ideal generated by $j \times j$-minors of $\varphi$. By convention we set $I_j(\varphi) = R$ if $j \le 0$. Let $M$ be a finite type $R$-module, and choose a presentation $F \xrightarrow{\varphi} E \to M \to 0$, where $E$ is a finite free module of rank $n$. Then the {\em $i$-th Fitting ideal} is defined as $\Fitt_i(M) := I_{n - i}(\varphi)$; Fitting lemma \cite[Cor.-Def. 20.4]{Ei} states that the Fitting ideal is independent of the choice of a presentation. By convention, we set $\Fitt_{j}(M) =0$ for $j<0$. For any $0 \ne f \in R$, $\Fitt_j(M)_f = \Fitt_j(M_f)$ and $I_j(\varphi)_f = I_j(\varphi_f)$. Hence these definitions extend to schemes. More precisely, for a scheme $X$ and a finite type quasi-coherent $\sO_X$-module $\sG$, the $j$-th Fitting ideal $\Fitt_j(\sG)$ of $\sG$ is defined as follows:
 for any affine open $U = \Spec R \subset X$, $\Fitt_j(\sG)|_U := \Fitt_j(\Gamma(U,\sG))^{\sim}$. For an $\sO_X$-module morphism $\varphi \colon \sF \to \sE$ of locally free sheaves, where $\sE$ is of finite type, the ideal $I_{j}(\varphi)$ is defined to be the image of $\bigwedge^j \sF \otimes \bigwedge^j \sE^\vee \to \sO_X$.

% Definition of degeneracy loci
\begin{definition}
\begin{enumerate}[leftmargin=*]
	\item  Let $X$ be a scheme, $\sG$ a finite type quasi-coherent $\sO_X$-module, $r \ge 0$ an integer. The {\em degeneracy locus of $\sG$ of rank $\ge r$}, denoted by $X^{\ge r}(\sG)$, is the closed subscheme defined by the Fitting ideal $\Fitt_{r-1}(\sG)$. By convention, $X^{\ge r}(\sG) = X$ if $r \le 0$.
	\item Let $X$ be a scheme, $\varphi: \sF \to \sE$ an $\sO_X$-module map of locally free sheave on $X$, where $\sE$ is of finite type, and let $\ell \ge 0$ be an integer. The {\em degeneracy locus of $\varphi$ of rank $\le \ell$}, denoted by $D_{\ell}(\varphi)$, is the closed subscheme defined by the ideal $I_{\ell+1}(\varphi)$.
\end{enumerate}
\end{definition}
By definition, if $\sF \xrightarrow{\varphi} \sE \to \sG \to 0$ is a presentation of $\sG$, where $\sE$ is a finite locally free of rank $n$, then there is a canonical identification $X^{\ge r}(\sG) = D_{n - r}(\varphi)$. 

% Lemma: degeneracy loci
\begin{lemma}[{See also \cite[\href{https://stacks.math.columbia.edu/tag/05P8}{Tag 05P8}]{stacks-project}}]\label{lem:degloci} Let $X$ be a scheme, and let $\sG$ be a finite type quasi-coherent $\sO_X$-module. Then there is a locally finite sequence of closed subschemes
	$$X = X^{\ge 0}(\sG) \supseteq X^{\ge 1}(\sG) \supseteq X^{\ge 2}(\sG) \supseteq \ldots$$
such that $X^{\ge 1}(\sG) \supseteq \Supp \sG$ and $X^{\ge 1}(\sG)_{\rm red} = (\Supp \sG)_{\rm red}$, where $\Supp \sG$ is the scheme-theoretic support of $\sG$, and $(\blank)_{\rm red}$ denotes the reduced scheme. Furthermore:
	\begin{enumerate}[leftmargin=*]
		\item \label{lem:degloci-1}
		{\em (The formation of degeneracy loci commutes with base change.)}  For any $X' \to X$, there is a canonical identification $X'^{\ge r} (\sG_{X'}) = X'^{\ge r}(\sG) \times_X X'$;
		\item \label{lem:degloci-2}
		A map $T \to X$ factors through $(X \backslash X^{\ge r}(\sG) ) \subseteq  X$ iff $\sG_T$ can be locally generated by $\le (r-1)$ elements. In particular, the underlying set $|X^{\ge r}(\sG)| $ of $X^{\ge r}(\sG)$ satisfies
		\begin{align*}
			|X^{\ge r}(\sG)| & = \{x \in X \mid  \dim_{\kappa(x)} \sG|_x \ge r\} \\
				&= \{x \in X \mid  \text{$\sG_x$ cannot be generated by $r$ elements} \}.
		\end{align*}
		\item \label{lem:degloci-3}
		For any map $T \to X$, the pullback $\sG_T$ is locally free of constant rank $r$ iff $T \to X$ factors through $T \to X^{\ge r}(\sG) \, \backslash X^{\ge r+1}(\sG) \hookrightarrow X$. 
	\end{enumerate}
\end{lemma}
\begin{proof} The sequence of closed subschemes follow from the the increasing sequence of the Fitting ideals $\Fitt_{-1}(\sG) \subset \Fitt_{0}(\sG) \subset \Fitt_{1}(\sG) \subset \ldots$. Since all the statements are local, we may assume $X = \Spec R$, and $\sG=\tilde{M}$ for a finite type $R$-module $M$. By \cite[Prop. 20.7]{Ei}, if $M$ is generated by $n$ elements, then $(\Ann M)^n \subseteq \Fitt_0 M \subseteq \Ann M$, where $\Ann(M)$ is the annihilator of $M$. This proves the stated relationship between the subscheme $X^{\ge 1}(\sG)$ and $\Supp \sG$. 
For the ``furthermore" part: the statement \eqref{lem:degloci-1} follows from $\Fitt_j(\sG_T) = \phi^{-1}(\Fitt_j \sG) \cdot \sO_T$, which is \cite[Cor. 20.5]{Ei}. For statement \eqref{lem:degloci-2}: by \cite[Prop. 20.6]{Ei}, for any $x \in X$, $x \not\in V(\Fitt_j(\sG))$ iff $\sG_x$ can be generated by $j$ elements; together with \eqref{lem:degloci-1} this implies \eqref{lem:degloci-2}. For \eqref{lem:degloci-3}: from \cite[Prop. 20.8]{Ei} , $\sG_T$ is locally free of constant rank $r$ iff $\Fitt_{r-1} \sG_T = 0$ and $\Fitt_{r} \sG_T = \sO_T$, together with \eqref{lem:degloci-1}, this proves \eqref{lem:degloci-3}. (See also \cite[\href{https://stacks.math.columbia.edu/tag/05P8}{Tag 05P8}]{stacks-project} for the proof of \eqref{lem:degloci-2} and \eqref{lem:degloci-3}.)
\end{proof}

One can easily translate the theorem into an analogous theorem about the degeneracy loci $D_{\ell}(\varphi)$ of an $\sO_X$-module map $\varphi \colon \sF \to \sE$; we leave it to the readers.

\begin{corollary} \label{cor:Quot:degloci} Let $X$ be a scheme and $\sG$ a finite type quasi-coherent $\sO_X$-module. For any given pair of integers $d \ge 0$ and $r \ge 0$, denote $\pi \colon \Quot_{X,d}(\sG) \to X$ the Quot scheme and $\mathring{X}^{r} (\sG): = X^{\ge r}(\sG) \, \backslash X^{\ge r+1}(\sG)$ the degeneracy locus. Then $\pi$ factorises through 
	$\Quot_{X,d}(\sG) \to X^{\ge d}(\sG) \subseteq X.$
If $\mathring{X}^{d}(\sG) \ne \emptyset$, then $\pi$ induces an isomorphism of schemes $\pi |_{\pi^{-1}(\mathring{X}^{d}(\sG))} \colon \Quot_{X,d}(\sG)|_{\mathring{X}^{d}(\sG) } \simeq \mathring{X}^{d}(\sG)$. If $r > d$, then $\sG|_{\mathring{X}^{r}(\sG)}$ is a locally free sheaf of the rank $r$, and the restriction 
$\pi |_{\pi^{-1}(\mathring{X}^{r}(\sG) )} \colon \Quot_{X,d}(\sG) |_{\mathring{X}^{r}(\sG) }= \Gr_d((\sG|_{\mathring{X}^{r}(\sG)})^\vee) \to \mathring{X}^{r}(\sG) $ is the rank $d$ Grassmannian bundle of the rank $r$ locally free sheaf $(\sG|_{\mathring{X}^{r}(\sG)})^\vee$.
\end{corollary}
\begin{proof} Let $\pi^* \sG \to \sQ$ denote the tautological quotient map, where $\sQ=\sQ_d$ is the universal locally free sheaf of rank $d$. Then over each point $z \in \Quot_{X,d}(\sG)$, $\sQ_z \simeq \sO_{z}^{\oplus d}$ is projective $\sO_z$-module, hence there is a splitting $(\pr^* \sG)_{z} \simeq \sO_{z}^{\oplus d} \oplus \sG'$. Therefore $\Fitt_{d-1}((\pr^* \sG)_{z}) = \Fitt_{d-1}(\pr^* \sG)_{z} =0$. This implies $\pi^{-1} \Fitt_{d-1}(\sG) \cdot \sO_{\Quot_d(\sG)} = \Fitt_{d-1}(\pr^* \sG) = 0$, therefore we obtain that $\pi$ factorizes through the closed immersion $X^{\ge d}(\sG) \subseteq X$. If $\mathring{X}^{d}(\sG) \ne \emptyset$, then since for any $T \to X$, the set of $T$-points $\Quot_{X,d}(\sG)|_{\mathring{X}^{d}(\sG)}(T)$ is the set of locally free quotients $\varphi \colon \sE_T \to \sP$ where $\sE_T$ and $\sP$ are both locally free of rank $d$. Hence $\varphi$ is necessarily an isomorphism. Therefore by Lem. \ref{lem:degloci}, the projection $\pi |_{\pi^{-1}(\mathring{X}^{d}(\sG))} \colon \Quot_{X,d}(\sG)|_{\mathring{X}^{d}(\sG) } \simeq \mathring{X}^{d}(\sG)$ is an isomorphism of functors, hence induces an isomorphism of schemes. The last statement is proved similarly, in view of Thm. \ref{thm:Quot} and Lem. \ref{lem:degloci}.
\end{proof}

\begin{definition}
Let $\sG$ be a finite presented quasi-coherent sheaf, with a local presentation $\sF \xrightarrow{\varphi} \sE \to \sG$, where $\sF$ and $\sE$ are finite locally free of ranks $m$ and $n$. Let $r \ge 0$ be an integer and set $\ell = n - r$. Then the {\em expected codimension} of $X^{\ge r}(\sG)=D_{\ell}(\varphi)$ is:
	$${\rm exp.codim}_X \, D_{\ell}(\varphi) = {\rm exp.codim}_X \, X^{\ge r}(\sG) := (m -\ell) (n - \ell) = (m - n + r) r.$$
\end{definition}

\begin{remark} Recall the following classical results regarding expected dimensions: 
\begin{enumerate}[leftmargin=*]
	\item (Macaulay--Eagon--Northcott) If $X$ is locally noetherian, equidimensional, and $X^{\ge r}(\sG) \ne \emptyset$, then for any irreducible component $Z$ of $X^{\ge r}(\sG)$, 
			$$\codim_X(Z) \le  {\rm exp.codim}_X \, X^{\ge r}(\sG) : =  (m - n + r) r.$$
	\item (Hochster--Eagon) If $X$ is Cohen-Macaulay, and the locus $X^{\ge r}(\sG) \ne \emptyset$ achieves the expected codimension (i.e. it is equidimensional, and for any irreducible component $Z \subseteq X^{\ge r}(\sG)$, $\codim_X(Z) = (m - n+ r) r$), then $X^{\ge r}(\sG)$ is Cohen-Macaulay.
\end{enumerate}
See, for example, \cite[Ex. 10.9, \S 18.5, Thm. 18.18]{Ei} for these statements in affine cases.
\end{remark}

\subsection{Quasi-coherent sheaves of finite homological dimensions}
Let $X$ be a scheme and $\sG$ a finite type quasi-coherent $\sO_X$-module. If $X$ is integral, then $\sG$ has a well-defined rank, namely the dimension of $\sG$ over the generic point. In general, the rank might not be well-defined; however, it is the case for a class of quasi-coherent sheaves.

\begin{definition} \label{def:hd} A finite type quasi-coherent $\sO_X$-module $\sG$ over a scheme $X$ is said to have {\em finite homological dimension} (in the strong senes) if there exists an integer $n \ge 0$ such that locally over $X$, $\sG$ admits a resolution by {\em finite} locally free sheaves of length $n$. The smallest number $n\ge 0$ satisfying the above condition is called the (strong) {\em homological dimension} of $\sG$, denoted by ${\rm hd}_X(\sG)$. By convention, ${\rm hd}_X(\sG) \le 0$ iff $\sG$ is finite locally free.
\end{definition}

\begin{remark} Notice that $\sG$ has homological dimension $n$ in the sense of Def. \ref{def:hd} is equivalent to it has perfect amplitude $[-n,0]$ in \cite[p. 121, Def. 4.7]{SGA}. Hence ``$\sG$ has perfect dimension $n$" or ``$\sG$ has finite perfect dimension" would be a more precise terminology. However, as ``perfect dimension" is a far less used term than ``homological dimension", we choose to use the latter in this paper.
\end{remark}

\begin{remark} If $X$ is a regular noetherian scheme of finite Krull dimension, then by Serre's theorem, any finite type quasi-coherent $\sO_X$-module $\sG$ has finite homological dimension, with ${\rm hd}_X(\sG)$ bounded above by the dimension of $X$.
\end{remark}

\begin{remark} If $X$ is noetherian, then our definition of ${\rm hd}_X(\sG)$ agrees with the {\em projective dimension} ${\rm pd}_X(\sG) = \sup_{x \in X} {\rm pd}_{\sO_{X,x}} (\sG_x)$ defined in \cite[IV, (17.2.14)]{EGA}. However, in the non-noetherian case, for example in the affine case $X = \Spec R$, $\sG = \tilde{M}$, where $M$ is a $R$-module, and $R$ is not necessarily noetherian, then ${\rm hd}_X(\sG)$ corresponds to the {\em restricted} projective dimension ${\rm pd}^*_R(M)$ defined in \cite[\S 3.4]{No} (which is defined as the minimal length of a supplementatble projective resolution of $M$) rather than the usual ${\rm pd}_R(M)$. In general, ${\rm pd}_R(M) \le {\rm pd}^*_R(M)$, and the inequality could be strict;  however, by \cite[Ch. 3, Lem. 3]{No}, if ${\rm pd}^*_R(M) < \infty$, then ${\rm pd}^*_R(M) = {\rm pd}_R(M)$. If $R$ is noetherian and $M$ is finitely generated, then there is no difference between ${\rm pd}^*_R(M)$ and  ${\rm pd}_R(M)$.
\end{remark}

\begin{definition}[Rank] Let $X$ be a scheme and $\sG$ a quasi-coherent $\sO_X$-module of finite homological dimension (in the sense of Def. \ref{def:hd}). Let $U \subset X$ be an open subset over which $\sG$ admits a finite locally free resolution $0 \to \sF_n \to \sF_{n-1} \to \ldots \to \sF_1 \to \sF_0 \to \sG$. Then for any point $x\in U$, the localization $\sF_{i, x}$ is a free $\sO_{X,x}$-module of finite rank. By  \cite[\S 4.4 Thm. 19]{No}, the function $|X| \to \ZZ$,
	$$x \mapsto \sum_{i=0}^n (-1)^i \rank_{\sO_{X,x}} \sF_{i , x} \in \ZZ $$
is a well-defined {\em locally constant} function in the Zariski topology of $X$, and is independent of the choice a local resolution $\sF_\bullet$ of $\sG$. We call this {\em locally constant function} the {\em rank} of $\sG$, and denote it by $\rank \sG$. If $X$ is connected, then $\rank \sG$ is a constant number.
\end{definition}

\begin{lemma}\label{lem:rank} Let $X$ be a connected scheme, $0 \ne \sG$ a quasi-coherent $\sO_X$-module of finite homological dimension, and let $\delta = \rank \sG$. Then:
	\begin{enumerate}[leftmargin=*]
		\item \label{lem:rank-1}
		{\em (Positivity of rank)}  $\delta \ge 0$; $\delta=0$ iff $\, \Fitt_0(\sG) \ne 0$;
		\item \label{lem:rank-2}
		If  $r \le \delta$, then $X^{\ge r}(\sG) = X$. Furthermore, $X^{\ge \delta+1}(\sG) \ne X$; 
		\item \label{lem:rank-3}
		For any weak associated point $x \in {\rm WeakAss}(X)$ (in particular, for any generic point $x$ of an irreducible component of $X$), $\sG_x$ is a {\em free} $\sO_{X,x}$-module of rank $\delta$.
	\end{enumerate}
\end{lemma}

\begin{proof} \eqref{lem:rank-1} follows from \cite[\S 4.3 Cor. 1 \& Cor. 3]{No}, and \eqref{lem:rank-2} follows from \cite[\S 4.3 Cor. 2]{No}. For \eqref{lem:rank-3}, let $x \in {\rm WeakAss}(X)$, and choose a finite locally free resolution $\sF_{\bullet} = (\sF_{i})_{i\in[0,n]}\to \sG$ in an open neighbourhood of $x$. Since $x$ is a weak associated point, $\pdepth_{\sO_{X,x}}(\sO_{X,x}) = 0$ (see e.g. \cite[Lem. 2.8]{HM}).
Then by the non-noetherian {\em Auslander--Buchsbaum--Hochster theorem} of Northcott \cite[Ch. 6, Thm. 2]{No} applied to $R= \sO_{X,x}$, $M = \sG_x$, we have 
	$$\pdepth_{\sO_{X,x}}(\sG_x) + {\rm pd}_{\sO_{X,x}}(\sG_x) = \pdepth_{\sO_{X,x}}(\sO_{X,x})=0.$$ 
Hence ${\rm pd}_{\sO_{X,x}}(\sG_x) =0$, i.e. $\sG_x$ is a free $\sO_{X,x}$-module. Then $0 \to \sF_{n, x} \to \ldots \to \sF_{0, x} \to \sG_x\to 0$ is an exact sequence of free modules, therefore by \cite[\S 3.5, Thm. 19]{No} $\rank \sG_x = \sum_{i=0}^n (-1)^i \rank_{\sO_{X,x}} \sF_{i, x}$, and the latter is by definition $\rank \sG =\delta$.
\end{proof}

\begin{remark} Our proof of \eqref{lem:rank-3} works for any {\em attached prime} in the sense of \cite{No} (which is equivalent to the notion of a {\em strong Krull prime} of \cite{IR}). The readers are referred to \cite{IR, AT} for various notions of associated primes and grades in the non-noetherian case. \end{remark}

Recall the determinant of a locally free sheaf $\sF$ of finite rank over $X$ is the line bundle $\det \sF : = \bigwedge^{\rank \sF} \sF$. This definition could be generalized as follows:
 
\begin{definition}[Determinant] \label{def:det}  Let $X$ be a scheme and $\sG$ a quasi-coherent $\sO_X$-module of finite homological dimension. The {\em determinant of $\sG$} is the line bundle $\det \sG$ defined as follows: if $U \subset X$ be an open subset such that $\sG$ admits a finite locally free resolution $0 \to \sF_n \to \ldots  \to \sF_0 \to \sG$, then $\det \sG$ over $U$ is defined by the formula:
	$$\det \sG : = \bigotimes_{i=1}^n(\det \sF_i)^{\otimes (-1)^i}.$$
The line bundle on the right hand side is independent of the choice of the resolution $\sF_\bullet$ of $\sG$, hence we obtain a well-defined line bundle $\det \sG \in \Pic(X)$.
\end{definition}

\subsubsection{Quasi-coherent sheaves of homological dimension $\le 1$}
Let $X$ be a connected scheme and $\sG$ a quasi-coherent $\sO_X$-module of homological dimension $\le 1$ with $\rank \sG = \delta$. Then by Lem. \ref{lem:rank}, we have $\delta \ge 0$; and $\delta =0$ iff $X^{\ge 1}(\sG) \ne X$ iff $\Fitt_0(\sG)$ is a principal ideal generated by a non-zerodivisor (see \cite[\S 3.5 Theorem 21]{No}). The case ${\rm hd}_X(\sG)=1$ is especially interesting, since the ``interesting" piece of information of the ``derived dual" of $\sG$ is reflected by another sheaf 
	$$\sK : = \sExt^1_{\sO_X}(\sG, \sO_X),$$
which, under mild assumptions on $\sG$, also has finite homological dimension.
	
\begin{lemma}\label{lem:hdK} Let $X$ be a connected scheme, and let $\sG$ be a quasi-coherent $\sO_X$-module of homological dimension $\le 1$ with $\rank \sG = \delta \ge 0$, and denote $\sK= \sExt^1_{\sO_X}(\sG, \sO_X)$. Then 
\begin{enumerate}[leftmargin=*]
	\item \label{lem:hdK-1}
	$\Fitt_{\delta+j}(\sG) = \Fitt_{j}(\sK)$ and $X^{\ge \delta+j}(\sG) = X^{\ge j}(\sK)$ for all $j \in \ZZ$;
	\item \label{lem:hdK-2}
	$X = X^{\ge 0}(\sK) = X^{\ge \delta}(\sG)$. Denote by $Z: =  X^{\ge 1}(\sK)=X^{\ge \delta+1}(\sG) \subsetneq X$ the {\em first degeneracy locus} of $\sG$ and $\sK$ (cf. Lem. \ref{lem:rank} \eqref{lem:rank-2}), and assume further that: 
	\begin{equation}\label{eqn:pdepthZ}
	\pdepth(Z, X) \ge \delta+1.
	\end{equation}
	Then $\sK$ has homological dimension $\le \delta+1$.
	\item \label{lem:hdK-3}
	If $X$ is Cohen--Macaulay, then \eqref{eqn:pdepthZ} holds iff $\codim_X(Z) =\delta+1$ (the expected codimension). If any of these conditions holds, then $Z$ is also  Cohen--Macaulay, and $\sK$ is a maximal Cohen--Macaulay $\sO_Z$-module on $Z$. 
\end{enumerate}
\end{lemma}

\begin{proof} Since the problem is local, we may assume $\sG$ admits a presentation $0 \to \sW \xrightarrow{\sigma} \sV \to \sG$, where $\sW$ and $\sV$ are finite locally free of rank $m$ and $n$. Then $\delta = n-m \ge 0$, and $\sK$ fits into an exact sequence: $\sV^\vee \xrightarrow{\sigma^\vee} \sW^\vee \to \sK \to 0$. Since for any $j \in \ZZ$, $I_{j}(\sigma) = I_{j}(\sigma^\vee)$, \eqref{lem:hdK-1} immediately follows. For \eqref{lem:hdK-2}, we consider the {\em Buchsbaum--Rim complex}:
	$$\sF_\bullet \colon \qquad 0 \to \sF_{\delta+1} \xrightarrow{\varphi_{\delta+1}} \sF_{\delta} \xrightarrow{\varphi_{\delta}}  \ldots \to \sF_{2} \xrightarrow{\varphi_{2}} \sF_{1}\xrightarrow{\varphi_{1}}  \sF_0,$$
where $\sF_0 = \sW^\vee$, $\sF_1 = \sV^\vee$, $\varphi_1 = \sigma^\vee$, and for $t= 1, 2, \ldots, \delta$, 
	$$\sF_{1+t}: = \bigwedge^{m+t} \sV^\vee \otimes \Gamma^{t-1} (\sW) \otimes \det \sW,$$
where $\Gamma^{k}(\sW)$ is the $k$-th divided power of $\sW$ (i.e., $\Gamma^{k}(\sW) = (S^k \sW^\vee)^\vee$), and $\det \sW = \bigwedge^m \sW$ is the determinant line bundle. This is the second complex in the family of Eagon--Northcott complexes, which is denoted by $\shC^{1}$ in \cite[\S A.2.6.1]{Ei} and by ${\rm EN}_1$ in \cite[\S B.2]{Laz04}; We refer the readers to these references for details of the definition. If we denote $r_k: = \sum_{i=k}^{\delta} (-1)^{i-k} \rank (\sF_i)$, the expected rank of $\varphi_k$. Then it is shown by Eisenbud \cite[Thm. A.2.10]{Ei} that $\sF_\bullet$ is indeed a complex, $I_{r_k}(\varphi_k) \subseteq I_{m}(\sigma)$ and $\sqrt{I_{r_k}(\varphi_k)} = \sqrt{I_{m}(\sigma)}$ for all $k\in [1, \delta+1]$. Therefore if (\ref{eqn:pdepthZ}) holds, which in particular implies for all $x \in X$, 
	$$\pdepth_{\sO_{X,x}}(I_{r_k}(\varphi_k)_x; \sO_{X,x})  = \pdepth_{\sO_{X,x}}(I_{m}(\sigma)_x; \sO_{X,x}) \ge \delta+1, \quad k=1,\ldots, \delta+1.$$
Then by Northcott's generalized Eisenbud--Buchsbaum criterion for exactness of complexes \cite[\S 6.4, Thm. 15]{No} (see also \cite[page 250, 7.2.3]{Gl}), $(\sF_\bullet)_x$ is exact at each point $x \in X$, and hence $\sF_\bullet$ is a locally free resolution of $\Coker(\varphi_1) = \sK$. This proves \eqref{lem:hdK-2}. For \eqref{lem:hdK-3}, notice for a noetherian scheme $X$, polynomial depth agrees with usual depth: $\pdepth(Z,X) = \depth(Z,X)$, hence \eqref{lem:hdK-3}  follows from \cite[Cor. A.2.13]{Ei}.
\end{proof}

Regarding the determinant of $\sG$ when ${\rm hd}(\sG) \le 1$, we have the following:

\begin{lemma} Let $X$ be a connected scheme, and let $\sG$ be a quasi-coherent $\sO_X$-module of homological dimension $\le 1$ with $\rank \sG = \delta \ge 0$. Let $Z: = X^{\ge \delta+1}(\sG) \subsetneq X$ be the first degeneracy locus, and denote by $\sI_Z \subseteq \sO_X$ the ideal sheaf of $Z$. Assume (\ref{eqn:pdepthZ}) holds, i.e.
	$
	\pdepth(Z, X) \ge \delta+1.
	$
(If $X$ is Cohen--Macaulay, then this is equivalent to $Z \subseteq X$ has expected codimension $\delta+1$.) Then there is a natural isomorphism of sheaves:
	$$\bigwedge^{r} \sG \otimes (\det \sG)^{-1} \xrightarrow{\sim} \sI_Z \subset \sO_X.$$
In particular, if $\delta=0$, then $\sI_Z \simeq  (\det \sG)^{-1}$ is an invertible ideal.
\end{lemma}

\begin{proof} As before we may assume $\sG$ admits a presentation $0 \to \sW \xrightarrow{\sigma} \sV \to \sG$, where $\sW$ and $\sV$ are finite locally free of rank $m$ and $n$, $\delta = n - m$. Then the first three terms of the EN complex $\shC^{0}$ of  \cite[\S A.2.6.1]{Ei} take the form:
	$$\bigwedge^{m+1} \sV^\vee \otimes \sW \otimes \det \sW \xrightarrow{\partial_1} \bigwedge^{m} \sV^\vee \otimes \det \sW \xrightarrow{\partial_0} \sO_X,$$
where the second map $\partial_0$ is given by the map $\bigwedge^m \sigma^\vee \colon \bigwedge^m \sV \to \bigwedge^m \sW =\det \sW$, hence by definition the image of $\partial_0$ is $\sI_Z$. Then by Eisenbud \cite[Thm. A.2.10]{Ei} as in previous lemma, the condition (\ref{eqn:pdepthZ}) implies that above complex is indeed {\em exact}. If $\delta=0$, this already implies $\partial_0 \colon (\det \sV)^{-1} \otimes \det \sW \xrightarrow{\sim} \sI_Z $. If $\delta \ge 1$, we have canonical isomorphisms $\bigwedge^{m+i} \sV^\vee \simeq \bigwedge^{\delta-i} \sV \otimes (\det \sV)^{-1}$ for $i=0,1$, then from the construction of $\partial_1$ (see e.g. \cite[\S A.2.10]{Ei}), through the above isomorphisms the map $\partial_1$ coincides with the multiplication map $\partial_1' \colon \bigwedge^{\delta-1} \sV \otimes \sW \to \bigwedge^{\delta} \sV$ up to tensoring with the line bundle $(\det \sV)^{-1} \otimes \det \sW = (\det \sG)^{-1}$. On the other hand, by the canonical exact sequences of  exterior products we have $\Coker \partial_1' \simeq \bigwedge^\delta \sG$, see e.g. \cite[Prop. A.2.2 (d)]{Ei}, hence the map $\partial_0$ factorises through an isomorphism $\Coker \partial_1'  \otimes (\det \sG)^{-1} = \Coker \partial_1 \xrightarrow{\sim} \sI_Z$. 
\end{proof}

%%%
\subsection{Correspondences as Quot schemes} Let $X$ be a connected scheme, let $\sG$ be a quasi-coherent $\sO_X$-module of homological dimension $\le 1$ with $\rank \sG = \delta \ge 0$, and denote $\sK= \sExt^1_{X}(\sG, \sO_X)$. Let $(d_+,d_-)$ be a pair of non-negative integers, and consider:
	$$\pi_+ \colon \shZ_+: = \Quot_{X,d_+}(\sG) \to X, \qquad \pi_- \colon \shZ_- : = \Quot_{X, d_-}(\sK) \to X.$$
Let $\pi_+^* \sG \twoheadrightarrow \sQ_+$ and $\pi_-^* \sK \twoheadrightarrow \sQ_-$ be the tautological quotients on $\shZ_+$ and $\shZ_-$, and set:
	\begin{align*}
	\sG_+ : = \Ker (\pi_+^* \sG \twoheadrightarrow \sQ_{+}), \quad \sK_+ : = \pi_+^* \sK; \qquad
	\sK_- : = \Ker (\pi_-^* \sK \twoheadrightarrow \sQ_{d_-}), \quad \sG_- : = \pi_-^* \sG.
	\end{align*}
Then Thm. \ref{thm:Quot} \eqref{thm:Quot-1} implies:
	$$\shZ_+ \times_X \shZ_- = \Quot_{\shZ_+, d_-}(\sK_+) = \Quot_{\shZ_-, d_+}(\sG_-).$$
\begin{lemma}\label{lem:hdK:iterated} In the above situation (let ``CM" stand for ``Cohen--Macaulay"):
\begin{enumerate}[leftmargin=*]
	\item Assume $\delta \ge d_+$. If ${\rm p.depth}(\pi_+^{-1}(X^{\ge \delta+1}(\sG)), \shZ_+) \ge 1$, e.g., if $X$ is CM and
		$$\codim_{X}(X^{\ge \delta + i}(\sG)) \ge i(\delta+i) - i(\delta - d_+) + 1-i^2, \qquad \forall i \ge 1.$$
	Then $\sG_+$ has homological dimension $\le 1$ on $\shZ_+$, and $\sK_+ \simeq \sExt^1_{\shZ_+}(\sG_+, \sO_{\shZ_+})$. If furthermore ${\rm p.depth}(\pi_+^{-1}(X^{\ge \delta+1}(\sG)), \shZ_+) \ge \delta - d_+ +1$, e.g., if $X$ is CM and
		$$\codim_{X}(X^{\ge \delta + i}(\sG)) \ge i(\delta+i) - (i-1)(\delta - d_+) + 1-i^2,  \qquad \forall i \ge 1.$$
	Then $\sK_+$ has homological dimension $\le \delta - d_+ + 1$ on $\shZ_+$;
	\item Assume $\delta \le d_+$. If ${\rm p.depth}(\pi_+^{-1}(X^{\ge d_+ + 1}(\sG)), \shZ_+) \ge 1$, e.g., if $X$ is CM and
		$$\codim_{X}(X^{\ge d_+ + i}(\sG)) \ge (d_+ + i)(d_+ - \delta + i) - i(d_+ - \delta) + 1 -i ^2, \qquad \forall i \ge 1.$$
	Then $\sK_+$ has homological dimension $\le 1$ on $\shZ_+$, and $\sG_+ \simeq \sExt^1_{\shZ_+}(\sK_+, \sO_{\shZ_+})$. If furthermore ${\rm p.depth}(\pi_+^{-1}(X^{\ge d_+ +1}(\sG)), \shZ_+) \ge d_+ - \delta +1$,  e.g., if $X$ is CM and 
		$$\codim_{X}(X^{\ge d_+ + i}(\sG)) \ge  (d_+ + i)(d_+ - \delta + i) - (i-1)(d_+ - \delta) + 1 -i ^2,  \qquad \forall i \ge 1.$$
	Then $\sG_+$ has homological dimension $\le d_+ - \delta + 1$ on $\shZ_+$;
	\item If ${\rm p.depth}(\pi_-^{-1}(X^{\ge d_- + \delta + 1}(\sG)), \shZ_-) \ge 1$,  e.g., if $X$ is CM and
			$$\codim_{X}(X^{\ge d_- +\delta + i}(\sG)) \ge  (d_- +i)(d_- + \delta+i) - i(d_- + \delta) + 1-i^2, \qquad \forall i \ge 1.$$
	Then  $\sG_-$ has homological dimension $\le 1$ on $\shZ_-$, and $\sK_- \simeq \sExt^1_{\shZ_-}(\sG_-, \sO_{\shZ_-})$. If furthermore ${\rm p.depth}(\pi_-^{-1}(X^{\ge d_- + \delta + 1}(\sG)), \shZ_-) \ge d_- + \delta + 1$, e.g., if $X$ is CM and
		$$\codim_{X}(X^{\ge d_- + \delta + i}(\sG)) \ge  (d_- +i)(d_- + \delta+i) - (i-1)(d_- + \delta) + 1-i^2,  \qquad \forall i \ge 1.$$	
	 Then $\sK_-$ has homological dimension $\le \delta + d_- + 1$ on $\shZ_-$.
\end{enumerate}
\end{lemma}
\begin{proof} This is a direct application of Cor. \ref{cor:Quot:degloci}, (the proof of) Lem. \ref{lem:hdK} \eqref{lem:hdK-2}, and Eisenbud--Buchsbaum's criterion \cite[\S 6.4, Thm. 15]{No}; we omit the details of the computations.
\end{proof}

%%% Sec: blowups
\subsection{Blowups as Quot schemes}
\subsubsection{Blowing up Koszul-regularly immersed centers}
For a closed subscheme $Z$ of a scheme $X$ defined by a quasi-coherent ideal $\sI$, the {\em blowup of $X$ along the center $Z$} is
	$$\pi \colon \Bl_Z X = \underline{\Proj}_{X} \bigoplus_{n \ge 0} \sI^n \to X.$$
The {\em exceptional divisor} $E = \pi^{-1} Z =  \underline{\Proj}_{Z} \bigoplus_{n \ge 0} \sI^n/\sI^{n+1}$ is an effective Cartier divisor on $X$, and the $\pi$-relative very ample line bundle $\sO_{\Bl_Z X}(1)  = \sO_{\Bl_Z X}(-E)$ is the ideal of $E$.

\begin{lemma}[Blowing up Koszul-regularly immersed centers commutes with Tor-independent base-change] \label{lem:blowup_bc} Let $i \colon Z \hookrightarrow X$ be a Koszul-regular closed immersion, let $g \colon X' \to X$ be a base-change, and denote $Z' = Z \times_X X'$. Suppose $g$ is Tor-independent with respect to $i$, then $\Bl_{Z'} X' = \Bl_Z X \times_{X} X'$ with exceptional divisor $E' = E \times_{Z} Z'$.
\end{lemma}

\begin{proof} Since blow-ups and fiber products can be computed affine-locally, we may assume $X = \Spec A$, $X' = \Spec B$, $g \colon X' \to X$ corresponds to a ring homomorphism $A \to B$, and $Z \subset X$ is defined by an ideal $I \subset A$ generated by an $A$-Koszul-regular sequence. Then $Z' \subset X'$ is defined by the ideal $I B$. Since relative Proj commutes with base change, see e.g. \cite[\href{https://stacks.math.columbia.edu/tag/01O3}{Tag 01O3}]{stacks-project}, $\Bl_{Z} X \times_X X' =\Proj_B \bigoplus_{n \ge 0} (I^n \otimes B)$. By Lem. \ref{lem:bc:CM}, $\Tor^A_i(A/I^n, B) = 0$ for all $i\ge1, n \ge 0$, therefore by tensoring the short exact sequence $0 \to I^n \to A \to A/I^n \to 0$ with $B$, we obtain that $I^n \otimes B = I^n B = (IB)^n$. Hence the lemma is proved.
\end{proof}

\begin{lemma} \label{lem:blowup:univ}
Let $X$ be a scheme, let $Z \subseteq X$ be a closed subscheme cut out by a Koszul-regular section $\sigma$ of a locally free sheaf $\sE$ on $X$ of constant rank $n \ge 1$, and set $\sG = \Coker(\sO_X \xrightarrow{\sigma} \sE)$. Then there is an isomorphism of $X$-schemes 
	$$\Bl_Z X \simeq \Quot_{X, n-1}(\sG) \to X,$$
such that $ \sO_{ \Quot_{X, n-1}}(1) = \sO_{\Bl_Z X}(1) \otimes \det \sE$.
Moreover, the map $\pi \colon \Bl_Z X \to X$ is a projective local complete intersection morphism, thus in particular perfect and proper.
\end{lemma}

\begin{proof} The problem being Zariski-local, we may assume $X = \Spec R$, $\sE = (R^{\oplus n})^{\sim}$, $\sigma$ is given by a $R$-Koszul-regular sequence $(f_1,\ldots, f_n)$. The ring homomorphism $A= \ZZ[x_1, \ldots, x_n] \to R$, $x_i \mapsto f_i$, induces a morphism of schemes $h \colon X \to \AA^n = \Spec A$, such that $Z = h^{-1}\{0\}$, where $\{0 \} \subset \AA^n$ is subscheme cut out by the ideal $(x_1,\ldots, x_n) \subset A$. By Lem. \ref{lem:bc:CM}, the base change $h$ is Tor-independent with respect to $\{0 \} \hookrightarrow \AA^n$. Denote $\sG_\ZZ : = \Coker(A \xrightarrow{(x_1,\ldots,x_n)} A^{\oplus n})^{\sim}$, then $\sG = h^* \sG_\ZZ$. However, for the inclusion $\{0 \} \hookrightarrow \AA^n$, we have $\Bl_{\{0  \} } \AA^n = \Quot_{\AA^n, n-1}(\sG_\ZZ)$, since both are represented by the subscheme of $\AA^n \times \Proj \ZZ[X_1,\ldots, X_n]$ cut out by the ideal $(x_i X_j - x_j X_i \mid 1 \le i < j \le n)$. By Lem. \ref{lem:blowup_bc}, $\Bl_Z X = \Bl_{\{0 \} } \AA^n \times_{\AA^n}  X$, and by Thm. \ref{thm:Quot} \eqref{thm:Quot-1}, $\Quot_{X, n-1}(\sG) =  \Quot_{\AA^n, n-1}(\sG_\ZZ) \times_{\AA^n} X$, hence $\Bl_Z X = \Quot_{X, n-1}(\sG)$. Finally, $\pi = q \circ \iota$ is the composition of the closed immersion $\iota \colon \Bl_Z X \hookrightarrow \PP_X(\sE^\vee)$ followed by the projection $q \colon \PP_X(\sE^\vee) \to X$, where $\iota$ is induced by $\sE^\vee \twoheadrightarrow 
\sI_Z$ via Thm. \ref{thm:Quot} \eqref{thm:Quot-2}, and it is a Koszul-regular closed immersion by Lem \ref{lem:3squares} and Lem. \ref{lem:bc:CM} \eqref{lem:bc:CM-1}. Hence $\pi$ is a projective local complete intersection morphism; in particular, it is perfect and proper. 

Finally, the comparison of $\sO(1)$'s follows from the fact that under the identification $\PP_X(\sE^\vee) = \Quot_{X,n-1}(\sE)$, if we denote $\sO_{\PP(\sE^\vee)}(1)$ the $\sO(1)$-bundle from projectivizaiton, and $\sQ_{n-1}$ the universal rank-$(n-1)$ quotient bundle from Quot scheme, then the Euler sequence $0 \to \sQ_{n-1}^\vee \to \sE^\vee \to \sO_{\PP(\sE^\vee)}(1) \to 0$ implies $\sO_{\PP(\sE^\vee)}(1)  \otimes \sO_{\Quot}(-1)  \simeq \det \sE^\vee$.
\end{proof}

\subsubsection{Blowing up along determinantal subschemes}
The following is proved in \cite{J20}.
\begin{lemma}[{\cite{J20}}]\label{lem:Quot=Bldet} Let $X$ be a connected scheme, let $\sG$ be a quasi-coherent $\sO_X$-module of homological dimension $\le 1$, and assume $\delta: = \rank \sG \ge 1$. Let $Z = X^{\ge \delta+1}(\sG) \subset X$ be the first degeneracy loci.
Consider the Quot scheme $\pi \colon \Quot_{X, \delta}(\sG) \to X$. Then $\pi^{-1}(Z) \subseteq \Quot_{X, \delta}(\sG)$ is a  locally principal closed subscheme. Furthermore:
	\begin{enumerate}[leftmargin=*]
	\item \label{lem:Quot=Bldet-1}
	If $\pi^{-1}(Z) \subseteq \Quot_{X, \delta}(\sG)$ is an effective Cartier divisor, then $\pi \colon \Quot_{X, \delta}(\sG) \to X$ is isomorphic to the blowup of $X$ along $Z$, and  
	$\sO_{ \Quot_{X, \delta}(\sG)}(1) = \sO_{\Bl_Z X}(1) \otimes \det \sG$;
	\item \label{lem:Quot=Bldet-2}
	If \eqref{lem:Quot=Bldet-1} holds and $X$ is reduced (resp. irreducible, integral), then so is $\Quot_{X, \delta}(\sG)$;
	\item \label{lem:Quot=Bldet-3}
	If $X$ is Cohen--Macaulay, and the following weak dimension conditions are satisfied:
	$$\codim_X (X^{\ge\delta +i}(\sG) ) \ge i\cdot \delta +1, \qquad \forall i \ge 1.$$
	Then the condition of \eqref{lem:Quot=Bldet-1} is satisfied, and $\Quot_{X, \delta}(\sG) = \Bl_Z X$ is also Cohen--Macaulay.
	\end{enumerate}
\end{lemma}

\subsubsection{Blowing up of determinantal subschemes along further determinantal  subschemes}
The following is a generalisation of Lem. \ref{lem:Quot=Bldet}, and is closely related to the geometry of Grassmannian flips \S \ref{sec:Grass_flips}.

\begin{lemma}\label{lem:Quotfib=Bl} Let $X$ be a connected scheme, $\sG$ a quasi-coherent $\sO_X$-module of homological dimension $\le 1$ and rank $\delta$, and denote $\sK: = \sExt^1_{\sO_X}(\sG, \sO_X)$. Let $d \ge \delta$ be an integer. Consider the degeneracy loci $Y := X^{\ge d} (\sG)$ and $Z := X^{\ge d+1} (\sG)$, and the Quot schemes:
	$$\pi_+ \colon Y_+ := \Quot_{X, d}(\sG) \to X, \qquad \pi_- \colon Y_- := \Quot_{X, d - \delta}(\sK) \to X.$$
and their fiber product 
	$$\widehat{\pi} \colon \widehat{Y} := \Quot_{X, d}(\sG) \times_X \Quot_{X, d - \delta}(\sK)  \to X .$$	 
Then by Cor. \ref{cor:Quot:degloci}, the projections $\pi_{\pm} \colon Y_{\pm} \to X$ and $\widehat{\pi} \colon \widehat{Y} \to X$ factorise through $Y_{\pm} \to Y \subseteq X$ and $\widehat{Y} \to Y\subseteq X$ respectively, and they induce isomorphisms: $Y_\pm \backslash \pi_\pm^{-1}(Z) \simeq \widehat{Y} \backslash \widehat{\pi}^{-1}(Z) \simeq Y \backslash Z$. We claim that $\widehat{\pi}^{-1}(Z) \subseteq \widehat{Y}$ is a locally principal closed subscheme. Furthermore:
	\begin{enumerate}[leftmargin=*]
	\item \label{lem:Quotfib=Bl-1}
	If $\widehat{\pi}^{-1}(Z) \subseteq \widehat{Y}$ is an effective Cartier divisor, then $\widehat{\pi} \colon \widehat{Y} \to Y \subseteq X$ is isomorphic to the blowup of $Y$ along $Z$, and the following holds:
		$$\sO_{ \Quot_{X, d}(\sG)}(1) \boxtimes \sO_{ \Quot_{X, d-\delta}(\sK)}(1)  = \sO_{\Bl_Z Y}(1) \otimes \widehat{\pi}^* \det \sG;$$
	\item \label{lem:Quotfib=Bl-2}
	If \eqref{lem:Quotfib=Bl-1} holds, and $Y$ is reduced (resp. irreducible, integral), then so is $\widehat{Y}$;
	\item \label{lem:Quotfib=Bl-3}
	If $X$ is Cohen--Macaulay, and the following weak dimension conditions hold:
	\begin{align*}
	& \codim_X (Y) = d(d-\delta);\\
	& \codim_X (X^{\ge d +i}(\sG)) \ge (d+i)(d-\delta+i) + 1 - i^2,  \qquad \forall i \ge 1.
	\end{align*}
	(Notice that $X^{\ge d +i}(\sG) $ has expected codimension $(d+i)(d-\delta+i)$ in $X$.) Then the condition of \eqref{lem:Quotfib=Bl-1} is satisfied, and $Y$, $Y_\pm$ and $\widehat{Y} = \Bl_Z Y$ are also Cohen--Macaulay.
	\end{enumerate}
\end{lemma}

\begin{proof} Since all statements are local, we may assume that there is a presentation $0 \to \sW \xrightarrow{\sigma} \sV \to \sG$, where $\rank \sW = m$, $\rank \sV = m+ \delta$. Then $X^{\ge d+i}(\sG) = D_{\ell-i}(\sigma)$ for $i \ge 0$, where $\ell: = n - d$. Consider $\GG_- = \Gr_{d-\delta}(\sW)$, $\GG_+ = \Gr_{d}(\sV)$ and the tautological sequences 
	$$0 \to \shU_- \to \sW \otimes \sO_{\GG_-} \to \shQ_- \to 0, \qquad 0 \to \shU_+ \to \sV^\vee \otimes \sO_{\GG_+} \to \shQ_+ \to 0.$$
By definition, the ideal $\sI_Z$ of $Z$ is the image of the map
	$\rho \colon \bigwedge^\ell \sV^\vee \otimes_{\sO_X} \bigwedge^\ell \sW \to \sO_X$
induced by $\bigwedge^\ell \sigma^\vee \colon \bigwedge^\ell \sV^\vee \to \bigwedge^\ell \sW$; Hence the ideal $\widehat{\pi}^{-1} \sI_Z \cdot \sO_{\widehat{Y}}$ is given by the image of $\widehat{\pi}^* (\rho)$. By Thm. \ref{thm:Quot} (see also the proof of Lem. \ref{lem:Sym}), the map $\widehat{\pi}^* \sigma$ factorises through 
	$$\widehat{\pi}^*(\sW) \twoheadrightarrow \shQ_-|_{\widehat{Y}} \xrightarrow{\widehat{\sigma}} \shQ_+^\vee|_{\widehat{Y}} \to \widehat{\pi}^* \sV.$$
Therefore $\widehat{\pi}^* (\rho)$ factorises through:
	$$\widehat{\pi}^* (\rho) \colon \widehat{\pi}^* (\bigwedge^\ell  \sV^\vee) \otimes_{\sO_{\widehat{Y}}}  \widehat{\pi}^* (\bigwedge^\ell  \sW) \twoheadrightarrow \bigwedge^\ell \shQ_+|_{\widehat{Y}}  \otimes_{\sO_{\widehat{Y}}} \bigwedge^\ell \shQ_-|_{\widehat{Y}}   \xrightarrow{\bigwedge^\ell \widehat{\sigma}} \sO_{\widehat{Y}}.$$
This shows that the ideal $\widehat{\pi}^{-1} \sI_Z \cdot \sO_{\widehat{Y}}$ is locally principal. 

Next, we show $\widehat{Y}$ enjoys the universal property of a blowup. Let $f \colon Y' \to Y \subseteq X$ be any morphism such that the dieal $f^{-1}\sI_Z \cdot \sO_{Y'} $ is invertible. Then there is a surjection:
 	$$f^* (\bigwedge^\ell \sV^\vee \otimes \bigwedge^\ell \sW) \twoheadrightarrow f^{-1}\sI_Z \cdot \sO_{Y'} \subseteq \sO_{Y'}.$$
Since $f^{-1}\sI_Z \cdot \sO_{Y'}$ is invertible, by Ex. \ref{ex:proj} this defines a morphism:
	$$\phi \colon Y' \to \PP(\bigwedge^\ell \sV^\vee \otimes_{\sO_X} \bigwedge^\ell \sW)$$
which is a lift of $f$, such that $\phi^* \sO_{ \PP(\bigwedge^\ell \sV^\vee \otimes_{\sO_X} \bigwedge^\ell \sW)}(1) \simeq f^{-1}\sI_Z \cdot \sO_{Y'}$.
On the other hand, denote by $\varpi_+ \colon Y_+ \to \PP(\bigwedge^\ell \shQ_+) \subseteq \PP(\bigwedge^\ell \sV^\vee)$ and  $\varpi_- \colon Y_- \to \PP(\bigwedge^\ell \shQ_-) \subseteq \PP(\bigwedge^\ell \sW)$ the Pl{\"u}cker embeddings of Thm. \ref{thm:Quot} \eqref{thm:Quot-3}, then there is a closed immersion:
	$$\widehat{\varpi} \colon \widehat{Y} \xrightarrow{\varpi_+ \times_X \varpi_-} \PP(\bigwedge^\ell \sV^\vee) \times_X  \PP(\bigwedge^\ell \sW) \xrightarrow{\varsigma_{\bigwedge^\ell \sV^\vee, \bigwedge^\ell \sW}} \PP(\bigwedge^\ell \sV^\vee \otimes_{\sO_X} \bigwedge^\ell \sW),$$
induced by the surjection $\widehat{\pi}^* (\bigwedge^\ell  \sV^\vee \otimes  \widehat{\pi}^* \bigwedge^\ell  \sW) \twoheadrightarrow \bigwedge^\ell \shQ_+|_{\widehat{Y}}  \otimes \bigwedge^\ell \shQ_-|_{\widehat{Y}}$, where $\varsigma_{\bigwedge^\ell \sV^\vee, \bigwedge^\ell \sW}$ is the Segre embedding (see \cite[Prop. 9.8.7]{EGAI}). Through the isomorphism $\widehat{Y} \backslash \widehat{\pi}^{-1}(Z) \simeq Y \backslash Z$ of Cor. \ref{cor:Quot:degloci}, the map $\phi|_{Y' \backslash f^{-1}(Z) } \colon Y' \backslash f^{-1}(Z)  \to \PP(\bigwedge^\ell \sV^\vee \otimes_{\sO_X} \bigwedge^\ell \sW)$ factorises through $\widehat{\varpi}$. Since $Y' \backslash f^{-1}(Z) $ is scheme-theoretically dense in $Y'$ (see \cite[\href{https://stacks.math.columbia.edu/tag/07ZU}{Tag 07ZU}]{stacks-project}) and $\widehat{\varpi}$ is a closed immersion, $\phi$ itself factorises through $\widehat{\varpi}$. Thus we obtain a lifting $Y' \to \widehat{Y}$ of $f$. Since $\widehat{Y} = Y_+ \times_X Y_-$ is separated over $X$ (Thm. \ref{thm:Quot}), and $\phi|_{Y' \backslash  f^{-1}(Z)}$ is determined by $f$ through $\widehat{Y} \backslash \widehat{\pi}^{-1}(Z) \simeq Y \backslash Z$, the lifting $Y' \to \widehat{Y}$ of $f$ is unique (see, e.g., \cite[\href{https://stacks.math.columbia.edu/tag/01RH}{Tag 01RH}]{stacks-project}).

Hence \eqref{lem:Quotfib=Bl-1} holds, which implies \eqref{lem:Quotfib=Bl-2}. For \eqref{lem:Quotfib=Bl-3}, it suffices to observe that by Cor. \ref{cor:Quot:degloci}, the dimension conditions of the lemma precisely imply that $Y \subset X$, $Y_\pm \subseteq \GG_\pm$ and $\widehat{Y} \subseteq \GG_+ \times_X \GG_-$ are closed subschemes of Cohen--Macauly schemes with expected codimensions (cf. Lem. \ref{lem:Tor-ind:quot:CM}), and that $\widehat{\pi}^{-1}(Z) \subseteq \widehat{Y}$ has codimension $\ge 1$. Since the codimension of a subscheme of a Cohen--Macaulay scheme coincides with the depth of its ideal, \eqref{lem:Quotfib=Bl-3} is proved. 
\end{proof}

 \addtocontents{toc}{\vspace{0.5\normalbaselineskip}}	
\section{Relative Fourier--Mukai transforms} \label{sec:FM}
The theory of relative Fourier-Mukai transforms has been a crucial ingredient in the study of derived categories. However, different references tend to make different, sometimes strong, assumptions on the schemes or the morphisms. We find it helpful to have a uniform {\em commutative} framework under which the various robust theories of derived categories -- mutation theory, relative Serre duality, base-change, descent theory, etc.; see \cite{Bo, BK, Kuz06, Kuz07, Kuz11, Huy, FMN, BS, P19} -- could be simultaneously applied. This section investigates the theory of relative Fourier-Mukai transforms for the category of {\em quasi-compact, quasi-separated} schemes.

\subsection{Generators of Triangulated Categories}
This subsection briefly reviews various notions of generators for triangulated categories (see \cite{Huy, Nee, BB}). 

Let $\shD$ be a triangulated category with translation functor $[i]$, $i \in \ZZ$ (see \cite[\href{https://stacks.math.columbia.edu/tag/0145}{Tag 0145}]{stacks-project}). To avoid awful terminology, in this paper, a full triangulated subcategory always means a {\em strictly} full triangulated subcategory (\cite[\href{https://stacks.math.columbia.edu/tag/001D}{Tag 001D}]{stacks-project}); direct sums (resp. products, filtered colimits, etc) always mean small direct sums (resp. products, filtered colimits, etc). 

For a class $\shE \subseteq {\rm Ob}(\shD)$ of objects of $\shD$, we let $\langle \shE \rangle$ denote the smallest full triangulated subcategory of $\shD$ containing $\shE$ (notice that our notation $\langle \shE \rangle$ is different from \cite{BB}, where the notation ``$\langle \shE \rangle$" there means the thick closure of our $\langle \shE \rangle$), and let $\shE^\perp \subseteq \shE$ (resp. $ {}^\perp \shE \subseteq \shD$) denote the full subcategory of $\shD$ spanned by objects $A \in \shD$ such that $\Hom_{\shD}(E,A[i]) = 0$ (resp. $\Hom_{\shD}(A,E[i]) = 0$) for all $E \in \shE$ and $i \in \ZZ$.  We will refer to $\shE^\perp$ (resp. ${}^\perp \shE$) as the right (resp. left) orthogonal of $\shE$. Then $\shE^\perp$ and ${}^\perp \shE$ are full, thick (i.e., closed under taking direct summands) triangulated subcategories of $\shD$, and it is easy to see that $\shE^\perp = \langle \shE \rangle^\perp$ and ${}^\perp\shE = {}^\perp \langle \shE \rangle$. 
Assume now that $\shD$ has (arbitrary, small) direct sums (i.e., coproducts). An object $K \in \shD$ is called {\em compact} if the functor $\Hom_{\shD}(K,\blank)$ commutes with direct sums. The compact objects of $\shD$ form a full, thick triangulated subcategory $\shD^c \subseteq \shD$.

\begin{definition}[Generators of Triangulated Categories; {see \cite{Huy, Nee, BB}}] Let $\shD$ be a triangulated category and let $\shE \subseteq {\rm Ob}(\shD)$ be a subset of objects.
\begin{itemize}
	\item We say $\shE$ {\em generates $\shD$ in the triangulated sense} if $\langle \shE \rangle = \shD$. 
	\item We say $\shE$ {\em thickly generates} (or {\em classically generates}) $\shD$ if the smallest full, thick (i.e., closed under direct summands) triangulated subcategory containing $\shE$ is $\shD$ itself. In the case where $\shE = \{E\}$ consists of a single object and $\{ E\}$ thickly generates $\shD$, we say that $E$ is a {\em thick (or classical) generator} of $\shD$. 
	\item  We say $\shE$ {\em spans} $\shD$ (or $\shE$ is a {\em spanning class of $\shD$}) if $\shE^\perp = {}^\perp \shE = 0$. 
	\item  We say $\shE$ {\em weakly generates} or {\em generates} $\shD$ (or $\shE$ is a {\em generating set} of $\shE$) if  $\shE^{\perp} = 0$. If $\shE = \{E\}$ consists of a single object and $\{ E\}$ generates $\shD$, we say $E$ is a {\em generator} of $\shE$. 
	\item Assume that $\shD$ has direct sums. We say that $\shE$ is a set of {\em compact generators} for $\shD$ (or $\shD$ is compactly generated by $\shE$) if $\shE$ consists of compact objects and generates $\shD$ in the weak sense (that is, $\shE \subseteq \shD^c$ and $\shE^\perp = 0$). A triangulated category $\shD$ is said to be {\em compactly generated} if it has arbitrary direct sums and has a set of compact generators (or equivalently, $\shD$ is generated by $\shD^{c}$). 
\end{itemize}
\end{definition}

Notice that we have the following result: if $\shD$ is compactly generated and let $\shE \subseteq \shD^{c}$ be a set of compact objects, then $\shD$ is generated by $\shE$ if and only if $\shD^c$ is thickly generated by $\shE$ (see \cite[Theorem 2.1]{Nee} or \cite[Theorem 2.1.2]{BB}).

The following representability theorems, due to Neeman \cite{Nee} (see also Krause \cite{Kra}) in the triangulated setting, play a dominant role in the theory of the next subsection:

\begin{theorem}[{Brown Representability Theorem; Neeman \cite[Theorem 3.1]{Nee}, Krause \cite[Theorem A]{Kra}}]
\label{thm:Brown}
Let $\shD$ be a compactly generated triangulated category. Then a functor $H \colon \shD^{\rm op} \to {\rm Ab}$ is representable (i.e., $H \simeq \Hom_{\shD}(\blank, D)$ for some $D \in \shD$) if and only if it is homological (i.e., carries exact triangles to long exact sequences) and carries direct sums in $\shD$ to products of abelian groups.
\end{theorem}

\begin{theorem}[{Adjoint Functor Theorem; Neeman \cite[Theorem 4.1]{Nee}}] 
\label{thm:Adjoint}
Let $F \colon \shD \to \shE$ be an exact functor between triangulated categories and assume that $\shD$ is compactly generated. Then $F$ admits a right adjoint if and only if it preserves direct sums. 
\end{theorem}

Notice that adjoint functors of an exact functor between triangulated categories are automatically exact functors (\cite[Lemma 5.3.6]{Nee01}).

\begin{remark}[Dual Statements] 
\label{rmk:dualBrown}
The ``dual" statements of Theorems \ref{thm:Brown} and \ref{thm:Adjoint} are only true under certain technical assumptions on the triangulated category $\shD$ or on the functors $F \colon \shD \to \shE$. For example, we could consider the following technical condition on $\shD$:
\begin{itemize}
	\item[(*)] $\shD$ is a compactly generated triangulated category which satisfies one of the following two conditions: either $\shD$ is $\kappa$-compactly generated for some regular cardinal $\kappa$ and the abelian category $\sE x(\{\shD^{\kappa}\}^{\rm op}, {\rm Ab})$ has enough injectives (see \cite[Theorem 1.18]{Nee01} for details), or $\shD$ has a set of symmetric generators in the sense of \cite[Definition 2]{Kra}.
\end{itemize}
Then we have the following dual version of the Brown Representability Theorem:
	\begin{itemize}
		\item If $\shD$ satisfies condition $(*)$, then a functor $H \colon \shD \to {\rm Ab}$ is representable if and only if it is homological and carries products in $\shD$ to products of abelian groups; see \cite[Theorem 1.18]{Nee01} and \cite[Theorem B]{Kra}.
	\end{itemize}
Consequently, we also have the following dual version of the Adjoint Functor Theorem: 
\begin{itemize}
	\item If $\shD$ satisfies condition $(*)$, then an exact functor $F \colon \shD \to \shE$ to a triangulated category $\shE$ admits a left adjoint if and only if it preserves products; see \cite[Theorem 8.4.4]{Nee01}.
\end{itemize}
The author does not know how to verify condition $(*)$ in practice; it would be extremely helpful to know the answer to the following question(s):
\begin{itemize}
	\item If $X$ is a quasi-compact, quasi-separated scheme, does $\Dqc(X)$ satisfy the condition $(*)$? (And if the answer is ``not always", when does $\Dqc(X)$ satisfy the condition $(*)$?)
\end{itemize}
There is potentially an alternative version of the dual Adjoint Functor Theorem as follows. 
In the context of (locally) presentable categories, it is known that a functor between (locally) presentable categories has a left adjoint if and only if it is accessible and preserves small limits (see \cite[Theorem 1.66]{AR}, \cite[Corollary 5.5.2.9]{HTT}). In view of this result, we believe the following theorem is true:
\begin{itemize}
	\item (Dual Brown Representability Theorem for Triangulated Categories; speculation) Let $F \colon \shD \to \shE$ be an exact functor between well generated triangulated categories (in the sense of \cite[Definition 1.15]{Nee01}). Then $F$ admits a left adjoint if and only if $F$ preserves small products and there exists a regular cardinal $\kappa$ such that $F$ preserves $\kappa$-filtered homotopy colimits (\cite[\href{https://stacks.math.columbia.edu/tag/090Z}{Tag 090Z}]{stacks-project}).  
\end{itemize}
We will not use this theorem in this paper, but it would be desirable to know whether it is true or not and have a reference.
\end{remark}

\begin{theorem}[{\cite[Theorem 5.1]{Nee}}]
\label{thm:Brown:cpt}
 Let $F \colon \shD \to \shE$ be an exact functor between triangulated categories and assume that $\shD$ is compactly generated by a set $S$ of objects, and let $G \colon \shE \to \shD$ be the right adjoint functor of $F$. Then $G$ preserves coproducts if and only if $F$ carries objects of $S$ to compact objects of $\shE$. Consequently, if we take $S = \shD^c$, then we obtain that $G$ preserves coproducts if and only if $F$ preserves compact objects. 
\end{theorem}

We find the following assertion very useful (whose various versions are well-known to experts):

\begin{corollary}[{Compare with \cite[Theorem 2.1.2 \& Lemma 3.2]{Nee}}]
\label{cor:Adjoint}
Let $\shD$ be a compactly generated triangulated category and let $\shC \subseteq \shD$ be a full triangulated subcategory which is closed under formation of direct sums in $\shD$. If $\shC$ contains a compact generating set of $\shD$, then $\shC= \shD$.
\end{corollary}
\begin{proof}
This assertion is essentially proved in \cite[Lemma 3.2]{Nee}; we present a different proof here as a fun exercise of the Adjoint Functor Theorem. Let $i_{\shC} \colon \shC \to \shD$ denote the inclusion functor. Then the condition of $\shC$ implies that $\shC$ is compactly generated and $i_{\shC}$ preserves direct sums. Consequently, by Theorem \ref{thm:Adjoint}, $i_{\shC}$ admits a right adjoint $i_{\shC}^! \colon \shD \to \shC$. For any $D \in \shD$, we consider the counit map $\eta \colon i_{\shC} i_{\shC}^! (D) \to D$ and complete it to an exact triangle in $\shD$:
	$$i_{\shC} i_{\shC}^! (D) \xrightarrow{\eta} D \to E \xrightarrow{[1]}.$$
From the adjunction $i_{\shC}  \dashv i_{\shC}^!$, we obtain that, for any object $C \in \shC$, the composite map
	$$\Hom_{\shC}(C, i_{\shC}^!(D)) \xrightarrow{\simeq} \Hom_{\shD}(i_{\shC}(C), i_{\shC} i_{\shC}^!(D)) \xrightarrow{\circ \eta} \Hom_{\shD}(i_{\shC}(C), D)$$
is an isomorphism. Therefore, $\Hom_{\shC}(C, E) \simeq 0$ for any $C$. Since $\shC$ contains a generating set of $\shD$, we conclude that $E \simeq 0$, and hence $  i_{\shC} i_{\shC}^! (D) \simeq D$. This proves $\shC = \shD$.
\end{proof}

Consequently, we have the following result: 

\begin{corollary}
\label{cor:Adjoint.isom}
Let $\Phi_1, \Phi_2 \colon \shD_1 \to \shD_2$ be two exact functors between triangulated categories. Assume that $\shD_1$ is compactly generated and the functors $\Phi_1,\Phi_2$ preserve direct sums. Then a natural transform $\phi \colon \Phi_1 \to \Phi_2$ is an isomorphism if and only if it is an isomorphism on a generating set of $\shD_1$.
\end{corollary}
\begin{proof}
We let $\shC \subseteq \shD_1$ be the full subcategory spanned by objects $A \in \shD_1$ such that the map 
	$$\phi(A[n]) \colon \Phi_1(A[n]) \to \Phi_2(A)[n]$$ 
is an isomorphism in $\shD_2$ for all $n \in \ZZ$. Then $\shC$ is a full triangulated subcategory which is closed under the formations of direct sums of $\shD_1$. By virtue of Corollary \ref{cor:Adjoint}, we have $\shC = \shD_1$. 
\end{proof}

\begin{corollary}
\label{CorollaryAdjoint.ff}
Let $\Phi \colon \shD_1 \to \shD_2$ be an exact functor between triangulated categories. Assume that $\shD_1$ is compactly generated and $\Phi$ admits a right adjoint. Then:
\begin{enumerate}
	\item 
	\label{CorollaryAdjoint.ff-1}
	If $\Phi$ admits a left adjoint, then $\Phi$ is fully faithful if and only if it is fully faithful on a generating set of $\shD_1$, that is, for a generating set $\shE$ of $\shD_1$, the canonical map
			$$\Hom(E_1, E_2) \xrightarrow{\Phi(\blank)} \Hom(\Phi(E_1), \Phi(E_2))$$
		is an isomorphism for all $E_1, E_2 \in \shE$.
	\item 
	\label{CorollaryAdjoint.ff-2}
	If $\shD_2$ is compactly generated and $\Phi$ is fully faithful, then $\Phi$ is an equivalence if and only if the essentially image of $\Phi$ contains a compact generating set of $\shD_2$.
\end{enumerate}
\end{corollary}
\begin{proof}
We first prove assertion \eqref{CorollaryAdjoint.ff-1}. Let $\Phi^L$ denote a left adjoint of $\Phi$ and let $\eta \colon \Phi^L \circ \Phi \to \id$ denote the counit map. For every $E \in \shE$, from the commutative diagram of functors
	$$
	\begin{tikzcd}%[column sep = 3 em, row sep = 3 em]
		\Hom(\blank, E) \ar{rd}[swap]{\Phi} \ar{r}{\eta_{E}}& \Hom(\blank, \Phi^L \circ \Phi (E) ) \ar{d}{\simeq} \\
		 	&\Hom(\Phi(\blank), \Phi(E))
	\end{tikzcd}
	$$
(see \cite[Lemma 1.21]{Huy}), we obtain that $\Phi$ is fully faithful on a generating set $\shE$ if and only if the counit map $\eta \colon \Phi^L \circ \Phi \to \id$ is an isomorphism on all objects of $\shE$. Now assertion \eqref{CorollaryAdjoint.ff-1} follows from applying Corollary \ref{cor:Adjoint.isom} to the functors $\Phi_1 = \Phi^L \circ \Phi$ and $\Phi_2 = \Id$.  

Assertion \eqref{CorollaryAdjoint.ff-2} follows from applying Corollary \ref{cor:Adjoint} to the essential image of $\Phi$.
\end{proof}

For later reference, we list the following known results regarding generators:

\begin{lemma} 
\label{lem:span:f.f.-1} 
{(Bridgeland \cite{Br}, Orlov \cite{Or}; see \cite[Proposition 1.49]{Huy})} 
Let $\Phi \colon \shD_1 \to \shD_2$ be an exact functor between triangulated categories. If $\Phi$ has both a left and a right adjoint, then $\Phi$ is fully faithful if and only if it is fully faithful on some spanning class. %In the case where $\Phi$ is fully faithful, $\Phi$ is essentially surjective if and only if it contains a generating set of $\shD_2$.
\end{lemma}

\begin{lemma}[{\cite{BB}}]\label{lem:generator} Let $f_i \colon X_i \to S$ be a morphisms of quasi-compact, quasi-separated schemes, where $i = 1,2$, and let $p_i \colon X_1 \times_S X_2 \to X_i$ denote the natural projections. Then $\Dqc(X_1 \times_S X_2)$ (resp. $\Perf(X_1 \times_S X_2)$) is compactly generated (resp.  thickly generated) by
	$$R := \{p^*_1 F_1 \otimes p^*_2 F_2 \mid F_1 \in \Perf(X_1), F_2 \in \Perf(X_2) \} \subseteq \Perf(X_1 \times_S X_2).$$ 
In particular, 
%the localizing envelope (resp. thick closure) of $R$ inside $\Dqc(X_1 \times_S X_2)$ is $\Dqc(X_1 \times_S X_2)$ (resp. $\Perf(X_1 \times_S X_2)$). In other words, 
the smallest full triangulated subcategory of $\Dqc(X_1 \times_S X_2)$ containing $R$ which is closed under taking arbitrary direct sums (resp. taking direct summands) is $\Dqc(X_1 \times_S X_2)$ (resp. $\Perf(X_1 \times_S X_2)$).
\end{lemma}
\begin{proof}
By virtue of Theorem \ref{thm:BB}, we can choose compact generators $F_i$ of $\Dqc(X_i)$, $i=1,2$. Then by \cite[Lemma 3.4.1]{BB}, $F_1 \boxtimes_S F_2: =  p^*_1 F_1 \otimes p^*_2 F_2\in R$ is a compact generator for $\Dqc(X_1 \times_S X_2)$. Hence $\Dqc(X_1 \times_S X_2)$ is compactly generated by $R$. Since $R$ is closed under degree shift, if we let $\shR$ denote the smallest full triangulated subcategory of $\Dqc(X_1 \times_S X_2)$ containing $R$ that is closed under taking direct sums, then \cite[Theorem 2.1.2]{Nee} implies that $\shR = \Dqc(X_1 \times_S X_2)$. Therefore, \cite[Theorem 2.1.3]{Nee} %and \cite[Lemma 2.2]{Nee92} 
implies that $\shR^c =  \Dqc(X_1 \times_S X_2)^c = \Perf(X_1 \times_S X_2) $ is the thick closure of $R$. 
\end{proof}

\begin{lemma} \label{lem:span} Let $f \colon Y \to X$ be an affine morphism between quasi-compact, quasi-separated schemes. Then $\{f^* F \mid F \in \Perf(X)\}$ spans $\Perf(Y)$ and compactly generates $\Dqc(Y)$.
\end{lemma}
\begin{proof} Observe that $(f^*, f_*)$ forms an adjoint pair between $\Dqc(X)$ and $\Dqc(Y)$, $f_*$ is conservative, and $f^*$ restricts to a functor $\Perf(X) \to \Perf(Y)$. For any $G \in (f^*\Perf(X))^\perp \subseteq \Dqc(Y)$, $\Hom_{Y}(f^*F,G) = \Hom_{X}(F,f_* G)=0$ for any $F \in \Perf(X)$, hence $f_*G =0$, therefore $G=0$. Similarly, for any $G \in {}^\perp(f^*\Perf(X)) \cap \Perf(Y)$, $\Hom_{Y}(G,f^*F) = \Hom_{X}(F^\vee,f_* (G^\vee))=0$ for any $F \in \Perf(X)$, therefore $f_*(G^\vee)=0$, which implies $G^\vee=0$, and hence $G=0$. 
\end{proof}

\subsection{Derived Categories of Quasi-Compact, Quasi-Separated Schemes} \label{sec:generalities:derived}
\subsubsection{Quasi-compact, quasi-separated schemes}
For a scheme $X$, we let $\Mod \sO_X$ denote the abelian category of $\sO_X$-modules, $\Qcoh X \subseteq \Mod \sO_X$ the abelian subcategory of quasi-coherent sheaves and $\coh X  \subseteq \Mod \sO_X$ the abelian subcategory of coherent sheaves. We let $\D(\sO_X) = D (\Mod \sO_X)$, $\D(\Qcoh X)$ and $\D(\coh X)$ denote the unbounded derived categories of these abelian categories (\cite[\S 1.2]{Lip}, \cite{Sp}), respectively. For a pair of integers $a \le b$, $a,b \in \ZZ$, we let $\D^{[a,b]}(\sO_X)$,   $\D^{[a,b]}(\Qcoh X)$, and $\D^{[a,b]}(\coh X)$ denote the full subcategories of $\D(\sO_X) = D (\Mod \sO_X)$, $\D(\Qcoh X)$ and $\D(\coh X)$, respectively, spanned by those complexes with vanishing cohomology sheaves in degrees $< a$ or $>b$. Similarly, for $? \in \{ -,+, \b\}$, we let $\D^{?}(\sO_X)$,   $\D^{?}(\Qcoh X)$, and $\D^{?}(\coh X)$ denotes the full subcategories spanned by complexes with bounded above, bounded below, and bounded cohomologies, respectively.

We will be mostly interested in following full triangulated subcategories of $\D(\sO_X)$:
	$$\Perf(X) \subseteq \Db(X) \subseteq \Dqc(X),$$
where $\Dqc(X)$ is full subcategory spanned by unbounded complexes of $\sO_X$-modules with quasi-coherent cohomologies, $\Db(X)$ is the full subcategory spanned by pseudo-coherent complexes with bounded cohomologies, and $\Perf(X)$ is the full subcategory spanned by perfect complexes. 

\begin{remark}
If $X$ is {\em noetherian} and has finite Krull dimension, then: 
	\begin{enumerate}
		\item There is a canonical equivalence $\Db(\coh X) \simeq \Db_{\coh}(X)$ (\cite[Corollary II.2.2.2.1]{SGA}) and a canonical identification $\Db(X) = \Db_{\coh}(X)$ as subcategories of $\Dqc(X)$ (\cite[\href{https://stacks.math.columbia.edu/tag/08E8}{Tag 08E8}]{stacks-project}).		
		\item There is a canonical equivalence $\D(\Qcoh X) \simeq \Dqc(X)$; see  \cite[\href{https://stacks.math.columbia.edu/tag/09T4}{Tag 09T4}]{stacks-project};
	\end{enumerate}
	
If $X$ is {\em quasi-compact} and {\em semi-separated} (i.e., $X$ has affine diagonal), then: 
	\begin{enumerate}
		\item There is canonical equivalence $\D(\Qcoh X) \simeq \Dqc(X)$; see \cite[\href{https://stacks.math.columbia.edu/tag/08DB}{Tag 08DB}]{stacks-project};  
		\item $X$ is a perfect stack in the sense of \cite{BFN}; see \cite[Proposition 3.1.9]{BFN} and note that $\Dqc(X)$ is compactly generated by perfect complexes \cite{LN, BB}. 
	\end{enumerate}
\end{remark}

In general, we have the following implications for a scheme $X$:
	\begin{align*}
	&	 \text{$X$ is separated} \implies \text{$X$ is semi-separated} \implies \text{$X$ is quasi-separated}.\\
	&	 \text{$X$ is noetherian} \implies \text{$X$ is quasi-compact, quasi-separated}.
	\end{align*}
All these implications are strict; see \cite{TLRG} for counterexamples of the inverse directions of the first two implications. In the following, we will work with the weakest hypothesis of the above implications, that is, we will work in the category of {\em quasi-compact, quasi-separated schemes}. \footnote{In literatures, the adjective ``quasi-compact, quasi-separated" is sometimes abbreviated as ``concentrated"; for example, quasi-compact, quasi-separated morphisms are also called concentrated maps (see \cite[\S 3.9]{Lip}).}
Notice that morphisms between quasi-compact, quasi-separated schemes are automatically quasi-compact, quasi-separated. The derived categories of quasi-compact, quasi-separated schemes are particularly well-behaved from a categorical perspective:

\begin{theorem}[{\cite[Theorem 3.1.1]{BB}; see also \cite{Nee, TT}}]
\label{thm:BB}
 Let $X$ be a quasi-compact, quasi-separated scheme. Then $\Dqc(X)$ is compactly generated, and the compact objects of $\Dqc(X)$ are precisely with perfect complexes. %In fact, \cite[Theorem 3.1.1]{BB} proves something even stronger: 
 In fact, $\Dqc(X)$ can be (classically) generated by a single perfect complex. 
\end{theorem}

\subsubsection{Morphisms and Grothendieck--Neeman Duality}

Recall that a morphism $f \colon X \to Y$ between schemes is called {\em pseudo-coherent} if $f$ is locally of finite type and $\sO_X$ is pseudo-coherent relative to $Y$ (
\cite[\href{https://stacks.math.columbia.edu/tag/067Z}{Tag 067Z}]{stacks-project}); $f$ is called {\em perfect} if it is pseudo-coherent and has finite Tor-dimension (\cite[\href{https://stacks.math.columbia.edu/tag/0685}{Tag 0685}]{stacks-project}).  Let $f \colon X \to Y$ be a morphism, we let $f^*$ denote the derived pullback functor $f^* \colon \Dqc(Y) \to \Dqc(X)$ (which is well-defined by virtue of \cite[Proposition 3.9.1]{Lip}); then it is direct to see that $f^*$ restricts to a functor $\Perf(Y) \to \Perf(X)$, also denoted by $f^*$, on the subcategories of perfect complexes.

We summarize the results of Lipman and Neeman in \cite{LN, Nee, Nee10, Lip} that are relevant to this paper in the following theorem:

\begin{theorem}[{Lipman--Neeman \cite{LN, Nee, Nee10, Lip}; see also Lurie \cite{SAG}}] \label{thm:Neeman-Lipman} 
Let $f \colon X \to Y$ be a morphism between quasi-compact, quasi-separated schemes %(which is automatically a quasi-compact, quasi-separated morphism) 
and let $f^* \colon \Dqc(Y) \to \Dqc(X)$ denote the derived pullback functor.
\begin{enumerate}[leftmargin=*]
	\item \label{thm:Neeman-Lipman-1}
	\begin{enumerate}[label=(\theenumi\alph*), ref=\theenumi\emph{\alph*}]
		\item \label{thm:Neeman-Lipman-1i}
		The functor $f^*$ admits a right adjoint $f_* \colon \Dqc(X) \to \Dqc(Y)$, called derived pushforward functor, which has finite cohomological amplitude and preserves direct sums. Consequently, the derived pushforward functor $f_* \colon \Dqc(X) \to \Dqc(Y)$ admits a right adjoint $f^! \colon \Dqc(Y) \to \Dqc(X)$.
		\item \label{thm:Neeman-Lipman-1ii}
		For every Tor-independent square of quasi-compact, quasi-separated schemes
			$$
			\begin{tikzcd}
				X' \ar{d}{f'} \ar{r}{g'} & X \ar{d}{f} \\
				Y' \ar{r}{g} & Y,
			\end{tikzcd}
			$$
	the canonical base-change transformation $g^* f_* \to f'_* g'^*$ is a natural isomorphism of functors from $\Dqc(X)$ to $\Dqc(Y')$.
		\item \label{thm:Neeman-Lipman-1iii}
		For every pair of objects $\sF \in \Dqc(X)$ and $\sG \in \Dqc(Y)$, the canonical map $\sF \otimes f_* \sG \to f_*(f^* \sF \otimes \sG)$ is an equivalence.
	\end{enumerate}
	\item \label{thm:Neeman-Lipman-2}
	If $f$ is proper and perfect, then $f^* \colon \Dqc(Y) \to \Dqc(X)$ has finite cohomological amplitude and preserves (bounded) pseudo-coherent complexes. Moreover:
		\begin{enumerate}[label=(\theenumi\alph*), ref=\theenumi\emph{\alph*}]
		\item \label{thm:Neeman-Lipman-2i}
		The pushforward functor $f_* \colon \Dqc(X) \to \Dqc(Y)$ preserves (bounded) pseudo-coherent complexes and perfect complexes, respectively, and admits a right adjoint $f^! \colon \Dqc(Y) \to \Dqc(X)$ which preserves direct sums \footnote{Consequently, $f^!$ also admits a right adjoint $f_{(-1)}$; the functor $f_{(-1)}$ will not be used in this paper.} and has finite cohomological amplitude.
		\item \label{thm:Neeman-Lipman-2ii}
		 For every Tor-independent square of quasi-compact, quasi-separated schemes
			$$
			\begin{tikzcd}
				X' \ar{d}{f'} \ar{r}{g'} & X \ar{d}{f} \\
				Y' \ar{r}{g} & Y,
			\end{tikzcd}
			$$
	the canonical base-change transformation $g'^* f^! \to f'^!g'^*$ is a natural isomorohism of functors from $\Dqc(Y)$ to $\Dqc(X')$.
		\item \label{thm:Neeman-Lipman-2iii}
		For every pair of objects $\sF, \sG \in \Dqc(Y)$, the canonical map $f^!\sF \otimes f^* \sG \to f^!(\sF \otimes \sG)$ is an equivalence.
		\item \label{thm:Neeman-Lipman-2iv}
		The derived pullback functor $f^*$ admits a left adjoint $f_! \colon \Dqc(X) \to \Dqc(Y)$ which preserves perfect objects. Furthermore, for all $\sF \in \Perf(X)$, there is a canonical equivalence $f_!(\sF) \simeq (f_* \sF^\vee)^\vee \in \Perf(Y)$.
		\item \label{thm:Neeman-Lipman-2v}
		 For every Tor-independent square of quasi-compact, quasi-separated schemes
			$$
			\begin{tikzcd}
				X' \ar{d}{f'} \ar{r}{g'} & X \ar{d}{f} \\
				Y' \ar{r}{g} & Y,
			\end{tikzcd}
			$$
	the canonical base change transformation $f'_! g'^*\to g^*f_!$ is a natural isomorphism of functors from $\Dqc(X)$ to $\Dqc(Y')$.
		\item \label{thm:Neeman-Lipman-2vi}
		For every pair of objects $\sF \in \Dqc(X)$ and $\sG \in \Dqc(Y)$, the canonical map $f_!(f^* \sF \otimes \sG) \to \sF \otimes f_! \sG$ is an equivalence.
		\item \label{thm:Neeman-Lipman-2vii}
		Let $\omega_f = \omega_{X/Y}$ denote the {\em relative dualizing complex}, that is, $\omega_f = f^! \sO_Y \in \Dqc(X)$. Then there are canonical natural isomorphisms of functors
				$$f^!(\blank) \simeq f^*(\blank) \otimes \omega_f \quad \text{and} \quad f_!(\blank) \simeq f_*(\blank \otimes \omega_f).$$
		\end{enumerate}
	\item \label{thm:Neeman-Lipman-3}
	In the situation of \eqref{thm:Neeman-Lipman-2}, if we furthermore assume that the relative dualizing complex $\omega_f$ is perfect, then each member of the adjunction sequence $f_! \dashv f^* \dashv f_*\dashv f^!$ has finite cohomological amplitude, preserves perfect complexes and (bounded) pseudo-coherent complexes, respectively (in particular, in the case where $X$ and $Y$ are locally Noetherian, each of the functors $f_! , f^* , f_*$ and $f^!$ preserves bounded complexes of coherent sheaves).
\end{enumerate}
\end{theorem}

\begin{proof}
First, we focus on the results that are well documented in literatures:
\begin{itemize}
	\item Assertion \eqref{thm:Neeman-Lipman-1} is a consequence of Brown Representability Theorem. Concretely:
		\begin{itemize}
			\item Assertion \eqref{thm:Neeman-Lipman-1i}: Since $f^*$ preserves direct sums and perfect complexes (which coincide with compact objects by Theorem \ref{thm:BB}), by virtue of Theorems \ref{thm:Brown} and \ref{thm:Brown:cpt}, we obtain that $f^*$ admits a right adjoint $f_*$ which preserves direct sums (see also \cite[Corollary 3.9.3.3]{Lip}). Moreover, $f_*$ has finite cohomological amplitude (\cite[Proposition 3.9.2]{Lip}).
			\item Assertion \eqref{thm:Neeman-Lipman-1ii} follows from \cite[Theorem 3.10.3]{Lip}.
			\item Assertion \eqref{thm:Neeman-Lipman-1iii} is \cite[Proposition 3.9.4]{Lip}.
		\end{itemize}
	\item Assertion \eqref{thm:Neeman-Lipman-2}: generally, $f^*$ preserves perfectness and pseudo-coherence. Since $f$ has finite Tor-dimension, $f^*$ has finite cohomological amplitude.
		\begin{itemize}
			 \item Assertion \eqref{thm:Neeman-Lipman-2i}: by \cite[Proposition 2.1 \& Examples 2.2 (a)]{LN} or \cite[Proposition 4.7.1 \& Example 4.7.3(a)]{Lip}, $f_*$ preserves perfect complexes and $f^!$ preserves direct sums (or equivalently, has a right adjoint $f_{(-1)}$). By \cite[Theorem 1.2]{LN}, $f$ is quasi-proper, that is, $f_*$ preserves pseudo-coherent complexes, and $f^!$ has finite cohomological amplitude.
			 \item Assertion \eqref{thm:Neeman-Lipman-2ii} is \cite[Theorem 4.7.4]{Lip} or \cite[Theorem 2.7]{LN}.
			  \item Assertion \eqref{thm:Neeman-Lipman-2iii} follows from  %\cite[Proposition 4.7.5]{Lip} under the condition that $\sG$ has finite Tor-dimension [So check whether $\infty$-categorical one holds without condition?]. 
			 $f^!(\sF) \simeq  f^*(\sF) \otimes \omega_{f}$ of assertion \eqref{thm:Neeman-Lipman-2vii}.
			\item Assertion \eqref{thm:Neeman-Lipman-2vii}: the equivalence $f^! (\blank)= f^*(\blank) \otimes \omega_f$ is \cite[Proposition 4.7.1]{Lip} (see also \cite[Proposition 2.1, (iv)]{LN}).
		\end{itemize}
	\item Assertion \eqref{thm:Neeman-Lipman-3} follows from \eqref{thm:Neeman-Lipman-2}: if $\omega_f$ is a perfect complex, then $\otimes \omega_f$ has finite cohomological amplitude and preserves perfectness and pseudo-coherence, respectively.
\end{itemize}

It only remains to prove the statements about the left adjoint functor $f_!$ of $f^*$ of assertions \eqref{thm:Neeman-Lipman-2iv}, \eqref{thm:Neeman-Lipman-2v}, \eqref{thm:Neeman-Lipman-2vi} and \eqref{thm:Neeman-Lipman-2vii}. These are consequences of the above assertions \footnote{If the speculation of Remark \ref{rmk:dualBrown} is true, then we can deduce the existence of $f_!$ by the Adjoint Functor Theorem since $f^*$ preserves small filtered homotopy colimits.}; 
since these results appear to be not well-documented in the context of triangulate categories, we present a proof for the sake of completeness. Following an idea of Lurie presented in \cite[\S 6.4.5]{SAG}, we {\em define} a functor $f_! \colon \Dqc(X) \to \Dqc(Y)$ by the formula $f_!(\sF) = f_*(\sF \otimes \omega_f)$, where $\omega_f = f^!(\sO_Y)$, and we will show that $f_!$ is a left adjoint of $f^*$. Concretely, let $v_0 \colon f_*(\omega_f) = f_* f^!(\sO_Y) \to \sO_Y$ denote the counit map, then there is a natural transformation $v \colon f_! f^* \to \id$ given by the formula
	$$v \colon f_! f^*(\sG) = f_* (f^*\sG \otimes \omega_f) \simeq \sG \otimes (f_* \omega_f) \xrightarrow{\id \otimes v_0} \sG$$
for all $\sG \in \Dqc(Y)$. We wish to show that $v \colon f_! f^* \to \id$ exhibits $f_!$ as a left adjoint of $f^*$. In other words, we wish to show for all $\sF \in \Dqc(X)$ and $\sG \in \Dqc(Y)$, the composite map
	$$\rho_{\sF, \sG} \colon \Hom_{\Dqc(X)}(\sF, f^*\sG) \xrightarrow{f_!} \Hom_{\Dqc(Y)}(f_! \sF, f_! f^* \sG) \xrightarrow{v} \Hom_{\Dqc(Y)}(f_! \sF, \sG)$$
is an isomorphism. First, we consider the case where $\sF \in \Perf(X)$; in this case, the composite map
	$$f_*(\sF^\vee) \otimes f_*(\sF \otimes \omega_f) \to f_*(\sF^\vee \otimes \sF \otimes \omega_f) \to f_*(\omega_f) \xrightarrow{v_0} \sO_Y$$
induces a functorial map $\psi \colon f_! (\sF) \to f_*(\sF^\vee) ^\vee$ (here, the first morphism is given by the canonical lax monoidal structure, and $f_*(\sF^\vee)$ is perfect by assertion \eqref{thm:Neeman-Lipman-2i}). We wish to show that $\psi$ is an isomorphism for all $\sF \in \Perf(X)$. This follows from Yoneda's lemma and the following functorial isomorphisms for all $\sH \in \Dqc(Y)$:
	\begin{align*}
		&\Hom_{\Dqc(Y)}(\sH, f_!(\sF)) \simeq \Hom_{\Dqc(X)}(f^* \sH, \omega_f \otimes \sF) \simeq \Hom_{\Dqc(X)}((f^* \sH) \otimes \sF^\vee, \omega_f)  \\
		& = \Hom_{\Dqc(X)}((f^* \sH) \otimes \sF^\vee, f^!(\sO_Y)) \simeq \Hom_{\Dqc(Y)}(f_*(f^* \sH \otimes \sF^\vee), \sO_Y) \\
		& \simeq  \Hom_{\Dqc(Y)}(\sH \otimes f_*(\sF^\vee), \sO_Y) \simeq \Hom_{\Dqc(Y)}(\sH, (f_* \sF^\vee)^\vee).
	\end{align*}
Next, we fix $\sF \in \Perf(X)$. Then the above argument shows that $f_!(\sF) \simeq (f_* \sF^\vee)^\vee \in \Perf(Y)$. Consequently, both the functors $\Hom(\sF, f^*(\blank))$ and $\Hom(f_!(\sF), \blank)$ preserve direct sums. Let $\shE \subseteq \Dqc(Y)$ denote the full subcategory spanned by $\sG$ such that $\rho_{\sF, \sG}$ is an equivalence. Then $\shE$ is a full triangulated subcategory which is closed under formation of direct sums. By virtue of Corollary \ref{cor:Adjoint}, to show $\rho_{\sF, \sG}$ is an isomorphism for all $\sG \in \Dqc(Y)$, we are reduced to the case where $\sG$ is perfect, in which case the isomorphism $\rho_{\sF,\sG}$ follows from the isomorphism $\psi \colon f_! (\sF) \simeq (f_* \sF^\vee)^{\vee}$. Hence we have proved that $\rho_{\sF, \sG}$ is an isomorphism for all $\sF \in \Perf(X)$. However, for any fixed $\sG \in \Dqc(Y)$, both the source and target of $\rho_{\blank, \sG} \colon \Hom(\blank, f^* \sG) \to \Hom(f_*(\blank \otimes \omega_f), \sG)$ carry direct sums to products of abelian groups. Consequently, by applying Corollary \ref{cor:Adjoint}, we obtain that $\rho_{\sF, \sG}$ is an isomorphism for all $\sF \in \Dqc(X)$ from the case where $\sF \in \Perf(X)$. This proves assertions \eqref{thm:Neeman-Lipman-2iv} and \eqref{thm:Neeman-Lipman-2vii}. 

Finally, assertions \eqref{thm:Neeman-Lipman-2iii} and \eqref{thm:Neeman-Lipman-2iv} follow from assertions \eqref{thm:Neeman-Lipman-1ii} and \eqref{thm:Neeman-Lipman-1iii}, respectively, by virtue of the formula $f_! (\sF)= f_*(\sF \otimes \omega_f)$.
\end{proof}

\begin{remark}[Quasi-Properness and Quasi-Perfectness]
\label{rmk:quasi-perfect}
In Lipman and Neeman's papers \cite{LN, Lip}, the slightly more general concepts of quasi-properness and quasi-perfectness were introduced. Specifically, let $X$ and $Y$ be quasi-compact, quasi-separated schemes, then:
\begin{itemize}
	\item A map $f \colon X \to Y$ is said to be {\em quasi-proper} if the derived pushforwad $f_*$ preserves pseudo-coherent complexes. Then {\em Kiehl's Finiteness Theorem} can be formulated as: every proper, pseudo-coherent map is quasi-proper (\cite[Corollary 4.3.3.2]{Lip}). In the case where $Y$ is noetherian, a finite type, separated map $f \colon X \to Y$ is  quasi-proper if and only if it is proper (\cite[Corollary 4.3.3.3]{Lip}).
	\item A map $f \colon X \to Y$ is said to be {\em quasi-perfect} if $f_*$ preserves perfect complexes. By virtue of Theorems \ref{thm:Brown:cpt} and \ref{thm:Adjoint}, this condition is equivalent to requiring that $f^!$ has a right adjoint, or equivalently, $f^!$ preserves direct sums. The main result of \cite{LN} asserts that $f$ is quasi-perfect if and only if $f$ is quasi-proper and has ﬁnite tor-dimension, if and only if $f$ is quasi-proper and the functor $f^!$ has finite cohomological amplitude (see \cite[Theorem 1.2]{LN}). Consequently, Kiehl's Finiteness Theorem implies that any proper, perfect map is quasi-perfect; in the case where $X$, $Y$ are noetherian, a finite type separated map $f \colon X \to Y$ is quasi-perfect if and only if it is  proper and perfect.
\end{itemize}		
%We found that the above version of the theorem is sufficient (and convenient) for our applications. 

%Slightly more general; strengthened. We find that it seems in practice, the above version is more straightforward and more useful.
\end{remark}

\begin{remark} In the situation of Theorem \ref{thm:Neeman-Lipman} \eqref{thm:Neeman-Lipman-2}, \cite[Proposition 4.1]{BDS} asserts that the dualizing complex $\omega_{f}$ is a perfect complex if and only if $\omega_{f}$ is $\otimes$-invertible. Consequently, in the situation of \eqref{thm:Neeman-Lipman-2}, this result implies that $\omega_f$ is a perfect complex if and only if it is {\em invertible}, that is, it is isomorphic to objects of the form $\sL [n]$ on each connected component, where $\sL$ is a line bundle and $n \in \ZZ$ (see \cite[Corollary 2.9.5.7]{SAG}). Moreover, the main result \cite[Theorem 1.9]{BDS} implies that, in the situation of Theorem \ref{thm:Neeman-Lipman} \eqref{thm:Neeman-Lipman-3}, the adjunction sequence $f_! \dashv f^* \dashv f_*\dashv f^! \dashv f_{(-1)}$ fits into an infinite tower of adjoints which are connected to each other by so-called Wirthm\"uller isomorphisms; see \cite[Theorem 1.9]{BDS} for details. %This paper does not need these strong results.
\end{remark}

%\begin{proof} In the situation of classical algebraic geometry where $X$, $Y$ are quasi-compact, quasi-separated schemes (and the involved base change are Tor-independent), the assertions are consequences of the results of \cite{LN, Nee, Nee10, Lip}; see also our summary in \cite[Theorem 3.1]{J21}. In the context of spectral algebraic geometry, assertion \eqref{thm:Neeman-Lipman-1} is a consequence of \cite[Proposition 6.3.4.1]{SAG} and \cite[Corollary 6.3.4.3]{SAG}, assertions \eqref{thm:Neeman-Lipman-2i} through \eqref{thm:Neeman-Lipman-2vii} follows from a combination of \cite[Proposition 6.4.2.1, Lemma 6.4.2.2, Proposition 6.4.5.3, and Proposition 6.4.5.4]{SAG}, and assertion \eqref{thm:Neeman-Lipman-3} is a direct consequence of \eqref{thm:Neeman-Lipman-2}. The same proof strategy of the above results work in our setting with minimal modifications. \end{proof}

%A morphism is called {\em quasi-perfect} if the right adjoint $f^!$ (which always exists in this case) of $f_*$ preserves direct sums, or equivalently $f_*$ send perfect complexes to perfect complexes \cite[Proposition 2.1]{LN}. In particular, {\em a proper and perfect morphism is quasi-perfect}  -- this is called Kiehl’s Finiteness Theorem, see \cite[p. 236, Theorem 2.2]{SGA}, \cite[Example 2.2 (a)]{LN}.

%% Example: lci
\begin{example} \label{eg:lci} 
Let $f \colon X \to Y$ be a morphism between quasi-compact, quasi-separated schemes. Then in each of the following examples, $f$ is proper, perfect, with invertible %\footnote{A perfect complex is said to be {\em invertible} if it is a degree shift of a line bundle.} 
$\omega_f$:
	\begin{enumerate}[leftmargin=*]
		\item \label{eg:lci-1} If $f \colon X \hookrightarrow Y$ is a Koszul-regular closed immersion (see \cite[\href{https://stacks.math.columbia.edu/tag/063D}{Tag 063D}]{stacks-project}), then $f$ is perfect (see \cite[\href{https://stacks.math.columbia.edu/tag/068C}{Tag 068C}]{stacks-project}) with $\omega_f \simeq \bigwedge^r \shN [-r]$ (see \cite[\href{https://stacks.math.columbia.edu/tag/0BR0}{Tag 0BR0}]{stacks-project}), where $\shN$ denotes the normal sheaf of $f \colon X \hookrightarrow Y$, which is locally free of rank $r$.
		\item  \label{eg:lci-2} $f \colon X \twoheadrightarrow Y$ is a smooth proper morphism of relative dimension $d$, then $f$ is perfect (by \cite[\href{https://stacks.math.columbia.edu/tag/068A}{Tag 068A}]{stacks-project}, since $f$ is flat and locally of finite presentation), with $\omega_f = \bigwedge^d \Omega_{X/Y} [d]$, where $\Omega_{X/Y}$ is the sheaf of relative differentials, see \cite[\href{https://stacks.math.columbia.edu/tag/0BRT}{Tag 0BRT}]{stacks-project}. \footnote{Notice this result is stated in \cite[\href{https://stacks.math.columbia.edu/tag/0BRT}{Tag 0BRT}]{stacks-project} under the condition that $Y$ is {\em noetherian}. However, the only place the noetherian condition is used is the second equality of the equation there, where it is referred to \cite[\href{https://stacks.math.columbia.edu/tag/0A9R}{Tag 0A9U}]{stacks-project}, which holds without noetherian assumption \cite[Proposition 2.1]{LN}; Alternatively, as Johan de Jong pointed out, one could also prove this by using absolute noetherian reduction and the fact that base change for the relative dualizing complex behaves well for flat proper morphisms of finite presentations.}%, to prove the lemma directly in our stated generality.}
		\item  \label{eg:lci-3} If $f \colon X \to Y$ is a morphism between schemes which are smooth and proper over a quasi-compact, quasi-separated base scheme $S$, then $\omega_f = \omega_{X/S} \otimes f^*(\omega_{Y/S})^\vee$.
	\end{enumerate}
In particular, all these are examples of {\em local complete intersection morphisms}, that is, morphisms which locally factorize through a Koszul-regular closed immersion followed by a smooth morphism. Conversely, \eqref{eg:lci-1} and \eqref{eg:lci-2} imply that any proper local complete intersection morphism is perfect, with an invertible relative dualizing complex.
\end{example}

\subsection{Semiorthogonal Decompositions and Mutations}
For a triangulated category $\shD$ and objects $A,B \in \shD$, $\Hom^k(A,B): = \Ext^k(A,B) = \Hom(A,B[k])$ for $k \in \ZZ$, and $\Hom^\bullet(A,B)$ denotes $\ZZ$-graded the complex $\{\Hom^k(A,B)\}_{k \in \ZZ}$ with zero differential. For a family of objects $\shE = \{E\}$ of a triangulated category $\shD$, denote:
	$$\shE^\perp := \{ A \in \shD \mid \Hom^\bullet(E,A) = 0, ~~ \forall E \in \shE \}, \quad {}^\perp \shE :=\{ A \in \shD \mid \Hom^\bullet(A,E) = 0,~~  \forall E \in \shE\}$$ 
to be the {\em right} and respectively {\em left orthogonal} of $\shE$ inside $\shD$. The subcategory {\em generated} by $\shE$, denote by $\langle \shE \rangle$, is the smallest full triangulated subcategory of $\shD$ containing $\shE$. 

%%% Defn SOD
\begin{definition}
A {\em semiorthogonal decomposition}\label{def:SOD:D} of a triangulated category $\shD$, written as:
	\begin{equation} \label{eqn:SOD:D}
	\shD= \langle \shA_1, \shA_2, \ldots, \shA_{n} \rangle,
	\end{equation}
is formed by a sequence of full triangulated subcategories $\shA_1, \ldots, \shA_n$ of $\shT$ such that
	\begin{enumerate}
		\item  \label{eqn:SOD:D-1} (Semiorthogonality) $\shA_i \subset \shA_{j}^\perp$ for all $i < j$, i.e., $\Hom_\shD (A_j ,A_i) = 0$ for all $A_j \in \shA_j$ and $A_i \in \shA_i$ if $i < j$.
		\item  \label{eqn:SOD:D-2} (Generation) For any object $D \in \shD$, there is a sequence of objects $D_i$ and a diagram:
		\begin{equation*}
				\begin{tikzcd} [back line/.style={dashed}, row sep=1.8 em, column sep=1.6 em]
	0=D_{n} \ar{rr} 	& 	& D_{n-1} \ar{rr} \ar{ld}		&		& D_{n-2} \ar{r} \ar{ld}	&\cdots \ar{r} 	& D_{1} \ar{rr} 	&	& D_{0}=D \ar{ld}. \\
							& A_{n} \ar[dashed]{lu} &	& A_{n-1}  \ar[dashed]{lu} & & & & A_{1} \ar[dashed]{lu}
				\end{tikzcd}
			\end{equation*}
such that $A_i = \cone(D_i \to D_{i-1}) \in \shA_i $ for all $i \in [1,n]$.
	\end{enumerate}
The subcategories $\shA_i$'s are called {\em components} of $\shD$ with respect to \eqref{eqn:SOD:D}.  The condition \eqref{eqn:SOD:D-1} implies that the objects $T_i$ and $A_i$ in the diagram of \eqref{eqn:SOD:D-2} are unique (up to canonical isomorphisms), and the assignments $D \to D_i \in \shT$ and $D \mapsto A_i \in \shA_i$ are functorial. The exact functor
	$$\pr_i \colon \shD \to \shA_i, \qquad D \mapsto A_i \in \shA_i$$
is called the $i$-th \emph{projection functor} of the semiorthogonal decomposition \eqref{eqn:SOD:D}. A sequence $\shA_1, \ldots, \shA_n$ satisfying the condition \eqref{eqn:SOD:D-1} is called a {\em semiorthogonal sequence}.
\end{definition}

\begin{remark}[Partial order] \label{rmk:poset} The notion of semiorthogonal decomposition could be generalized to a sequence of full triangulated subcategories $\{\shA_i\}_{i \in I}$ indexed by a {\em finite well-ordered partial order set $(I, \prec)$}: we say the semiorthogonal order is {\em compatible} with the partial order set $(I, \prec)$ if $\shA_i \subseteq \shA_j^\perp$ if $i \prec j$; we say it is {\em strongly compatible} with $(I, \prec)$ if $\shA_i \subseteq \shA_j^\perp$ whenever $i \nsucceq j$. If the semiorthogonal order of a sequence $\{\shA_i\}_{i \in I}$ is strongly compatible with $(I, \prec)$, and $\{\shA_i\}_{i \in I}$ generates $\shD$ in the triangulated sense, then {\em any total order} (or called {\em linear order}) extending the partial order $(I, \prec)$ give rises to a semiorthogonal decomposition of $\shD$ in the sense of Def. \ref{def:SOD:D} of the form \eqref{eqn:SOD:D}. 
\end{remark}

%\begin{remark}[Exceptional collection] In the case when all triangulated categories are assumed to be $\kk$-linear for a field $\kk$, an object $E \in \shD$ is called {\em exceptional} if $\Hom^\bullet(E,E) \simeq \kk$ and for any $A \in \shD$, $\Hom^\bullet(E,A)$ and $\Hom^\bullet(A,E)$ are bounded. A family of exceptional objects $\shE = \{E_i\}_{i \in I}$, where $I$ is a finite total order set, is called an {\em exceptional collection} if $E_i$ is exceptional and $\Hom(E_i,E_j) = 0$ if $i > j$. The collection $\shE$ is called {\em strong} if $\Hom(E_i, E_j[k]) = 0$ for any $k \ne 0$ and $i,j \in I$, and is called {\em full} if $\shE$ generates $\shD$ in the sense of triangulated categories. A full exceptional collection $\shE =  \{E_i\}_{i \in I}$ induces a semiorthogonal decomposition $\shD = \langle  \{ \langle E_i \rangle \}_{i \in I} \rangle$, which for simplicity usually written as $\shD = \langle  \{ E_i \}_{i \in I} \rangle$, with each subcategory $\langle E_i \rangle \simeq \Db(\Spec \kk)$.\end{remark}

\begin{definition} A full triangulated subcategory $\shA$ of $\shD$ is called {\em left admissible} (resp. {\em right admissible}) if the inclusion functor $i_{\shA}: \shA\hookrightarrow \shD$ has a left adjoint $i_{\shA}^* \colon \shD \to \shA$ (resp. a right adjoint $i_{\shA}^! \colon \shD \to \shA$). $\shA$ is called {\em admissible} in $\shD$ if it is both left and right admissible. A semiorthogonal decomposition \eqref{eqn:SOD:D} is called {\em admissible} if each component $\shA_i$ is admissible.
\end{definition}

Here are some basic properties about admissibility:
\begin{lemma}[{\cite{Bo, BK}}] \label{lem:mut:admissible}
	\begin{enumerate}[leftmargin=*]
		\item If $\shA$ is left (resp. right) admissible in $\shD$, and $\shD$ is left (resp. right) admissible in $\shD'$, then $\shA \subset \shD'$ is also left (resp. right) admissible. 
		\item $\shA \subset \shD$ is left admissible iff there is a semiorthogonal decomposition $\shD = \langle \shA, \shB \rangle$ (then in this case $\shB = {}^\perp \shA$), iff $\shD$ is generated by $\shA$ and ${}^\perp \shA$ as a triangulated category. Similarly for right admissibility.
		\item If $\shA \subset \shD$ is left admissible, then $( {}^\perp \shA)^\perp = \shA$. Similarly for right admissibility.
		\item If $\shA, \shB \subseteq \shD$ are left (resp. right) admissible in $\shD$ and $\shB \subseteq {}^\perp \shA$, then the triangulated subcategory $\langle \shA, \shB \rangle$ generated by $\shA$ and $\shB$ is left (resp. right) admissible in $\shD$.
	\end{enumerate}
\end{lemma}

%%% Definition mutation
\begin{definition}[\cite{Bo, BK}] Let $\shA \subset \shD$ be admissible, hence we have semiorthogonal decompositions $\shD = \langle \shA^{\perp}, \shA \rangle = \langle \shA , {}^\perp \shA \rangle$. Let $\shA_\bullet = (\shA_1, \ldots, \shA_n)$, $n \ge 2$ be a semiorthogonal sequence of admissible subcategories inside $\shD$.
\begin{enumerate}[leftmargin=*]
	\item The {\em left mutation}, respectively, {\em right mutation} functors {\em through $\shA$} are defined as:
	$$\LL_\shA: = i_{\shA^\perp} i^*_{\shA^{\perp}} : \shA \to \shA \quad \text{and} \quad \RR_{\shA} : =  i_{{}^\perp \shA} i^!_{ {}^\perp \shA}: \shA \to \shA$$

	\item For $1 \le i \le n-1$, {\em the left mutation of $\shA_\bullet$ at position $i$} is the sequence of subcategories:
	$$\LL_{i}(\shA_\bullet) = (\shA_1, \ldots, \shA_{i-1}, \LL_{\shA_{i}} (\shA_{i+1}), \shA_{i}, \shA_{i+2}, \ldots, \shA_n).$$
For $2 \le i \le n$, the {\em right mutation of $\shA_\bullet$ at position $i$} is the sequence of subcategories:
	$$\RR_{i}(\shA_\bullet) = (\shA_1, \ldots, \shA_{i-1}, \shA_{i+1}, \RR_{\shA_{i+1}} (\shA_{i}),  \shA_{i+2}, \ldots, \shA_n).$$
\end{enumerate}
\end{definition}

The following is a collection of standard results about mutations, see also \cite{Bo, BK, Kuz07, Kuz09} for further references.

% Lemma: mutation
\begin{lemma}[{\cite{Bo, BK}}] \label{lem:mut}  Let $\shA$ be an admissible subcategory of $\shD$, and $\shA = (\shA_1, \ldots, \shA_n)$ be a semiorthogonal sequence of admissible subcategories inside $\shD$, $n \ge 2$.
	\begin{enumerate}[leftmargin=*]
		\item \label{lem:mut-1} For any $E \in \shD$, there are distinguished triangles
		$$ i_\shA i^!_{\shA} (E) \to E \to \LL_{\shA} \,E \xrightarrow{[1]}{},\qquad  \RR_{\shA} \,E \to E \to  i_\shA i^*_{\shA} (E)  \xrightarrow{[1]}{}.$$
		\item \label{lem:mut-2} The left mutation $\LL_{\shA}$ is zero on $\shA$, fully faithful on ${}^\perp \shA$; The right mutation $\RR_{\shA}$ is zero on $\shA$, fully faithful on $\shA^\perp$; The restrictions of left and right mutations induce mutually inverse equivalences $\LL_{\shA}\,|_{{}^\perp \shA} : {}^\perp \shA \simeq \shA^\perp$ and  $\RR_{\shA}\,|_{\shA^\perp } : \shA^\perp \simeq {}^\perp \shA$.
		\item \label{lem:mut-3} Denote by $\langle \shA_\bullet \rangle = \langle \shA_1, \ldots, \shA_n \rangle$ the triangulated subcategory generated by $\shA_1, \ldots, \shA_n$ as usual. Then there are canonical isomorphisms of functors:
		$$\LL_{\langle \shA_1,\shA_2, \ldots, \shA_n \rangle} = \LL_{\shA_1} \circ \LL_{\shA_2} \circ \cdots \circ \LL_{\shA_n}, \qquad \RR_{\langle \shA_1,\shA_2, \ldots, \shA_n \rangle} = \RR_{\shA_n} \circ \RR_{\shA_{n-1}} \circ \cdots \circ \RR_{\shA_1}.$$
		 \item \label{lem:mut-4} For any $2 \le i \le n$ (resp. $1 \le i \le n-1$), $\LL_i(\shA_\bullet)$ (resp. $\RR_i(\shA_\bullet)$) is also a semiorthogonal sequence. Furthermore $\shA_\bullet$ and $\LL_i(\shA_\bullet)$ (resp. $\RR_i(\shA_\bullet)$) generate the same triangulated subcategory inside $\shD$, i.e. $\langle \shA_\bullet \rangle = \langle \LL_{i}(\shA_\bullet) \rangle$ (resp. $\langle \shA_\bullet \rangle = \langle \RR_{i}(\shA_\bullet) \rangle$).
		\item \label{lem:mut-5} If $\Phi: \shD \to \shD$ is an autoequivalence, then there are canonical isomorphisms:
			$$\Phi \circ \LL_\shA  \simeq \LL_{\Phi(\shA)} \circ \Phi, 
				\quad \text{and} \quad \Phi \circ \RR_{\shA} = \RR_{\Phi(\shA)} \circ \Phi.$$
	\end{enumerate}
\end{lemma}

\begin{remark} If $\shD$ has a relative Serre functor $\SS_{\shD/S}$ over some base scheme $S$ in the senes of Def. \ref{def:Serre}, such that all objects of $\shD$ have $\Perf(S)$-valued $\Hom$-objects, then it follows from Lem. \ref{lem:Serre} that $\SS_{\shD/S}({}^\perp \shA) = \shA^\perp$, and $\SS_{\shD/S}^{-1}(\shA^\perp) = {}^\perp \shA$.
\end{remark}

\begin{definition} A triangulated subcategory $\shA \subseteq \shD$ is called {\em $\infty$-admissible} if it is admissible all its iterated right and left orthogonals are all admissible. A semiorthogonal decomposition \eqref{eqn:SOD:D} is called {\em $\infty$-admissible} if each component $\shA_i$ is admissible and all its iterated right and left mutations are admissible. 
\end{definition}

\begin{proposition}[\cite{BK}] Let $\shD = \langle \shA_1, \ldots, \shA_n \rangle$ be a  semiorthogonal decomposition that is $\infty$-admissible, they all its iterated left and right mutations are $\infty$-admissible semiorthogonal decompositions of $\shD$. Moreover, the mutation functors define a braid group action on the set of all $\infty$-admissible semiorthogonal decompositions of $\shD$, i.e. they satisfies the braid group relations: $\LL_i \circ \RR_i = \id = \RR_i \circ \LL_i$, $\LL_i \circ \LL_{i+1} \circ \LL_i  = \LL_{i+1} \circ \LL_i \circ \LL_{i+1}$, $\RR_i \circ \RR_{i+1} \circ \RR_i  = \RR_{i+1} \circ \RR_i \circ \RR_{i+1}$, and $\LL_i \circ \LL_j = \LL_j \circ \LL_i$, $\RR_i \circ \RR_j = \RR_j \circ \RR_i$ for $|i-j|\ge 2$.
\end{proposition}

\begin{definition} Let $\shA_\bullet = (\shA_1, \ldots, \shA_n)$ be a semiorthogonal sequence inside $\shD$ with admissible components. The {\em left dual semiorthogonal semiorthogonal sequence $\shB_\bullet=\foL(\shA_\bullet)$ of $\shA_\bullet$} is the semiorthogonal sequence $\shB_\bullet = (\shB_n, \ldots, \shB_1)$ defined by setting 
	$$\shB_1 = \shA_1, \qquad \shB_{i} = \LL_{\langle \shA_1, \ldots, \shA_{i-1}\rangle} \shA_i \quad \text{for $2 \le i \le n$}.$$
The {\em right dual semiorthogonal sequence $\shC_\bullet  = \foR(\shA_\bullet)$ of $\shA_\bullet$} 
is defined by setting 
	$$\shC_n = \shA_n, \qquad \shC_{i} = \RR_{\langle \shA_{i+1} , \ldots, \shA_{n}\rangle} \shA_i  \quad \text{for $1 \le i \le n-1$}. $$
\end{definition}

% Lemma: dual SOD
\begin{lemma}[\cite{Bo,Kuz09}] \label{lem:dualsod} Let $\shD = \langle \shA_1, \ldots, \shA_n \rangle$ be a semiorthogonal decomposition with admissible components, let $\shB_\bullet  = \foL(\shA_\bullet)$ and $\shC_\bullet = \foR(\shA_\bullet)$ be the left and right dual sequences.
\begin{enumerate}[leftmargin = *] 
	\item There are semiorthogonal decompositions 
		$$\shD = \langle \shB_n, \ldots, \shB_1 \rangle =  \langle \shA_1, \ldots, \shA_n \rangle= \langle \shC_n, \ldots, \shC_1 \rangle.$$ 
	\item $\forall i \le k \le n$, the following holds:
	$$\langle \shB_i, \ldots, \shB_1 \rangle = \langle \shA_1, \ldots, \shA_{i}  \rangle  \quad \text{and} \quad \langle \shC_n, \ldots, \shC_i\rangle = \langle \shA_i, \ldots, \shA_{n}  \rangle,$$
	$$\shB_i =\langle \shA_1, \ldots, \shA_{i-1}, \shA_{i+1}, \ldots, \shA_{n} \rangle^\perp, \quad \shC_i ={}^\perp \langle \shA_1, \ldots, \shA_{i-1}, \shA_{i+1}, \ldots, \shA_{n} \rangle.$$
	\item $\forall  1 \le k \le n$, the following functors induce mutually inverse equivalences of categories:
	$$\LL_{\langle \shA_1, \ldots, \shA_{k-1}\rangle}=\LL_{\shA_1} \cdots  \LL_{\shA_{k-1}} \colon \shA_k \simeq \shB_k, \qquad \RR_{\langle \shB_{k-1}, \ldots, \shB_1 \rangle} =\RR_{\shB_1}  \cdots  \RR_{\shB_{k-1}}   \colon \shB_k \simeq \shA_k.$$ 
Similarly, the following functors induce mutually inverse equivalence of categories:
	$$\RR_{\langle  \shA_{k+1}, \ldots, \shA_{n}  \rangle} = \RR_{\shA_n} \cdots \RR_{\shA_{k+1}} \colon \shA_k \simeq \shC_k, \qquad  \LL_{\langle \shC_n, \ldots, \shC_{k+1} \rangle} = \LL_{\shC_n} \cdots \LL_{\shC_{k+1}} \colon \shC_k \simeq \shA_k.$$ 
	\end{enumerate}
\end{lemma}

\begin{proof} This is an easy consequence of Lem. \ref{lem:mut}. % (1),(2) follow from (4) of Lem. \ref{lem:mut}; (3) follows from (2), (3) of Lem. \ref{lem:mut}. %(4) follows from (2) of Lem. \ref{lem:mut} and (2). See also the discussions of \cite[Sec. 2.4]{Kuz09}.
\end{proof}

% rmk: if Serre exists.
\begin{remark} In the situation of Lem. \ref{lem:dualsod}, if we assume that $\shD$ has a relative Serre functor $\SS_{\shD/S}$ over some base scheme $S$ in the sense of \S \ref{sec:Serre}, then it will follow from Lem. \ref{lem:Serre} that the semiorthogonal decompositions $\foL(\shA_\bullet)$ and $\foR(\shA_\bullet)$ are also admissible, and the following holds:
	$$\foL^2 = \SS_{\shD/S}, \quad \text{and} \quad \foR^2 = \SS_{\shD/S}^{-1}.$$ 
\begin{comment}
To see this, for any $k$, statement (3) of above lemma implies 
	\begin{align*}
	& \big\langle \shB_k, \langle \shA_1, \ldots, \shA_{k-1}, \shA_{k+1}, \ldots, \shA_n \rangle \big \rangle 
	= \big \langle \langle \shA_1, \ldots, \shA_{k-1} \rangle, \shA_k, \langle \shA_{k+1}, \ldots, \shA_n \rangle \big\rangle \\
	 =& \big \langle \langle \shA_1, \ldots, \shA_{k-1}, \shA_{k+1}, \ldots, \shA_n \rangle, \shC_k  \big\rangle,
	\end{align*}
hence the result follows from Lem. \ref{lem:Serre}. 
\end{comment}
\end{remark}

% sec: Postnikov and Convolution
\subsection{Postnikov Systems and Convolutions}
Let $\shD$ be a triangulated category, and  $a\le b$ be two integers. Let $X^\bullet = \{ X^{a} \xrightarrow{d^{a}} X^{a+1} \xrightarrow{d^{a+1}} \cdots  \xrightarrow{d^{b-1}} X^b\}$ be a complex over $\shD$, that is, a collection of objects and morphisms in $\shD$ such that $d^{i+1} \circ d^{i} = 0$. 

\begin{definition} \label{def:Postnikov} A {\em (right) Postnikov system} attached to $X^\bullet$ is a diagram of the form:
 \begin{equation*}
	\begin{tikzcd} [row sep=1.8 em, column sep=0.5 em]
		&X^{a}[-a-1] \ar{rd}{i_{a}}	&	& &\cdots \cdots &	 &	&X^{b-1}[-b] \ar{rd}{i_{b-1}}&		& X^{b}[-b-1] \\
		Y=Y^{a} \ar[dashed]{ru}{j_{a-1}} &	&Y^{a+1}  \ar{ll}  & & \ar{ll} \cdots  \cdots    &  &  \ar{ll}  Y^{b-1} \ar[dashed]{ru}{j_{b-2}}	&	&  \ar{ll} Y^{b} = X^{b} [-b] \ar[dashed]{ru}[swap]{j_{b-1}=\Id} &
	\end{tikzcd}
	\end{equation*}
such that $j_k$ has degree $+1$, each triangle is distinguished, and $j_k \circ i_k = d^{k}[-k-1] \colon X^{k}[-k-1] \to X^{k+1}[-k-1]$ for each $k$. So in particular, $j_{b-1} = \Id \colon Y^b \simeq X^b[-b]$ and $i_{b-1} = d^{b-1}[-b]$. The object $Y= Y^{a}$ is called the {\em (right) convolution} of the (right) Postnikov system.
 \end{definition}
 
 We will call the object $Y^i$ in above diagram the ``$Y$-terms of the Postnikov system", and the object $X^i[-i]$ (notice the degree shift) the ``associated graded object of the right Postnikov system".
 
 \begin{remark}  There is also a notion of a {\em left Postnikov system} attached to $X^\bullet$, which is by definition a diagram of the form:
		\begin{equation*}
	\begin{tikzcd} [row sep=1.8 em, column sep=0.5 em]
		X^{a}  \ar{rd}{i_{a}=\Id} &	&X^{a+1}   \ar{rd}{i_{a+1}} &  &      \cdots \ar{rd}{i_{b-2}}  & 	&   X^{b-1} \ar{rd}{i_{b-1}}  	&	&  X^b  \ar{rd}{i_{b}}   & \\
		& Z^{a} = X^{a} \ar{ru}[swap]{j_{a}} &	 & Z^{a+1} \ar[dashed]{ll}  \ar{ru}[swap]{j_{a+1}}	&\cdots\cdots	     &Z^{b-2} \ar{ru}[swap]{j_{b-2}}&	&Z^{b-1} \ar[dashed]{ll} \ar{ru}[swap]{j_{b-1}}&		& Z^b = Z[b] \ar[dashed]{ll}
	\end{tikzcd}
	\end{equation*}
	such that each triangle is distinguished, and $j_k \circ i_k  = d^{k}\colon X^k  \to X^{k+1}$ for all $k$. So in particular $i_{a} = \Id$ and $j_a = d^{a} \colon X^{a} \to X^{a+1}$. The object $Z= Z^{b}[-b]$ is called the {\em left convolution} of the left Postnikov system. By octahedron axiom, there is a one-to-one correspondence between the class of right convolutions and the class of left convolutions. Therefore in this paper we will only use the {\em right} Postnikov system and omit the word ``right".
\end{remark}

\begin{remark}Note that our definition is slightly unconventional, but it is easy to see that up to degree shifts, above definition is equivalent to the ones given in \cite[\S IV.2]{GM} (in the case $[a,b] = [0,n]$), and the ones given in \cite[\S 1.2]{Kap88}, \cite[\S 1.3]{Or} (in the case $[a,b] =[-n,0]$). The benefits of our convention will be clear in the following examples.
\end{remark}

\begin{example} If $X^\bullet = \{X^{-1} \xrightarrow{d^{-1}} X^0 \}$, then the (right) convolution of $X^\bullet $ is uniquely given by $\cone(X^{-1} \xrightarrow{d^{-1}} X^0)$. If $X^{\bullet} = \{X^{0} \xrightarrow{d^{0}} X^{1} \}$, then the (right) convolution of $X^\bullet$ is uniquely given by the ``kernel" of $d^0$, i.e., $\cone(X^{0} \xrightarrow{d^{0}} X^{1})[-1]$.
\end{example}

\begin{example}[``Stupid" truncation] \label{eg:conv:stupid} If $\shA$ is an abelian category, $\shD = D^b(\shA)$ the bounded derived category of $\shA$, and $X^\bullet$ is given by a genuine bounded complex $E:=A^\bullet =  \{ A^{a} \xrightarrow{d^{a}} A^{a+1}  \xrightarrow{d^{a+1}} \cdots  \xrightarrow{d^{b-1}}  A^b\} \in D^b(\shA)$ with $a \le b$, where $A^i \in \shA$. Then the ``stupid" truncations (cf. \cite[\href{https://stacks.math.columbia.edu/tag/0118}{Tag 0118}]{stacks-project}) give rise to distinguished triangles:
	$$ \sigma^{\ge s+1} E \to \sigma^{\ge s} E \to A^s[-s] \xrightarrow{[1]}  \qquad \forall s \in \ZZ,$$
which induce a canonical (right) Postnikov system with $Y^i = \sigma^{\ge i} A^\bullet$:
  \begin{equation*}
	\begin{tikzcd} [row sep=1.8 em, column sep=0.5 em]
		&A^{a}[-a-1] \ar{rd}	&	 &\cdots \cdots 	 &	&A^{b-1}[-b] \ar{rd}&		& A^{b}[-b-1] \\
		E=\sigma^{\ge a} E \ar[dashed]{ru} &	& \sigma^{\ge a+1} E   \ar{ll}  & \cdots  \cdots    &    \sigma^{\ge b-1}  E \ar[dashed]{ru}	&	&  \ar{ll}  \sigma^{\ge b}  E  \ar[dashed]{ru} &
	\end{tikzcd}
	\end{equation*}
whose the right convolution is given by the complex itself $E=A^\bullet \in D^b(\shA)$.
\end{example}

\begin{example}[Koszul complex]\label{eg:conv:lci} Let $j \colon Z \hookrightarrow X$ be the regular immersion of the zero locus of a regular section $\xi$ of a vector bundle $\shE$ of rank $r$ over a quasi-compact, quasi-separated scheme $X$ (cf. case \eqref{eg:lci-1} of Example \ref{eg:lci}) and let $\shD = \Perf(X)$. The {\em Koszul complex} of $j$ is a complex of locally free sheaves in degree $[-r,0]$ given by $\shK^\bullet(j) := \{\shK^{-k} = \wedge^k(\shE^\vee), d^{-k} = \lrcorner \,\xi\}_{k=0,\ldots, r}$, see for example \cite[VII, \S 1]{SGA} or \cite[\S 17.4]{Ei}; we will often omit the differentials $d^{-k}$ in the expressions if there is no confusion. If we apply Example \ref{eg:conv:stupid} to the {Koszul complex} $\shK^\bullet(j)$, we obtain a canonical (right) Postnikov system attached to $\shK^\bullet(j)$ whose convolution is $j_{*} \sO_Z \in \Perf(X)$. Dually, by case \eqref{eg:lci-1} of Example \ref{eg:lci} and Grothendieck duality, if we apply Example \ref{eg:conv:stupid} to the {\em dual} Koszul complex $\shK^\bullet(j)^\vee = \{ (\shK^\vee)^{k} = \wedge^k \shE, d^k = (d^{-k})^\vee \}_{k=0,\ldots, r}$ (regarded as a complex in degree $[0,r]$), 
we obtain a canonical (right) Postnikov system attached to $\shK^\bullet(j)^\vee$ whose convolution is $j_{!} \sO_Z  \in \Perf(X)$.
\end{example}

The next lemma follows easily from definition:

\begin{lemma}[\cite{Or}] \label{lem:conv} If $\Phi \colon \shD_1 \to \shD_2$ is an exact functor between triangulated categories, $Y$ (resp. $Z$) is the right (resp. left) convolution of a right (resp. left) Postnikov system attached to a complex $X^\bullet$ over $\shD_1$. Then $\Phi(X^\bullet)$ is a complex over $\shD_2$, and the image under $\Phi$ of the right (resp. left) Postnikov is naturally a right (resp. left) Postnikov system in $\shD_2$ attached to $\Phi(X^\bullet)$, whose right (resp. left) convolution is given by  $\Phi(Y)$ (resp. $\Phi(Z)$).
\end{lemma}

%%% sec: monoidal
\subsection{Closed Symmetric Monoidal Structures}
This subsection considers the closed symmetric monoidal structures on $\Dqc(X)$ for a quasi-compact, quasi-separated scheme $X$. Recall that a {\em symmetric monoidal category} is a category $\shC$ together with the following data:
	\begin{itemize}
			\item (Tensor product functor) A functor $\otimes \colon \shC \times \shC \to \shC$;
			\item (Associativity constraints) A natural isomorphism 
					$$\alpha \colon ((\blank) \otimes (\blank) ) \otimes (\blank) \to (\blank)  \otimes ((\blank) \otimes (\blank) ) $$
					called {\em the associativity constraints of $\shC$}, with components of the form %: every triple of objects $X, Y, Z \in \shC$,
					$$\alpha _{X,Y,Z}: X \otimes (Y \otimes Z) \simeq (X \otimes Y) \otimes Z.$$
			\item (Unit) An element $\mathbf{1} \in \shC$ called {\em unit} of $\shC$, and functors $\lambda_X \colon 1 \otimes X \simeq X$ and $\rho_X \colon X \otimes 1 \simeq X$ called {\em left unit constraint} and {\em right unit constraint}, respectively.
			\item (Braidings) A natural isomorphism $\beta_{X, Y} \colon X \otimes Y \simeq Y \times X$ called {\em braidings}.
	\end{itemize}
Furthermore, these data are required to satisfy Pentagon identity, Hexagon identity, Braiding identity, and Triangle identity (see \cite[Definition 3.4.1]{Lip})
%The collection of data $(\shC, \otimes, \mathbf{1}, \alpha, \beta, \lambda, \rho)$ is said to be a symmetric monoidal structure on $\shC$. 
\footnote{We can equivalently define a symmetric monoidal category as a Grothendieck op-fibration $p \colon \shC^\otimes \to \shF{\rm in}_*$ over the Segal's category of pointed finite sets which satisfies Segal's condition; see \cite[Definition 2.0.0.7]{HA}.}.
A {\em closed symmetric monoidal category} is a symmetric monoidal category $(\shC, \otimes, \mathbf{1}, \alpha, \beta, \lambda, \rho)$ together with an {\em internal hom} functor $\ihom \colon \shC^{\rm op} \times \shC \to \shC$ and a functorial isomorphism
	$$\pi \colon \Hom(A \otimes B, C) \xrightarrow{\sim} \Hom(A, \ihom(B,C))$$
for all $A,B,C \in \shC$; see \cite[Definition 3.5.1]{Lip}.

\begin{example} 
\label{eg:D(O_X)^otimes}
Let $X$ be a scheme, and let $\otimes$ and $\RsHom_X$ denote the (derived) tensor product and (derived) sheaf $\Hom$ in the derived category $\D(\sO_X)$ of $\sO_X$-modules, respectively, let $\mathbf{1} = \sO_X$, and let $\alpha$, $\beta$, $\lambda$, $\rho$ be the natural isomorphisms, and let $\pi$ be the natural isomorphism
	$$\Hom_X(\sE \otimes \sF, \sG) \xrightarrow{\sim} \Hom_X(\sE, \RsHom_X(\sF,\sG))$$
from applying global section functor to \cite[Proposition 2.6.1]{Lip}. Then the collection of data 
	$$(\D(\sO_X), \otimes, \RsHom_X, \sO_X, \alpha, \beta, \lambda, \rho, \pi)$$
 forms a closed symmetric monoidal category (see \cite[Example (3.5.2)(d)]{Lip}).  Moreover, the formations of the symmetric monoidal categories $\D(\sO_X)$ are well-behaved under the (derived) pushforward and pullback functors for a composable pair of maps $X \xrightarrow{f} Y \xrightarrow{g} Z$ of schemes (\cite[\S 3.6, \S 3.7]{Lip}). Using the terminology of \cite[\S 3.6]{Lip}, let $\mathbf{S}$ denote the category of schemes, for each $X \in \mathbf{S}$, let $\mathbf{X}^* = \mathbf{X}_* = \D(X)$, and for each map $f \colon X \to Y$ in $\mathbf{S}$, let $f^*$ and $f_*$ denote the derived pushforward and derived pullback for $\sO_X$ and $\sO_Y$-modules, respectively. For a composable pair of maps $X \xrightarrow{f} Y \xrightarrow{g} Z$  in $\mathbf{S}$, we have functorial isomorphisms $(g\circ f)_* \simeq g_* \circ f_*$ and $(g\circ f)^* \simeq f^* \circ g^*$. Then the above collection of data defines an adjoint pair $({}^*, {}_*)$ of monoidal $\Delta$-pseudofunctors on $\mathbf{S}$; see \cite[(3.6.10) Scholium]{Lip}.

%, for a map of schemes $f \colon X \to S$, the desired properties and compatibility conditions for $f_*, f^*, \otimes, \RsHom$ hold on the level of $\D(\sO_X)$; see \cite[Sections 3.6, 3.7]{Lip} for details.
\end{example}

For a quasi-compact, quasi-separated scheme $X$, since the inclusion $j_{\qc} \colon \Dqc(X) \hookrightarrow \D(\sO_X)$ preserves direct sums and $\Dqc(X)$ is compactly generated, $j_\qc$ admits a right adjoint functor $\bQ_X \colon \D(\sO_X) \to \Dqc(X)$, called {\em coherator}  \cite{SGA, TT}. In the case where $X$ is furthermore semi-separated or noetherian, the coherator can be constructed as follows: let $Q_X$ be the right adjoint to the inclusion map $\Qcoh X \hookrightarrow \Mod \sO_X$, then the right derived functor $ \bR Q_X \colon \D(\sO_X) \to \D(\Qcoh X)$ is right adjoint to the natural map $\varphi_X \colon \D(\Qcoh X) \to \D(\sO_X)$. If $X$ is semi-separated or noetherian, then $\varphi_X$ is fully faithful and induces $\D(\Qcoh X)  \simeq \Dqc(X)$ (see \cite[\href{https://stacks.math.columbia.edu/tag/09T4}{Tag 09T4}, \& \href{https://stacks.math.columbia.edu/tag/08DB}{Tag 08DB}]{stacks-project}; consequently, in this case, $\bQ_X$ could be defined as 
	$\bQ_X : = \varphi_X \circ \bR Q_X \colon \D(\sO_X) \to \Dqc(X).$
In general, we can deduce the existence of $\bQ_X$ from the 
 from Neeman--Brown's representability theorem (Theorem \ref{thm:Adjoint}). Alternatively, one can reduce the general case to semi-separated cases by {\em reduction principle} \cite[Proposition 3.3.1]{BB}; we refer the readers to  \cite[\href{https://stacks.math.columbia.edu/tag/0CQZ}{Tag 0CQZ}]{stacks-project} for details. Consequently, the inclusion $j_{\qc} \colon \Dqc(X) \hookrightarrow \D(\sO_X)$ identifies $\Dqc(X)$ as a colocalizating (i.e., right admissible) subcategory of $\D(\sO_X)$, and for any $F \in \Dqc(X)$, the unit map $F \to \bQ_X \circ j_\qc (F)$ is an isomorphism. 

\begin{definition}[Internal Hom] 
\label{def:ihom}
For a quasi-compact, quasi-separated scheme $X$, we  define the {\em internal hom} $\ihom_X$ via the coherator $\bQ_X$ as follows: for any $\sF, \sG \in D(\sO_X)$,
	$$\ihom_X(F,G) : = \bQ_X \circ \RsHom_X(\sF,\sG) \in \Dqc(X).$$  
\end{definition}

For a quasi-compact and quasi-separated scheme $X$, we let $\otimes \colon \Dqc(X) \times \Dqc(X) \to \Dqc(X)$ denote the derived tensor product (\cite[(2.5.8.1)]{Lip}), and for all $\sE,\sF,\sG \in \Dqc(X)$, we let
	$$\alpha _{\sE, \sF, \sG}: \sE \otimes (\sF \otimes \sG)  \simeq (\sE \otimes \sF) \otimes \sG$$ 
	$$\lambda_{\sE} \colon \sO_X \otimes \sE \simeq \sE \qquad \rho_{\sE} \colon \sE \otimes \sO_X \simeq \sE \qquad \beta_{\sE, \sF} \colon \sE \otimes \sF \simeq \sF \otimes \sE$$
denote the natural functorial isomorphisms as in the case of $\D(\sO_X)$. Then we have the following result regarding the closed symmetric monoidal structure on $\Dqc(X)$ \footnote{A similar statement for the category $\D(\Qcoh X)$ is obtained by \cite{TLRG} under the condition that $X$ is {\em semi-separated}. The semi-separated condition is required in their case because they work with $\D(\Qcoh X)$, and the equivalence $\D(\Qcoh X) \simeq \Dqc(X)$ requires $X$ to be semi-separated or noetherian. However, since $\Dqc(X)$ is the category of interest for our purposes, we are able to drop these extra conditions.}: 

%%% Theorem : Closed monoidal structure on Dqc
\begin{theorem}[{Lipman \cite[\S 4.2]{Lip}}] \label{thm:Lipman} For any quasi-compact and quasi-separated scheme $X$ and all $\sE, \sF, \sG \in \Dqc(X)$, there is a natural fuctorial isomorphism
	$$\pi_{\sE,\sF,\sG} \colon \Hom_{\Dqc(X)}(\sE \otimes \sF, \sG) \xrightarrow{\sim} \Hom_{\Dqc(X)}(\sE, \ihom_X(\sF,\sG))$$
such that the collection of data $(\Dqc(X), \otimes, \sO_X, \alpha, \beta, \lambda, \rho)$, together with the internal hom $\ihom_X$ (Definition \ref{def:ihom}) and the functorial isomorphisms $\pi_{\sE,\sF,\sG}$ form a closed symmetric monoidal category. Furthermore, for any morphism $f \colon X \to S$ of quasi-compact, quasi-separated schemes, the following statements are true:
\begin{enumerate}[leftmargin=*]
	\item \label{thm:Lipman-1} Let $\bR \Gamma$ denote the derived global section functor, then there are natural isomorphisms:
	$$\bR \Gamma(X, \bQ_X(\blank)) \xrightarrow{\sim} \bR \Gamma(X, \blank), 
	\quad f_* \, \bQ_X \xrightarrow{\sim}  \bQ_S \, f_*, 
	\quad f^! \, \bQ_S \xrightarrow{\sim} f^! .$$
	\item \label{thm:Lipman-2} For any $\sF \in \Dqc(X)$ $\sG \in \Dqc(S)$, there is a functorial isomorphism
		$$f_* \, \ihom_X(\sF, f^!\sG) \xrightarrow{\sim} \ihom_S(f_* \sF, \sG)$$
	to which the application of the functor $H^0 \bR \Gamma(S, \blank)$ produces the (usual) adjunction isomorphism (usually referred to as the global Grothendieck duality):
		$$\Hom_X(\sF, f^! \sG) \xrightarrow{\sim} \Hom_S(f_* \sF, \sG).$$
	\item \label{thm:Lipman-3} For any $\sF,\sG \in \Dqc(S)$, there is a functorial isomorphism
		$$\ihom_X(f^* \sF, f^! \sG) \xrightarrow{\sim} f^! \, \ihom_S(\sF,\sG).$$
	\item \label{thm:Lipman-4} For any $\sE \in \Dqc(X)$, $\sF \in \Dqc(S)$, there is a functorial isomorphism
		 $$f_* \, \ihom_X(f^*\sF, \sE) \xrightarrow{\sim} \ihom_S(\sF, f_* \sE).$$
\end{enumerate}
\end{theorem}
\begin{proof} Most of the above results are asserted in \cite[(4.2.3)]{Lip} under the condition that all schemes are {\em separated}. However, it will be clear in the following proof that separateness condition could be {dropped}. 
Since the tensor structure $\otimes \colon \Dqc(X) \times \Dqc(X) \to \Dqc(X)$ is induced from the derived tensor product $\otimes \colon \D(X) \times \D(X) \to \D(X)$ (\cite[(2.5.8.1)]{Lip}), we immediately obtain that the symmetric monoidal structure on $\D(\sO_X)$ (Example \ref{eg:D(O_X)^otimes}) induces a symmetric monoidal structure $(\Dqc(X), \otimes, \sO_X, \alpha, \beta, \lambda, \rho)$ on $\Dqc(X)$ and that the inclusion $j_{\qc} \colon \Dqc(X) \to \D(\sO_X)$ is a symmetric monoidal functor. For all $\sE,\sF,\sG \in \Dqc(X)$, we let $\pi_{\sE, \sF, \sG}$ denote the composition of functorial isomorphisms:
	\begin{align*}
	&\Hom_{\Dqc(X)}(\sE, \ihom_X(\sF,\sG)) \\
	&= \Hom_{\Dqc(X)}(\sE, \bQ_X \RsHom_X(\sF,\sG)) & (\text{by definition of $\ihom_X$})\\
	& \simeq \Hom_{\D(\sO_X)}(j_{\qc}(\sE), \RsHom_X(\sF, \sG)) & (j_\qc \dashv \bQ_X) \\
	& \simeq \Hom_{\D(\sO_X)}(\sE \otimes \sF, \sG) & (\otimes \dashv \RsHom \quad \text{in} \quad \D(\sO_X))\\
	& \simeq \Hom_{\Dqc(X)}(\sE \otimes \sF, \sG) & (\sE\otimes \sF, \sG \in \Dqc(X)). 
	\end{align*}
This establishes the adjunction $\otimes \dashv \ihom$ and gives rise to the closed structure of $\Dqc(X)$.

For the ``furthermore" part \eqref{thm:Lipman-1}, the first formula follows from the second, since the global section functor is a pushforward. For the last two isomorphisms: $f_* \, \bQ_X \simeq \bQ_S \, f_*$ holds since it is right conjugate to the natural isomorphism $f^* \circ j_\qc \simeq j_\qc \circ f^*$, and $f^! \, \bQ_S \simeq f^!$ holds since it is right conjugate to the natural isomorphism $f_* \circ j_\qc \simeq f_*|_{\Dqc}$. 

For \eqref{thm:Lipman-2}, we deduce it from the global Grothendieck duality as follows: for any $\sE \in \Dqc$,
	\begin{align*}
		&\Hom_{\Dqc(S)}(\sE, f_* \, \ihom_X(\sF, f^!\sG))   \\
		& \simeq \Hom_{\Dqc(S)}(\sE, \bQ_S \, f_* \, \RsHom_(\sF, f^!\sG)) & (f_* \, \bQ_X \simeq \bQ_S \,f_*) \\
		& \simeq \Hom_{\D(\sO_S)} (\sE, f_* \RsHom_X(\sF, f^! \sG))	 & (j_\qc \dashv \bQ_S )\\
		& \simeq \Hom_{\D(\sO_S)} (\sE, \RsHom_S(f_* \sF,\sG))	& (\text{Grothendieck duality \cite[Theorem 4.2]{Lip}}) \\
		& \simeq \Hom_{\Dqc(S)}(\sE, \ihom_S(f_* \sF, \sG)) 	&  (j_\qc \dashv \bQ_S).
	\end{align*}
By Yoneda's lemma,  \eqref{thm:Lipman-2} is proved.  

For \eqref{thm:Lipman-3} and  \eqref{thm:Lipman-4}:  \eqref{thm:Lipman-3} holds since it is right conjugate to the projection formula isomorphism $p_1 \colon f_* E \otimes (\blank) \simeq f_*(E \otimes f^*(\blank))$ of \cite[Proposition 3.9.4]{Lip}, and \eqref{thm:Lipman-4} holds since it is right conjugate to the canonical isomorphism $f^*(E \otimes \blank) \xrightarrow{\sim} f^*(E) \otimes f^*(\blank)$ of \cite[Proposition 3.2.4]{Lip}. In particular, in all the above steps, we require no other condition except for the schemes and morphisms being quasi-compact and quasi-separated. Hence the theorem is proved.
\end{proof}

\begin{remark} The theorem, together with Theorem \ref{thm:BB}, implies that $\Dqc(X)$ is a {\em unital algebraic stable homotopy category} in the sense of \cite{HPS}. In particular, the triangulated category $\Dqc(X)$ is compactly generated, and the following objects of $\Dqc(X)$ coincide: (i) (strong) dualizable objects, (ii) compact objects, (iii) perfect complexes. This theorem enables us to apply the theory of unital algebraic stable homotopy category of \cite{HPS} to $\Dqc(X)$. 
\end{remark}

\begin{remark}
\label{rmk:ihomcomposition}
As a direct consequence of the theorem, for all $\sE, \sF, \sG \in \Dqc(X)$, the identity map of $\ihom_X(\sF,\sG)$ induces the {\em evaluation map} $\ihom_X(\sF,\sG) \otimes \sF \to \sG$ for any $\sF,\sG$. Moreover, the composition of evaluation maps 
	$$\ihom_X(\sF,\sG) \otimes \ihom_X(\sE,\sF) \otimes \sE \to \ihom_X(\sF,\sG) \otimes \sF \to \sG$$
 induces the {\em composition map} of internal homs $\ihom_X(\sF,\sG) \otimes \ihom_X(\sE,\sF) \to \ihom_X(\sE,\sG)$.
\end{remark}

\begin{remark} For any $f \colon X \to S$ between quasi-compact and quasi-separated schemes, $f^*$ and $f_*$ restricts to $f^* \colon \Dqc(S) \to \Dqc(X)$ (\cite[Proposition 3.9.1]{Lip}) and $f_* \colon \Dqc(X) \to \Dqc(S)$ (\cite[Proposition 3.9.2]{Lip}). Since the symmetric monoidal structure $\Dqc(X)^{\otimes }$ is inherited from that of $\D(\sO_X)^{\otimes}$, the compatibility conditions \cite[(3.6.10)]{Lip} for $\D(\sO_X)$ (mentioned in the ``moreover" part of Example \ref{eg:D(O_X)^otimes}) are also true for $\Dqc(X)$. If we use the terminology of \cite[(3.6.10)]{Lip}, then we obtain: let $\mathbf{S}$ be the category of quasi-compact and quasi-separated schemes. For any $X \in \mathbf{S}$, we let $\mathbf{X}^* = \mathbf{X}_*$ denote the closed symmetric monoidal category $\Dqc(X)$ with product $\otimes$, unit $\sO_X$ and internal hom $\ihom_X$. For any morphism $f \colon X \to Y$ in $\mathbf{S}$, we let $f^* \colon \mathbf{Y}^* \to \mathbf{X}^*$ and $f_* \colon \mathbf{X}_* \to \mathbf{Y}_*$ denote the derived pullback and pushforward functors (Theorem \ref{thm:Neeman-Lipman}), respectively. For a composable pair of maps $X \xrightarrow{f} Y \xrightarrow{g} Z$  in $\mathbf{S}$, there are functorial isomorphisms $(g\circ f)_* \simeq g_* \circ f_*$ and $(g\circ f)^* \simeq f^* \circ g^*$. Then the above collection of data defines an adjoint pair $({}^*, {}_*)$ of monoidal $\Delta$-pseudofunctors on $\mathbf{S}$ (see \cite[(3.6.10)\&(4.2.3b)]{Lip} for more details about adjoint pair of monoidal $\Delta$-pseudofunctors).
\end{remark}

\begin{definition}\label{def:shom} For a morphism $f \colon X \to S$ of quasi-compact, quasi-separated schemes, $\Dqc(X)$ is naturally enriched over $\Dqc(S)$ with {\em $\Dqc(S)$-valued $\Hom$-object} given by:
	$$\shom_S(\sF,\sG) : = f_* \, \ihom_X(\sF,\sG) \in \Dqc(S), \qquad \forall \sF,\sG \in \Dqc(X).$$
By virtue of Theorem \ref{thm:Lipman}, there are functorial isomorphisms
	\begin{align*}
	\Hom_{\Dqc(X)}(f^*(\sE) \otimes \sF, \sG) & \xrightarrow{\simeq} \Hom_{\Dqc(X)}(f^*\sE, \ihom_X(\sF, \sG)) \\
	& \xrightarrow{\simeq} \Hom_{\Dqc(S)}(\sE, f_* \ihom_X(\sF,\sG)) \\
	&=\Hom_{\Dqc(S)}(\sE, \sHom_S(\sF,\sG))
	\end{align*}
for all $\sE \in \Dqc(S)$, $\sF, \sG \in \Dqc(X)$ which exhibits $\sHom_S(\sF,\sG)$ is the mapping object for $\sF$ and $\sG$.
There is a natural composition map $\shom_S(\sF,\sG) \otimes \shom_S(\sE,\sF) \to \shom_S(\sE,\sG)$ induced from the composition map of $\ihom_X$ of Remark  \ref{rmk:ihomcomposition} and the natural map $f_* A \otimes f_* B \to f_*(A \otimes B)$ since $f_*$ is lax symmetric monoidal (\cite[(3.2.4.2)]{Lip}). 
\end{definition}

As a corollary, we obtain the $S$-linear version Grothendieck--Serre duality:

%%% Corollary: Grothendieck duality
\begin{corollary}[Grothendieck--Serre Duality]\label{cor:GD} For a morphism $f \colon X \to S$ of quasi-compact, quasi-separated schemes, any $A \in \Perf(X)$ and $B \in \Dqc(X)$, denote $\omega_f = f^!\sO_S$ as in Theorem \ref{thm:Neeman-Lipman}. Then there is a functorial isomorphism in $\Dqc(S)$:
	$$\shom_S(B,  A \otimes \omega_f) \xrightarrow{\sim}  \ihom_S(\shom_S(A,B), \sO_S).$$
\end{corollary}
\begin{proof} Apply Theorem \ref{thm:Lipman} \eqref{thm:Lipman-2} to the case $F = \ihom_X(A,B)$ and $G= \sO_S$. 
\end{proof}

\begin{remark}[Quasi-Perfect Cases] 
\label{rmk:quasi-perfect:hom_S}
If $f$ is quasi-perfect (see Remark \ref{rmk:quasi-perfect}; for example, if $f$ is proper and perfect), then $f_*$ preserves perfect complexes and
	$$\shom_S(A,B) = f_* \RsHom_X(A,B) \in \Perf(S) \qquad \text{for all $A, B \in \Perf(X)$}.$$
Consequently, the Grothendieck--Serre duality implies functorial equivalences 
	$$f_*\RsHom_X(B, A \otimes \omega_f) \simeq (f_* \RsHom_X(A,B))^\vee \in \Perf(S) \qquad \text{for all $A, B \in \Perf(X)$}.$$
\end{remark}

 \subsection{Linear categories} \label{sec:linear:cat}
\begin{definition} Let $f \colon X \to S$ be a morphism of schemes. A triangulated subcategory $\shD \subseteq \Dqc(X)$ is called {\em $S$-linear} if it is stable with respect to tensoring by pullbacks of perfect complexes on $S$, that is, $A \otimes f^*P \in \shD$ holds for any $A \in \shD$, $P \in \Perf(S)$.  A semiorthogonal decomposition \eqref{eqn:SOD:D} is called $S$-linear if all of its components are $S$-linear.
\end{definition}

\begin{lemma}[{\cite[Lem. 2.7]{Kuz11}}]\label{lem:linearSO} Let $f \colon X \to S$ be a morphism of quasi-compact quasi-separated schemes. A pair of $S$-linear subcategories $\shA, \shB \subseteq \Dqc(X)$ is semiorthogonal (i.e. $\Hom_X(B, A) = 0$ for any $A \in \shA$, $B \in \shB$) iff $\shom_S(B,A) = 0$ any $A \in \shA$, $B \in \shB$.
\end{lemma}

\begin{proof} The ``if" direction follows by taking $H^0\bR\Gamma(S, \blank)$ to $\shom_S(B,A)$ in the view of Thm. \ref{thm:Lipman} \eqref{thm:Lipman-1}. For any $A,B \in \Dqc(X)$ and $P \in \Perf(S)$, from the adjunction $\otimes \dashv \ihom$ we have:
	\begin{align*}
	\Hom_S(P, f_* \ihom_X(B,A)) = \Hom_X(f^*P, \ihom_X(B,A)) = \Hom_X(f^* P \otimes B, A) = 0.
	\end{align*}
(The last equality holds by $S$-linearity of $\shB$.) Since $\Dqc(S)$ is compactly generated by $P \in \Perf(S)$, $\shom_S(B,A) =  f_* \ihom_X(B,A) =0$. This proves the ``only if" direction.
\end{proof}

\begin{definition} Let $f \colon X \to S$ and $g \colon Y \to S$ be morphisms of schemes, and let $\shA \subseteq \Dqc(X)$ and $\shB \subseteq \Dqc(Y)$ be $S$-linear triangulated subcategories. An exact functor $F \colon \shA \to \shB$ is called {\em $S$-linear} if it preserves the action of $\Perf(S)$, i.e., for any $A \in \shA$, $P \in \Perf(S)$, there is a natural isomorphism $F(A \otimes f^* P) \simeq F(A) \otimes g^*P \in \shB$.
\end{definition}

By definition, if a subcategory is an $S$-linear subcategory, then the inclusion functor is an $S$-linear functor. If a semiorthogonal decomposition \eqref{eqn:SOD:D} is $S$-linear, then the projection functors $\pr_i$ are $S$-linear for all $i$; see \cite[Lem. 2.8]{Kuz11}.

\begin{lemma} \label{lem:linearLR} Let $f \colon X \to S$ and $g \colon Y \to S$ be morphisms of quasi-compact, quasi-separated schemes, and let $\shA \subseteq \Dqc(X)$ and $\shB \subseteq \Dqc(Y)$ be $S$-linear triangulated subcategories. Assume $L \colon \shA  \to \shB$ and $R \colon \shB \to \shA$ are exact functors.
	\begin{enumerate}[leftmargin=*]
		\item \label{lem:linearLR-1} If $L \dashv R$ is an adjoint pair, then $L$ is $S$-linear iff $R$ is $S$-linear;
		%$L \colon \shA \rightleftarrows \shB \colon R$ 
		\item  \label{lem:linearLR-2} If $L$ is $S$-linear, then $L \dashv R$ iff for any $A \in \shA$, $B \in \shB$, there is a functorial isomorphism 
			$$\shom_S(LA, B) \simeq \shom_S(A, RB)$$
		to which the application of the functor $H^0 \bR \Gamma(S, \blank)$ produces the usual adjunction isomorphism $\Hom(LA, B) \simeq \Hom(A,RB)$.
	\end{enumerate}
\end{lemma}
\begin{proof} Let $P\in \Perf(S)$. \eqref{lem:linearLR-1} follows easily from definition, since $f^* P$ and $g^* P$ are (strong) dualizable. For \eqref{lem:linearLR-2}, the ``if" part is trivial by Thm. \ref{thm:Lipman} \eqref{thm:Lipman-1}; For ``only if" direction:
	\begin{align*}
	& \Hom_S(P, \shom_S(LA, B))  \\
	  & \simeq 	\Hom_Y(g^*P, \ihom_Y(LA,B))
	  \simeq	 \Hom_Y(g^* P \otimes LA, B)  & (g^* \dashv g_* \,\, \text{and} \, \,\otimes \dashv \ihom_Y)\\
	 & \simeq  \Hom_Y(L(f^*P \otimes A), B) 	& (\text{$L$ is $S$-linear} ) \\
	  & \simeq \Hom_X(f^*P \otimes A, RB) 		& (L \dashv R)\\
	 & \simeq \Hom_X(f^*P, \ihom_X(A,RB)) 
	 \simeq \Hom_S(P, \shom_S(A, RB)) 	&(\otimes \dashv \ihom_X \,\, \text{and}\,\, f^* \dashv f_*).
	\end{align*}
Since $\Dqc(S)$ is compactly generated by $P \in \Perf(S)$, the lemma is proved.\end{proof}

% Relative Serre functor
\subsection{Relative Serre functors} \label{sec:Serre}
In this subsection fix a morphism $f \colon X \to S$ between quasi-compact, quasi-separated schemes and an $S$-linear admissible subcategory $\shD \subseteq \Dqc(X)$.

\begin{definition}\label{def:proper}
The $S$-linear category $\shD$ is said to be {\em proper over $S$} if for any $D_1, D_2 \in \shD$, the Hom-object  (Definition \ref{def:shom}) $\shom_S(D_1,D_2) \in \Perf(S)$. 
\end{definition}

\begin{example}%[{\cite[Example ]{HTT}}]
If $f \colon X \to S$ is proper and perfect, then any $S$-linear admissible subcategory $\shD \subseteq \Perf(X)$ is proper over $S$ (Theorem \ref{thm:Neeman-Lipman} \eqref{thm:Neeman-Lipman-2i}). In particular, if $S = \Spec \kk$ for a field $\kk$, and $X$ is a proper, finitely presented scheme over $\kk$, then $\Perf(X)$ is proper over $\kk$.
\end{example}

%[{\cite[Definition ]{HA}}] We say that $\shD$ is {\em proper over $S$} if $\shD$ is compactly generated, the tensor action of $\Perf(S)$ restricts to a morphism on compact objects:	$$\Perf(S) \otimes \shD^c \to \shD^c$$ and that for any pair of compact objects $D_1, D_2 \in \shD^c$, $\sHom_S(D_1, D_2) \in \Perf(S)$.  is a perfect complex on $S$. For example, if $X \to S$ is proper and perfect, then $\Dqc(X)$ is proper over $S$. \end{definition}

% definition: relative Serre functor
\begin{definition}\label{def:Serre} Let $\shD$ be a proper $S$-linear category (Definition \ref{def:proper}). A {\em relative Serre functor} for $\shD$ over $S$ is an $S$-linear autoequivalence $\SS_{\shD/S} \colon \shD \xrightarrow{\simeq} \shD$, such that for any $A, B \in \shD$, there is a functorial isomorphism:
	$$\shom_S(B,  \SS_{\shD/S} (A)) \xrightarrow{\sim} \shom_S(A,B)^\vee,$$
where $\shom_S(\blank, \blank)$ is the $\Hom$-object of Definition \ref{def:shom}.
\end{definition}
If a relative Serre functor exists, then it is unique up to canonical isomorphisms. The next result shows the existence in geometric situations:

% Proposition relative Serre for schemes
\begin{proposition}\label{prop:Serre} Let $f \colon X \to S$ be a quasi-perfect (e.g., proper and perfect) morphism between quasi-compact, quasi-separated schemes, and assume $\omega_f$ is invertible. Then $\Perf(X)$ admits a relative Serre functor over $S$ given by $\SS_{X/S}  = (\blank) \otimes \omega_f \colon \Perf(X) \simeq \Perf(X)$. 
\end{proposition}
\begin{proof} This follows from Grothendieck--Serre duality Corollary \ref{cor:GD}. 
%The condition is in fact optimal in the sense that: the quasi-perfection is equivalent to $f^! = f^* \otimes \omega_f$ \cite[Proposition 2.1]{LN}, and $\omega_f$ being invertible is equivalent to $\SS_{X/S}=(\blank) \otimes \omega_f$ is an autoequivalence of $\Perf(X)$. 
\end{proof}

Recall that a triangulated subcategory $\shA \subseteq \shD$ is called {\em $\infty$-admissible} if it is admissible and all its iterated right and left orthogonals are admissible. 
		
\begin{lemma}\label{lem:Serre} Let $\shD$ be a proper $S$-linear category which admits a relative Serre functor $\SS_{\shD/S}$ over $S$ (Definition \ref{def:Serre}). 
%	\begin{enumerate}[label=(\alph*),ref=\alph*]
%		\item \label{lem:Serre-a} $\shD$ is proper, i.e., for any $D_1, D_2 \in \shD$, $\shom_S(D_1,D_2) \in \Perf(S)$;
%		\item \label{lem:Serre-b} $\shD$ has a relative Serre functor $ \SS_{\shD/S} \colon \shD \to \shD$ over $S$.
%	\end{enumerate}
Let $\shD = \langle \shA, \shB \rangle$ be an $S$-linear semiorthogonal decomposition. Then:% the following holds:
	\begin{enumerate}%[leftmargin=*]
		\item \label{lem:Serre-1} $\shA$ is admissible iff $\shB$ is admissible. If this happens, then both $\shA$ and $\shB$ are $\infty$-admissible,  $\shA^{\perp \perp} = \SS_{\shD/S} (\shA)$ and ${}^{\perp \perp} \shA = \SS_{\shD/S}^{-1} (\shA)$, and similarly for $\shB$.
		\item \label{lem:Serre-2} If $\shA \subseteq \shD$ is admissible, then $\shA$ also admits a relative Serre functor $\SS_{\shA/S} \colon \shA \simeq \shA$ over $S$, which is related to $\SS_{\shD/S}$ via the following isomorphisms:
			$$\SS_{\shA/S} \simeq i_{\shA}^! \circ \SS_{\shD/S} \circ i_{\shA} \quad \text{and} \quad \SS_{\shA/S}^{-1} \simeq  i_{\shA}^* \circ \SS_{\shD/S}^{-1} \circ i_{\shA}.$$
		\item \label{lem:Serre-3} If $\shB \subset \shD$ is admissible, and we have $S$-linear admissible semiorthogonal decompositions $\shD = \langle \shA, \shB \rangle = \langle \shB, \shC \rangle$ i.e. $\shA = \shB^\perp$ and $\shC = {}^\perp \shB$, then the following holds:
			\begin{align*}
				& \SS_{\shD/S}\,|_{\, \shC} \simeq \SS_{\shA/S} \circ (\LL_{\shB} |_{\, \shC} )\simeq (\LL_{\shB}|_{\, \shC}) \circ \SS_{\shC/S} \colon \qquad  \shC \xrightarrow{\sim} \shA; \\
				& \SS_{\shD/S}^{-1}\, |_{\, \shA} \simeq \SS_{\shC/S}^{-1} \circ (\RR_{\shB} |_{\, \shA}) \simeq (\RR_{\shB}|_{\, \shA}) \circ \SS_{\shA/S}^{-1} \colon	  \qquad \shA \xrightarrow{\sim} \shC.
			\end{align*}
	In other words, Serre functor and mutation agrees up to autoequivalences.
	\end{enumerate}
\end{lemma}

\begin{proof} Since $\shD$ is proper, and $(\blank)^\vee \colon \Perf(S)^{\rm op} \xrightarrow{\sim} \Perf(S)$ is an equivalence, we obtain that for all $F, G \in \shD$,
	$\shom_S(F,G) = 0 \iff \shom_S(G, \SS_{\shD/S}(F)) =0 \iff \shom_S( \SS_{\shD/S}^{-1}(G), F) = 0.$
	
For \eqref{lem:Serre-1}, assume $\shB$ is admissible and let $\shD = \langle \shB, \shC \rangle$. Then from Lemma \ref{lem:linearSO}, we have:
	$$A \in \shA=\shB^\perp \iff  \SS_{\shD/S}^{-1}(A) \in \shC = {}^\perp \shB.$$
Hence the equivalence $\SS_{\shD/S}^{-1}$ induces $\shA \xrightarrow{\sim} \shC = \SS_{\shD/S}^{-1}(\shA)$. Since $\shC$ is right admissible, we obtain that $\shA$ is also right admissible, hence admissible. Consequently, $\shC = {}^{\perp \perp} \shA$ is admissible. Continuing the process, we see that $\shA$ and $\shB$ are $\infty$-admissible. The other case is similar. 

For \eqref{lem:Serre-2}, thanks to Lemma \ref{lem:linearLR} \eqref{lem:linearLR-2}, we could upgrade all $S$-linear adjunction isomorphisms to $\Perf(S)$-valued ones, hence for any $E,F \in \shA$, we have functorial isomorphisms
	\begin{align*}
		& \shom_S(F, i_{\shA}^! \, \SS_{\shD/S} \, i_{\shA} \, E)
		 \simeq \shom_S(i_{\shA} F, \SS_{\shD/S} \, i_{\shA} \,E) 
		 \simeq \shom_S(i_{\shA} E, i_{\shA} \, F)^\vee 
		 \simeq \shom_S(E,F)^\vee; \\
		& \shom_S(i_{\shA}^* \, \SS_{\shD/S}^{-1} \, i_{\shA} F, E) 
		\simeq 	\shom_S( \SS_{\shD/S}^{-1} \, i_{\shA} F, i_{\shA} E)  
		\simeq 	 \shom_S(i_{\shA} E, i_{\shA} \, F)^\vee 
		\simeq \shom_S(E,F)^\vee.
	\end{align*}
By assertion \eqref{lem:Serre-1}, $i_{\shA}^! \, \SS_{\shD/S} \, i_{\shA}$ and $i_{\shA}^* \, \SS_{\shD/S}^{-1} \, i_{\shA}$ are mutually inverse equivalence of categories, hence above computations show that there indeed the relative Serre functor and its inverse.

For \eqref{lem:Serre-3}, for any $A \in \shA$, $C \in \shC$, similar to \eqref{lem:Serre-2}, we have the $\Perf(S)$-valued isomorphisms:
	\begin{align*}
	& \shom_S(A, \SS_{\shA/S} \circ \LL_\shB (C)) 
	\simeq \shom_S(\LL_\shB(C), A)^\vee 
	\simeq \shom_S(C,  i_{\shA} A)^\vee \\
	& \simeq \shom_S(i_{\shA} A, \SS_{\shD/S} (C))
	\simeq  \shom_S(A, \SS_{\shD/S} (C)).
	\end{align*}
This shows $\SS_{\shA/S} \circ \LL_\shB|_{\, \shC}  \simeq  \SS_{\shD/S} \,|_{\, \shC}$; taking right adjoint we obtain $\RR_{\shB}|_{\, \shA}  \circ \SS_{\shA/S}^{-1} \simeq  \SS_{\shD/S}^{-1}|_{\shA}$.
	\begin{align*}
	& \shom_S(\SS_{\shC/S}^{-1} \circ \RR_{\shB}(A), C)
	\simeq \shom_S(C, \RR_\shB(A))^\vee 
	\simeq \shom_S(i_{\shC} \,C,  A)^\vee \\
	& \simeq \shom_S(\SS_{\shD/S}^{-1} A, i_{\shC} \, C) 
	\simeq \shom_S(\SS_{\shD/S}^{-1} A, C).
	\end{align*}
This shows $\SS_{\shC/S}^{-1} \circ \RR_{\shB}|_{\, \shA} \simeq \SS_{\shD/S}^{-1}|_{\, \shA}$. Taking left adjoint we obtain $\LL_{\shB}|_{\shC} \circ \SS_{\shC/S} \simeq  \SS_{\shD/S}|_{\shC}$.
\end{proof}

%\begin{remark} The condition \eqref{lem:Serre-a} holds if $\shD \subseteq \Perf(X)$ is an admissible $S$-linear subcategory for a scheme $X$ such that $f \colon X \to S$ is a quasi-perfect morphism of quasi-compact, quasi-separated schemes; \end{remark}

We have the following analogue of \cite[Proposition 1.10]{BK}:
\begin{lemma}[Cf. {\cite[Proposition 1.10]{BK}}]\label{lem:iterated_mutations}  
Let $\shD$ be a proper $S$-linear category which admits a relative Serre functor $\SS_{\shD/S}$ over $S$ (Definition \ref{def:Serre}) and 
%Let $\shD$ be an $S$-linear proper category satisfying the conditions \eqref{lem:Serre-a} and \eqref{lem:Serre-b} of Lemma \ref{lem:Serre}, 
let $\shA \subseteq \shD$ is an $S$-linear admissible subcategory. 
Let $i_{\shA} \colon \shA \to \shD$ the inclusion functor and $i_{\shA}^*$ its left adjoint as usual. We define an infinite sequence of $S$-linear exact functors as follows: for $n \in \ZZ$, 
	$$i_{(n)}  : =  \SS_{\shD/S}^{\circ (n)} \circ i_{\shA} \circ \SS_{\shA/S}^{\circ (-n)} \colon \shA \to \shD \quad \text{and} \quad i^{(n)} :=  \SS_{\shA/S}^{\circ (n)} \circ i_{\shA}^* \circ \SS_{\shD/S}^{\circ (-n)}  \colon \shD \to \shA.$$
(Here, if $n <0$, $S^{\circ (n)} = (S^{-1})^{\circ (-n)}$ denotes the $(-n)$-fold composition of the inverse Serre functor $S$.) Then $i_{(n)}$ is fully faithful for each $n \in \ZZ$, $i_{(0)} = i_{\shA}$, $i^{(0)} = i_{\shA}^*$, and there are adjoint sequences $i^{(n)} \dashv i_{(n)} \dashv i^{(n+1)}$ for all $n \in \ZZ$. Furthermore, for $n > 0$ (resp. $n<0$), the exact functor $i_{(n)} \colon \shA \hookrightarrow \shD$ agrees with the $n$-fold iterated left mutations $\shA \xrightarrow{\sim} \shA^{\perp (2n)} \hookrightarrow \shD$ (resp. $(-n)$-fold iterated right mutations $\shA \xrightarrow{\sim} {}^{\perp (-2n)} \shA \hookrightarrow \shD$ ).
\end{lemma}
\begin{proof}  The lemma for the cases $n =1$ and $n=-1$ is a direct consequence of Lemma \ref{lem:Serre}; The claim for general $n \in \ZZ$ then follows from induction.
\end{proof}

\begin{proposition}[{\cite[Proposition 4.8]{BK}}] \label{prop:Serre2sod}  
Let $\shD$ be a proper $S$-linear category and let $\shD = \langle \shA, \shB \rangle $ be an $S$-linear semiorthogonal decomposition such that $\shA$, $\shB$ are admissible with relative Serre functors $\SS_{\shA/S}$ and $\SS_{\shB/S}$, respecitvely. Then $\shD$ admits a relative Serre functor over $S$.
\end{proposition}

\begin{proof} Similar to \cite{BK}, we only need to show for any object $D \in \shD$, the functors of the form $\shom_S(D, \blank)^\vee$ and $\shom_S(\blank, D)^\vee$ are representable. We only need to consider the case $h(\blank)= \shom_S(D, \blank)^\vee \colon \shD \to \Perf(S)$; the other case is similar. We use $\shom_{\shA} = \shom_{\shA, S}$ to denote $\Perf(S)$-valued hom object for the category $\shA$; the notations $\shom_{\shB}$, $\shom_{\shD}$ are similarly defined. By assumption, $\shD = \langle \shA^\perp, \shA \rangle  = \langle  \shA, \shB \rangle = \langle \shB, {}^\perp \shB \rangle$, and we know that
	\begin{align*}
		&h|_{\shB} (\blank)= \shom_{\shB}(\blank, E), \qquad E = \SS_{\shB/S} \circ i_{\shB}^* (D) \in \shB; \\
		 &h|_{\shA} (\blank) = \shom_{\shA}(\blank, F), \qquad F = \SS_{\shA/S} \circ i_{\shA}^* (D) \in \shA;\\
		&\shom_{\shD}(A, E) = \shom_{\shA}(A, E_{\shA}) \quad \text{for all $A \in \shA$} , \qquad  E_{\shA} = i_{\shA}^! (E) \in \shA. 
	\end{align*}
Let $E_{\shA^\perp} = i_{\shA^\perp}^* (E)$, then there is a distinguished triangle $E_{\shA} \to E \to E_{\shA^\perp} \xrightarrow{[1]}$. Apply the functor $h(\blank)$ to the map $E_{\shA} \to E$, we obtain a map $\shom_{\shA}(E,E) \to \shom_{\shD}(E_{\shA}, F)$. Taking global section, the identity map $\id_E$ then corresponds to a canonical map $\gamma \colon E_{\shA} \to F$. Consider the composition $\delta \colon E_{\shA^\perp}[-1] \to E_{\shA} \xrightarrow{\gamma} F$, and let $X$ be the cone of $\delta$, then the same argument of \cite[Proposition 4.8]{BK} shows that $X$ represents the functor $h$.
\end{proof}

\begin{remark} For future reference, in the above proof we have shown that for any $D \in \shD$, the application of Serre functor $\SS_{\shD/S}(D)$ sits into an exact triangle:
	$$\SS_{\shA/S} \circ i_{\shA}^* (D)\to \SS_{\shD/S}(D) \to i_{\shA^\perp}^* \circ \SS_{\shB/S} \circ i_{\shB}^*(D) \xrightarrow{[1]}.$$
\end{remark}

\begin{proposition}[{\cite[Proposition 4.11, \& 4.12]{BK}}] \label{prop:Serresod}  Let $\shD$ be a proper $S$-linear category and let $\shD = \langle \shA_1, \ldots, \shA_n \rangle$ be an admissible $S$-linear semiorthogonal decomposition. Then $\shD$ admits a relative Serre functor over $S$ iff each $\shA_i$ admits a relative Serre functor over $S$. If this happens, then the semiorthogonal decomposition $\shA_\bullet$ is $\infty$-admissible.
\end{proposition}
\begin{proof} For the first statement, if $n=2$, then the ``if" part is Proposition \ref{prop:Serre2sod} and the ``only if" part follows from Lemma \ref{lem:Serre} \eqref{lem:Serre-2}; The general case follows from induction. The second statement follows from Lemma \ref{lem:Serre} \eqref{lem:Serre-1} and induction.
\end{proof}

%%%
\subsection{Base change of linear categories} \label{sec:bc}

% defn: bc
\begin{definition}[Tor-independent base-change] \label{def:bc} A base change $\phi: T \to S$ is called {\em Tor-independent with respect to a morphism $f: X \to S$} if the Cartesian square
	\begin{equation}\label{eqn:fiber}
	\begin{tikzcd}
	X_T \arrow{d}{}[swap]{f_T}  \arrow{r}{\phi_T} & X \arrow{d}{f} \\
	T\arrow{r}{\phi}  & S
	\end{tikzcd}
	\end{equation}
is Tor-independent, i.e. for all $t \in T$, $x \in X$, and $s \in S$ with $\phi(t) = s = f(x)$, $\Tor_i^{\sO_{S,\,s}}(\sO_{T,\,t}, \sO_{X,\,x}) = 0$ for all $i > 0$. 
A base change $\phi: T \to S$ is called \emph{Tor-independent with respect to a pair $(X, Y)$ of schemes over $S$}, if $\phi$ is Tor-independent with respect to the  morphisms $f: X \to S$, $g: Y \to S$ and $f \times_S g : X \times_S Y \to S$.
\end{definition}

\begin{remark}\label{rmk:Lipbc} By \cite[Thm. 3.10.3]{Lip}, a base-change $\phi: T \to S$ as above is Tor-independent iff the natural transformation $\phi^* \,f_* \to f_{T\,*} \, \phi_{T}^*: \Dqc(X) \to \Dpc(T)$ is  an isomorphism. The condition $\phi^* \,f_* \simeq f_{T\,*} \, \phi_{T}^*$ was taken as the definition of {\em faithful base-change} in \cite{Kuz11}. Consequently, if we assume \eqref{eqn:fiber} is Tor-independent, and all  schemes in \eqref{eqn:fiber} are quasi-compact and quasi-separated, then for any $F \in \Perf(X)$, $G \in \Dqc(X)$, the following holds:
	$$\shom_S(F, G)_T \simeq \shom_T(F_T, G_T) \in \Dqc(T),$$
where $\shom_S$ (resp. $\shom_T$) is the $\Dqc(S)$- (resp. $\Dqc(T)$-) valued $\Hom$-object of Def. \ref{def:shom}.
\end{remark}

\begin{lemma}[K\"unneth formula]\label{lem:Kunneth} Let $\phi: T \to S$ be a Tor-independent base-change with respect to a morphism $f: X \to S$ as in diagram \eqref{eqn:fiber}, and assume all the schemes are quasi-compact and quasi-separated. For any $F_1\in \Perf(X)$, $F_2 \in \Dqc(X)$, $G_1\in \Perf(T)$, $G_2 \in \Dqc(T)$, if we denote $F_i \boxtimes G_i := \phi_T^* F_i \otimes f_T^* G_i$ for $i=1,2$, then the following holds:
	$$\shom_S(F_1 \boxtimes G_1, F_2 \boxtimes G_2) \simeq \shom_S(F_1, F_2) \otimes \shom_S(G_1,G_2),$$
where $\shom_S$ is the $\Dqc(S)$-valued $\Hom$-object of Def. \ref{def:shom}.
\end{lemma}
\begin{proof} It is easy to see Tor-independence implies K\"unneth independence i.e. if we denote $h = f \circ \phi_T = \phi \circ f_T \colon X_T \to S$, then for any $A \in \Dqc(X)$ and $B \in \Dqc(T)$, there is a functorial isomorphism $h_* (\phi_T^* \otimes f_T^* B) \simeq \phi_* A \otimes f_* B$, see e.g. \cite[(3.10.3)]{Lip}. Take $A = \ihom_X(F_1,F_2)$ and $B = \ihom_T(G_1,G_2)$, then the lemma follows from Thm. \ref{thm:Lipman}. % (1), the commutativity of pushforward and coherator. 
\end{proof}

% 3 squares
\begin{lemma} \label{lem:3squares} In the following commutative diagram of fiber squares of  quasi-compact and quasi-separated schemes
\begin{equation*}
	\begin{tikzcd}
	X'' \ar{d}{f''} \ar{r}{\psi'} & X' \ar{d}{f'} \ar{r}{\psi} & X \ar{d}{f} \\
	S'' \ar{r}{\phi'}		& S' \ar{r}{\phi}		& S,
	\end{tikzcd}
\end{equation*}
if the right square is Tor-independent, then the outer square is Tor-independent if and only if the left square is Tor-independent.
\end{lemma}
\begin{proof} \cite[Lemma 3.10.3.2]{Lip} states Tor-independence is stable under compositions, this proves the ``if" part. The proof of the ``only if" part is similar to the proof of \cite[Lemma 2.25]{Kuz06}. More precisely, suppose the outer square is Tor-independent, hence by Rmk. \ref{rmk:Lipbc}, $\phi'^*\,\phi^*f_*\simeq f''_*\,\psi'^*\,\psi^*$. To show that left square is Tor-independent is equivalent to show $\phi'^*\,f'_*  \simeq f''_*\,\psi'^*$. 
By Lemma \ref{lem:generator}, $\Dqc(X')$ is compactly generated by elements of the form $f'^* A \otimes \psi^* B$ for $A \in \Perf(S')$ and $B \in \Perf(X)$. Moreover, there are functorial isomorphisms:
	\begin{align*}
	& f''_*\,\psi'^*(f'^* A \otimes \psi^* B) \\
	&\simeq f''_*(f''^*\phi'^* A \otimes \psi'^*\,\psi^* B)  &   (\text{$\psi'^*$ preserves $\otimes$}) \\
	& \simeq \phi'^*A \otimes  ( f''_*\,\psi'^*\,\psi^* B) & (\text{projection formula; Theorem \ref{thm:Neeman-Lipman} \eqref{thm:Neeman-Lipman-1iii})} \\
	& \simeq  \phi'^*A \otimes (\phi'^*\,\phi^*f_* B)  &(\text{Tor-indep. of outer square $\implies$} f''_*\,\psi'^*\,\psi^* \simeq \phi'^*\,\phi^*f_*) \\
	& \simeq \phi'^*(A \otimes \phi^*f_* B) & (\text{$\phi'^*$ preserves $\otimes$})   \\
	& \simeq \phi'^*(A \otimes f'_*\,\psi^* B) &(\text{Tor-indep. of right square $\implies$} \phi^*f_* \simeq f'_*\,\psi^*)\\
	& \simeq \phi'^*\,f'_* (f'^* A \otimes \psi^* B) & (\text{projection formula; Theorem \ref{thm:Neeman-Lipman} \eqref{thm:Neeman-Lipman-1iii})}).
	\end{align*}
Since $\phi'^*\,f'_*$ and $f''_*\,\psi'^*$ preserve direct sums, by Corollary \ref{cor:Adjoint.isom} we have $\phi'^*\,f'_*\simeq f''_*\,\psi'^*$.
\end{proof}

Recall a closed immersion of schemes $Z \hookrightarrow X$ is called Koszul-regular if locally the ideal $\sI_Z$ of $Z$ is generated by a Koszul-regular sequence, see \cite[\href{https://stacks.math.columbia.edu/tag/0638}{Tag 0638}]{stacks-project}.

\begin{lemma} \label{lem:bc:CM} In the situation of a Cartesian diagram \eqref{eqn:fiber}, and suppose $\phi \colon T \hookrightarrow S$ is a Koszul-regular closed immersion of codimension $r$, where $r>0$ is an integer. 
	\begin{enumerate}[leftmargin=*]
		\item \label{lem:bc:CM-1} The square \eqref{eqn:fiber} is Tor-independent iff the closed immersion $\phi_T \colon X_T \hookrightarrow X$  is also Koszul-regular of codimension $r$. In particular, if $X$ is locally {\em Cohen--Macaulay}, then the square \eqref{eqn:fiber} is Tor-independent iff ${\rm codim}_{X} (X_T) ={\rm codim}_{S}(T)$.
		\item \label{lem:bc:CM-2} Suppose \eqref{eqn:fiber} is Tor-independent. For $n \ge 0$, denote by $T^{(n)}$ the $n$-th infinitesimal neighborhood of $T$ inside $S$ (i.e., $T^{(n)}$ is the closed subscheme defined by the ideal $\sI_T^{n+1}$). Then $\sTor^{\sO_S}_i(\sO_{T^{(n)}}, \sO_X) =0$ for any $n \ge 0$, $i \ge 0$.
	\end{enumerate} 
\end{lemma}

\begin{proof} Since the problem is stalk-local, we may assume $f \colon X \to S$ is given by a local ring morphism $\varphi \colon (A,\fom) \to (B, \fon)$, and the Koszul-regular immersion $T \subseteq S$ is given by an ideal $I\subset A$ generated by an $A$-Koszul-regular sequence $\mathbf{x}=(x_1, \ldots, x_r)$, where $x_i \in \fom$, $r = {\rm codim}_S(T)$. Denote $y_i = \varphi (x_i) \in \fon$ the image of $x_i$ in $B$, then the closed immersion $X_T \subset X$ is given by the ideal $I B =(y_1,\ldots, y_r) \subset B$. For \eqref{lem:bc:CM-1}, since the Koszul complex $\shK_\bullet(\mathbf{x})$ (see \cite[\href{https://stacks.math.columbia.edu/tag/0621}{Tag 0621}]{stacks-project}) is a acyclic resolution of $A/I$, therefore:
	$$\Tor_{i}^A(A/I, B) = H_i(\shK_\bullet(\mathbf{x}) \otimes_A B) = H_i(\shK_\bullet(\mathbf{y})),$$
where $\shK_\bullet(\mathbf{y})$ is the Koszul complex of the sequence $\mathbf{y}=(y_1,\ldots, y_r)$. Then $H_i(\shK_\bullet(\mathbf{y}))=0$ for all $i \ne 0$ iff $\mathbf{y}$ is a $B$-Kosuzl-regular sequence. If $B$ is Cohen--Macaulay, then $H_i(\shK_\bullet(\mathbf{y}))=0$ for all $i \ne 0$ iff $\mathbf{y}$ is a $B$-regular sequence, iff ${\rm codim}_{B} (B/I B) = {\rm codim}_{X} X_T = r$.

For \eqref{lem:bc:CM-2}, we prove by induction that $\Tor^A_i(A/I^{n+1}, B) = 0$ for all $n \ge 0, i \ge 1.$ The base case $n=0$ is our Tor-independent assumption. Since $\mathbf{x}$ is Koszul-regular, hence it is quasi-regular (see \cite[\href{https://stacks.math.columbia.edu/tag/063C}{Tag 063C}]{stacks-project}), in particular $I^n/I^{n+1}$ is a finite free $A/I$-module for any $n \ge 1$. Hence $\Tor_i^A(I^n/I^{n+1},B)=\Tor_i^A(A/I, B)=0$ for all $i \ge 1, n \ge 1$. By induction and the short exact sequence $0 \to I^n/I^{n+1} \to  A/I^{n+1} \to A/I^n \to 0$, the lemma is proved. 
\end{proof}

%%% Prop: BC linear
\begin{proposition}[Base-change of linear perfect complexes categories]\label{prop:bclinear} Let $f \colon X \to S$, $\phi \colon T \to S$ be morphisms of quasi-compact, quasi-separated schemes, and let 
	$$\Perf(X) = \langle \shA_1, \shA_2, \ldots, \shA_n \rangle$$
be an $S$-linear semiorthogonal decomposition, and assume $\phi$ is Tor-independent with respect to $f$. Then there are $T$-linear semiorthogonal decompositions induced from base change:
	\begin{align*}
		\Perf(X_T) = \langle \shA_{1\, T},   \shA_{2\, T}, \ldots,  \shA_{n\, T}\rangle; \\
		\Dqc(X_T) = \langle  \hat{\shA}_{1\, T},   \hat{\shA}_{2\, T}, \ldots,  \hat{\shA}_{n\, T}\rangle
	\end{align*}
which are compatible with base-change functors $\phi_*$ and $\phi^*$. The base-change component $\shA_{i \, T}$ is defined as the minimal thick triangulated category subcategory which contains objects of the form $\phi^* F \otimes f^* G$ for $F \in \shA_i$ and $G \in \Perf(T)$, and the component $\hat{\shA}_{i \, T}$ is the minimal triangulated subcategory which is closed under direct sums and contains $\shA_{i \, T}$.
\end{proposition}		

\begin{proof} This is \cite[Prop. 5.1 \& 5.3]{Kuz11} which states the results for quasi-projective varieties over a field, and \cite[Lem. 3.15 (1)]{BLM+} which states the result under semi-separateness condition. The reason that the theorem holds in our stated generality is as follows:
	\begin{enumerate}[leftmargin=*]
		\item For any objects $F_i \in \shA_{i \, T}, F_j \in \shA_{j \, T}$ and $G, G' \in \Perf(T)$, $i < j$, the following holds thanks to Lem. \ref{lem:linearSO} and the $\Dqc(S)$-valued K\"unneth formula Lem. \ref{lem:Kunneth}:
			$$\shom_S(F_j \boxtimes G, F_i \boxtimes G) \simeq \shom_S(F_j, F_i) \otimes \shom_S(G,G') = 0.$$
		Hence by Lem. \ref{lem:linearSO}, this shows the semi-orthogonality of $\shA_{i \, T}$'s.
		\item The generation result follows from Lem. \ref{lem:generator}: $\Perf(X_T)$ coincides with the minimal thick (i.e. idempotent complete) triangulated subcategory of $\Dqc(X_T)$ which contains objects of the form $F \boxtimes G$ for all $F \in \Perf(X)$ and $G \in \Perf(T)$.
	\end{enumerate}
The reason that we could drop the semi-separateness condition of \cite[Lem. 3.15 (1)]{BLM+} is that we deduce the second step from Lem. \ref{lem:generator} instead of from \cite[Thm. 1.2]{BFN}, and the latter requires semi-separateness but the former does not. \end{proof}		

\subsection{Relative Fourier--Mukai transforms} \label{sec:relFM}
{\em In this subsection, all schemes are assumed to be quasi-compact, quasi-separated.} Notice that quasi-compactness and quasi-separateness are stable  under composition and arbitrary base-change, and any morphism between quasi-compact and quasi-separated are themselves quasi-compact and quasi-separated; see \cite{TT} for details. 
 
%Let $f \colon X \to S$ and $g \colon Y \to S$ be morphisms (in particular, they are quasi-compact and quasi-separated).
Following Bergh--Schn{\"u}rer \cite{BS}, we introduce the following definition:

\begin{definition}[{See \cite[Definition 3.3]{BS}}] \label{Def:FM:BS}
Let $X$ and $Y$ be schemes over a base scheme $S$ (in particular, they are all quasi-compact and quasi-separated).
A {\em relative Fourier--Mukai transform (datum) from $X$ to $Y$ over $S$} is a quadruple $\Omega= (K, p, q, \shK) \colon X \to Y$, where 
	$$
	\begin{tikzcd}[column sep = 3 em, row sep = 1.5 em]
		 & K \ar{dr}{q} \ar{dl}[swap]{p}& \\
		X  & & Y
	\end{tikzcd}
	$$
is a diagram of $S$-schemes and $\shK \in \Perf(K)$ is a perfect complex, such that $p$ and $q$ are proper and perfect, and the relative dualizing complex $\omega_q = q^!(\sO_Y)$ is a perfect complex. We will refer to $\shK$ as the {\em (Fourier--Mukai) kernel}. We say that a relative Fourier--Mukai transform $\Omega= (K, p, q, \shK) \colon X \to Y$ is {\em strong} if $\omega_p = p^!(\sO_Y)$ is a perfect complex. A relative Fourier--Mukai transform datum $\Omega = (K, p, q, \shK)$ over $S$ give rises to an exact functor 
	$$\Phi = \Phi(K,p,q,\shK) \colon \Dqc(X) \to \Dqc(Y) \qquad \sF \mapsto q_*(\shK \otimes p^*(\sF));$$
we will refer to $\Phi$ as the {\em (relative Fourier--Mukai) functor associated with $\Omega=(K, p, q, \shK)$
\footnote{The functor $\Phi$ is automatically $S$-linear in the sense of \S \ref{sec:linear:cat}. The $S$-linear functors obtained in this way are sometimes called {\em geometric $S$-linear functors} in literatures.}.}
We will also use notations $\Phi=\Phi_\shK$ to indicate the dependence on the kernel $\shK$.
\end{definition}

%Denote $p \colon X \times_S Y \to X$ and $q \colon X \times_S Y \to Y$ the natural projections. For a perfect complex $\shK \in \Perf(X \times_S Y)$, called {\em kernel}, we can associate an $S$-linear Fourier-Mukai functor: $\Phi (\blank) = q_*(\shK \otimes p^*(\blank)) \colon \Dqc(X) \to \Dqc(Y).$ It always admits a right adjoint functor:	$\Phi^R(\blank) = p_*(\shK^\vee \otimes q^!(\blank)) \colon \Dqc(X) \to \Dqc(Y).$
%We will also use notations $\Phi=\Phi_\shK$ and $\Phi^R=\Phi^R_\shK$ to indicate the dependence on $\shK$, or use $\Phi = \Phi(p,q,\shK)$ to indicate the dependence on the whole data $(p,q,\shK)$. Following \cite{BS}, we introduce the following concept:
	
The preceding setup of the theory of relative Fourier-Mukai transforms guarantees the existence of left and right adjoints (Proposition \ref{prop:FM}) and that the formulation of the Fourier--Mukai transforms (and their adjoints) is stable under any Tor-independent base change (Lemma-Definition \ref{lem-def:relFM_bc}).  
Furthermore, semiorthogonal decompositions obtained from relative Fourier--Mukai transforms satisfy base-change properties (Theorem \ref{thm:bc}) and flat descent (Theorem \ref{thm:fppf}).
%These properties, together with Lipman--Neeman's theory that we reviewed in \S \ref{sec:generalities:derived},
% in turn guarantee that 

\begin{remark}
There are several ways to relax the conditions in the preceding definition. Instead of requiring $p$ and $q$ to be proper and perfect, we can require them to be quasi-perfect (Remark \ref{rmk:quasi-perfect}); see Remark \ref{rmk:relFM_bc}. Moreover, instead of requiring $\shK$ to be perfect, we can require $\shK$ to be relative perfect with respect to one of $p$ and $q$ (\cite[{E}xpos{\'e} III \S 4]{SGA}). We find that the above definition sufficient for many applications that we have in mind, and we hope that our exposition is clear enough to allow for generalization when necessary.
\end{remark}

%By abuse of notations, we will sometimes call the functor $\Phi = \Phi(p,q, \shK) \colon  \Dqc(X) \to \Dqc(Y)$ a {\em relative Fourier--Mukai transform} over $S$.

We have the following slight generalization of \cite[Proposition 3.5]{BS} in our setting: % \&Rmk. 3.6

% Lem: FM
\begin{proposition}\label{prop:FM} 
Let $\Omega=(K, p,q, \shK)$ be a relative Fourier--Mukai transform $X\to Y$ over $S$, and let $\Phi = \Phi_{\shK} \colon \Dqc(X) \to \Dqc(Y)$ denote the associated relative Fourier--Mukai functor. 
\begin{enumerate}
	\item \label{prop:FM-1} 
	Then $\Phi_{\shK}$ admits both a left adjoint $\Phi_{\shK}^L$ and a right adjoint $\Phi_{\shK}^R$ given by the formula:
	\begin{align*}
	\Phi_{\shK}^L(\blank) &= p_!(\shK^\vee \otimes q^*(\blank)) =  p_{*} (\shK^\vee \otimes \omega_{p} \otimes q^*(\blank))  \colon \Dqc(Y) \to \Dqc(X)\\
	 \Phi_{\shK}^R(\blank) &=p_{*} (  \shK^\vee \otimes q^!(\blank)) = p_{*} (  \shK^\vee \otimes \omega_{q} \otimes q^*(\blank))  \colon \Dqc(Y) \to \Dqc(X).
	\end{align*}
	Moreover:
	\begin{itemize}
		\item All three functors $\Phi_{\shK}$, $\Phi_{\shK}^L$, and $\Phi_{\shK}^R$ preserve perfect complexes. Consequently,  $\Phi_{\shK}^L \dashv \Phi_{\shK} \dashv \Phi_{\shK}^R$ forms an adjoint sequence on the subcategories of perfect complexes. % $\Perf(X) \stackrel{}{\rightleftarrows} \Perf(Y)$.
		Moreover, $\Phi_{\shK}^L|_{\Perf(Y)}$ can be given by the formula $\Phi_{\shK}^L(\sF) = p_*(\shK \otimes q^*(\sF^\vee))^\vee$ for $\sF \in \Perf(Y)$.
		\item The functors $\Phi_{\shK}$ and $\Phi_{\shK}^R$ have finite cohomological amplitudes and preserve bounded pseudo-coherent complexes. Consequently, they induce an adjoint pair on the subcategories of  bounded pseudo-coherent complexes $\Phi_{\shK} \colon \Db(X) \stackrel{}{\rightleftarrows} \Db(Y) \colon \Phi_{\shK}^R$.
	\end{itemize} 
	\item \label{prop:FM-2} 
	If the relative Fourier--Mukai transform $\Omega=\Omega(K, p,q, \shK) \colon X \to Y$ is {\em strong}, then
	$$\Omega^L =(K, q,p,\shK^\vee \otimes \omega_p)\colon Y \to X \qquad \Omega^R= (K, q,p,\shK^\vee \otimes \omega_q) \colon Y \to X$$
	 are both {\em strong} relative Fourier--Mukai transforms from $Y$ to $X$ over $S$, whose associated Fourier-Mukai functors are the left and right adjoint functors $\Phi^L = \Phi_{\shK}^L$ and $\Phi^R =\Phi_{\shK}^R$ of $\Phi$ described in \eqref{prop:FM-1}, respectively. Moreover, all the three functors $\Phi_{\shK}^L, \Phi_{\shK}$ and $\Phi_{\shK}^R$ have finite cohomological amplitudes, preserve perfect complexes and (bounded) pseudo-coherent complexes
	%(We call $\Phi^L$ resp. $\Phi^R$ the {\em left} resp. {\em right} adjoint of $\Phi$.) 
%	In particular, . Hence the adjoint sequence $\Phi_{\shK}^L \dashv \Phi_{\shK} \dashv \Phi_{\shK}^R$ restricts to an adjoint sequence on the categories of perfect complexes $\Perf(X) \stackrel{}{\rightleftarrows} \Perf(Y)$ and the bounded pseudo-coherent categories $\Db(X) \stackrel{}{\rightleftarrows} \Db(Y)$.
\end{enumerate}
\end{proposition}

\begin{proof} The desired results follow from Theorem \ref{thm:Neeman-Lipman}. Concretely, the existence of left and right adjoints is a consequence of \eqref{thm:Neeman-Lipman-2}. The property of $\Phi^R(\blank) = p_*(\shK^\vee \otimes q^!(\blank))$ in assertion \eqref{prop:FM-1} follows from Theorem \ref{thm:Neeman-Lipman}\eqref{thm:Neeman-Lipman-2i}, and Theorem \ref{thm:Neeman-Lipman} \eqref{thm:Neeman-Lipman-2iv} implies $\Phi^L(\blank) = p_!(\shK^\vee \otimes q^*(\blank))$ preserves perfect complexes with the desired formula. The assertion \eqref{prop:FM-2} follows from Theorem \ref{thm:Neeman-Lipman} \eqref{thm:Neeman-Lipman-3} as in this case both the morphisms $p$ and $q$ satisfy the conditions of \eqref{thm:Neeman-Lipman-3}.
\end{proof}

\begin{definition}
\label{def:bcFM}
We say a base-change $\phi \colon T \to S$ is {\em Tor-independent} for a relative Fourier--Mukai transform $\Omega =(K, p,q, \shK) \colon X \to Y$ if the base change $\phi$ are Tor-independent for the structural morphisms $X \to S$, $Y \to S$, and $K \to S$ (Definition \ref{def:bc}).
\end{definition}

\begin{lemma-definition}[Base Change of Relative Fourier--Mukai Transforms] \label{lem-def:relFM_bc} 
Let $\Omega =(K, p,q, \shK)$ be a relative Fourier--Mukai transform $X \to Y$ over $S$ and $\phi \colon T \to S$ a {\em Tor-independent} base-change of $\Omega$ (Definition \ref{def:bcFM}). Then there is an induced relative Fourier--Mukai transform $\Omega_T = (K_T, p_T, q_T, \shK_T) \colon X_T \to Y_T$ over $T$, where 
%$(K_T, p_T, q_T)$ is the diagram of $T$-schemes
 	$$
	\begin{tikzcd}[column sep = 3 em, row sep = 1.5 em]
		 & K_T \ar{dr}{q_T} \ar{dl}[swap]{p_T}& \\
		X_T  & & Y_T
	\end{tikzcd}
	$$
is the diagram of $T$-schemes obtained from base change of the diagram in Definition \ref{Def:FM:BS} along $\phi$, and $\shK_T \in \Perf(K_T)$ is the (derived) pullback of $\shK \in \Perf(K)$ along the morphism $\phi_K \colon K_T \to K$. We will refer to the relative Fourier--Mukai transform $\Omega_T = (K_T, p_T, q_T, \shK_T)$ as the {\em base change of $\Omega =(K, p,q, \shK)$ along $\phi$}. Moreover, 
let  $\phi_X \colon X_T \to X$ and $\phi_Y \colon Y_T \to Y$ denote the morphisms induced by $\phi$, then there are natural isomorphisms of functors
	\begin{align*}
		\phi_Y^* \, \Phi_{\shK} \xrightarrow{\simeq} \Phi_{\shK_T} \, \phi_X^* 
		\qquad  
		\phi_Y^* \, \Phi_{\shK}^R \xrightarrow{\simeq}  \Phi_{\shK_T}^R \, \phi_X^* 
		\qquad 
		\phi_Y^* \,  \Phi_{\shK}^L \xrightarrow{\simeq}  \Phi_{\shK_T}^L \, \phi_X^*.
	\end{align*}
In addition, if $\Omega= (K, p,q, \shK)$ is strong, then $\Omega_T = (K_T, p_T,q_T, \shK_T)$ is strong. 
\end{lemma-definition}

\begin{proof} We verify that the conditions of Definition \ref{Def:FM:BS} are satisfied by $\Omega_T$:
	\begin{itemize}
		\item $p_T$ and $q_T$ are proper, perfect morphisms. This follows from the following facts: properness is stable under base change (\cite[\href{https://stacks.math.columbia.edu/tag/01W4}{Tag 01W4}]{stacks-project}), pseudo-coherence and relative perfectness are stable under Tor-independence base change, respectively (\cite[Page 233, Colollaire 1.10 \& Page 257, Colollaire 4.7.2]{SGA}).
		\item $\omega_{q_T}$ is perfect. By virtue of Theorem \ref{thm:Neeman-Lipman} \eqref{thm:Neeman-Lipman-2ii}, we have natural isomorphisms $\phi_K^* \, q^!(\sO_S) \simeq q_T^! \, \phi^*(\sO_S) = q_T^! (\sO_T)$. Consequently, $\omega_{q_T} \simeq \phi_K^* \omega_q$ is perfect.
		\item $\omega_{p_T}$ is perfect if $\Omega$ is strong. Similarly as above, $\omega_{p_T} \simeq \phi_K^* \omega_p$ is perfect if $\omega_p$ is perfect.
	\end{itemize}
It only remains to establish the three desired natural isomorphisms of functors. By virtue of Lemma \ref{lem:3squares}, the following squares are Tor-independent:
	$$
	\begin{tikzcd}%[column sep = 3 em, row sep = 3 em]
		K_T \ar{d}{p_T} \ar{r}{\phi_{K}}& K \ar{d}{p} \\
		X_T \ar{r}{\phi_X} & X
	\end{tikzcd}
	\qquad  \qquad
	\begin{tikzcd}%[column sep = 3 em, row sep = 3 em]
		K_T \ar{d}{q_T} \ar{r}{\phi_K}& K \ar{d}{q} \\
		Y_T \ar{r}{\phi_Y} & Y,
	\end{tikzcd}
	$$
where the horizontal arrows are the natural morphisms induced by the base change $\phi \colon T \to S$. By virtue of Tor-independent base change \eqref{thm:Neeman-Lipman-1ii} and the isomorphism $\omega_{q_T} \simeq \phi_K^* \omega_q$ established above, we obtain, for any $\sF \in \Dqc(Y)$, functorial isomorphisms
\begin{align*}
&\phi_X^*(\Phi_{\shK}^R(\sF)) = \phi_X^* p_*(\shK^\vee \otimes \omega_q \otimes q^*(\sF)) \xrightarrow{\simeq}  p_{T\,*} \phi_K^*(\shK^\vee \otimes \omega_q  \otimes q^*(\sF)) \\
&\simeq  p_{T\,*} (\phi_K^*(\shK^\vee) \otimes \phi_K^*(\omega_q) \otimes \phi_K^*q^*(\sF)) \simeq p_{T\,*} (\shK_T^\vee \otimes \omega_{q_T} \otimes q_T^* \phi_Y^*(\sF)) = \Phi_{\shK_T}^R(\phi_Y^*(\sF)).
\end{align*}
The other two desired natural isomorphisms are proved similarly.
%and $\omega_{p_T} \simeq \phi^* \omega_p$. %
%$\phi^* \Phi_{\shK} \simeq \Phi_{\shK_T} \phi^*$ follows from Tor-independent base-change \eqref{thm:Neeman-Lipman-1ii}, and the other two equivalences follow from  \eqref{thm:Neeman-Lipman-1ii} and the isomorphisms $\omega_{q_T} \simeq \phi^* \omega_q$ and $\omega_{p_T} \simeq \phi^* \omega_p$, respectively.
 \end{proof}
 
\begin{remark} \label{rmk:relFM_bc}
If we modify the theory of relative Fourier--Mukai transforms in Definition \ref{Def:FM:BS} by requiring $p$ and $q$ to be quasi-perfect (Remark \ref{rmk:quasi-perfect}) rather than proper and perfect, then this setup
 is also stable under Tor-independent base change because quasi-perfectness is stable under any Tor-independent base change (\cite[Proposition (4.7.3.1)]{Lip}). 
\end{remark}

We will need the following generalization of \cite[Proposition 2.44]{Kuz06}:

\begin{proposition}\label{prop:relFM_bc} 
Let $\Omega =(K, p,q, \shK)$ be a relative Fourier--Mukai transform $X \to Y$ over $S$ with associated functor $\Phi \colon \Dqc(X) \to \Dqc(Y)$, let $\phi \colon T \to S$ be a {\em Tor-independent} base-change of $\Omega$, and let $\Omega_T = (K_T, p_T, q_T, \shK_T)$ be the base-change of $\Omega$ along $\phi$ (Lemma--Definition \ref{lem-def:relFM_bc}) with associated functor $\Phi_T \colon \Dqc(X_T) \to \Dqc(Y_T)$. Let $\Phi^L$ and $\Phi^R$ (resp. $\Phi_T^L$ and $\Phi_T^R$) denote the left and right adjoints of $\Phi$ (resp. $\Phi_T$) as in Proposition \ref{prop:FM}, respectively. Then
\begin{enumerate}
	\item \label{prop:relFM_bc-1} 
		\begin{enumerate}
			\item \label{prop:relFM_bc-1i} 
			If $\Phi|\Perf(X) \colon \Perf(X) \to \Perf(Y)$ is fully faithful, then 
			$\Phi \colon \Dqc(X) \to \Dqc(Y)$ and 
			$\Phi_T \colon \Dqc(X_T) \to \Dqc(Y_T) $ is fully faithful. %  on $\Dqc(X_T)$ (hence on $\Perf(X_T)$ and $\Db(X_T)$);
			Conversely, if $\phi$ is faithfully flat and $\Phi_T|\Perf(X_T)$ is fully faithful, then $\Phi \colon \Dqc(X) \to \Dqc(Y)$ is fully faithful. 
			\item \label{prop:relFM_bc-1ii} 
			If $\Phi^R|\Perf(Y) \colon \Perf(Y) \to \Perf(X)$ is fully faithful, then $\Phi^R \colon \Dqc(Y) \to \Dqc(X)$ and $\Phi^R_T \colon \Dqc(Y_T) \to \Dqc(X_T) $ are fully faithful. Conversely, if $\phi$ is faithfully flat and $\Phi_T^R|\Perf(Y_T)$ is fully faithful, then $\Phi^R \colon \Dqc(Y) \to \Dqc(X)$ is fully faithful. 
			\item \label{prop:relFM_bc-1iii} 
			If $\Phi$ induces an equivalence $\Perf(X) \simeq \Perf(Y)$, then $\Phi_T$ induces an equivalence $\Dqc(X_T) \simeq \Dqc(Y_T)$ which restricts to equivalences $\Perf(X_T) \simeq \Perf(Y_T)$ and $\Db(X_T) \simeq \Db(Y_T)$ which are compatible with the natural inclusion functors.
		\end{enumerate}	
	\item \label{prop:relFM_bc-2} Let $\Omega' =(K', p',q',\shK') \colon \Dqc(X') \to \Dqc(Y)$ be another relative Fourier-Mukai transform $X' \to Y$ over $S$ with associated functor $\Phi'$ such that $\phi \colon T \to S$ is Tor-independent for $\Omega'$, and let $\Omega_T'$ denote the base change of $\Omega'$ along $\phi$. If $\Im (\Phi'|_{\Perf(X')}) \subseteq \Im (\Phi|_{\Perf(X)})^\perp$ in $\Perf(Y)$, then $\Im \Phi'_T \subseteq (\Im \Phi_T)^{\perp}$ in $\Dqc(Y_T)$ (and hence also in $\Perf(Y_T)$ and $\Db(Y_T)$).
\end{enumerate}
\end{proposition}

\begin{proof} %The proof of \cite[Proposition 2.44]{Kuz06} works in our generality with minor modifications. More precisely, for \eqref{prop:relFM_bc-1}, notice that $\Phi$ (resp. $\Phi^R$) %, $\Phi^L$ is fully faithful iff the natural transform $\id \to \Phi^R \Phi$ (resp. $\Phi\Phi^R \to \id$)
%, resp. $\id \to \Phi \Phi^L$ is an isomorphism; and similarly for the functors under base-change. We only show one of the cases. Assume $\id \simeq \Phi^R \Phi$, and we need to show $\id_{X_T} \simeq \Phi^R_T \Phi_T$. Since by Lemma \ref{prop:FM}\eqref{prop:FM-1}, both the functors $\id_{X_T}$ and $\Phi^R_T \Phi_T$ preserve direct sums, hence by Lemma \ref{lem:span:f.f.} we only need to show the isomorphism on a set of compact generators. 
For assertion \eqref{prop:relFM_bc-1i}: by virtue of Corollary \ref{CorollaryAdjoint.ff} \eqref{CorollaryAdjoint.ff-1}, we only need to verify $\Phi_T$ is fully faithful on the generating set of $\Dqc(X_T)$ consisting of elements of the form $\phi^* A \otimes f_T^* F$ for $A \in \Perf(X)$ and $F \in \Perf(T)$ (Lemma \ref{lem:generator}). It follows from Lemma-Definition \ref{lem-def:relFM_bc} and $T$-linearity that there are functorial isomorphisms
  	$$\Phi_T^R \Phi_T (\phi_X^* A \otimes f_T^* F) \simeq \Phi_T^R \Phi_T (\phi_X^* A) \otimes f_T^* F \simeq \phi_X^*(\Phi^R \Phi(A)) \otimes f_T^* F \simeq \phi_X^* A \otimes f_T^* F.$$
Conversely, if $\phi$ if faithfully flat, then $\phi_X^*$ is conservative, and hence (taking $F = \sO_T$) the above functorial isomorphisms $\Phi_T^R \Phi_T (\phi_X^* A) \simeq   \phi_X^*(\Phi^R \Phi(A)) \simeq \phi_X^*(A)$ imply functorial isomorphism $\Phi^R \Phi(A) \simeq A$ for all $A \in \Perf(X)$. By Corollary \ref{CorollaryAdjoint.ff} \eqref{CorollaryAdjoint.ff-1}, assertion \eqref{prop:relFM_bc-1i} is proved. Assertion \eqref{prop:relFM_bc-1ii} is proved in the same way. Assertion \eqref{prop:relFM_bc-1ii} follows similarly from Corollary \ref{CorollaryAdjoint.ff} \eqref{CorollaryAdjoint.ff-2}.

Assertion \eqref{prop:relFM_bc-2} is also proved in a similar manner. Observe that 
$\Im \Phi' \subseteq (\Im \Phi)^\perp$ (resp. $\Im \Phi'_T \subseteq (\Im \Phi_T)^{\perp}$) if and only if $\Phi^R \Phi' = 0$ (resp. $\Phi_T^R \Phi_T' = 0$).  For compact generators of $\Dqc(X'_T)$ of the form $\phi^* A' \otimes f_T'^* F$, $A' \in \Perf(X')$ and $F \in \Perf(T)$, by Lemma-Definition \ref{lem-def:relFM_bc} and $T$-linearity, we have
  	$$\Phi_T^R \Phi_T' (\phi^* A' \otimes f_T'^* F) \simeq \Phi_T^R ( \Phi_T' \phi^* A'  \otimes g_T^* F) \simeq \Phi_T^R ( \phi^* \Phi'(A'))  \otimes f_T^* F \simeq  \phi^*(\Phi^R \Phi'(A')) \otimes f_T^* F.$$
Hence $\Phi^R \Phi' = 0  \implies \Phi^R_T \Phi'_T =0$, by virtue of Corollary \ref{cor:Adjoint.isom}. Therefore, we are done.
 \end{proof}

One of the dominant features of our setup of the relative Fourier--Mukai transform theory is that it behaves well under Tor-independent base-change \cite{Kuz06, Kuz11}.

%Let $\Omega =(K, p,q, \shK)$ be a relative Fourier--Mukai transform $X \to Y$ over $S$ with associated functor $\Phi \colon \Dqc(X) \to \Dqc(Y)$, $\phi \colon T \to S$ a {\em Tor-independent} base-change of $\Omega$, and $\Omega_T = (K_T, p_T, q_T, \shK_T)$ the base-change of $\Omega$ along $\phi$ (Lemma--Definition \ref{lem-def:relFM_bc}) with associated functor $\Phi_T \colon \Dqc(X_T) \to \Dqc(Y_T)$.

%%% Base-change theorem
\begin{theorem}[Tor-independent Base-Change Theorem; see Kuznetsov {\cite[Theorem 2.46]{Kuz06}, \cite[Proposition 5.1 \& Theorem 5.6]{Kuz11}}] \label{thm:bc} Let $n \ge 2$ be an integer, and let $X_i \to S$ and $Y \to S$ be morphisms of quasi-compact, quasi-separated schemes, and let $\Omega_i =(K_i,p_i,q_i, \shK_i)$ be relative Fourier--Mukai transforms $X_i \to Y$ over $S$ with associated functors $\Phi_i \colon \Dqc(X_i) \to \Dqc(Y)$ where $i=1,\ldots, n$. Let $\phi \colon T \to S$ be a base-change which is Tor-independent with respect to $\Omega_i$ for all $i$. Let $\Omega_{i\,T} = (K_{i\,T}, p_{i\,T}, q_{i\,T}, \shK_{i\,T})$ be the relative Fourier--Mukai transforms $X_{i\, T} \to Y_T$ obtained from base change along $\phi$ (Lemma-Definition \ref{lem-def:relFM_bc}) and let $\Phi_{\shK_{i\,T}} \colon \Dqc(X_{i\,T}) \to \Dqc(Y_{T})$ denote the associated Fourier--Mukai functors. Assume that the restrictions of $\Phi_i$ to perfect complexes, $\Phi_i|_{\Perf(X_i)} \colon \Perf(X_i) \to \Perf(Y)$, are fully faithful for all $i$, and that there is an induced semiorthogonal decomposition
		$$\Perf(Y) = \langle \Phi_1(\Perf(X_1)),  \Phi_2(\Perf(X_2)), \ldots, \Phi_n(\Perf(X_n))\rangle.$$
Then the functor $\Phi_{i\,T} \colon \Dqc(X_{i\,T}) \to \Dqc(Y_T)$ is fully faithful for all $i$. (Consequently, the restrictions $\Phi_{i\,T}|_{\Perf(X_{i\,T})} \colon \Perf(X_{i\,T}) \to \Perf(Y_T)$ and $\Phi_{i \,T}|_{\Db(X_{i\,T})} \colon \Db(X_{i\,T}) \to \Db(Y_T)$ are also fully faithful for all $i$). Moreover, the essential images induce $T$-linear semiorthogonal decompositions with right admissible components:
		\begin{align*}
			\Perf(Y_T)  &= \langle \Phi_{1\,T}(\Perf(X_{1\,T})),  \Phi_{2\,T}(\Perf(X_{2\,T})), \ldots, \Phi_{n\,T}(\Perf(X_{n\,T}))\rangle, \\
			  \Db(Y_T) & = \langle \Phi_{1\,T}(\Db(X_{1\,T})), \, \Phi_{2\,T}(\Db(X_{2\,T})), \ldots, \Phi_{n\,T}(\Db(X_{n\,T}))\rangle,  \\
			   \Dqc(Y_T)  &= \langle \Phi_{1\,T}(\Dqc(X_{1\,T})),  \, \Phi_{2\,T}(\Dqc(X_{2\,T})), \ldots, \Phi_{n\,T}(\Dqc(X_{n\,T}))\rangle.
		\end{align*}		
Here, the first semiorthogonal decomposition is admissible, and these semiorthogonal decompositions are compatible with the inclusions $\Perf(X_{i\,T}) \subseteq \Db(X_{i\,T}) \subseteq \Dqc(X_{i\,T})$ and $\Perf(Y_T) \subseteq \Db(Y_T) \subseteq \Dqc(Y_T)$. Furthermore, if $\Phi_i$ are {\em strong} for all $i$, then $\Phi_{i \, T}$ are strong for all $i$, and all the above semiorthogonal decompositions are admissible.
\end{theorem}
	
\begin{proof} The results regarding fully-faithfulness and semiorthogonality follow directly from Proposition \ref{prop:relFM_bc}. It only remains to show fullness and compatibility. 

Denote $\shD$ the triangulated subcategory of $\Dqc(Y_T)$ generated by $\Perf(X_{1 \,T}), \ldots, \Perf(X_{n\, T})$. Then by using the decomposition for $\Perf(X)$ and the compatibility of these functors with base-change Lemma-Definition \ref{lem-def:relFM_bc}, $\shD$ contains all elements of the form $\phi^* A \otimes f_T^* F$ for $A \in \Perf(X)$ and $F \in \Perf(T)$. By Lemma \ref{lem:generator} these elements compactly generate $\Dqc(Y_T)$, hence $\shD^\perp=0$, and this shows the fullness of all these decompositions. (Alternatively, the generation for $\Perf(Y_T)$ could also follows from a similar argument as with Proposition \ref{prop:relFM_bc}.)

To show compatibility, denote $\shA_{i \,T} : = \Phi_{i \,T}(\Perf(X_{i \,T}))$, $\shA_{i \,T}^\b : = \Phi_{i \,T}(\Db(X_{i \,T}))$, and $\hat{\shA}_{i \,T} : = \Phi_{i \,T}(\Dqc(X_{i \,T}))$. Then $\shA_{i \,T} \subseteq \shA_{i \,T}^\b \subseteq \hat{\shA}_{i \,T}$. Denote by $\hat{\pr_i} \colon \Dqc(Y_T) \to \hat{\shA}_{i \,T}$ the projection functor, and by $\LL_i = \LL_{\hat{\shA}_{i \,T}}$ the left mutation functor through $\hat{\shA}_{i \,T}$, then it follows immediately from Lemma \ref{prop:FM} that $\hat{\pr_i} = (\Phi_{i \, T} \, \Phi_{i \, T}^R) \, \LL_{i+1} \ldots \LL_{n}$ preserves perfect complexes, pseudo-coherence and boundedness. Hence $\shA_{i \,T} = \hat{\shA}_{i \,T} \cap \Perf(Y_T)$ and $ \shA_{i \,T}^\b = \hat{\shA}_{i \,T}\cap \Db(Y_T)$. This shows compatibility. (Alternatively, similar to the proof of \cite[Theorem 6.4]{Kuz11}, since $\Phi_{i \,T}$ and $\Phi_{i\,T}^R$ preserve direct sums, $\hat{\shA}_{i \,T}$ coincides with the minimal subcategory of $\Dqc(Y_T)$ which is $T$-linear, triangulated, closed under direct sums and contains $\Phi_{i \,T}(\Perf(X_{i \, T})) = \shA_{i \,T}$ thanks to Lemma \ref{lem:generator}, hence our categories $\shA_{i \,T} \subseteq \shA_{i \,T}^\b \subseteq \hat{\shA}_{i \,T}$ coincide with the ones defined in \cite[Proposition 4.2 \& 4.3]{Kuz11} and compatibility follows.)
\end{proof}

%%% Conservative descent
Another key feature of the theory of relative Fourier--Mukai transforms is that it enjoys flat descent \cite{BS, BOR, AE}. The following result slightly generalizes Bergh--Schn{\"u}rer {\cite[Theorem 6.1\& 6.2]{BS} in the case of quasi-compact, quasi-separated schemes:

\begin{theorem}[Faithfully Flat Descent; Bergh--Schn{\"u}rer {\cite[Theorem 6.1\& 6.2]{BS}}] \label{thm:fppf} Let $X_i \to S$ and $Y \to S$ be morphisms of quasi-compact qusi-separated schemes, and let $\Omega_i = (K_i, p_i,q_i, \shK_i)$ be relative Fourier--Mukai transforms $X_i \to Y$ over $S$ with associated Fourier--Mukai functors $\Phi_i \colon \Dqc(X_i) \to \Dqc(X)$, where $i=1,\ldots, n$. Let $\phi \colon T \to S$ be a {\em faithfully flat} base change. Let 
	$$\Omega_{i\,T} = (K_{i\,T}, p_{i\,T}, q_{i\,T}, \shK_{i\,T}) \colon X_{i\, T} \to Y_T$$ 
be the base-change of $\Omega_i$ along $\phi$ (Lemma-Definition \ref{lem-def:relFM_bc}) and let $\Phi_{i\, T} \colon \Dqc(X_{i\,T}) \to \Dqc(Y_T)$ denote the associated Fourier--Mukai functors. 
	\begin{enumerate}[leftmargin=*]
		\item  \label{thm:fppf-1} 
		If $\Phi_{i \,T}|_{\Perf(X_{i\,T})}$ is fully faithful, then $\Phi_i$ is fully faithful (hence so are $\Phi_i|_{\Perf(X_i)}$ and $\Phi_i|_{\Db(X_i)}$). If $\Im (\Phi_{i \, T}|_{\Perf(X_{i\,T})}) \subseteq \Im (\Phi_{j \, T}|_{\Perf(X_{j\,T})})^{\perp}$, then $\Im \Phi_{i} \subseteq (\Im \Phi_{j})^{\perp}$.
		\item  \label{thm:fppf-2} 
		If $\Phi_{i \,T}|_{\Perf(X_{i\,T})}$ are fully faithful for all $i$ and induce a semiorthogonal decomposition
			$$\Perf(Y_T)  = \langle \Phi_{1\,T}(\Perf(X_{1\,T})),  \Phi_{2\,T}(\Perf(X_{2\,T})), \ldots, \Phi_{n\,T}(\Perf(X_{n\,T}))\rangle, $$
	then $\Phi_i \colon \Dqc(X_i) \to \Dqc(Y)$ (are fully faithful for all $i$, by the previous assertion, and their essential images) induce $S$-linear semiorthogonal decompositions:
		\begin{align*}
			\Perf(Y)  &= \langle \Phi_{1}(\Perf(X_{1})),  \Phi_{2}(\Perf(X_{2})), \ldots, \Phi_{n}(\Perf(X_{n}))\rangle, \\
			  \Db(Y) & = \langle \Phi_{1}(\Db(X_{1})), \, \Phi_{2}(\Db(X_{2})), \ldots, \Phi_{n}(\Db(X_{n}))\rangle, \\
			 \Dqc(Y)  &= \langle \Phi_{1}(\Dqc(X_{1})),  \, \Phi_{2}(\Dqc(X_{2})), \ldots, \Phi_{n}(\Dqc(X_{n}))\rangle.
		\end{align*}	
Here, the first semiorthogonal decomposition is admissible and the other two are right admissible, and the three semiorthogonal decompositions are compatible with the respective inclusions $\Perf(X_{i}) \subseteq \Db(X_{i}) \subseteq \Dqc(X_{i})$ and $\Perf(Y) \subseteq \Db(Y) \subseteq \Dqc(Y)$. Furthermore, if $\Phi_{i}$ are strong for all $i$, then all the above semiorthogonal decompositions are admissible.
	\end{enumerate}
\end{theorem}
\begin{proof} 
Assertion \eqref{thm:fppf-1} follows from Proposition \ref{prop:relFM_bc} \eqref{prop:relFM_bc-1i} and \eqref{prop:relFM_bc-1ii}. 
To prove assertion \eqref{thm:fppf-2}, we apply Theorem \ref{thm:bc} to the case of $\id \colon T \to T$ and obtain a semiorthogonal decomposition
	$$\Dqc(Y_T) = \Phi_{1\,T}(\Dqc(X_{1\,T})),  \Phi_{2\,T}(\Dqc(X_{2\,T})), \ldots, \Phi_{n\,T}(\Dqc(X_{n\,T}))$$
which is compatible with the given semiorthogonal decomposition of $\Perf(Y_T)$. Then we could apply \cite[Theorem 6.1\& 6.2]{BS}. %; Compare with Proposition \ref{prop:relFM_bc} and Theorem \ref{thm:bc}. 
%Notice all the properties on the data of relative Fourier--Mukai transforms are fppf local. 
\end{proof}

%Notice the theorem is especially useful when the base-change is given by a fppf covering $\{T_\alpha \to S\} $; All the properties of our definition for (strong) relative Fourier--Mukai transform -- $p_i,q_i$ being quasi-perfect, or proper and perfect (perfect morphisms are fppf local \cite[\href{https://stacks.math.columbia.edu/tag/069D}{Tag 069D}]{stacks-project}), $\omega_{q_i}$ (and $\omega_{p_i}$) being a perfect complex (being a perfect complex is fppf local, even fpqc local, see \cite[\href{https://stacks.math.columbia.edu/tag/09UG}{Tag 09UG}]{stacks-project}) -- could be checked fppf locally. % and the theorem states that whether a sequence of relative Fourier--Mukai transforms are fully faithful, resp. semiorthogonal, resp. give rise to a semiorthogonal decomposition or not could be checked fppf locally.

\begin{remark} %The theory of semiorthogonal decompositions enjoys faithfully flat descent, which allows one to globalize fppf-local results. 
There are various other versions of descent theory for semiorthogonal decompositions: Elagin \cite{Ela} showed that semiorthogonal decompositions satisfy descent along certain comonads (see also Shinder \cite{Shi}); Belmans--Okawa--Ricolfi \cite{BOR} and Antieau--Elmanto \cite{AE} independently showed that semiorthogonal decompositions satisfy fppf descent. %We follow Bergh--Schun{\"u}rer's conservative descent \cite{BS} as it fits best into the framework of this paper. 
\end{remark}

%\subsection{Convolutions of relative Fourier--Mukai kernels}
\subsection{Compositions of relative Fourier--Mukai transforms}\label{sec:relFM:cov}

\begin{definition}[Composability and Convolution] \label{defn:convolution} Let 
	$$\Omega=(K, p,q, \shK) \colon {X\to Y} \quad \text{and} \quad \Omega' = (K', p', q', \shK') \colon {Y \to Z}$$ 
be two relative Fourier--Mukai transforms over $S$. Then $\Omega$ and $\Omega'$ are said to be {\em composable} (over $S$) if the following fiber square is Tor-independent:
\begin{equation*} % \label{eqn:convolution_sq}
\begin{tikzcd} 
	K\times_Y K'  \ar{r}{q_{K'}} \ar{d}{p_{K}}& K' \ar{d}{p'}\\
	K \ar{r}{q} & Y.
\end{tikzcd}
\end{equation*}
In this case, their {\em composition} is defined to be the quadruple 
	$$\Omega' \circ \Omega = (K \times_Y K',  p \circ p_K,  q \circ q_{K'}, p_K^*(\shK) \otimes q_{K'}^* (\shK')) \colon X \to Z.$$ %, where $u \colon X\times_S Z \to X$ and $v \colon X \times_S Z \to Z$ are the natural projections, and the
%Assume that $\Omega$ and $\Omega'$ is composable, then 
A {\em convolution of $\Omega$ and $\Omega'$} is a relative Fourier--Mukai transform $\Omega'' = (K'', p'', q'', \shK'')$ together with a morphism $h \colon K \times_Y K' \to K''$ and an isomorphism 
	$\beta \colon \shK'' \xrightarrow{\simeq} h_*(p_K^*\shK \otimes q_{K'}^* \shK').$
 \end{definition}

By definition, the convolutions of $\Omega$ and $\Omega'$ are not unique. However,  different choices of convolutions give rise to isomorphic Fourier--Mukai functors (Lemma \ref{lem:conv:FM}). 
There is always a trivial convolution given by letting $\Omega'' = \Omega \circ \Omega'$, $h = \id$ and $\beta = \id$. In practice, there are usually more ``economic" choices of involutions than the trivial ones (see Example \ref{eg:conv:can}). 
\begin{remark} In the case where $K = X \times_S Y$ and $K' = Y \times_S Z$, the composability condition is always satisfied holds if $S = \Spec \kk$, where $\kk$ is a field. The composability condition naturally showed up in our study of relative HPD in \cite{JLX17}, and it worths noting that this condition is not automatically satisfied over a general base $S$.
%The composablity condition always holds if $S$ is a field, thus it never showed up in the study of absolute Fourier--Mukai transforms. This condition showed up naturally in our study of relative HPD in \cite{JLX17}, and it worths noting that it is not automatically satisfied over a general base $S$. In this paper, though, in all the places where we apply the next lemma this condition is always satisfied, so the readers need no to worry about it.
\end{remark}

\begin{example}\label{eg:conv:can}
 In the case where $\Omega= (X \times_S Y, p, q, \shK)$ and $\Omega' = (Y \times_S Z, r, s, \shL)$, a natural convolution is given by $\Omega'' = (X \times_S Z, u,v,  \shL * \shK)$, where $u \colon X\times_S Z \to X$ and $v \colon X \times_S Z \to Z$ are the natural projections, and $\shL * \shK$ is defined by
	$$ \shL * \shK = p_{XZ*}\, (p_{XY}^*\, \shK \otimes p_{YZ}^*\, \shL) \in \Perf(X \times_S Z),$$
 where $p_{XY} \colon X \times_S Y \times_S Z \to X \times_S Y$, $p_{YZ} \colon X \times_S Y \times_S Z \to Y \times_S Z$, and $p_{XZ} \colon X \times_S Y \times_S Z \to X \times_S Z$ are the natural projections. In the case where $S=\Spec \kk$, the above formula for $\shL * \shK$ is the usual definition of convolution of Fourier--Mukai kernels (see \cite[\S 5.1]{Huy}). 
 %Proposition 5.10). 
\end{example}

\begin{lemma}[Convolutions of Fourier--Mukai Functors] 
\label{lem:conv:FM} 
Let $\Omega=(K, p,q, \shK) \colon {X\to Y}$ and $\Omega' = (K', p', q', \shK') \colon {Y \to Z}$ be {\em composable} relative Fourier--Mukai transforms over $S$, and let 
	$$(\Omega'' = (K'', p'', q'', \shK''), h \colon K \times_Y K' \to K'', \beta \colon \shK'' \xrightarrow{\simeq} h_*(p_K^*\shK \otimes q_{K'}^* \shK'))$$ 
be a convolution of $\Omega$ and $\Omega'$ over $S$. Then there is a naturally induced isomorphism of Fourier--Mukai functors	
	$$\Phi_{\shK''}^{X \to Z} \xrightarrow{\simeq} \Phi_{\shK'}^{Y \to Z} \circ \Phi_{\shK}^{X \to Y}  \colon \Dqc(X) \to \Dqc(Z).$$
%Furthermore, if both $\Omega$ and $\Omega'$ are strong, then $\Omega''$ is also strong.
\end{lemma}

\begin{proof} 
Similar to the absolute case (see \cite[Proposition 5.10]{Huy}), let $u=p \circ p_K$ and $v = q_{K'} \circ q$, and consider the commutative diagram
\begin{equation*} %\label{diagram:convolution}
\begin{tikzcd}[back line/.style={}]		
		&	& K\times_Y K'  \ar{ld}[swap]{p_{K}} \ar{d}{h}\ar{dr}{q_{K'}}	 \ar[bend right]{ddll}[swap]{u}
		 \ar[bend left]{ddrr}{v} \\
		& K  \ar{ld}[swap]{p} \ar[back line]{rd}[swap]{q} & K'' \ar[crossing over]{lld}{p''} \ar[back line]{rrd}[swap]{q''} & K'   \ar[crossing over]{ld}{p'} \ar{rd}{q'} \\
	X	&	& Y &	& Z.
\end{tikzcd}
\end{equation*}
Then we have functorial isomorphisms for all $\sE \in \Dqc(X)$:
%it follows from projection formula \cite[Proposition 3.9.4]{Lip} that for any $A \in \Dqc(X)$,
\begin{align*} 
\Phi_{\shK''}^{X \to Z}(\sE) 
& = q''_* \,(\shK'' \otimes p''^*(\sE)) \\
& \xrightarrow[\simeq]{\beta} q''_*\, \big(h_*(p_K^*\shK \otimes q_{K'}^* \shK') \otimes p''^*(\sE)\big)  \\
&  \simeq q''_*\, h_*\big(p_K^*\shK \otimes q_{K'}^* \shK'  \otimes u^*\sE) \big) &(\text{projection formula \eqref{thm:Neeman-Lipman} \eqref{thm:Neeman-Lipman-1iii})}) \\
& = v_* \big(p_K^*\shK \otimes q_{K'}^* \shK'  \otimes u^*\sE \big) \\
& = q'_*\, q_{K' \,*} \big(p_K^*\shK \otimes q_{K'}^* \shK'  \otimes p_K^*\, p^* \sE \big) \\
	& \simeq q'_* \big(\shK'  \otimes q_{K' \,*}  p_K^*(\shK \otimes p^* \sE) \big) &(\text{projection formula \eqref{thm:Neeman-Lipman} \eqref{thm:Neeman-Lipman-1iii})})\\
	& \simeq  q'_* \big(\shK'  \otimes p'^* q_* (\shK \otimes p^* \sE) \big)  & (\text{Tor-independent base-change \eqref{thm:Neeman-Lipman} \eqref{thm:Neeman-Lipman-1ii})}) \\
	& =   \Phi_{\shK'}^{Y \to Z}  \circ \Phi_{\shK}^{X \to Y}  (\sE).
\end{align*}
\end{proof}

\begin{lemma} Let $\Phi_i= \Phi(X_i \times_S Y, p_i,q_i,\shK_i) \colon \Dqc(X_i) \to \Dqc(Y)$ be strong relative Fourier-Mukai transforms over $S$, $i=1,2$, and suppose that the functors $\Phi_1, \Phi_2$ are fully faithful and their images form a semiorthogonal pair $(\Im \Phi_1, \Im \Phi_2)$. Then the categories obtained from left and right mutations 
	$(\LL_{\Im \Phi_1} (\Im \Phi_2), \Im \Phi_1)$ and $(\Im \Phi_2, \RR_{\Im \Phi_2}(\Im \Phi_1))$
 are also $S$-linear semiorthogonal pairs. Suppose furthermore that the following fiber squares are Tor-independent 
\begin{equation*} % \label{eqn:convolution_sq}
\begin{tikzcd} 
	X_1 \times_S Y \times_S X_2  \ar{r} \ar{d} & X_2 \times_S Y  \ar{d}{p_2}\\
	X_1 \times_S Y \ar{r}{q_1} & Y,
\end{tikzcd} \qquad
\begin{tikzcd} 
	X_1 \times_S Y \times_S X_2  \ar{r} \ar{d} & X_1 \times_S X_2  \ar{d}\\
	X_i \times_S Y \ar{r} & X_i
\end{tikzcd}
\end{equation*}
for for $i=1,2$. Then the $S$-linear fully faithful functors
	\begin{align*}
		\Dqc(X_2) \xrightarrow{\Phi_2} \Dqc(Y) \xrightarrow{\LL_{\Im \Phi_1}} \Dqc(Y) \quad \text{resp.} \quad
		\Dqc(X_2) \xrightarrow{\Phi_1} \Dqc(Y) \xrightarrow{\RR_{\Im \Phi_2}} \Dqc(Y)	
	\end{align*}
are given by the strong relative Fourier--Mukai transforms $(p_2, q_2, L_{\shK_1}(\shK_2))$ and $(p_1, q_1, R_{\shK_2}(\shK_1))$ over $S$ respectively, where the kernels $L_{\shK_1}(\shK_2) \in \Perf(X_2 \times_S Y)$ and $R_{\shK_2}(\shK_1) \in \Perf(X_1 \times_S Y)$ fit into triangles in $\Perf(X_2 \times_S Y)$ and respectively in  $\Perf(X_1 \times_S Y)$:
	$$
	\shK_1 * (\shK_1^R * \shK_2) \to \shK_2 \to L_{\shK_1}(\shK_2) \xrightarrow{[1]}, \quad \text{resp.} \quad R_{\shK_2} (\shK_1) \to \shK_1 \to \shK_2 * (\shK_2^L * \shK_1) \xrightarrow{[1]},
	$$
where $\shK_1^R = \shK_1^\vee \otimes \omega_{q_1}$, $\shK_2^L = \shK_2^\vee \otimes \omega_{p_2}$, and $*$ is the convolution of kernels Def. \ref{defn:convolution}.
\end{lemma}
\begin{proof} This is a direct consequence of the properties of mutations and Lem. \ref{lem:conv:FM}.
\end{proof}

\subsection{Relative exceptional collections} \label{sec:relexc} In this subsection, we fix a map $f \colon X \to S$ between quasi-compact, quasi-separated schemes.

\begin{definition}[{\cite[Def. 3.19]{BLM+}}] Let $\shD \subseteq \Dqc (X)$ be a $S$-linear category. A {\em relative exceptional object in $\shD$} is an object $E \in \shD \cap \Perf(X)$ such that $\sO_S \xrightarrow{\sim} \shom_S(E,E)$. A {\em relative exceptional pair (resp. sequence) in $\shD$ }is a pair $(E_1, E_2)$ (resp. a sequence $(E_1, E_2, \ldots, E_n)$) of relative exceptional objects in $\shD$ such that $\sHom_S(E_i, E_j) = 0$ for all $i > j$.
\end{definition}

Notice that if $E \in \Perf(X)$, $F \in \Dqc(X)$, then their Hom-object can be expressed without using the coherators:
	$$\shom_S(E,F) \simeq f_* (E^\vee \otimes F) \simeq f_* \RHom_X(E,F) \in \Dqc(S).$$
If furthermore $f$ is quasi-perfect and $F \in \Perf(X)$, then  $\shom_S(E,F)  \in \Perf(S)$ (see Remark \ref{rmk:quasi-perfect:hom_S}). 
 The following is an analogue of {\cite[Lemma 3.23]{BLM+}} in our setting.

\begin{lemma}\label{lem:ff:relexc} If $f \colon X \to S$ is proper and perfect,
%quasi-perfect (e.g. proper and perfect), 
and $E \in \Perf(X)$ is a relative exceptional object over $S$. Then:
\begin{enumerate}[leftmargin=*]
	\item \label{lem:ff:relexc-1} The $S$-linear relative Fourier--Mukai functor
	$$\alpha_E (\blank) = f^*(\blank) \otimes E \colon \Dqc(S) \to \Dqc(X)$$
is fully faithful, 
%whose the essential image $\alpha_E (\Dqc(S)) \subseteq \Dqc(X)$ is a $S$-linear admissible subcategory, 
and admits a left adjoint $\alpha_E^L \colon \Dqc(X) \to \Dqc(S)$ and a right adjoint  $\alpha_E^R \colon \Dqc(X) \to \Dqc(S)$, given respectively by the formula:
	\begin{equation}\label{eqn:relexc:adjoints}
	\alpha_{E}^L(\blank) = f_{!}(E^\vee \otimes \blank) \qquad \text{and} \qquad \alpha_E^R(\blank) = f_{*} ( E^\vee \otimes \blank) = \sHom_S(E, \blank).
	\end{equation}
	\item \label{lem:ff:relexc-2} The restriction of $\alpha_E$ to perfect complexes category, $\alpha_E|_{\Perf(S)}  \colon \Perf(S) \to \Perf(X)$, is also fully faithful, %(still denoted by $\alpha_{E}$ by abuse of notations) 
	with left and right adjoints given respectively by the restrictions of \eqref{eqn:relexc:adjoints} to $\Perf(X)$. Furthermore, the left adjoint could also be expressed by the formula
		\begin{equation*}%{eqn:relexc:adjoints:perf}
		\alpha_{E}^L(\blank)|_{\Perf (X)} \simeq \shom_S(\blank, E)^\vee  \colon \Perf(X) \to \Perf(S).
		\end{equation*}
	\item \label{lem:ff:relexc-3} The restriction of $\alpha_{E}$ to bounded psuedo-coherent complexes $\alpha_E|_{\Db(S)} \colon \Db(S) \to \Db(X)$ is also fully faithful, and admits a right adjoint $\Db(X) \to \Db(S)$ given by the restriction of $\alpha_E^R$ of \eqref{eqn:relexc:adjoints}. If furthermore $\omega_f$ is a perfect complex, then $\alpha_E|_{\Db(S)} \colon \Db(S) \hookrightarrow \Db(X)$ also admits a left adjoint $\Db(X) \to \Db(S)$ given the restriction of $\alpha_{E}^L$ of \eqref{eqn:relexc:adjoints}.
	\end{enumerate}
	\end{lemma}
\begin{proof} 
The functor $\alpha_{E}$ is Fourier--Mukai functor associated with the relative Fourier--Mukai transform %(Definition \ref{Def:FM:BS}) 
$\Omega = (K=X, p=f, q=\id, \shK=E) \colon S \to X$ which is strong if $\omega_f$ is a perfect complex. By virtue of Proposition \ref{prop:FM}, $\Omega$ has left adjoint $\Omega^L = (X, \id, f, E^\vee \otimes \omega_f) \colon X \to S$ and right adjoint $\Omega^R = (X, \id, f, E^\vee) \colon X \to S$ with the desired properties. It only remains to prove that $\alpha_E$ is fully faithful; this is a consequent of $E$ being relative exceptional:  $\alpha_{E}^R \circ \alpha_{E} = f_*(E^\vee \otimes f^*(\blank) \otimes E) \simeq (\blank) \otimes \sHom_S(E,E) \simeq \Id$.
\end{proof}

\begin{definition} A relative exceptional sequence $(E_1, E_2, \ldots, E_n)$ is called a {\em full relative exceptional collection of $X$ over $S$} if the images $\alpha_{E_1}(\Perf(S)), \ldots, \alpha_{E_n}(\Perf(S))$ classically generate $\Perf(X)$, i.e. the right orthogonal of the images are zero.
\end{definition}

\begin{lemma} \label{lem:relexc} Let $f \colon X \to T$ be quasi-perfect, and let $(E_1, E_2, \ldots, E_n)$ be a relative exceptional sequence of $X$ over $S$. Then the $S$-linear subcategory 
	$$\langle \alpha_{E_1} (\Perf(S)), \alpha_{E_2}(\Perf(S)), \ldots, \alpha_{E_n}(\Perf(S)) \rangle \subset \Perf(X)$$
is admissible in $\Perf(X)$, and admits a relative Serre functor over $S$. (We call this subcategory the subcategory {\em spanned by $(E_1, \ldots, E_n)$}.) In particular, 
if $(E_1, \ldots, E_n)$ is a full relative exceptional collection, then there is an admissible $S$-linear semiorthogonal decomposition:
	$$\Perf(X) = \langle \alpha_{E_1} (\Perf(S)), \alpha_{E_2}(\Perf(S)), \ldots, \alpha_{E_n}(\Perf(S)) \rangle,$$
and $\Perf(X)$ itself admits a relative Serre functor over $S$; Furthermore this semiorthogonal decomposition is $\infty$-admissible.
\end{lemma}

\begin{proof} It suffices to observe that for each $i$, $\alpha_{E_i}(\Perf(S)) \subseteq \Perf(X)$ is admissible, and admits a relative Serre functor over $S$ given by identity. The rest follows from Prop. \ref{prop:Serresod}.
\end{proof}

The following is an immediate consequence of Thm. \ref{thm:bc}:

\begin{corollary}\label{cor:relexc} Let $f \colon X \to S$ be quasi-perfect, $(E_1, E_2, \ldots, E_n)$ a relative exceptional sequence of $X$ over $S$, and let $\phi \colon T \to S$ be a {\em Tor-independent} base-change of $f$ (this holds, for example, if either $f$ or $\phi$ is flat), where $T$ is a quasi-compact, quasi-separated scheme. Then $(E_{1 \,T}, E_{2\,T}, \ldots, E_{n\,T})$ is a relative exceptional sequence of $X_T=X \times_S T$ over $T$, where $E_{i \, T} \in \Perf(X_T)$ is the base-change of $E_i$. Moreover, if $(E_1, E_2, \ldots, E_n)$ is a full relative exceptional collection, then $(E_{1 \,T}, E_{2\,T}, \ldots, E_{n\,T})$ is also a full relative exceptional collection, and it induces  $T$-linear semiorthogonal decompositions with admissible components:
		\begin{align*}
			\Perf(X_T)  &= \langle \alpha_{E_{1\,T}}(\Perf(T)),  \alpha_{E_{2\,T}}(\Perf(T)), \ldots, \alpha_{E_{n\,T}}(\Perf(T)) \rangle, \\	
			\Dqc(X_T)  &= \langle \alpha_{E_{1\,T}}(\Dqc(T)), \alpha_{E_{2\,T}}(\Dqc(T)), \ldots, \alpha_{E_{n\,T}}(\Dqc(T)) \rangle,
		\end{align*}
and a $T$-linear semiorthogonal decomposition with right admissible components:
		\begin{align*}
			\Db(X_T)  = \langle \alpha_{E_{1\,T}}(\Db(T)),  \alpha_{E_{2\,T}}(\Db(T)), \ldots, \alpha_{E_{n\,T}}(\Db(T)) \rangle.
		\end{align*}
These semiorthogonal decompositions are compatible with the inclusions $\Perf(T) \subseteq \Db(T) \subseteq \Dqc(T)$ and $\Perf(X_T) \subseteq \Db(X_T) \subseteq \Dqc(X_T)$. If furthermore $\omega_f$ is a shift of line bundle, then the last semiorthogonal decomposition of $\Db(X_T)$ is also admissible.
\end{corollary}

Thus we will mainly focus on the perfect complexes; the corresponding statements about $\Dqc$ and $\Db$ then follows from above corollary applied to $T=S$.

On the other hand, the faithfully flat descent holds for relative exceptional collections.

\begin{corollary}\label{cor:relexc:descent} Let $f \colon X \to S$ be quasi-perfect, $E_i \in \Perf(X)$ for $i=1,\ldots, n$, and let $\phi \colon T \to S$ be a {\em faithfully flat} morphism between quasi-compact, quasi-separated schemes. Denote $E_{i \, T} \in \Perf(X_T)$ the base-change of $E_i$ along $T \to S$. If $E_{i \, T}$ is relative exceptional over $T$, then $E_i$ is relative exceptional over $S$. If $(E_{1 \, T}, \ldots, E_{n \, T})$ is a relative exceptional sequence (resp. a full relative exceptional collection) of $X_T$ over $T$, then $(E_{1}, \ldots, E_{n})$ is a relative exceptional sequence (resp. a full relative exceptional collection) of $X$ over $S$.
\end{corollary}

\begin{proof} By Rmk. \ref{rmk:Lipbc}, the pullback $\phi^*$ takes $\sO_S \to \shom_S(E_i, E_i)$ to $\sO_T \to \shom_T(E_{i \,T}, E_{i \, T})$. Since $\phi$ is faithfully flat, $\phi^*$ reflects isomorphisms, hence $\sO_T \simeq  \shom_T(E_{i \,T}, E_{i \, T})$ implies $\sO_T \simeq \shom_S(E_i, E_i)$. The rest follows from Thm. \ref{thm:fppf}.
\end{proof}

Let $E \in \Perf(X)$ be a relative exceptional object, then the {\em left (resp. right) mutation functor through $E$}, denoted by $\LL_E$ and $\RR_E$, are defined as the left and right mutation functors through the image of $\alpha_E$. Then it is clear that $E$ is relative exceptional iff $E[1]$ is, and the corresponding mutation functors through $E$ and through $E[1]$ are identical.

The following lemma is analogous to the absolute cases \cite{Bo, BK, Go}.

\begin{lemma} \label{lem:mut:relexc} Let $f \colon X \to S$ be quasi-perfect, $E \in \Perf(X)$ a relative exceptional object, and let $A \in \Perf(X)$ be any object.
	\begin{enumerate} [leftmargin=*]
	\item \label{lem:mut:relexc-1} The left and right mutations $\LL_E(A)$ and $\RR_E(A)$ fit into exact triangles:
			\begin{equation}\label{eqn:mut:relexc}
			f^* \shom_S(E,A) \otimes E \to A \to \LL_E(A), \qquad \RR_{E}(A) \to E \to f^* \shom_S(A,E)^\vee \otimes E.
			\end{equation}
		Hence in particular, $\shom_S(E, \LL_E(A)) =0$ and $\shom_S(\RR_E(A), E)=0$.
	\item \label{lem:mut:relexc-2} If $\shom_S(A,E) = 0$, then there is a bi-functorial isomorphism 
		$$\shom(\LL_E(A), E) \simeq \shom_S(E,A)^\vee[-1].$$
	% hence in particular $\RR_E \circ \LL_E (A) = A$. 
	Furthermore, for any $B \in \Perf(X)$, there are bi-functorial isomorphisms 
	 	$$\shom_S(A,B) \xrightarrow{\sim} \shom_S(A, \LL_E(B)) \xleftarrow{\sim} \shom_S(\LL_E(A), \LL_E(B)).$$
	\item \label{lem:mut:relexc-3} If $\shom_S(E,A) = 0$, then there is a bi-functorial isomorphism 
		$$\shom(E,\RR_E(A)) \simeq \shom_S(A,E)^\vee[-1].$$
	% hence in particular $\LL_E \circ \RR_E (A) = A$. 
	Furthermore, for any $B \in \Perf(X)$, there are bi-functorial isomorphisms 
	 	$$\shom_S(B,A) \xrightarrow{\sim} \shom_S(\RR_E(B), A) \xleftarrow{\sim} \shom_S(\RR_E(B), \RR_E(A)).$$
	\item \label{lem:mut:relexc-4} If $(E,F)$ is a relative exceptional pair, then $(\LL_E(F), E)$ and $(F, \RR_F (E))$ are also relative exceptional pairs. Moreover, the following holds: 
		\begin{align*}
		&\shom_S(\LL_E(F), E) \simeq \shom_S(E,F)^\vee[-1] \simeq \shom_S(F, \RR_F(E));\\
		&\shom_S(F, \LL_E(F)) \simeq \sO_S, \qquad \shom_S(\RR_F(E), E)\simeq \sO_S;\\
		&\RR_E \circ \LL_E (F) = F, \qquad \LL_F \circ \RR_F (E) = E; \\
		&\LL_{\LL_E(F)} \circ \LL_E = \LL_E \circ \LL_F, \qquad \RR_{\RR_F(E)} \circ \RR_F = \RR_F \circ \RR_E.
		\end{align*}
	\end{enumerate}
\end{lemma}
\begin{proof} The proof is similar to absolute cases. \eqref{lem:mut:relexc-1} follows from Lem. \ref{lem:mut} \eqref{lem:mut-1} and Lem. \ref{lem:ff:relexc} \eqref{lem:ff:relexc-2}. For \eqref{lem:mut:relexc-2}, the first statement follows from applying $\shom_S(\blank, E)$ to the first equation of \eqref{eqn:mut:relexc}, the isomorphism $\shom_S(A,B) \simeq \shom_S(A, \LL_E(B))$ follows from applying $\shom_S(A,\blank)$ to the triangle $f^* \shom_S(E,B) \otimes E \to B \to \LL_E(B)$, and the isomorphism $\shom_S(A, \LL_E(B)) \simeq \shom_S(\LL_E(A), \LL_E(B))$ follows from applying $\shom_S(\blank, \LL_E(B))$ to the first equation of (\ref{eqn:mut:relexc}). \eqref{lem:mut:relexc-3} is similar to \eqref{lem:mut:relexc-2}, and finally \eqref{lem:mut:relexc-4} follows easily from \eqref{lem:mut:relexc-2} and \eqref{lem:mut:relexc-3}.
\end{proof}

\begin{definition} Let $E_\bullet=(E_1, E_2, \ldots, E_n)$ be a relative exceptional sequence of $X$ over $S$. Then its {\em left dual exceptional sequence $D_\bullet$} is defined by setting
	$$D_1 = E_1, \qquad D_{i} = \LL_{E_1} \circ \ldots \circ \LL_{E_{i-1}} (E_i) \quad \text{for $2 \le i \le n$},$$
and its {\em right dual exceptional sequence $F_\bullet$} is defined by setting
	$$F_n = E_n, \qquad F_{i} = \RR_{E_{n}} \circ \ldots \circ \RR_{E_{i+1}} (E_i)  \quad \text{for $1 \le i \le n-1$}. $$
\end{definition}

 \begin{lemma} Let $f \colon X \to T$ be quasi-perfect and $E_\bullet = (E_1, E_2, \ldots, E_n)$ a relative exceptional sequence of $X$ over $S$. Then the left dual  $D_\bullet = (D_n ,\ldots, D_1)$ and right dual $F_\bullet =(F_1, \ldots, F_n)$ of $E_\bullet$ are both relative exceptional sequences of $X$ over $S$. Moreover, the relative exceptional sequences $E_\bullet$, $D_\bullet$ and $F_\bullet$ span the same $S$-linear admissible subcategory (in the sense of Lem. \ref{lem:relexc}), and enjoy the following ``relations of dual basis":
  	$$
	\shom_S(E_i, D_j) = \begin{cases} \sO_S & i = j; \\
	0 & i\ne j, \end{cases}
	\quad \text{resp.} \quad
	\shom_S(F_i, E_j) = \begin{cases} \sO_S & i = j; \\
	0 & i\ne j. \end{cases}
	$$
Furthermore, if $E_\bullet$ is full, then its left dual $D_\bullet$ and right dual $F_\bullet$ are both full, and they are respectively uniquely determined by the above ``relations of dual basis".
 \end{lemma}
\begin{proof} This follows from an iterated application of Lem. \ref{lem:mut:relexc} \eqref{lem:mut:relexc-4}.
\end{proof}

The cases of {\em projective bundles} Thm. \ref{thm:proj.bundle} and {\em Grassmannian bundles} Thm. \ref{thm:Grass.bundle} provide examples of full relative exceptional collections and dual exceptional collections.
 
 % sec: Grassmannian bundles
 \subsection{Grassmannian bundles} \label{sec:Grass_bundles} 
 %We will review the combinatorials about Young diagram in detail for characteristic zero in section \S \ref{sec:Young_Grassmannian}. For this subsection, we present the characteristic-free results about Grassmannian bundles. 
 Let $S$ be a quasi-compact, quasi-separated scheme and $\sE$ a locally free sheaf of rank $n$ over $S$. Let $n,d$ be two integers such that $n \ge 1$ and $1 < d < n$. Let $\pi \colon \Gr_d(\sE) = \Quot_{S,d}(\sE^\vee) \to S$ be the rank $d$ Grassmannian bundle of $\sE$ over $S$, see Ex. \ref{ex:Grass} for details, and let $\shU$ and $\shQ$ be the tautological rank $d$ subbundle resp. rank $\ell := n-d$ quotient bundle of $\sE$. There is a tautological sequence over $\Gr_d(\sE)$:
	$$0 \to \shU \to \pi^*\sE \to \shQ \to 0.$$
For a commutative ring $R$, a free $R$-module $F$, and
a partition $\lambda = (\lambda_1 \ge \lambda_2 \ge \cdots \ge \lambda_\ell \ge 0)$, the {\em Schur module $\Sigma^\lambda F$} is defined as the image of the composition:
	$$\Sigma^{\lambda} F := {\rm image} \big(\bigotimes_{j \in [1, \lambda_1]} \bigwedge^{\lambda_j^t} \nolimits F \xrightarrow{\otimes_j \Delta_j} \bigotimes_{(i,j) \in \lambda} F(i,j) \xrightarrow{\otimes_i m_i } \bigotimes_{i \in [1, \ell]} \Sym^{\lambda_i} F \big),$$
where $F(i,j)$ denotes a copy of $F$ labeled by the index $(i,j) \in \lambda$, $\Delta_j$ is the comultiplication map along the $j$-th column, and $m_i$ is the multiplication map along the $i$-th row, $\lambda^t$ denotes the transpose partition of $\lambda$, see \cite{Wey, Ful} for details. Since the Schur functor $\Sigma^\lambda$ is universally free, for any scheme $X$, the Schur module construction defines an endofunctor $\Sigma^\lambda \colon (\text{Locally free sheaves}/X) \to  (\text{Locally free sheaves}/X)$.

\begin{remark} Our notation ``$\Sigma^\lambda$" of Schur functor follows the convention of \cite{Kap85, Kap88}, and corresponds to ``$L_{\lambda^t}$" of \cite{Wey},  ``$L^{\lambda}$" of \cite{BLV}, and ``$S_{\lambda}$" of \cite{Ef}.
\end{remark}

We will review the details about Young diagram (in characteristic zero) in \S \ref{sec:Young_Grassmannian}. For non-negative integers $\ell$ and $d$, $B_{\ell,d}$ denotes the set of partitions $\lambda = (\lambda_1, \lambda_2, \ldots, \lambda_\ell)$ such that $d \ge \lambda_1 \ge \lambda_2 \ge \ldots \ge \lambda_\ell \ge 0$. Let  $B_{\ell,d}^{\le}$ be the set $B_{\ell,d}$ equipped with the ``canonical" {\em total} ordering (graded reverse lexicographic ordering): for any $\lambda, \mu \in B_{\ell,d}$, $\lambda < \mu$ if either $|\lambda| < |\mu|$, or $|\lambda| = |\mu|$ and $\lambda_i> \mu_i$ for the smallest $i$ such that $\lambda_i \ne \mu_i$. For example,
	$$(0) < (1) < (2) < (1,1) < (3) < (2,1) < (1,1,1) < \ldots \qquad \text{in $B_{\ell,d}^{\le}$}.$$
Let $B_{\ell,d}^{\ge}$ be the set $B_{\ell,d}$ equipped with opposite ordering $\ge$. The following is a globalization of the characteristic-free Kapranov's theorem  \cite[Thm. 1.5]{BLV}, \cite[Thm. 1.6]{Ef}:

\begin{theorem}[Grassmannian bundles] \label{thm:Grass.bundle} Let $\pi \colon \Gr_d(\sE) = \Quot_{S,d}(\sE^\vee) \to S$ be the rank $d$ Grassmannian bundle of a locally free sheaf $\sE$ of rank $n$ over a scheme $S$, $\ell: = n -d \ge 0$, and let $\shU$ and $\shQ$ be the tautological subbundle and quotient bundle as above. 
\begin{enumerate}[leftmargin=*]
	\item  \label{thm:Grass.bundle-1} 
	$\big\{ \Sigma^\lambda \shQ \big\}_{\lambda \in B_{\ell,d}^{\le}}$ 
%	$$\big\{ \Sigma^\lambda \shQ \big\}_{\lambda \in B_{\ell,d}^{\, \mathrm{can}}}  =(\sO_{\Gr_d(\sE)}, \, \shQ, \, \Sym^2 \shQ, \, \bigwedge^2 \nolimits \shQ, \ldots )$$
 is a full relative exceptional collection of $\Gr(\sE)$ over $S$, and its left dual exceptional collection over $S$ is given by 
 	$\big\{ \Sigma^{\lambda^t} \shU [|\lambda|]\big\}_{\lambda \in B_{\ell,d}^{\ge}}$.
 	%$$\{ \Sigma^{\lambda^t} \shU [|\lambda|]\}_{\lambda \in (B_{\ell,d}^{\, \mathrm{can}})^{\rm op}} =(\ldots, \Sym^2 \shU [2], \, \bigwedge^2 \nolimits \shU [2], \, \shU[1],  \, \sO_{\Gr_d(\sE)} ).$$
	\item  \label{thm:Grass.bundle-2} 
	$\Perf(\Gr_d(\sE))$ admits a relative Serre functor over $S$ given by $\SS_{\Gr(\sE)/S} = (\blank) \otimes \omega_\pi$, where
	\begin{align*}
	\omega_{\pi} &= (\det \shU)^{\otimes \ell} \otimes (\det \shQ)^{\otimes -d} [\ell d]  \\
	&= (\det \sE)^{\otimes \ell} \otimes  (\det \shQ)^{\otimes -n} [\ell d] = (\det \sE)^{\otimes -d} \otimes  (\det \shU)^{\otimes n} [\ell d].
	\end{align*}
	\item  \label{thm:Grass.bundle-3} 
	There are $S$-linear semiorthogonal decompositions with admissible components:
		\begin{align*}
		&\Perf(\Gr_d(\sE)) = \big \langle \Sigma^{\lambda^t} \shU  \otimes \pi^* \Perf(S) \big \rangle_{\lambda \in B_{\ell,d}^{\ge}}, \qquad &\Perf(\Gr_d(\sE)) = \big \langle \Sigma^\lambda \shQ  \otimes \pi^* \Perf(S) \big \rangle_{\lambda \in B_{\ell,d}^{\le}};\\
		&\Db(\Gr_d(\sE)) = \big \langle \Sigma^{\lambda^t} \shU  \otimes \pi^* \Db(S) \big \rangle_{\lambda \in B_{\ell,d}^{\ge}}, \qquad &\Db(\Gr_d(\sE)) = \big \langle \Sigma^\lambda \shQ  \otimes \pi^* \Db(S) \big \rangle_{\lambda \in B_{\ell,d}^{\le}};\\
		&\Dqc(\Gr_d(\sE)) = \big \langle \Sigma^{\lambda^t} \shU  \otimes \pi^* \Dqc(S) \big \rangle_{\lambda \in B_{\ell,d}^{\ge}}, \qquad &\Dqc(\Gr_d(\sE)) = \big \langle \Sigma^\lambda \shQ  \otimes \pi^* \Dqc(S) \big \rangle_{\lambda \in B_{\ell,d}^{\le}}
		\end{align*}
which are compatible with the respective natural inclusions $\Perf \subseteq \Db \subseteq \Dqc$.
\end{enumerate}
\end{theorem}

\begin{proof} The results of \eqref{thm:Grass.bundle-1} hold in the case $S = \Spec \ZZ$ by \cite[Thm. 1.5]{BLV}, \cite[Thm. 1.6]{Ef}, hence it holds for $S = \Spec R$, where $R$ is any ring, by Tor-independent base-change Cor. \ref{cor:relexc}; Thus \eqref{thm:Grass.bundle-1} holds for any base $S$ by fppf descent Cor. \ref{cor:relexc:descent}. For \eqref{thm:Grass.bundle-2}, since the sheaf of relative K\"ahler differentials is given by $\Omega_{\pi} = \shU \otimes \shQ^\vee$ (see e.g. \cite[Prop. 3.3.5]{Wey}), the claim follows from Ex. \ref{eg:lci}\eqref{eg:lci-2} and Prop. \ref{prop:Serre}. Finally, \eqref{thm:Grass.bundle-3} follows from \eqref{thm:Grass.bundle-1} and Cor. \ref{cor:relexc}.
\end{proof}

% Part 2: Local geometry
\newpage
\part{Local geometry} \label{part:local}
\addtocontents{toc}{\vspace{0.5\normalbaselineskip}}	
% Young diagram ans Grassmannian
\section{Young diagrams and Grassmannians} \label{sec:Young_Grassmannian}
For this section we work over a ground field $\kk$ of {\em characteristic zero}. 

\subsection{Young diagrams and Schur functors} 
The standard references for this part are \cite{FulYoung,FH,Wey}.
Let $B_{\ell,d}$ denote the set of Young diagrams inscribed in a rectangle of height $\ell$ and width $d$. The elements $\lambda \in B_{\ell, d}$ can be identified with a non-increasing integral sequences $\lambda = (\lambda_1, \lambda_2, \ldots, \lambda_\ell)$ such that $d \ge \lambda_1 \ge \lambda_2 \ge \ldots \ge \lambda_\ell \ge 0$. For $\lambda \in B_{\ell, d}$, denote $|\lambda| = \sum_{i=1}^{\ell} \lambda_i$.
By convention, $B_{\ell,d} = \{\emptyset\} = \{(0,0,\ldots,0)\}$ is a singleton if $\ell = 0$ or $d=0$, and $B_{\ell,d} = \emptyset$ is empty if $\ell <0$ or $d <0$. Let $B_{\ell,d}^{\ge }$ be theset $B_{\ell,d}$ equipped with the natural {\em partial order $\preceq$} of inclusions of Young diagrams, i.e. $\lambda \preceq \mu$ if and only if $\lambda_i \le \mu_i$ for all $i$, if and only if $\lambda \subseteq \mu$ as Young diagrams. Notice the ``canonical" total order $\le$ defined in \S \ref{sec:Grass_bundles} is a {\em refinement} of the partial order $\preceq$. Denote $B_{\ell,d}^{\succeq}$ the same set $B_{\ell,d}$ but with the opposite partial order $\succeq$. For $\lambda \in B_{\ell, d}$, denote $\lambda^t \in B_{d, \ell}$ the {\em transpose} of $\lambda$. %Then clearly $\lambda \preceq \mu$ iff $\lambda^t \preceq \mu^t$.

Let $W$ be a $\ell$-dimensional $\kk$-vector space, $\lambda \in B_{\ell, d}$ a partition, and let $\Sigma^{\lambda}W$ be the corresponding {\em Schur module} defined in \S \ref{sec:Grass_bundles}; our notation ``$\Sigma^\lambda$" follows the convention of Kapranov \cite{Kap85, Kap88}. By our convention, for $m \ge 0$, $\Sigma^{m} W = S^m W$ is the symmetric power and $\Sigma^{(1^m)} W : =\Sigma^{(1, 1, \ldots,1)} W= \bigwedge^m W$ is the exterior power. The Schur functor can be extended to $\lambda$ with negative entries by the following formula:
	$$\Sigma^{(\lambda_1 + k, \lambda_2 + k, \ldots, \lambda_\ell +k)} W = \Sigma^{(\lambda_1, \lambda_2, \ldots, \lambda_\ell)} W \otimes (\wedge^\ell W)^{\otimes k}, \qquad \text{for any $k \in \ZZ$}.$$
Since we are working in characteristic zero, there are canonical isomorphisms:
	$$(\Sigma^{(\lambda_1, \lambda_2, \ldots, \lambda_\ell)} W)^\vee \simeq \Sigma^{(\lambda_1, \lambda_2, \ldots, \lambda_\ell)} W^\vee \simeq \Sigma^{(-\lambda_\ell, \ldots, -\lambda_2, -\lambda_1)} W.$$
From above two formulae, it is convenient to use the following notations: for $k \in \ZZ$, 
	$$\lambda+k := (\lambda_1 + k, \lambda_2 + k, \ldots, \lambda_\ell +k) \quad \text{and} \quad - \lambda: = (-\lambda_\ell, \ldots, -\lambda_2, -\lambda_1).$$ 
We will use this convention whenever there is no confusion.

Next we review the standard formulae for Schur functors that will be extensively used in this paper. For Young diagrams $\lambda, \mu \in B_{\ell, d}$, the {\em Littlewood-Richardson rule} states that:
	$$\Sigma^\lambda W \otimes \Sigma^\mu W = \bigoplus_{\nu} (\Sigma^{\nu} W)^{\oplus m_{\lambda, \mu}^\nu},$$
where the non-negative integer $m_{\lambda, \mu}^\nu$ is the {\em Littlewood-Richardson coefficient}, which equals the number of ways of expressing Young diagram $\nu$ as a strict $\mu$-expansion of $\lambda$, see \cite{FulYoung, Wey}. In particular the Young diagrams $\nu$ appearing in the sum satisfies 
	$$|\nu| = |\lambda| + |\mu| \quad \text{and} \quad \lambda_i + \mu_\ell \le \nu_i \le \lambda_1 + \mu_i, \quad \text{for all} \quad 1 \le i \le \ell.$$
The following special case of Littlewood-Richardson rule is the {\em Pieri's formula}:
	$$\Sigma^\lambda W \otimes S^m W = \bigoplus_{\nu} (\Sigma^{\nu}W)^{\oplus m_{\lambda}^\nu} \quad \text{and} \quad \Sigma^\lambda W \otimes \bigwedge^m  W = \bigoplus_{\mu} (\Sigma^{\mu}W)^{\oplus n_{\lambda}^\mu}, $$
 where the non-negative integer $m_{\lambda}^\nu = m_{\lambda, m}^\nu$ (resp. $ n_{\lambda}^\mu = m_{\lambda, (1^m)}^\mu$) is number of ways of expression $\nu$ (resp. $\mu$) as a $m$-expansion of $\lambda$ according to Pieri's rules: 
where the non-negative integer $m_{\lambda}^\nu = m_{\lambda, m}^\nu$ (resp. $ n_{\lambda}^\mu = m_{\lambda, (1^m)}^\mu$) is Pieri's number, i.e. the number of ways of obtaining $\nu$ (resp. $\mu$) from $\lambda$ by adding $m$ boxes, such that no two added boxes are in the same column (resp. row), see \cite{FulYoung, Wey}. In particular, the Young diagrams $\nu$ and $\mu$ appearing in the sums 
$\nu$ and $\mu$ satisfy:
 	$$|\nu| = |\lambda| + m, \quad  \lambda_i \le \nu_i \le \lambda_i+m  \quad \text{and} \quad \lambda^t_j \le \nu^t_j \le \lambda^t_j+1, \quad \forall 1 \le i \le \ell, 1 \le j \le d;$$ 
 	$$|\mu| = |\lambda| + m \quad \lambda_i \le \mu_i \le \lambda_i+1 \quad \text{and} \quad \lambda^t_j \le \nu^t_j \le \lambda^t_j+m, \quad \forall 1 \le i \le \ell, 1 \le j \le d.$$
 The following two equalities are known as {\em Cauchy's formulae}: for any $m \ge 0$,
 	$$\bigwedge^m(V \otimes W) = \bigoplus_{|\alpha| = m} \Sigma^\alpha V \otimes \Sigma^{\alpha^t} W, \qquad S^m(V \otimes W) =  \bigoplus_{|\alpha| = m} \Sigma^\alpha V \otimes \Sigma^{\alpha} W.$$
 
 \subsection{Borel--Bott--Weil theorem and Kapranov's collections}
Let $V$ be a $\kk$-vector space of dimension $n$, and $0 < k < n$ an integer. Let $\GG: = \Gr_k(V)$ be the Grassmannian of $k$-dimensional linear subspaces, see Ex. \ref{ex:Grass}, and let $\shU$ (resp. $\shQ$) be the universal subbundle (resp. quotient bundle) of $V$ of rank $k$ (resp. $n -k$). All irreducible homogeneous vector bundles on $\Gr_k(V)$ are of the form:
	$$\shE_{\lambda} = \shE_{(\alpha,\beta)} =  \Sigma^{\alpha} \shU^\vee \otimes \Sigma^{\beta} \shQ^\vee,$$
where $\alpha=(\alpha_1, \alpha_2, \ldots,\alpha_k)$ and $\beta = (\beta _{1}, \beta _{2}, \ldots, \beta _{n-k})$ are two non-increasing sequences of integers, and $\lambda= (\alpha,\beta)$ is the concatenation. Then the Borel--Bott--Weil (BBW) theorem for Grassmannian states that (see \cite{Kap88, FH, Wey} for references):
	% Borel--Bott--Weil theorem
\begin{theorem}[Borel--Bott--Weil (BBW) theorem] \label{thm:BBW} If $\lambda= (\alpha,\beta)$ above is {\em singular}, %i.e. not all entires of $\lambda + \rho$ are distinct, 
i.e. $\alpha_i  -i = \beta_j  -k -j$ for some $i \in [1,k]$, $j \in [1,n-k]$, then
		$H^\bullet(\GG , \shE_{\lambda}) = 0.$
If the weight $\lambda= (\alpha,\beta)$ is {\em non-singular}, i.e. the entires of $\lambda + \rho$ are pairwise distinct, where $\rho = (n, n-1, \ldots, 1)$, 
let $w \in \SS_n$ be the unique permutation of entries such that $w \cdot (\lambda + \rho)$ is strictly decreasing, and let $\ell(w)$ be the length of $w$. Then 
	$H^\bullet(\GG , \shE_{\lambda})  = H^{\ell(w)}(\GG , \shE_{\lambda}) = \Sigma^{w\cdot(\lambda + \rho) - \rho} V^\vee.$
\end{theorem}

The following lemma is direct consequence of BBW theorem:

\begin{lemma}[{\cite{Kap85}}]
	\begin{enumerate}[leftmargin=*]
	\item If $\alpha_1 \ge \alpha_2 \ge \ldots \ge \alpha_k \ge -(n-k)$,  then $\shE_{(\alpha,0)} = \Sigma^{\alpha} \shU^\vee$ has no higher cohomologies, and it has zeroth cohomology if and only if $\alpha_i \ge 0$ for all $i$. In this case $(\alpha,0)$ is strictly decreasing, and 
	$$H^\bullet(\GG, \Sigma^{\alpha} \shU^\vee ) = H^0 (\GG, \Sigma^{\alpha} \shU^\vee) =  \Sigma^{\alpha} V^\vee;$$
	\item If $\beta_1 \ge \beta_2 \ge \ldots \ge \beta_{n-k} \ge -k$, then $E_{(0,\beta^{-1})} = \Sigma^{\beta} \shQ$ has no higher cohomologies, and it has zeroth cohomology if and only if $\beta_i \ge 0$ for all $i$. In this case $(0,\beta^{-1}) = (0, \ldots 0, -\beta_{n-k}, \ldots, - \beta_1)$ is strictly decreasing, and
	$$H^\bullet(\GG, \Sigma^{\beta} \shQ) = H^0 (\GG, \Sigma^{\beta}  \shQ) =  \Sigma^{(0,\beta^{-1})} V^\vee = \Sigma^{\beta} V.$$	 
	\end{enumerate}
\end{lemma}

\begin{lemma}[\cite{Kap88}] \label{lem:Kap:resolution} For any $\shF^\bullet \in \Db(\GG)$, there is a complex of vector bundles $\shV^\bullet$, quasi-isomorphic to $\shF^\bullet$, whose $p$-th term is:
	$$\shV^p = \bigoplus_{i-j=p} ~ \bigoplus_{\substack{\alpha \in B_{k,n-k}, \\ |\alpha| = j}} \mathbf{H}^i(\GG, \shF^\bullet \otimes \Sigma^{\alpha^t} \shQ^\vee) \otimes \Sigma^\alpha \shU.$$
In particular, if $\shF^\bullet = \shF$ is a coherent sheaf, then the following is a resolution of $\shF$:
	$$\shV^\bullet = \big\{ \cdots \to \bigoplus_{i \ge 0} ~\bigoplus_{\substack{\alpha \in B_{k,n-k}, \\ |\alpha| = i+p}} H^i(\GG, \shF \otimes\Sigma^{\alpha^t} \shQ^\vee) \otimes  \Sigma^\alpha \shU  \to \cdots  \big\}_{p \in \ZZ}.$$
\end{lemma}

% Kapranov's collection
%\subsection{Kapranov's collections} As previous subsection, $V$ be a $\kk$-vector space of dimension $n$, and $0 < k < n$ be an integer, and set $\GG: = \Gr_k(V)$. 
Now we summarise the main results of \cite{Kap85, Kap88} (in characteristic zero) as follows: 

%Thm: Kapranov 
\begin{theorem}[Kapranov \cite{Kap85, Kap88}] \label{thm:Kap} $\{\Sigma^{\lambda} \shQ \}_{\lambda \in B_{n-k,k}^{\preceq}}$ is a full strong exceptional collection of vector bundles of $\GG$ over $\kk$, and its left dual exceptional collection over $\kk$ is given by the full strong exceptional collection 
 	$\big\{ \Sigma^{\lambda^t} \shU [|\lambda|]\big\}_{\lambda \in B_{n-k,k}^{\succeq}}$.
In particular, the following holds:
\begin{enumerate}[leftmargin=*]
	\item \label{thm:Kap-1} 
	For any $\lambda, \mu \in B_{k,n-k}$, $\alpha, \beta \in B_{n-k,k}$:
		\begin{align*}
		& \Hom^*(\Sigma^\lambda \shU, \Sigma^\mu  \shU) =
		 \begin{cases} 
		 	0, & \text{if} \quad \lambda \nsucceq \mu; \\
			\Hom^0(\Sigma^\lambda \shU, \Sigma^\mu \shU),  & \text{if}\quad \lambda \succeq \mu. \end{cases} \\
		&  \Hom^*(\Sigma^{\alpha} \shQ, \Sigma^{\beta} \shQ) =
		 \begin{cases} 
		 	0, & \text{if} \quad \alpha \npreceq \beta; \\
			\Hom^0(\Sigma^{\alpha} \shQ, \Sigma^{\beta} \shQ),  & \text{if} \quad  \alpha \preceq \beta. \end{cases}
			\end{align*}
	\item  \label{thm:Kap-2} For any $\lambda, \mu \in B_{k,n-k}$, the following holds:
		\begin{align*}
		 \Hom^*(\Sigma^{\lambda^t} \shQ, \Sigma^\mu  \shU) =
		 \begin{cases} 
		 	0, & \text{if} \quad \lambda \ne \mu; \\
			\Hom^{|\lambda|}(\Sigma^{\lambda^t} \shQ, \Sigma^\lambda \shU) = \kk & \text{if} \quad \lambda = \mu. \end{cases}
			\end{align*}
\end{enumerate}
\end{theorem}

Notice the relationship of \eqref{thm:Kap-1} is stronger than its characteristic-free version \S \ref{sec:Grass_bundles}.

\begin{remark}[Dual version] \label{rmk:dualGr} Under the canonical identification $\Gr_k(V) \simeq \Gr_{n-k}(V^\vee)$, one has $\shU \simeq \shQ'^\vee$ and $\shQ \sim \shU'^\vee$, where $\shU'$ and $\shQ'$ are the respective universal subbundle and quotient bundle for $\Gr_{n-k}(V^\vee)$. {\em Hence every statement of this subsection has a corresponding dual statement.} In particular, $\{\Sigma^{\alpha} \shU^\vee \}_{\lambda \in B_{k,n-k}^{\preceq}}$ is a also full strong exceptional collection for $\GG$ over $\kk$, whose left dual is the full strong exceptional collection
 	$\big\{ \Sigma^{\lambda^t} \shQ^\vee [|\lambda|]\big\}_{\lambda \in B_{k,n-k}^{\succeq}}$.
\end{remark}

% sec: Mutation on Grassmannian
\subsection{Mutations on Grassmannians} In this subsection we perform mutations for Kapranov's collections. We fix the box $B = B_{n-k, k}$, and for any $\gamma \in B$, denote by $B_{\prec \gamma}$ the set $\{ \alpha \in B \mid \alpha  \prec \gamma \}$ equipped with the partial order $\preceq$. The partial ordered sets $B_{\preceq \gamma}$, $B_{\nprec \gamma}$, $B_{\npreceq \gamma}$ and $B_{\succeq \gamma}$, $B_{\nsucc \gamma}$, $B_{\nsucceq \gamma}$ are similarly defined, with natural partial order $\preceq$ inherited from that of $B$. We use $B_{*}^{op}$ to denote the same set $B_{*}$ with the opposite partial order $\succeq$. The following are immediate from the properties of partial order sets:
	\begin{itemize}
		\item For any $\gamma \in B$, $B = B_{\preceq \gamma} \sqcup B_{\npreceq \gamma}$, and $\forall \alpha \in B_{\preceq \gamma},\beta \in  B_{\npreceq \gamma}$: $\beta \npreceq \alpha$;
		\item For $\beta, \gamma \in B$: $\beta \preceq \gamma \iff B_{\preceq \beta} \subseteq B_{\preceq \gamma} \iff B_{\npreceq \beta} \supseteq B_{\npreceq \gamma}$;
		\item For $\beta, \gamma \in B$: $\beta \nsucceq \gamma \iff B_{\npreceq \beta} \cap B_{\preceq \gamma} \neq \emptyset$.
	\end{itemize}
Similar statements hold if we replace $\preceq$ by $\prec$. %Note also that $B_{\preceq \gamma} = B_{\prec \gamma} \sqcup \{\gamma\}$. 
For any subset $S\subset B$, we denote $\langle \Sigma^{ \alpha} \shQ \rangle_{\alpha \in S}$ the subcategory of $\Db(\GG)$ generated by the exceptional collection $\{\Sigma^{ \alpha} \shQ\}_{\alpha \in S}$ with order strongly compatible with the partial order of $S$ (Rmk. \ref{rmk:poset}); Similarly for other cases. %the subcategories generated by other collections of vector bundles.

The next lemma summarises the mutation results on Kapranov's collections that we will use in this paper; see also \cite{Kap88, BLV,Pi20} for related results.

\begin{lemma}[Mutation for Grassmannians]\label{lem:mutation}
	\begin{enumerate}[leftmargin=*]
		\item \label{lem:mutation-1}
		For any $\gamma \in B=B_{n-k,k}$, we have:	
			\begin{align*}
				& \big \langle \Sigma^{ \alpha} \shQ  \big\rangle_{\alpha \in B_{\prec (\preceq, \nsucc, \nsucceq)  \, \gamma}}  =  \big \langle \Sigma^{\alpha^t} \shU \big \rangle_{\alpha \in B_{\prec (\preceq, \nsucc, \nsucceq) \, \gamma}^{op}} ; \\
				& \big \langle \Sigma^{ \alpha} \shQ \rangle_{\alpha \in B_{\succ  (\succeq, \nprec, \npreceq) \, \gamma}}  = \big \langle \Sigma^{\alpha^t} \shU \otimes \omega_{\GG}^{-1} \big \rangle _{\alpha \in B_{\succ  (\succeq, \nprec, \npreceq) \, \gamma}^{op}}; \\
				&  \big \langle \Sigma^{\alpha^t} \shU \big \rangle _{\alpha \in B_{\succ  (\succeq, \nprec, \npreceq) \, \gamma}^{op}} =
				\big \langle \Sigma^{ \alpha} \shQ \otimes \omega_{\GG} \big \rangle_{\alpha \in B_{\succ  (\succeq, \nprec, \npreceq) \, \gamma}};
			\end{align*}		
		\item \label{lem:mutation-2}
		For any $\gamma \in B=B_{n-k,k}$, we have the following mutation results:
			\begin{align*}
				 \LL_{\big \langle \Sigma^{\alpha} \shQ \big \rangle_{\alpha \in B_{\prec (\nsucceq)  \, \gamma} } } \colon \Sigma^{ \gamma} \shQ \mapsto \Sigma^{\gamma^t} \shU [|\gamma|] \quad \text{and} \quad \RR_{\big \langle \Sigma^{\alpha} \shQ \big \rangle_{\alpha \in B_{\succ (\npreceq)  \, \gamma} } }  \colon \Sigma^{ \gamma} \shQ \mapsto \Sigma^{\gamma^t} \shU [|\gamma|] \otimes \omega_{\GG}^{-1}; \\
				\RR_ {\big \langle \Sigma^{\alpha^t} \shU \big \rangle_{\alpha \in B_{\prec (\nsucceq) \, \gamma}^{op}}}  \colon \Sigma^{\gamma^t} \shU \mapsto  \Sigma^{\gamma} \shQ [-|\gamma|] \quad \text{and} \quad \RR_ {\big \langle \Sigma^{\alpha^t} \shU \big \rangle_{\alpha \in B_{\succ (\npreceq) \, \gamma}^{op}}}  \colon \Sigma^{\gamma^t} \shU \mapsto  \Sigma^{\gamma} \shQ [-|\gamma|] \otimes \omega_{\GG}.
			\end{align*}
			\end{enumerate}
\end{lemma}

{\em Notation:}  We use the notation $B_{\prec (\preceq) \gamma}$ to indicate the results holds in both the two situations $B_{\prec \gamma}$ and $B_{\preceq \gamma}$; similarly for other cases. Since the degree shift does not affect the subcategory which the objects generate, we will sometimes omit the degree shifts in the expressions of the subcategories for simplicity of expressions. 

\begin{proof} Set $\shF = \Sigma^{\gamma^t} \shQ$, $\gamma \in B_{k,n-k}$ in Lem. \ref{lem:Kap:resolution}, we have a left resolution of $\Sigma^{\gamma^t} \shQ$ by:
	\begin{align*}
	\big\{ \Sigma^{\gamma} \shU \to & \bigoplus_{|\alpha| = |\gamma| - 1} H^0(\Sigma^{\gamma^t}  \shQ \otimes\Sigma^{\alpha^t} \shQ^\vee )  \otimes  \Sigma^\alpha \shU \to \cdots 
	 \to & H^0(\Sigma^{\gamma^t} \shQ\otimes \shQ^\vee )  \otimes  \shU \to    H^0(\Sigma^{\gamma^t} \shQ) \otimes \sO_\GG \big \}
	\end{align*}
which is concentrated in degree $[-|\gamma|, 0]$. Therefore the left mutation of $\Sigma^{\gamma^t} \shQ$ (resp. right mutation of $\Sigma^{\gamma} \shU$) passing through the subcategory $\langle \Sigma^{\alpha} \shU \rangle_{\alpha \in B^{op}_{\prec \gamma}}$ is exactly $\Sigma^{\gamma} \shU [|\gamma|]$ (resp. $\Sigma^{\gamma^t} \shQ [-|\gamma|] $). On the other hand, if we set $\shF = \Sigma^{\gamma^t} \shQ \otimes \sO_\GG(-n)$, then by Serre duality, we obtain a right resolution of the sheaf $\Sigma^{\gamma^t} \shQ \otimes \sO_\GG(-n)$ by:
	\begin{align*}
	\big\{ H^0(\Sigma^{\gamma^t}  \shQ^\vee \otimes \sO_{\GG}(k) )^\vee  \otimes  \Sigma^{(n-k)^k}  \shU \to \ldots 
	 \to & \bigoplus_{|\alpha| = |\gamma|+1} H^0(\Sigma^{\gamma^t}  \shQ^\vee \otimes\Sigma^{\alpha^t} \shQ)^\vee  \otimes  \Sigma^\alpha \shU \to    \Sigma^\gamma \shU \},
	\end{align*}
whose terms are concentrated in degree $[0, N - |\gamma|]$ (where $N:=k(n-k) = \dim \GG$). Notice the dualizing complex of $\GG$ is $\omega_{\GG} = \sO_{\GG}(-n)[N]$. Therefore the right mutation of $\Sigma^{\gamma^t} \shQ \otimes \omega_{\GG}$ (resp. left mutation $\Sigma^{\gamma} \shU$) passing through the subcategory $\langle \Sigma^{\alpha} \shU \rangle_{\alpha \in B^{op}_{\succ \gamma}}$ is exactly $\Sigma^{\gamma} \shU [|\gamma|]$ (resp. $\Sigma^{\gamma^t} \shQ [-|\gamma|]  \otimes \omega_{\GG}$). This proves \eqref{lem:mutation-2}. Now statements about $\prec$ and $\preceq$ (resp. $\succ$ and $\succeq$) in \eqref{lem:mutation-1} follow  from above process by induction, since $\Sigma^{(0)} \shU = \Sigma^{(0)}\shQ$ (resp. $\Sigma^{(k^{n-k})} \shU [N] = \Sigma^{(n-k)^k} \shQ \otimes \omega_{\GG} $). The rest of \eqref{lem:mutation-1} follows from  Thm. \ref{thm:Kap} \eqref{thm:Kap-1} and Lem. \ref{lem:mut}. %Finally (3) follows from the first two statements,  (1) of Thm. \ref{thm:Kap}, and Lem. \ref{lem:dualsod}.
\end{proof}

As mentioned in Remark \ref{rmk:dualGr}, from $\Gr_k(V) = \Gr_{n-k}(V^\vee)$, we also have a dual version of Lem. \ref{lem:mutation}. From this we have the following immediate consequence:

\begin{lemma} \label{lem:mutation:dual} Set $B=B_{n-k,k}$, and $\GG = \Gr_k(V) \simeq \Gr_k(n)$.
	\begin{enumerate}[leftmargin=*]
		\item \label{lem:mutation:dual-1}
		For any $\beta, \gamma \in B=B_{n-k,k}$ such that $B_{\preceq \beta} \cap B_{\npreceq \gamma} = \emptyset$ holds (e.g. if $\beta \preceq \gamma$), then there is a semiorthogonal decomposition
		$$\Db(\GG) = \big \langle 
		\langle \Sigma^{ \alpha} \shQ^\vee \rangle_{\alpha \in B_{\npreceq \gamma}^{op}},
		 \langle \Sigma^{ \alpha} \shQ^\vee \rangle_{\alpha\in B_{\npreceq \beta, \preceq \gamma}^{op}}, 
		 \langle \Sigma^{\alpha} \shQ^\vee \rangle_{\alpha  \in B_{\preceq \beta}^{op}} 
		 \big \rangle,
		$$
where we dente $B_{\npreceq \beta, \preceq \gamma} := B_{\npreceq \beta} \cap B_{\preceq \gamma}$. Furthermore, the following holds: 
	\begin{align*}
		\big \langle 
		\langle \Sigma^{\alpha} \shQ^\vee \rangle_{\alpha \in B_{\npreceq \beta, \preceq \gamma}^{op}}, 
		\langle \Sigma^{\alpha} \shQ^\vee \rangle_{\alpha  \in B_{\preceq \beta}^{op}} \big \rangle = 
		\big \langle  
		\langle \Sigma^{\alpha} \shQ^\vee \rangle_{\alpha  \in B_{\preceq \beta}^{op}},  \langle \Sigma^{\alpha^t} \shU^\vee \rangle_{\alpha \in B_{\npreceq \beta, \preceq \gamma}} \big \rangle, \\
		\big \langle  \langle  \Sigma^{ \alpha} \shQ^\vee \rangle_{\alpha \in B_{\npreceq \gamma}^{op}}, \langle \Sigma^{ \alpha} \shQ^\vee \rangle_{\alpha \in B_{\npreceq \beta, \preceq \gamma}^{op}} \big \rangle = \big \langle  \langle \Sigma^{\alpha^t} \shU^\vee \otimes \omega_{\GG} \rangle_{\alpha \in B_{\npreceq \beta, \preceq \gamma}},  \langle \Sigma^{ \alpha} \shQ^\vee \rangle_{\alpha \in B_{\npreceq \gamma}^{op}} \big \rangle.
		\end{align*}
		\item \label{lem:mutation:dual-2}
		For any sub-boxes $B_1, B_2 \subseteq B$ such that $B_2 \nsubseteq B_1$ (e.g. if $B_1 \subseteq B_2$), we have:
	\begin{align*} \Db(\GG) & =  \big \langle \langle \Sigma^{ \alpha} \shQ^\vee \rangle_{\alpha \in B^{op} \backslash (B_1^{op} \cup B_2^{op})}, ~ \langle \Sigma^{ \alpha} \shQ^\vee \rangle_{\alpha \in B_2^{op} \backslash B_1^{op}},~ \langle \Sigma^{\alpha} \shQ^\vee \rangle_{\alpha \in B_1^{op}} \big \rangle \\
	& =  \big \langle \langle \Sigma^{ \alpha} \shQ^\vee \rangle_{\alpha \in B^{op} \backslash (B_1^{op} \cup B_2^{op})},~ \langle \Sigma^{\alpha} \shQ^\vee \rangle_{\alpha \in B_1^{op}}, ~\langle \Sigma^{ \alpha^t} \shU^\vee \rangle_{\alpha \in B_2 \backslash B_1} \big \rangle \\
	& = \big \langle \langle \Sigma^{ \alpha^t} \shU^\vee \otimes \omega_\GG \rangle_{\alpha \in B_2 \backslash B_1}, ~ \langle \Sigma^{ \alpha} \shQ^\vee \rangle_{\alpha \in B^{op} \backslash (B_1^{op} \cup B_2^{op})},~ \langle \Sigma^{\alpha} \shQ^\vee \rangle_{\alpha \in B_1^{op}} \big \rangle
	\end{align*}
	\end{enumerate}
\end{lemma}
\begin{proof} \eqref{lem:mutation:dual-1} follows directly from the dual version of Lem. \ref{lem:mutation} \eqref{lem:mutation-1} \& \eqref{lem:mutation-2}, and Lem. \ref{lem:mut}; Then \eqref{lem:mutation:dual-2} follows from \eqref{lem:mutation:dual-1} by taking $\gamma$ to be the maximal Young diagram contained in the region $B_1 \cup B_2$, and $\beta \in B_1$ be the maximal Young diagram of $B_1$.\end{proof}

% Further mutation results
\subsubsection{Some further mutation results}
Set $\GG = \Gr_d(n)$ and $\ell = n-d$, where $d, \ell \ge 1$. The following is an application of the mutation results of previous subsection. Since in this subseciton we are only concerned with generation results, we will omit the symbols of partial order for all boxes $B_{*}$, for simplicity of notations.

% Lemma on mutation
\begin{lemma} \label{lem:G:mut}  
\begin{enumerate}[leftmargin=*]
		\item \label{lem:G:mut-1}  
		For any $1 \le r \le \ell$, consider the following subcategories of $\Db(\GG)$:
	\begin{align*}
		\shV : =  \big \langle \Sigma^{\lambda} \shQ^\vee \big \rangle_{\lambda \in B_{\ell,d} \backslash B_{\ell-r,d}}  \quad \text{and} \quad
		 \shS : =  \big \langle \big\{ \Sigma^{\lambda} \shQ^\vee \otimes \sO(-t) \big\}_{t \in [1,r], \lambda \in B_{\ell,d-1}} \big \rangle.
	\end{align*}
Then $\shV = \shS$. In particular, the following holds:
	\begin{enumerate}
	\item \label{lem:G:mut-1i} 
	For any $t \in [1,r]$, $\lambda \in B_{\ell,d-1}$, $\Sigma^{\lambda} \shQ^\vee \otimes \sO(-t)  \in  \shV$;
	\item \label{lem:G:mut-1ii} 
	For any $t \in [1,r]$, $\lambda \in B_{\ell,d-1}$, $\Sigma^{\lambda^t} \shU^\vee \otimes \sO(-t) \in  \shV$.
	\end{enumerate}
	\item \label{lem:G:mut-2} 
	Dually, for any $1 \le r \le \ell$, consider the following subcategories of $\Db(\GG)$:
	\begin{align*}
		\shV' : =  \big \langle \Sigma^{\lambda} \shQ \big \rangle_{\lambda \in B_{\ell,d} \backslash B_{\ell-r,d}}  
		\quad \text{and} \quad
		 \shS' : =  \big \langle \big\{ \Sigma^{\lambda} \shQ \otimes \sO(t) \big\}_{t \in [1,r], \lambda \in B_{\ell,d-1}} \big \rangle.
	\end{align*}
Then $\shV' = \shS'$. In particular, the following holds:
\begin{enumerate}
	\item \label{lem:G:mut-2i} 
	For any $t \in [1,r]$, $\lambda \in B_{\ell,d-1}$, $\Sigma^{\lambda} \shQ \otimes \sO(t)  \in  \shV'$;
	\item \label{lem:G:mut-2ii} 
	For any $t \in [1,r]$, $\lambda \in B_{\ell,d-1}$, $\Sigma^{\lambda^t} \shU \otimes \sO(t) \in  \shV'$.
	\end{enumerate}
	\end{enumerate}
\end{lemma}

\begin{proof} As the two cases are dual to each other, we only need to show one of them, say \eqref{lem:G:mut-1}. Notice \eqref{lem:G:mut-1ii} follows from  \eqref{lem:G:mut-1i}, and  \eqref{lem:G:mut-1i} follows directly from $\shV = \shS$. To prove the direction ``$\shV \supseteq \shS$", we claim that for any $k=0,1,\ldots, r-1$, the following holds:
	$$\shV_k : =  \big \langle \Sigma^{\lambda} \shQ^\vee \big \rangle_{\lambda \in B_{\ell,d} \backslash B_{\ell-r+k,d}} \otimes \sO(-k)  \subseteq \shV.$$
Then since $\Sigma^{\lambda} \shQ^\vee \otimes \sO(-t) \in \shV_{t-1}$ for $t \in [1,r]$ and $\lambda \in B_{\ell,d-1}$, the claim directly implies $\shS \subseteq \shV$. The base case $k=0$ is trivial as $\shV_0 = \shV$. Now we assume $\shV_{k-1} \subseteq \shV$ holds for $1 \le k \le r-1$, and we want to prove $\shV_{k} \subseteq \shV$. Consider the following subcategory of $ \shV_{k-1}$:
	$$\shR_k : = \langle\Sigma^{\lambda} \shQ^\vee \otimes \sO(-k) \rangle_{\lambda \in B_{\ell,d-1}} \subseteq \shV_{k-1},$$
we only need to show the right orthogonal of $\shR_k \cap \shV_k$ inside $\shV_k$,  which is:
	$$\shR_k^\perp \cap \shV_k = (\shR_k \cap \shV_k)^\perp  \cap \shV_k = \langle  \Sigma^{\lambda} \shQ^\vee \big \rangle_{\lambda \in B_{\ell,d} \backslash (B_{\ell-r+k,d} \cup B_{\ell-1,d-1})} \otimes \sO(-k) \subseteq \shV_k$$
 is also contained in $\shV_{k-1}$. Let us first assume that $\ell - r+k - 1 \ge 0$. Then by Lem. \ref{lem:mutation:dual} \eqref{lem:mutation:dual-2} applied to the case $B_1 = B_{\ell,d-1}$, $B_2 = B_{\ell - r+k, d}$, consider the following decomposition: % \footnote{For simplicity of notations, we will omit the orientation label ``$op$" in $B$ in the computations.}
 	\begin{align*}
	\Db(\GG) & = \Big \langle \langle  \Sigma^{\lambda} \shQ^\vee \big \rangle_{\lambda \in B_{\ell,d} \backslash (B_{\ell-r+k,d} \cup B_{\ell,d-1})}, ~    \langle \Sigma^\alpha \shQ^\vee  \rangle_{\alpha \in B_{\ell-r+k,d} \backslash B_{\ell,d-1}} , ~ \langle \Sigma^\alpha \shQ^\vee  \rangle_{\alpha \in B_{\ell,d-1}}  \Big \rangle \otimes \sO(-k) \\
	&  \overset{\text{Lem. \ref{lem:mutation}}}{=}  \Big \langle \langle  \Sigma^{\lambda} \shQ^\vee \big \rangle_{\lambda \in B_{\ell,d} \backslash (B_{\ell-r+k,d} \cup B_{\ell,d-1})},  ~ \langle \Sigma^\alpha \shQ^\vee  \rangle_{\alpha \in B_{\ell,d-1}},   \langle \Sigma^{\alpha^t} \shU^\vee  \rangle_{\alpha \in B_{\ell-r+k,d} \backslash B_{\ell,d-1}} \Big \rangle \otimes \sO(-k) \\
	& = \Big \langle \shR_k^\perp  \cap \shV_k,~ \shR_k , ~    \langle \Sigma^{\alpha^t} \shU^\vee  \rangle_{\alpha \in B_{\ell-r+k, d} \backslash B_{\ell, d-1}}  \otimes \sO(-k)  \Big \rangle \\
	& =  \Big \langle \shR_k^\perp  \cap \shV_k ,~ \shR_k , ~    \langle \Sigma^{\alpha^t} \shU^\vee  \rangle_{\alpha \in B_{\ell-r+k-1,d}}  \otimes \sO(-k+1)  \Big \rangle.
	\end{align*}
 On the other hand, By Lem. \ref{lem:mutation:dual} \eqref{lem:mutation:dual-2}  applied to the case $B_1=B_{\ell-r+k-1,d}$, $B_2 = B_{\ell-1,d}$, then there is another semiorthogonal decomposition 
 \begin{align*}
	\Db(\GG) & = \Big \langle \langle  \Sigma^{\lambda} \shQ^\vee \big \rangle_{\lambda \in B_{\ell,d} \backslash B_{\ell-1,d} }, ~   \big \langle \Sigma^{\lambda} \shQ^\vee \big \rangle_{\lambda \in B_{\ell-1,d} \backslash B_{\ell-r+k-1,d}}, ~ \langle \Sigma^\alpha \shQ^\vee  \rangle_{\alpha \in B_{\ell-r+k-1,d}}   \Big \rangle \otimes \sO(-k+1) \\
	 &= \Big \langle  \shR_k , ~ \underbrace{ \big \langle \Sigma^{\lambda} \shQ^\vee \big \rangle_{\lambda \in B_{\ell-1,d} \backslash B_{\ell-r+k-1,d}} \otimes \sO(-k+1)}_{ = \shS_{k-1}} , ~ \langle \Sigma^{\alpha^t} \shU^\vee  \rangle_{\alpha \in B_{\ell-r+k-1,d}}  \otimes \sO(-k+1)  \Big \rangle,
	\end{align*}
where for any $k \in [0,r-1]$ the subcategory $\shS_k \subseteq \shV_k$ is defined to be:
	$$\shS_k : =  \big \langle \Sigma^{\lambda} \shQ^\vee \big \rangle_{\lambda \in B_{\ell-1,d} \backslash B_{\ell-r+k,d}} \otimes \sO(-k).$$
Compare the two semiorthogonal decompositions of $\Db(\GG)$, we obtain that for all $k \ge 1$,
		$$\langle \shR_k^\perp  \cap \shV_k, \shR_k \rangle = \langle \shR_k, \shS_{k-1} \rangle = \shV_{k-1},$$
	and the right mutation of $\shR_k^\perp  \cap \shV_k $ passing through the category $\shR_k$ inside $\Db(\GG_+)$ is $\shS_{k-1} \subseteq \shV_{k-1}$. (Notice this mutation result also holds in the case $ \ell - r + k - 1 <0$; in fact, % this is only possible when $r = \ell$ and $k=0$, and 
the same argument works -- it is even simpler in this case, with $B_{\ell - r + k - 1} = \emptyset$.) Therefore 
	$$\shV_k = \langle \shR_k^\perp  \cap \shV_k,  \shR_k \cap \shV_k \rangle \subseteq  \langle \shR_k^\perp  \cap \shV_k, \shR_k \rangle =  \langle \shR_k, \shS_{k-1} \rangle = \shV_{k-1} .$$ 
By induction, $\shV_{k} \subseteq \shV$ holds, and in particular $\shV \supseteq \shS$

To prove the other direction ``$\shV \subseteq \shS$", we reverse the above mutation process. First observe directly from definition that $\shR_k  \subseteq \shS$ for all $k \in [1,r]$, and
	$
	 \shV_k = \langle \shR_{k+1}, \shS_k \rangle
	$
for all $k \in [0,r-1]$.
We claim that for all $k  = r-1, r-2, \ldots, 1, 0$, the following holds: 
	$$\shV_k \subseteq \shS \quad \text{and} \quad \shS_k \subseteq \shS. $$
The claim holds trivially for the base case $k = r-1$ since $\shV_{r-1} = \shR_r \subseteq \shS$ and $\shS_{r-1} = \emptyset$. Now assume the claim holds for $k$, where $r-1 \le k \le 1$, we want to show it holds for $k-1$. In fact, from $\shV_k \subseteq \shS$ and $\shR_k \subseteq \shS$, we obtain that the right mutation of $\shR_k^\perp \subseteq \shV_k$ passing through $\shR_k$ is also contained in $\shS$, i.e. $\shS_{k-1} \subseteq \shS$. Hence $\shV_{k-1} = \langle \shR_{k}, \shS_{k-1}  \rangle \subseteq \shS$. By induction we are done. In particular $\shV = \shV_0 \subseteq \shS$. Hence the lemma follows.
\end{proof}

\begin{remark} In the above proof, we start with $\shV = \shV_0 = \langle \shR_1, \shS_0 \rangle$ and perform iterated left mutations of $\shS_{k-1}$ passing through $\shR_k$ for $k=1,\ldots, r-1$. This process produces a semiorthogonal decomposition of $\shV$ of the form similar to Lem. \ref{lem:dualsod}:
	$$\shV = \langle \shB_r, \shB_{r-1}, \shB_{r-2}, \ldots, \shB_{2}, \shB_{1} \rangle, \quad \text{where} \quad \shB_{k} : = \shV_{k-1}  \cap {}^\perp \shV_{k}  %= \shR_k \cap {}^\perp \shV_{k}
	,$$
then $\shV_{k-1}  = \langle \shB_{r}, \ldots, \shB_{k}\rangle$, for $k \in [1,r]$, and $\shV_{k-1} = \langle \shV_{k}, \shB_{k} \rangle$. Notice $\shB_{r} = \shV_{r} = \shR_{r}$, and $\shB_{k} \subseteq \shR_k$ in general.
For example, in the case $r=\ell$, then the mutation process of the proof induces a semiorthogonal decomposition of $\Db(\GG) = \langle \shV_0, \sO_\GG\rangle$, with $\shB_t =  \big \langle \Sigma^{\lambda} \shQ^\vee \otimes \sO(-t) \big \rangle_{\lambda \in B_{t,d-1}} $:
	\begin{align*}
	\Db(\GG)  & = \langle \shB_{\ell}, \ldots, \shB_{1}, \sO_{\GG} \rangle  = \big \langle \Sigma^{\lambda} \shQ^\vee \otimes \sO(-t) \big \rangle_{t \in [0,\ell],  \lambda \in B_{t,d-1}}   \\
	&=  \big\langle  \langle \Sigma^{\lambda}  \shQ^\vee \otimes \sO(-\ell) \rangle_{ \lambda \in B_{\ell,d-1}}, 
	 \ldots, \langle \Sigma^{\lambda} \shQ^\vee \otimes \sO(-1) \rangle_{ \lambda \in B_{1,d-1}}, \sO_{\GG} \big\rangle.
	 \end{align*}
Notice this semiorthogonal decomposition could also be obtained from simply observing $ \big \langle \Sigma^{\lambda} \shQ^\vee \big \rangle_{ \lambda \in B_{t,d-1}} =  \big \langle \Sigma^{\lambda^t} \shU^\vee \big \rangle_{ \lambda \in B_{t,d-1}} $ and consider $\Db(\GG) = \big \langle \Sigma^{\alpha} \shU^\vee \otimes \sO(\ell-t) \big \rangle_{t\in[0,\ell], \alpha\in B_{d-1,t}}$. 

For example, in the case $d=2$, $\GG = \Gr_2(n)$, then $\shB_t =  \langle \bigwedge^s \shQ^\vee \otimes \sO(-t) \rangle_{s\in [0,t]}$, and the mutation process of the proof induces a semiorthogonal decomposition of $\Db(\Gr_2(n))$: 
	\begin{align*}
	\Db(\Gr_2(n)) 
	=  \big\langle  \langle \bigwedge^s \shQ^\vee \otimes \sO(-\ell) \rangle_{s\in [0,\ell]}, 
	 \ldots, \langle \bigwedge^s \shQ^\vee \otimes \sO(-2) \rangle_{s\in [0,2]}, \langle \bigwedge^s \shQ^\vee \otimes \sO(-1) \rangle_{s\in [0,1]}, \sO_{\GG}
	 \big\rangle.
	 \end{align*}
\end{remark}

%% G_2 mutationcase
\medskip
\subsubsection*{The case of $\Gr_2(n)$} In the case of $\GG = \Gr_2(n)$, $\ell = n-2 \ge 1$, denote $\shQ$ be the universal rank $\ell$ quotient bundle, and $\shU$ be the universal rank $d=2$ subbundle. Note that if we write $\alpha^t = (a+b,b)$ where $a,b \ge 0$ and $a + b \le \ell$, then
	$$ \Sigma^{\alpha^t} \shU_+ = S^a \shU_+ \otimes \sO_+(-b), \qquad  \Sigma^{\alpha^t} \shU_+^\vee = S^a \shU_+^\vee \otimes \sO_+(b).$$
Then the subcategories $\shS$ and $\shS'$ of Lem. \ref{lem:G:mut} take the form:
	\begin{align*}
		 \shS  =  \big \langle \big\{ \bigwedge^s \shQ^\vee \otimes \sO(-t) \big\}_{t \in [1,r], s \in [0,\ell]} \big \rangle
				\quad \text{and} \quad
		 \shS'  =  \big \langle \big\{ \bigwedge^s \shQ \otimes \sO(t) \big\}_{t \in [1,r], s \in [0,\ell]} \big \rangle,
	\end{align*}
and Lem. \ref{lem:G:mut} implies that $S^a \shU^\vee \otimes \sO(-b) \in \shV=\shS$ and $S^a \shU \otimes \sO(b) \in \shV'=\shS'$ if $0 \le a \le \ell$, $1 \le b \le r$. However, in fact we could slightly improve this result:

\begin{lemma} \label{lem:G_2:mut}  For $1 \le r \le \ell$, the following holds:
		\begin{align*}
		 \big \langle \Sigma^{\lambda} \shQ \big \rangle_{\lambda \in B_{\ell,2} \backslash B_{\ell-r,2}} & = \big \langle \big\{ \bigwedge^s \shQ \otimes \sO(t) \big\}_{t \in [1,r], s \in [0,\ell]} \big \rangle. \\
		  \big \langle \Sigma^{\lambda} \shQ^\vee \big \rangle_{\lambda \in B_{\ell,2} \backslash B_{\ell-r,2}} & = \big \langle \big\{ \bigwedge^s \shQ^\vee \otimes \sO(-t) \big\}_{t \in [1,r], s \in [0,\ell]} \big \rangle.
	\end{align*}
Furthermore, for all $a,b \in \ZZ$ such that $0 \le a \le \ell$ and $1 \le b \le a + r + 1$, we have:
	$$S^a \shU \otimes \sO(b) \in  \big \langle \Sigma^{\lambda} \shQ \big \rangle_{\lambda \in B_{\ell,2} \backslash B_{\ell-r,2}}, \qquad S^a \shU^\vee \otimes \sO(-b) \in  \big \langle \Sigma^{\lambda} \shQ^\vee \big \rangle_{\lambda \in B_{\ell,2} \backslash B_{\ell-r,2}}.$$
\end{lemma}

\begin{proof} The first part of the lemma is a special case of Lem. \ref{lem:G:mut}. It remains to show the final statements about $S^a \shU \otimes \sO(b)$ (and its dual). It follows from previous Lem. \ref{lem:G:mut} in the case $d=2$ that $S^a \shU \otimes \sO(b) \in  \big \langle \Sigma^{\lambda} \shQ \big \rangle_{\lambda \in B_{\ell,2} \backslash B_{\ell-r,2}}$ if $1 \le b \le r$. On the other hand, since 
	$$\big\langle \bigwedge^s \shQ \big\rangle_{s \in [0,\ell]} =\big  \langle \bigwedge^s \shQ^\vee \otimes \sO(1) \big \rangle_{s \in [0,\ell]} =\big \langle S^s \shU^\vee \otimes \sO(1) \big\rangle_{s \in [0,\ell]} $$
and $S^a \shU \otimes \sO(b) = S^a \shU^\vee \otimes \sO(b-a)$, hence $S^a \shU \otimes \sO(b) \in  \big \langle \Sigma^{\lambda} \shQ \big \rangle_{\lambda \in B_{\ell,2} \backslash B_{\ell-r,2}}$ also holds if $2 \le b-a \le r+ 1$. Combine the two conditions, we are done.
\end{proof}

%%%% section: Local situation 
%\addtocontents{toc}{\medskip}	
\addtocontents{toc}{\vspace{0.5\normalbaselineskip}}	
\section{Local geometry and correspondences}  \label{sec:local}
In this section, we prove our main results in the universal local situation $X = \Hom_{\kk}(W,V)$, where $W$ and $V$ are finite free modules over a ring $\kk$ of rank $m$ and $n$, with $m \le n$. Denote $\delta: = n - m \ge 0$. For simplicity, {\em for the majority part of this section, we will assume $\kk$ is a field of characteristic zero, unless otherwise stated}.% to make the patterns of Schubert calculus more transparent.

By Lem. \ref{lem:Homspace}, there is a tautological map $\tau_\kk \colon W \to V$ over $X$, and we denote the cokernels by $\sG = \Coker(\tau_\kk)$ and $\sK = \Coker(\tau_\kk^\vee)$. For any pair of integers $(d_+, d_-)$ such that $d_- \le d_+$, $0 \le d_- \le m$ and $0 \le d_+ \le n$, we will study the relationships of the derived categories of the Quot schemes $\shZ_+ =\Quot_{X, d_+}(\sG)$ and $\shZ_- =\Quot_{X, d_-}(\sK)$. 

\subsection{The key lemma and Lascoux-type resolutions}
In this subsection, to treat the schemes $\shZ_+$ and $\shZ_-$ in a uniform manner, we introduce the following symmetric notations. Let $V_{-} = W^\vee$, $V_{+} = V$, $n_- = m$, $n_+ = n$. We use $i \in \{-,+\}$ to indicate the indices, and denote $-i = \mp$ for $i = \pm$. For $i \in \{-,+\}$, let $d_i$ be integers such that $1 \le d_i \le n_i$, and let $\ell_i = n_i - d_i$. 
%Let $\delta = n - m = n_+ - n_- \ge 0$ be the difference, and let $\delta_i: = d_i - d_{-i}$, then $\delta_{-i} = - \delta_i$ and $\delta_i':=\ell_i - \ell_{-i} = r - \delta_i$. 
Denote $\GG_i : = \Gr_{d_i}(V_i^\vee)$ the Grassmannian, and let $\shU_i$ be the tautological rank $d_i$ subbundle of $V_i^\vee$, $\shQ_i$ be the rank $\ell_i$ quotient bundle, $i = \pm$, therefore the tautological short exact sequence for the Grassmannian $\GG_i$ is:
	$$0 \to \shU_i \to V_i^\vee \to \shQ_i \to 0, \qquad i \in \{-,+\}.$$
Denote by $p_i \colon \shZ_i \to \GG_i$, $p_i' \colon \widehat{\shZ} \to \GG_i$ the natural projections. Consider the following schemes:
	$$\shZ_+ :=\Quot_{X, d_+}(\sG),\qquad \shZ_- :=\Quot_{X, d_-}(\sK), \qquad \widehat{\shZ} := \shZ_- \times_X \shZ_+,$$
where $\tau_\kk \colon W \to V$ is the tautological map over $X$, $\sG = \Coker(\tau_\kk)$, $\sK = \Coker(\tau_\kk^\vee)$, and the Quot schemes are defined in \S \ref{sec:Quot}.

\begin{lemma} \label{lem:Sym} Over $X = \Hom_\kk(V_-^\vee, V_+)$, we have the following identifications:
	\begin{align} \label{eqn:Quot:localexpressions}
		\shZ_+ = |\sHom_{\GG_+}(V_-^\vee, \shQ_+^\vee)|, \quad \shZ_- = |\sHom_{\GG_-}(\shQ_-, V_+)|, \quad
		\widehat{\shZ} = |\sHom_{\GG_- \times \GG_+}(\shQ_-, \shQ_+^\vee)|.
	\end{align}
(See \S \ref{sec:Homspace} for the precise definition of Hom space).
\end{lemma} 

\begin{proof} For any $\kk$-scheme $T$, by Thm. \ref{thm:Quot}, $\Quot_{X, d_+}(\sG)(T)$ is the set of (equivalence classes of) pair $(V_{-,T}^\vee \xrightarrow{\varphi} V_{+,T}, V_{+,T}  \twoheadrightarrow \shU_{+,T}^\vee)$ such that the composition $V_{-,T}^\vee  \xrightarrow{\varphi} V_{+,T} \twoheadrightarrow \shU_{+,T}^\vee$ is zero. This set is in canonical bijection with the set of maps $V_{-, T}^\vee \to \shQ_{+,T}^\vee$, where $\shQ_{+,T}^\vee := \Ker(V_{+,T} \twoheadrightarrow \shU_{+,T}^\vee) \subseteq V_{+,T}$.  %, and $\varphi: V_{-, T}^\vee \to \shQ_{+,T}^\vee \hookrightarrow V_{+,T}$, $ \shU_{+,T}^\vee)=  V_{+,T}/  \shQ_{+,T}^\vee$. 
By Lem. \ref{lem:Homspace}, the latter set is exactly the set of $T$-points of $\shZ_+=|\sHom_{\GG_+}(V_-^\vee, \shQ_+^\vee)|$; The arguments for $\shZ_-$ and $\widehat{\shZ}$ are similar.
\end{proof}

{\em For the rest of this section, for simplicity, we will assume $\kk$ is a field of characteristic zero unless otherwise stated.}

% lemma on generators
\begin{lemma} \label{lem:Z_i:generator} $\Db(\shZ_i)$ is spanned and generated by the collection of vector bundles $\{p_i^* \Sigma^{\alpha} \shQ_{i}\}_{\alpha \in B_{\ell_i,d_i}}$, or respectively by $\{p_i^* \Sigma^{\alpha} \shQ_{i}^\vee\}_{\alpha \in B_{\ell_i,d_i}}$. Furthermore the following bundles:
	$$\shT_i = \bigoplus_{\alpha \in B_{\ell_i,d_i}} p_i^* \Sigma^{\alpha} \shQ_{i} \quad \text{and} \quad \shT_i^\vee = \bigoplus_{\alpha \in B_{\ell_i,d_i}} p_i^* \Sigma^{\alpha} \shQ_{i}^\vee$$
are classical tilting bundles of $\shZ_i$, i.e. $\shT_i$ and respectively  $\shT_i^\vee$ classically generates $\Db(\shZ_i)$, and $\Hom_{\shZ_i}(\shT_i,\shT_i[k])=0 =\Hom_{\shZ_i}(\shT_i^\vee,\shT_i^\vee[k])$ for all $k >0$.
\end{lemma}
\begin{proof}The generation statements follow directly from Lem. \ref{lem:span}, Lem. \ref{lem:Sym}, Thm. \ref{thm:Kap}. The rest is a characteristic zero version of \cite[Prop. 3.1]{BLV3}: both tilting statements are equivalent to $\Hom_{\shZ_i}(p_i^*\Sigma^{\alpha} \shQ_{i}, p_i^*\Sigma^{\beta} \shQ_{i} [k]) = 0$, $\forall \alpha,\beta \in B_{\ell_i,d_i}$, $\forall k >0$, which follows from
	$$\Hom_{\shZ_i}(p_i^*\Sigma^{\alpha} \shQ_{i}, p_i^*\Sigma^{\beta} \shQ_{i} [k]) = H^k(\GG_i, \Sym(V_{-i} \otimes \shQ_i)\otimes \Sigma^{\alpha} \shQ_{i}^\vee \otimes \Sigma^{\beta} \shQ_{i}) = 0$$
which is zero by Littlewood--Richardson's rule and Borel--Bott--Weil theorem \ref{thm:BBW}. For example, if $i=+$, $ \Sym(V_{-i}^\vee \otimes \shQ_i) = \bigoplus_{m \ge 0} S^m(W \otimes \shQ_+)$ then by Cauchy's formula and Littlewood--Richardson's rule, we only need to show for $\alpha \in B_{\ell_+,d_+}$ and any partition $\theta$,
	$$H^k(\GG_+, \Sigma^{\alpha} \shQ_+^\vee \otimes \Sigma^\theta \shQ_+) = 0, \qquad \forall k>0.$$
This is an easy consequence of Borel--Bott--Weil theorem \ref{thm:BBW}.
\end{proof}

For $i=\pm$, denote by $r_i \colon \widehat{\shZ} \to \shZ_i$, $\pi_i \colon \shZ_i \to X$ and $\widehat{\pi} \colon \widehat{\shZ} \to X$ the natural projections. Therefore we have a commutative correspondence diagram:
	\begin{equation} \label{diagram:Corr}
	\begin{tikzcd}[row sep= 2.5 em, column sep = 4 em]
		\widehat{\shZ} \ar{rd}{\widehat{\pi}}\ar{r}{r_+}  \ar{d}[swap]{r_-} &  \shZ_{+} \ar{d}{\pi_+}\\
		\shZ_{-} \ar{r}{\pi_-} & X
	\end{tikzcd}
	\end{equation}
To analyse the above diagram, notice that for $i=\pm$, the projection $r_i \colon \widehat{\shZ} \to \shZ_i$ factors through the composition of a local complete intersection closed immersion followed by a smooth projective morphism as follows:	
	\begin{equation}\label{diag:fact}
	\begin{tikzcd} \widehat{\shZ} \ar[hook]{r}{j_-} \ar{d}{r_-} & \shZ_- \times \GG_{+} \ar{dl}{pr_-} \\
	\shZ_-
	\end{tikzcd}	
	\qquad\text{and}\qquad
	\begin{tikzcd} \widehat{\shZ} \ar[hook]{r}{j_+} \ar{d}{r_+} & \shZ_+ \times \GG_{-} \ar{dl}{pr_+} \\
	\shZ_+
	\end{tikzcd}
	\end{equation}
where $j_i$ is a inclusion of the zero locus of a regular section $\xi_i$ of the vector bundle $\shU_{-i}^\vee \otimes \shQ_i^\vee$; in particular, $j_{i*} \sO_{\widehat{\shZ}}$ is resolved by the Koszul complex $\shK^\bullet(j_i) = \{ \wedge^k(\shU_{-i} \otimes \shQ_i),\lrcorner \,\xi_i\}_{k=0,\ldots, \ell_i \cdot d_{-i}}$, $i \in \{-,+\}$. Note that above commutative diagram is compatible with their projections to $\GG_i$, and furthermore the composition of $j_i \colon \widehat{\shZ} \hookrightarrow \ZZ_i \times \GG_{-i} $ followed by projection to second factor $\ZZ_i \times \GG_{-i}  \to \GG_{-i}$ agrees with the projection $p_{-i}' \colon \widehat{\shZ}  \to  \GG_{-i}$. Similarly the map $\pi_{\pm}$ could also be factorised as regular immersion followed by smooth projections. To summarise:

%% Lem: dualizing complex
\begin{lemma} \label{lem:local:Serre} 
All the maps of diagram \ref{diagram:Corr} are projective and local complete intersection morphisms, with corresponding dualizing complexes given by: 
	\begin{align*} 
	 & \omega_{r_+} = \sO_-(-d_++\delta) \otimes \sO_+(-d_-)[d_- (\ell_- -\ell_+)], \quad\\
	 & \omega_{r_-}  = \sO_-(-d_+) \otimes \sO_+(-d_--\delta)[d_+ (\ell_+ -\ell_-)],  \\
	& \omega_{\pi_+}  =  \sO_+(-\delta)[d_+ (\delta-d_+)],   \qquad \text{and} \qquad
	   \omega_{\pi_-}  = \sO_-(\delta) [d_- (-\delta-d_-)], 	\\
	 &\omega_{\widehat{\pi} } = \sO_-(-d_++\delta) \otimes \sO_+(-d_--\delta)[(d_+-d_-)(\ell_+ - \ell_- )- d_+ d_- ]. 
	\end{align*}
Here we denote $\sO_{\GG_\pm}(1) = \det \shQ_\pm \in \Pic(\GG_\pm)$ the ample line bundle on $\GG_\pm$, and  denote $\sO_\pm (1)$ the pull back of $\sO_{\GG_\pm}(1) $ to the corresponding schemes.
\end{lemma}

\begin{proof} This follows directly from above discussion, diagram \eqref{diag:fact}, and Example \ref{eg:lci}.\end{proof}

From this lemma, Example \ref{eg:lci} and Lem. \ref{prop:FM}, we immediately obtain that:

\begin{lemma} \label{lem:local:FM} For any $\shK \in \Perf(\widehat{\shZ})$, the $\Hom_\kk(W,V)$-linear Fourier--Mukai functor
	 $$\Phi_\shK: = r_{+\,*} (\shK \otimes r_{-}^*(\blank) ) \colon \Dqc(\shZ_-) \to \Dqc(\shZ_+)$$
has finite cohomological amplitude, preserves perfect complexes and pseudo-coherent complexes. $\Phi_\shK$ admits both a left adjoint $\Phi_\shK^L$ and a right adjoint $\Phi_\shK^R$ given by:
	\begin{align*}
	\Phi_{\shK}^L = r_{-\,!} (\shK^\vee \otimes r_{+}^*(\blank)),  %=r_{-\,*} (\shK^\vee \otimes \omega_{r_-} \otimes r_{+}^*(\blank)), 
	\qquad \Phi_\shK^R =  r_{-\,*} ( \shK^\vee \otimes r_{+}^!(\blank)).
	%=r_{-\,*} ( \shK^\vee \otimes \omega_{r_+} \otimes r_{+}^*(\blank)).
	\end{align*}
\end{lemma}

\medskip
%The next lemma gives a complete analysis of the behaviour of the generators of Lem.\ref{lem:Z_i:generator} under the correspondence \ref{diagram:Corr}. 

\noindent \textit{Convention.} From now on throughout this section we will omit the symbols of pullbacks $p_i^*$ in the expressions $p_i^* E \in \Db(\shZ_i)$ for $E \in \Db(\GG_i)$, if there is no confusion.

\medskip
We will need the following lemma of Kapranov:

\begin{lemma}[{\cite[Lem. 1.6]{Kap88}}] \label{lem:Hot} If $E^\bullet, F^\bullet$ are bounded complexes of objects in abelian category $\shA$ such that $\Ext^p(E^i, F^j) = 0$ for all $p>0$ and all $i,j$, then $\Hom_{\Db(\shA)}(E^\bullet, F^\bullet) = \Hom_{{\rm Hot}(\shA)}(E^\bullet,F^\bullet)$, i.e. any morphism between $E^\bullet$ and $F^\bullet$ in $\Db(\shA)$ can be represented by a genuine homotopy class of morphisms between the two complexes $E^\bullet$ and $F^\bullet$. \qed
\end{lemma}

The next lemma gives a complete description of the behaviour of the generators of Lem.\ref{lem:Z_i:generator} under the Fourier-Mukai functor induced by the correspondence \ref{diagram:Corr}.  

% Key lemma 
\begin{lemma}[{\bf Key lemma}]\label{lem:key} 
In the above situation, and assume that $\ell_+ \ge \ell_-$.
% Assume $d_+ - \delta \le d_- \le d_+$ (equivalently, $\ell_+ - \delta \le \ell_- \le \ell_+$). 
\begin{enumerate}[leftmargin=*]
	% Case (1)
	\item \label{lem:key-1} The functor $r_{-\,*} \, r_{+}^* \colon \Db(\shZ_+) \to \Db(\shZ_-)$ sends the element $\Sigma^{\alpha} \shQ_{+}^\vee$, $\alpha \in B_{\ell_+,d_+}$, of the generating set $\{\Sigma^{\alpha} \shQ_{+}^\vee\}_{\alpha \in B_{\ell_+,d_+}}$ of $\Db(\shZ_+)$ to the following object: 
		$$r_{-\,*} \, r_{+}^* (\Sigma^{\alpha} \shQ_{+}^\vee) = \RR^0 r_{-\,*} \,  r_{+}^* (\Sigma^{\alpha} \shQ_{+}^\vee)  = \begin{cases}   \Sigma^{\alpha}  \shQ_{-}, & \text{if~} \alpha \in B_{\ell_{-}, d_+} \subseteq  B_{\ell_+, d_+}  ; \\  0, & \text{if~} \alpha \in B_{\ell_+, d_+} \backslash B_{\ell_{-}, d_+}. \end{cases} $$
	% Case (2)
	\item \label{lem:key-2} The functor $r_{-\,!} \, r_{+}^* \colon \Db(\shZ_+) \to \Db(\shZ_-)$ sends the element $\Sigma^{\alpha} \shQ_{+}$, $\alpha \in B_{\ell_+,d_+}$, of the generating set $\{\Sigma^{\alpha} \shQ_{+}\}_{\alpha \in B_{\ell_+,d_+}}$ of $\Db(\shZ_+)$ to the following object:
		$$r_{-\,!} \, r_{+}^* (\Sigma^{\alpha} \shQ_{+})  =  \sH^0(r_{-\,!} \, r_{+}^*  (\Sigma^{\alpha} \shQ_{+})) =  \begin{cases}   \Sigma^{\alpha}  \shQ_{-}^\vee, & \text{if~} \alpha \in B_{\ell_{-}, d_+} \subseteq  B_{\ell_+, d_+} ; \\  0, & \text{if~} \alpha \in B_{\ell_+, d_+} \backslash B_{\ell_{-}, d_+}. \end{cases} $$
	% Case (3)
	\item \label{lem:key-3} The functor $r_{+\,*} \, r_{-}^* \colon \Db(\shZ_-) \to \Db(\shZ_+)$ sends each element $\Sigma^{\alpha} \shQ_{-}^\vee$, $\alpha \in B_{\ell_-,d_-}$, of the generating set $\{\Sigma^{\alpha} \shQ_{-}^\vee\}_{\alpha \in B_{\ell_-,d_-}}$ of $\Db(\shZ_-)$ to a bounded complex $F^\bullet = \{F^p\}_{p \in [-(\ell_+-\ell_-)d_-,0]}$ of vector bundles:
		$$r_{+\,*} \, r_{-}^*(\Sigma^{\alpha} \shQ_{-}^\vee) \simeq F^\bullet =  \{0 \to F^{-(\ell_+-\ell_-)d_-} \to \ldots \to F^{-1} \to  F^0  \to 0\}$$
		with $F^0 = \Sigma^{\alpha} \shQ_{+}$, and in general each term $F^p$ is given by
			$$F^p =  \bigoplus_{\gamma \in B(\alpha), \, p(\gamma) = p} H^{\ell(\gamma,\alpha)}(\GG_{-}, \Sigma^{\alpha} \shQ_-^\vee \otimes \Sigma^{\gamma^t} \shU_{-}) \otimes \Sigma^{\gamma}  \shQ_+,$$ 
		where $B(\alpha) \subseteq B_{\ell_+,d_-}$ is the set of Young diagram $\gamma$ such that $\gamma^t$ is of the form \eqref{eqn:keylemma:gamma} below, with  
		 cardinality $|B(\alpha)|=\binom{\ell_+ - \ell_- + d_- }{d_-}$, and satisfies 
			$$B(\alpha) \backslash  \{\alpha\} \subseteq B_{\ell_+, d_-} \backslash B_{\ell_-, d_-};$$  
			$p(\gamma) \in [-(\ell_+-\ell_-)d_-,0]$ and $\ell(\gamma, \alpha) \ge 0$
		 are functions on $\gamma$ defined by (\ref{eqn:keylemma:lp}). In particular, if $\ell_+ = \ell_-$, then for all $\alpha \in B_{\ell_-, d_-}$ we have 
		\begin{align*}
		 r_{+\,*} \, r_{-}^*(\Sigma^{\alpha} \shQ_{-}^\vee) = \RR^0 r_{+\,*} \, r_{-}^*(\Sigma^{\alpha} \shQ_{-}^\vee) =  F^0 = \Sigma^{\alpha} \shQ_{+}.
		\end{align*}
	% Case (4)
	\item \label{lem:key-4} The functor $r_{+\,!} \, r_{-}^* \colon \Db(\shZ_-) \to \Db(\shZ_+)$ sends each element $\Sigma^{\alpha} \shQ_{-}$, $\alpha \in B_{\ell_-,d_-}$, of the generating set $\{\Sigma^{\alpha} \shQ_{-}\}_{\alpha \in B_{\ell_-,d_-}}$ of $\Db(\shZ_-)$ to a bounded complex $G^\bullet = \{G^p\}_{p \in [0, (\ell_+-\ell_-)d_-]}$ of vector bundles:
		$$r_{+\,!}  \, r_{-}^*(\Sigma^{\alpha} \shQ_{-})  \simeq G^\bullet = \{ 0 \to G^0 \to G^1 \to \ldots \to G^{(\ell_+-\ell_-)d_-} \to 0\},$$
		with $G^0 = \Sigma^{\alpha} \shQ_+^\vee$, and in general each term $G^p$ is given by
			$$G^p =  \bigoplus_{\gamma \in B(\alpha), p(\gamma) = -p} H^{\ell(\gamma,\alpha)}(\GG_{-}, \Sigma^{\alpha} \shQ_-^\vee \otimes \Sigma^{\gamma^t} \shU_{-})^\vee \otimes \Sigma^{\gamma}  \shQ_+^\vee,$$
		 where $B(\alpha)$, $p(\gamma) \in [-(\ell_+-\ell_-)d_-,0]$ and $\ell(\gamma, \alpha) \ge 0$ are the same as in (3). In particular, if $\ell_+ = \ell_-$, then for all $\alpha \in B_{\ell_-, d_-}$ we have 
		\begin{align*}
		r_{+\,!} \, r_{-}^*(\Sigma^{\alpha} \shQ_{-}) = \sH^0 (r_{+\,!} \, r_{-}^*(\Sigma^{\alpha} \shQ_{-}) ) =  G^0 = \Sigma^{\alpha} \shQ_{+}^\vee.
		\end{align*}
\end{enumerate}
\end{lemma}

 %%% Proof of key lemma.
\begin{proof}%[Proof of Lem. \ref{lem:key}]
We begin by proving something slightly more general. Fix $i \in \{-,+\}$, and consider the object of the form $p_{-i}^*\, E^\bullet \in \Db(\shZ_{-i})$, where $E^\bullet \in \Db(\GG_{-i})$, and we want to find a nice perfect complex representatives for $r_{i\,*} \, r_{-i}^*(p_{-i}^* \,E^\bullet)$ and $r_{i\,!} \, r_{-i}^*(p_{-i}^* \,E^\bullet)$ in $\Db(\shZ_{i})$. Notice that $r_{-i}^*(p_{-i}^* \,E^\bullet) = p_{-i}'^* E^\bullet = j_{i}^* (\sO_{\shZ_{i}} \boxtimes E^\bullet)$, therefore $j_{i\,*} (r_{-i}^* E^\bullet) = j_{i\,*} \sO_{\widehat{\shZ}} \otimes E^\bullet$ and $j_{i\,!} (r_{-i}^* E) = j_{i\,!} \sO_{\widehat{\shZ}} \otimes E^\bullet$ by projection formula. Therefore by Example \ref{eg:conv:lci}, using the stupid truncation of Koszul resolution $\shK^\bullet(j_i)$ and the dual $\shK^\bullet(j_i)^\vee$, we obtain two canonical Postnikov systems attached to the complex over $\Db(\shZ_{i} \times \GG_{-i})$,
	\begin{align*}
		P^\bullet = \{ \cdots \to P^{-k}:=\wedge^k (\shU_{-i} \boxtimes \shQ_i) \otimes E^\bullet \to \cdots \}_{k=0,1,\ldots, \ell_{i} d_{-i}} \\
		Q^\bullet = \{  \cdots \to Q^{k}:=\wedge^k (\shU_{-i}^\vee \boxtimes \shQ_i^\vee) \otimes E^\bullet \to \cdots\}_{k=0,1,\ldots, \ell_{i} d_{-i}}.
	\end{align*}
whose convolutions are $j_{i\,*} (r_{-i}^* E^\bullet)$ and respectively $j_{i\,!} (r_{-i}^* E^\bullet)$. By Cauchy's formula, the terms $P^{-k}$ and $Q^k$ are furthermore decomposed as direct sums of 
	\begin{align*} 
	P^{-k} =  \bigoplus_{\gamma \in B_{\ell_{i}, d_{-i}}, |\gamma| = k}  (\Sigma^{\gamma^t} \shU_{-i} \boxtimes \Sigma^{\gamma} \shQ_i) \otimes  E^\bullet \quad \text{and} \quad
	Q^k =  \bigoplus_{\gamma \in B_{\ell_{i}, d_{-i}}, |\gamma| = k} (\Sigma^{\gamma^t} \shU_{-i}^\vee \boxtimes \Sigma^{\gamma} \shQ_i^\vee) \otimes  E^\bullet.
	\end{align*}
Taking $pr_{i\, *}$ and $pr_{i\, !}$ to respectively $P^\bullet$ and $Q^\bullet$, by Lem. \ref{lem:conv} we obtain two canonical Postnikov systems attached to the following complexes over $\Db(\shZ_i)$,
	\begin{align*}
	X^\bullet &= \{ \cdots \to X^{-k} = \bigoplus_{\gamma \in B_{\ell_{i}, d_{-i}}, |\gamma| = k} \mathbf{H}^\bullet(\GG_{-i}; E^\bullet \otimes \Sigma^{\gamma^t}  \shU_{-i})  \otimes p_{i}^* \Sigma^{\gamma}  \shQ_{i} \to  \cdots  \}_{k \in [0, \ell_{i} d_{-i}]}, \\
	X'^\bullet &= \{ \cdots \to X'^{k} = \bigoplus_{\gamma \in B_{\ell_{i}, d_{-i}}, |\gamma| = k} \mathbf{H}^\bullet(\GG_{-i}; E^\bullet \otimes \Sigma^{\gamma^t}  \shU_{-i}^\vee \otimes \omega_{\GG_{-i}})  \otimes p_{-}^* \Sigma^{\gamma}  \shQ_{i}^\vee \to  \cdots  \}_{k \in [0, \ell_{i} d_{-i}]},
	\end{align*}
whose convolutions are exactly $r_{i\,*} \, r_{-i}^*(p_{-i}^* \,E^\bullet)$ and respectively $r_{i\,!} \, r_{-i}^*(p_{-i}^* \,E^\bullet)$. Following a key idea of Kapranov \cite{Kap88}, we show that these convolutions can be represented by a genuine complex of vector bundles whose terms are direct sums of copies of $p_{i}^* \Sigma^{\gamma}  \shQ_{i}$ and respectively $p_{i}^* \Sigma^{\gamma}  \shQ_{i}^\vee$. Without loss of generality, we only need to show for $X^\bullet$. We claim that this desired statement holds for any $Y$-term of the Postnikov system in Def. \ref{def:Postnikov} attached to $X^\bullet$, %i.e each $Y$-term can be can be represented by a genuine complex of vector bundles whose terms are direct sums of copies of $p_{i}^* \Sigma^{\gamma}  \shQ_{i}$
hence the result holds for the convolution $Y = Y^a \simeq r_{i\,*} \, r_{-i}^*(p_{-i}^* \,E^\bullet)$. (Notice in this case $[a,b] = [- \ell_{i} d_{-i},0]$). The base case $Y^b = X^b[-b]$ is trivial. For the induction step, assume it holds for $Y^i$, then consider the triangle $X^{i-1}[-i] \to Y^i \to Y^{i-1}$ in Def. \ref{def:Postnikov}. By Lem. \ref{lem:Hot} and Lem. \ref{lem:Z_i:generator}, we could represent the morphism $X^{i-1}[-i] \to Y^i$ by a genuine morphism of complexes, and hence represent $Y^{i-1}$ by the mapping cone. This implies that the desired statement also holds for $Y^{i-1}$. Hence by reverse induction the claim is proved.

Therefore we have represented $r_{i\,*} \, r_{-i}^*(p_{-i}^* \,E^\bullet)$ and $r_{i\,!} \, r_{-i}^*(p_{-i}^* \,E^\bullet)$ in $\Db(\shZ_i)$ by complexes of vector bundles $F^\bullet$ and respectively $G^\bullet$, whose terms are canonically determined by $E^\bullet$:
	\begin{align}
	F^p & =  \bigoplus_{\gamma \in B_{\ell_{i}, d_{-i}}} \Ext^p_{\GG_{-i}}(E^{\bullet \, \vee}, \Sigma^{\gamma^t}  \shU_{-i} [|\gamma|] ) \otimes p_i^* \Sigma^{\gamma}  \shQ_i;  \label{eqn:keylemma:F} \\ 
	G^p &= % \bigoplus_{\gamma \in B_{\ell_{i}, d_{-i}}} \Ext^{-p}_{\GG_{-i}}(\Sigma^{\gamma^t}  \shU_{-i} [|\gamma|], E^\bullet \otimes \omega_{\GG_{-i}}) \otimes p_i^* \Sigma^{\gamma}  \shQ_i^\vee
	   \bigoplus_{\gamma \in B_{\ell_{i}, d_{-i}}} \Ext^{-p}_{\GG_{-i}}(E^{\bullet}, \Sigma^{\gamma^t}  \shU_{-i} [|\gamma|] )^\vee \otimes p_i^* \Sigma^{\gamma}  \shQ_i^\vee.    \label{eqn:keylemma:G} 
	\end{align}
Notice the differentials of the complexes $F^\bullet$ and $G^\bullet$ are not canonically determined by $E^\bullet$, but the homotopy classes of the differentials are.

Next, we compute the non-zero terms $F^p$ and $G^p$ appearing in (\ref{eqn:keylemma:F}, \ref{eqn:keylemma:G}) in the case $E^\bullet = \Sigma^\alpha \shQ_{-i} $ and $E^\bullet = \Sigma^\alpha \shQ_{-i}^\vee$ respectively. Notice our assumption $\ell_+ \ge \ell_-$ implies $B_{\ell_-, d_+} \subseteq B_{\ell_+, d_+}$. In case \eqref{lem:key-1}, if we take $E = \Sigma^{\gamma}  \shQ_+^\vee$, then by Thm. \ref{thm:Kap}  \eqref{thm:Kap-2}, the only non-vanishing term of $F^p$'s is $F^0 = \Sigma^{\alpha}  \shQ_-$. In case \eqref{lem:key-2}, if we take $E = \Sigma^{\gamma} \shQ_+$, then by Thm. \ref{thm:Kap} \eqref{thm:Kap-2}, the only non-vanishing term of $G^p$'s is $G^0 = \Sigma^{\alpha}  \shQ_-^\vee$. Hence \eqref{lem:key-1} and \eqref{lem:key-2} are proved. 

In the cases \eqref{lem:key-3} and \eqref{lem:key-4}, to compute all the possible non-zero $F^p$ and $G^p$ terms, we need to find (for a given $\alpha \in B_{\ell_-, d_-}$) all possible $\gamma \in B_{\ell_+, d_-}$ and all possible $k$ such that 
	$$ \Ext^k_{\GG_{-}}( \Sigma^{\alpha}  \shQ_- , \Sigma^{\gamma^t}  \shU_{-})  = H^k(\GG_-, \Sigma^{\alpha}  \shQ_-^\vee \otimes  \Sigma^{\gamma^t}  \shU_{-} ) \ne 0.$$

From Borel--Bott--Weil theorem \ref{thm:BBW}, this happens exactly when $\lambda + \rho$ is nosingular, where $\lambda = (-\gamma^t, \alpha)$, and $\lambda + \rho  = (-\gamma^t + \rho^{(1)}, \alpha + \rho^{(2)})$, where 
	\begin{align*}
	 -\gamma^t + \rho^{(1)} & =  (d_-+\ell_- -\gamma^t_{d_-}, \ldots, 2 +\ell_- - \gamma^t_{2}, 1+ \ell_-  -\gamma^t_{1}), \\
	 \alpha + \rho^{(2)}  & = (\ell_- + \alpha_1, \ldots,  2 + \alpha_{\ell_--1}, 1 + \alpha_{\ell_-}).
	\end{align*}
The first part $-\gamma^t + \rho^{(1)}$ of $\lambda + \rho$, is a strictly decreasing sequence contained in the range $[\ell_--\ell_++1, d_-+ \ell_-]$, and the second part $\alpha + \rho^{(2)}$ is a strictly decreasing sequence contained in the range $[1,d_- + \ell_-]$. Therefore all possible choices for $-\gamma^t + \rho^{(1)}$ such that $\lambda$ is nonsingular is to choose $x$ elements from the set $[\ell_- - \ell_+ + 1,0]$ and $y$ elements from the set 
	$$[1,d_- + \ell_-] \backslash \{\ell_- + \alpha_1, \ldots, 1 + \alpha_{\ell_-} \}  = \{\ell_- + d_- - \alpha^t_{d_-}, \ldots,  \ell_- + 1 - \alpha^t_{1}\},$$
for some integers $x,y \ge 0$ such that $x+y=d_-$. Hence all possible choices of $\gamma^t$ are 
	\begin{align}\label{eqn:keylemma:gamma}
		\gamma^t = (\ell_- + i_1, \ldots, \ell_- + i_x; ~ x+1 - j_1 + \alpha^t_{j_1}, \ldots, x+y - j_y + \alpha^t_{j_y})
	\end{align}
for any two sequences of integers $i_1, \ldots, i_x$ and $j_1, \ldots, j_y$ such that
	$$\ell_+ - \ell_- \ge i_1 \ge i_2 \ge \ldots \ge i_x \ge x; \quad 1 \le j_1 < j_2 < \ldots < j_y \le d_-.$$
Denote $B(\alpha)$ the set of $\gamma$ such that $\gamma^t$ has the form \eqref{eqn:keylemma:gamma}, then it is clear from \eqref{eqn:keylemma:gamma} that $\alpha \in B(\alpha)$ and $B(\alpha) \subseteq \{\alpha\} \cup B_{\ell_+, d_-} \backslash B_{\ell_-, d_-}$. Notice that 
	$$\tau := (i_1 - x \ge \ldots \ge i_x -x) \quad \text{and} \quad \theta := (x+1-j_1 \ge  \ldots \ge x+y - j_y)$$
 can be any Young diagram $\tau \in B_{x, \ell_+ - \ell_- - x}$ and $\theta \in B_{y,x}$, hence $|B(\alpha)| = \binom{\ell_+ - \ell_- + d_- }{d_-}$.

Let $\gamma \in B(\alpha)$ be of the form \eqref{eqn:keylemma:gamma}, then 
	$$|\gamma| = (\ell_- + x) + (\alpha_{j_1}^t + \ldots + \alpha_{j_y}^t) + |\theta| + |\tau|.$$
Let $w$ be the permutation of BBW theorem \ref{thm:BBW}, i.e. the unique permutation of entries of $(\lambda + \rho)$ such that $w\cdot(\lambda + \rho)$ is strictly decreasing. Denote $\ell(\gamma;\alpha) = \ell(w)$ be the length of $w$. To compute $\ell(\gamma;\alpha)$, first notice that as 
	$$\lambda + \rho = (\ell_-+j_y -\alpha^t_{j_y}, \ldots, \ell_- + j_1  - \alpha^t_{j_1}; ~ x-i_x, \ldots, 1- i_1; ~ \ell_- + \alpha_1, \ldots, 1 + \alpha_{\ell_-}),$$
it  requires exactly $\ell_-  x$ permutations interchanging the last two parts, after which it becomes
	$$(\ell_-+j_y -\alpha^t_{j_y}, \ldots, \ell_- + j_1  - \alpha^t_{j_1}; ~ \ell_- + \alpha_1, \ldots, 1 + \alpha_{\ell_-}; ~ x-i_x, \ldots, 1- i_1 ).$$
If we take the subtraction of each term of the first two parts of above partition from $(\ell_- + d_- + 1)$, we obtain two {\em increasing} sequences:
	$$\alpha^t_{j_y} + d_- + 1 - j_y< \ldots <\alpha^t_{j_1} + d_-+1 - j_1;  \quad  1 + (d_- -  \alpha_{1}) < \ldots < \ell_- + (d_- - \alpha_{\ell_-}).$$
The second sequence is fixed, the first is nothing but $y$ choices from the sequence:
	$$(1 \le)~ 1+ \alpha_{d_-}^t < 2 + \alpha_{d_- -1}^t < \ldots <  d_- + \alpha_{1}^t ~(\le d_- + \ell_-).$$
Now we can consider the inverse permutation problem, which is: to start from the whole sequence 
	$1+ \alpha_{d_-}^t < 2 + \alpha_{d_- -1}^t < \ldots <  d_- + \alpha_{1}^t$
inside $[1,d_- + \ell_-]$, choose $y$ elements, delate the rest elements of the above sequence, then permute these $y$ element to the far left of all remaining numbers. Starting from the smallest element of $\alpha^t_{j_y} + d_- + 1 - j_y$, for each step, it requires exactly $\alpha^t_{j_s}$ transpositions to permute $\alpha^t_{j_s} + d_- + 1 - j_s$ to the far left of all remaining numbers. It is easy to see this is the minimal decomposition of the permutation, hence 
	\begin{equation}\label{eqn:keylemma:lp}
	\left\{ \begin{split}
	& \ell(\gamma;\alpha) = \ell_-  x + \alpha_{j_1}^t + \ldots + \alpha_{j_y}^t, \\
	& p(\gamma)  : = \ell(\gamma;\alpha) - |\gamma|  = -x^2 - |\tau| -  |\theta| = -x^2  - \sum_{s=1}^x (i_s - x) - \sum_{s=1}^y (x - s- j_s),
	\end{split}
	\right.
	\end{equation}
where $0 \le x \le \min\{\ell_+ - \ell_-, d_-\}$. Hence the range for the function $p(\gamma)$ is:
	$$p(\gamma) \in [-x(\ell_+ - \ell_- +d_- - x), -x^2] \subseteq [ - (\ell_+ - \ell_-) d_-, 0]$$
(Notice the minimal value $p_{\min} = - (\ell_+ - \ell_-) d_-$ is achieved precisely when $\gamma$ is given by $x =  \min\{\ell_+ - \ell_-, d_-\}$, and $\tau$ and $\theta$ to be the maximal Young diagrams in the corresponding boxes.) The term with factor $\Sigma^{\gamma} \shQ_+$ in (\ref{eqn:keylemma:F}) (resp. $\Sigma^{\gamma} \shQ_+^\vee$ in (\ref{eqn:keylemma:G})) appears exactly in degree $p = p(\gamma)$ (resp. $p = -p(\gamma)$), hence it is clear that 
	$$F^0 =  \Sigma^{\alpha} \shQ_{+},   \quad G^0 =  \Sigma^{\alpha} \shQ_{+}^\vee \quad \text{and} \quad F^p = G^{-p} = 0 \quad \text{if}  \quad p \notin [- (\ell_+ - \ell_-) d_-, 0].$$
If $\ell_+ = \ell_-$, then $F^\bullet = F^0, G^\bullet = G^0$. Now all statements of the lemma are proved.
\end{proof}

In the cases \eqref{lem:key-3} and \eqref{lem:key-4}, above proof of the lemma gives an effective way to compute each term $F^{p}$ and $G^{p}$. We compute a few special cases:

%%% remark: 
\begin{example}[$F^{-1}$ and $G^{1}$] In the same situation of Lem.  \ref{lem:key}, and we compute the terms $F^{-1}$ and $G^{1}$. If $\ell_+ = \ell_-$, then $F^\bullet = F^0, G^\bullet = G^0$ and in particular $F^{-1} = G^{1} = 0$. If $\ell_+ > \ell_-$, then the only term contributing to $F^{-1}$ and $P^{1}$ is the case $x=1$, $\theta= \tau = 0$, i.e. 
	$\gamma^{(1)} = (\ell_1 + 1, \alpha_2^t, \ldots, \alpha_{d_-}^t )^t.$
It remains to compute the space
	$H^{\ell(\gamma,\alpha)}(\GG_-, \Sigma^{\alpha} \shQ_-^\vee \otimes \Sigma^{\gamma^{(1), t}} \shU_- ),$
where $\ell(\gamma^{(1)} ,\alpha) = \ell_- + \alpha_2^t + \ldots + \alpha_{d_-}^t$. In this case
	$$\lambda + \rho = (\ell_- + d_- - \alpha_{d_-}^t, \ldots, \ell_- + 2 -  \alpha_{2}^t, 0; \ell_- + \alpha_1, \ell_--1 + \alpha_2 ,\ldots, 1+ \alpha_{\ell_-})$$
Therefore we see $w. (\lambda + \rho)$ is a strictly decreasing sequence in $[0, \ell_- + d_-]$ with only one term missing which is $\ell_- + 1 - \alpha_t^1$. Therefore 
	$$w. (\lambda + \rho)  - \rho = (\underbrace{0, \ldots, 0}_{d_--1+\alpha_1^t}, \underbrace{-1,-1,\ldots,-1}_{\ell_-+1 - \alpha_1^t}).$$
Hence by Borel--Bott--Weil theorem \ref{thm:BBW} for $\GG_- = \Gr_{d_-}(W)$ we have
	$$H^{\ell(\gamma,\alpha)}(\GG_-, \Sigma^{\alpha} \shQ_-^\vee \otimes \Sigma^{\gamma^{(1), t}} \shU_- ) = \Sigma^{w. (\lambda + \rho)  - \rho} W^\vee = \wedge^{\ell_-+1 - \alpha_1^t} W.$$
Therefore  
	$$F^{-1} =  \wedge^{\ell_-+1 - \alpha_1^t} \, W \otimes \Sigma^{(\ell_- + 1, \alpha_2^t, \ldots, \alpha_{d_-}^t )^t} \shQ_+ \quad \text{and} \quad G^{1} = \wedge^{\ell_-+1 - \alpha_1^t} \, W^\vee \otimes \Sigma^{(\ell_- + 1, \alpha_2^t, \ldots, \alpha_{d_-}^t )^t} \shQ_+^\vee.$$
\end{example}

%%%  Lascoux type resolutions for structure sheaf
\begin{example}[{\bf Lascoux complexes}] \label{ex:Lascoux} Assume $\ell_+ \ge \ell_-$, and we apply Lem. \ref{lem:key} to the structure sheaves $\sO_{\shZ_{\pm}} \in \Db(\shZ_{\pm})$. The parts \eqref{lem:key-1} and \eqref{lem:key-2} of Lem. \ref{lem:key} imply that:
	$$ r_{-\,*}  \sO_{\widehat{\shZ} } \simeq \RR^0 r_{-\,*}  \sO_{\widehat{\shZ} } \simeq  \sO_{\shZ_-} \quad \text{and} \quad r_{-\,!}  \sO_{\widehat{\shZ}} \simeq  \sH^0( r_{-\,!}  \sO_{\widehat{\shZ}})\simeq \sO_{\shZ_-}.$$
This result reflects the fact that the map $r_- \colon \widehat{\shZ} \to \shZ_-$ is surjective and has rationally connected fibres. On the other hand, \eqref{lem:key-3} and \eqref{lem:key-4} of Lem. \ref{lem:key} give rise to resolutions:
	\begin{align*}
	r_{+\,*} \sO_{\widehat{\shZ} } & \simeq  \{F^{-(\ell_+-\ell_-)d_-} \to \ldots \to F^{-k} \to \ldots \to F^1 \to F^0 \}_{k \in [0, (\ell_+-\ell_-)d_-]}, \\
	r_{+\,!}  \sO_{\widehat{\shZ} }  & \simeq \{G^0 \to G^1 \to \ldots  \to G^k \to \ldots \to G^{(\ell_+-\ell_-)d_-} \}_{k \in [0, (\ell_+-\ell_-)d_-]},
	\end{align*}
where the terms $F^{-k}$ and $G^k$ are explicitly given by \eqref{eqn:ex:Lascoux} below. More concretely, let $x,y$ be any pair of integers such that $x \ge 0, y \ge 0, x+y = d_-$, and let $\tau \in B_{x, \ell_+ - \ell_- - x}$ and $\theta \in B_{y,x}$ be any Young diagrams. Let $\gamma(\tau,\theta)$ be the Young diagram given by formula \eqref{eqn:keylemma:gamma}, i.e.
	$$\gamma(\tau, \theta)^t  = (\ell_- + x + \tau_1, \ldots, \ell_- +x + \tau_x; ~ \theta_1, \ldots, \theta_y).$$ 
To utilise BBW theorem \ref{thm:BBW}, we also set $\gamma^\natural(\tau, \theta) = -(w\cdot(\gamma(\tau,\theta) + \rho) - \rho)$, then by the computations in the proof of Lem. \ref{lem:key} we have:
	$$
	\gamma^\natural(\tau, \theta)^t =(\ell_- + x + \theta_1^t, \ldots,  \ell_- +x + \theta_x^t; ~\tau_1^t, \ldots, \tau_{\ell_+ - \ell_- - x}^t).
	$$
Equivalently, taking transposes, $\gamma(\tau, \theta)$ and $\gamma^\natural(\tau, \theta)$ are given by:
		\begin{equation}\label{eqn:ex:Lascoux:gamma}
		\left\{ 
		\begin{split}
		\gamma(\tau, \theta) &= (x + \theta_1^t, \ldots, x + \theta_{x}^t;  ~ \underbrace{ x, \ldots, x}_{\text{$\ell_-$-terms}};~\tau_1^t, \ldots, \tau_{\ell_+ - \ell_- - x}^t);\\
		\gamma^\natural(\tau, \theta) &= (x+ \tau_1, \ldots, x+ \tau_x;~ \underbrace{ x, \ldots, x}_{\text{$\ell_-$-terms}}; ~ \theta_1, \ldots, \theta_y).
		\end{split}
	\right.
	\end{equation}
Then $B(0)$ consists of exactly Young diagrams $\gamma(\tau, \theta)$ of the form  \eqref{eqn:ex:Lascoux:gamma} for some $x,y \ge 0$, $x+y = d_-$, $(\tau, \theta) \in B_{x, \ell_+ - \ell_- - x} \times B_{y,x}$. Furthermore, \eqref{eqn:keylemma:lp} implies:
	$$\ell(\gamma(\tau,\theta); (0))= \ell_- \cdot x, \qquad p(\gamma(\tau,\theta)) = -x^2 - |\tau| -  |\theta|.$$
Putting these together and using Thm. \ref{thm:BBW}, the terms $F^{-k}$ and $G^k$ are given by: $F^{-k} = 0 = G^k$ if $k \not \in [0, (\ell_+ - \ell_-)d_-]$, and for $k \in [0, (\ell_+ - \ell_-)d_-]$:
	\begin{equation}\label{eqn:ex:Lascoux}
	\left\{ 
		\begin{split}
	F^{-k} = & \bigoplus_{x= 0}^{\min\{\ell_+-\ell_-, d_-\}} \, \bigoplus_{(\tau, \theta) \in B_{x, \ell_+ - \ell_- - x} \times B_{d_- - x,x} , \,  |\tau| + |\theta| = k - x^2} \,  \Sigma^{\gamma^\natural(\tau,\theta)} W \otimes \Sigma^{\gamma(\tau,\theta)} \shQ_+; \\
	G^{k} =  & \bigoplus_{x= 0}^{\min\{\ell_+-\ell_-, d_-\}} \, \bigoplus_{(\tau, \theta) \in B_{x, \ell_+ - \ell_- - x} \times B_{d_- - x,x}, \,  |\tau| + |\theta| = k - x^2} \,  \Sigma^{\gamma^\natural(\tau,\theta)} W^\vee \otimes \Sigma^{\gamma(\tau,\theta)} \shQ_+^\vee.
	\end{split}
	\right.
	\end{equation}
(Here $\gamma(\tau, \theta)$ and $\gamma^\natural(\tau, \theta)$ are given by \eqref{eqn:ex:Lascoux:gamma} above.) In particular, the first few terms are: $F^0 = G^0 = \sO_{\shZ_-}$,  $F^{-1} = \wedge^{\ell_- + 1} W \otimes \wedge^{\ell_- + 1} \shQ_+$, $G^{1} =  \wedge^{\ell_- + 1} W^\vee \otimes \wedge^{\ell_- + 1} \shQ_+^\vee$. 

The ``last" terms $F^{-(\ell_+ - \ell_-)d_-}$ and $G^{(\ell_+ - \ell_-)d_-} =( F^{-(\ell_+ - \ell_-)d_-})^\vee$ are given as follows.
	\begin{itemize}
		\item If $\ell_+ - \ell_- \le d_-$, then the minimal value of $p(\gamma(\tau,\theta))$ is achieved precisely when $x =\ell_+ - \ell_-$, $y = d_- - (\ell_+ - \ell_-)$, $\tau = (0)$ and $\theta = (x^y)$. Hence:
			\begin{align*}
			F^{-(\ell_+ - \ell_-)d_-} %&= \Sigma^{(\ell_+ - \ell_-)^m} W \otimes  \Sigma^{(d_-^{\ell_+ - \ell_-}, (\ell_+ - \ell_-)^{\ell_-})} \shQ_+ \\
			 = (\det W )^{\otimes (\ell_+ - \ell_-)} \otimes (\det \shQ_+)^{\otimes (\ell_+ - \ell_-)} \otimes \Sigma^{((d_- - \ell_+ + \ell_-)^{\ell_+ - \ell_-})} \shQ_+.
			\end{align*}
		\item If $\ell_+ - \ell_- \ge d_-$, then the minimal value of $p(\gamma(\tau,\theta))$ is achieved precisely when $x =d_- $, $y = 0$, $\tau =( (\ell_+ - \ell_- - d_-)^{d_-})$ and $\theta =(0)$. Hence:
		\begin{align*}
			F^{-(\ell_+ - \ell_-)d_-} %&= \Sigma^{((\ell_+ - \ell_-)^{d_-}, ((d_-)^{\ell_-})} W \otimes  \Sigma^{((d_-)^{\ell_+})} \shQ_+ \\
			 = (\det W )^{\otimes d_-} \otimes (\det \shQ_+)^{\otimes d_-} \otimes \Sigma^{(( \ell_+ - \ell_- - d_-)^{d_-})} W.
			\end{align*}
	\end{itemize}

Notice that in the special case when $d_+ = 0$, $\ell_+ = n$, $d_- = m - \ell_-$, then $\shZ_+ = X$, $\widehat{\shZ} = \shZ_-$, $\shQ_+ = V^\vee \otimes \sO_{X}$, and $\shZ_-$ is the resolution of the degeneracy locus $\DD_{\ell_-} \subseteq X$ (which is the locus where the tautological map has rank $\le \ell_-$, see \S \ref{sec:deg}), and $r_{+\,*} \sO_{\widehat{\shZ}} = \sO_{\DD_{\ell_-}} \simeq F^\bullet$. Setting $\shQ_+ = V^\vee \otimes \sO_{X}$ in \eqref{eqn:ex:Lascoux}, our complex $F^\bullet$ is exactly the famous {\em Lascoux resolution} of $\sO_{\DD_{\ell_-}}$ of \cite{Lasc}, \cite[\S 6.1]{Wey}; The two cases above for the expression of the ``last" term $F^{-(m - \ell_-)(n-\ell_-)}$ correspond to the two cases $n \le m$ and $n \ge m$. If we further set $\ell_- = m-1$ and assume $m \le n$, then our $F^\bullet$ reduces to the Eagon--Northcott complex \cite[\S A.2.6.1]{Ei}, \cite[\S B.2]{Laz04}. Therefore the complexes $F^\bullet$ and $G^\bullet$ of \eqref{eqn:ex:Lascoux} are generalisations of Lascoux resolutions.
\end{example}

%% Twist case
In the statements \eqref{lem:key-3} and \eqref{lem:key-4} of Lem. \ref{lem:key}, we could twist the generators by line bundles $\sO(j)$ for $j \in [0, \ell_+ - \ell_-]$. For example, the part Lem. \ref{lem:key} \eqref{lem:key-3} can be generalized to:
%%% key lemma for twist
\begin{lemma} \label{lem:key:twist} In the situation of Lem. \ref{lem:key} \eqref{lem:key-3},  for any given $j \in [0,\ell_+  - \ell_-]$,  $r_{+\,*} \, r_{-}^* (\Sigma^{\alpha} \shQ_{-}^\vee  \otimes \sO_- (j))$ is isomorphic to a bounded complex $F^\bullet = \{F^p\}_{p \le 0}$ of vector bundles with 
		$$F^p =  \bigoplus_{\gamma \in B^{\{j\}}(\alpha), \, p^{\{j\}}(\gamma) = p} H^{\ell^{\{j\}}(\gamma;\alpha)}(\GG_{-}, \Sigma^{\alpha} \shQ_-^\vee(j) \otimes \Sigma^{\gamma^t} \shU_{-}) \otimes \Sigma^{\gamma}  \shQ_+,$$ 
		where $B^{\{j\}}(\alpha) \subseteq B_{\ell_+,d_-}$ is a set of Young diagram of the form \eqref{eqn:keylemma:twist:gamma} given below; it has cardinality $|B^{\{j\}}(\alpha)|=\binom{\ell_+ - \ell_- + d_- }{d_-}$, and satisfies 
	$$B^{\{j\}}(\alpha) \backslash \{(\alpha^t+j)^t\} \subseteq B_{\ell_+, d_-} \backslash B_{\ell_-, d_-}^{\{j\}},$$ where 
 $B_{\ell_-, d_-}^{\{j\}}$ is the set of Young diagrams:
		$$B_{\ell_-, d_-}^{\{j\}} = \big\{(\underbrace{d_-, \ldots, d_-}_{j}, \gamma_1, \ldots, \gamma_{\ell_-}) \in B_{j+\ell_-,d_-} \mid \gamma = (\gamma_1, \ldots, \gamma_{\ell_-}) \in B_{\ell_-,d_-}  \big\};$$
Here $p^{\{j\}}(\gamma)  \le 0$ and $\ell^{\{j\}}(\gamma, \alpha) \ge 0$ are functions on $\gamma$ to be explicitly given below. Hence there is exact one copy of summand $\Sigma^{(\alpha^t+j)^t} \shQ_+$ in the degree $-d_- j$ term $F^{-d_- j}$, and all other summands of nonzero $F^{p}$'s are copies of $\{ \Sigma^{\gamma}  \shQ_{+} \}_{\gamma \in  B_{\ell_+, d_-} \backslash B_{\ell_-, d_-}^{\{j\}}}$. 
\end{lemma}		

%%% remark: twist
\begin{proof}
Same as the proof of Lem. \ref{lem:key} \eqref{lem:key-3}, if we take $E^\bullet = \Sigma^{\alpha} \shQ_{-}^\vee  \otimes \sO_- (j)$ in \eqref{eqn:keylemma:F}, by BBW theorem the only possible cases for $F^p$ to be non-zero is when $\gamma^t$ is of the form:
\begin{align}\label{eqn:keylemma:twist:gamma}		
	\gamma^t =  (j+ \ell_- + i_1, \ldots, j+  \ell_- + i_x; ~~ j+ x+1 - j_1 + \alpha^t_{j_1}, \ldots, j+ x+y - j_y + \alpha^t_{j_y}; ~~k_1, \ldots, k_z)
	\end{align}
for some choice of integers $x, y ,z \ge 0$, $x \le \ell_+ - \ell_- -j$, $z \le j$, $x+ y + z = d_-$, and some choices of integers $i_1, \ldots, i_x$, $j_1, \ldots, j_y$, $k_1, \ldots, k_z$ such that
	\begin{align*}
	& \ell_+ - \ell_- \ge i_1 \ge i_2 \ge \ldots \ge i_x \ge x; \\
	& 1 \le j_1 < j_2 < \ldots < j_y \le d_-; \\
	& j - z \ge k_1\ge k_2 \ge \ldots \ge k_z \ge 0.
	\end{align*} 
Equivalently, we could choose $\tau, \theta, \xi$ to be any Young diagrams given by:
	\begin{align*}
	& \tau = \big( (\ell_+ - \ell_- - j - x \ge)~ \tau_1 =  i_1 - x  \ge  \tau_2 = i_2 - x  \ge \ldots \ge \tau_x = i_x - x  ~(\ge 0) \big); \\
	& \theta = \big( (x+z \ge)~ \theta_1 = x + z + 1 - j_1 \ge \ldots \ge \theta_y = x + z + y  - j_y ~( \ge 0 ) \big); \\
	& \xi = \big( (j - z \ge)~ \xi_1 = k_1\ge \xi_2 = k_2 \ge \ldots \ge \xi_z = k_z ~(\ge 0)\big).
	\end{align*} 
Denote $B^{\{j\}}(\alpha)$ the set of $\gamma$ such that $\gamma^t$ has the above form \eqref{eqn:keylemma:twist:gamma}. Hence the number of choices are 
	$$|B^{\{j\}}(\alpha)| = \sum_{x+y+z = d_1} \binom{\ell_+ - \ell_- - j}{x} \binom{j}{z} \binom{d_-}{y} = \binom{\ell_+ - \ell_- + d_-}{d_-}. $$	
For any $\gamma \in B^{\{j\}}(\alpha)$ of the form \eqref{eqn:keylemma:twist:gamma}, by the same argument as  Lem. \ref{lem:key} we have 
	\begin{align*}
	& \ell^{\{j\}}(\gamma;\alpha) = \ell_-  x + \alpha_{j_1}^t + \ldots + \alpha_{j_y}^t, \\
%	& |\gamma| =  \ell_-  x+  \alpha_{j_1}^t + \ldots + \alpha_{j_y}^t + x^2 + j(x+y) + |\tau| + |\theta| + |\xi|, \\
	& p^{\{j\}}(\gamma)  : = \ell(\gamma;\alpha) - |\gamma|  = -x^2 - j(x+y) - |\tau| -  |\theta|  - |\xi|.
	\end{align*}
%Therefore
%	$$ -x^2 - j(x+y) \ge p^{\{j\}}(\gamma) \ge  -x^2 - j(x+y) - x(\ell_+ - \ell_- - j  - x) - y(x+z) - z(j-z).$$
It follows from the expression of \eqref{eqn:keylemma:twist:gamma} that $(\alpha^t+j)^t \in B^{\{j\}}(\alpha)$ and 
	$B^{\{j\}}(\alpha) \backslash \{(\alpha^t+j)^t\} \subseteq B_{\ell_+, d_-} \backslash B_{\ell_-, d_-}^{\{j\}},$
where $(\alpha^t+j)^t \in B^{\{j\}}(\alpha)$ corresponds to the unique choice of $\gamma$ such that $x=z=0$ and $\tau = \theta = \xi = 0$; and hence
	$$p^{\{j\}}((\alpha^t+j)^t) = - d_- \cdot j .$$
(However, notice that in the case $j \ne 0$, it might happen that this degree is not anymore the highest degree, nor $\Sigma^{(\alpha^t+j)^t} \shQ_+$ being unique term with this degree. For example, in the case $d_- \le j \le \ell_+ - \ell_-$,  if we can choose $x=y=0$, $z = d_-$, and any $\xi \in B_{d_-, j}$, then in this case we see that $p^{\{j\}}(\gamma) = - |\xi| \in [- d_- \cdot j, 0]$ could be number within this range. In particular, there are terms of degree $0$, but also summands of $F^p$ rather than the case $\gamma=(\alpha^t+j)^t$ (if $\alpha \ne 0$) that contributes to degree $- d_- \cdot j$)
\end{proof}
Similarly, there is a twisted version for Lem. \ref{lem:key} \eqref{lem:key-4}, and we leave it to the readers.

% First results
\subsection{First implications} We continue to use the notations $0 \le d_- \le n_-=m$, $0 \le d_+ \le n_+=n$, $\ell_- = m-d_-$, $\ell_+ = n - d_+$. Furthermore, we assume:

\begin{itemize}
	\item {\em For the rest of this whole section, we will assume $d_- \le d_+$ and $\ell_- \le \ell_+$. Thus $\delta := n -m \ge 0$.} 
\end{itemize}

\noindent (This assumption guarantees $B_{\ell_-,d_-} \subseteq B_{\ell_+, d_+}$.) We will fix $d_+$ and let $d_-$ vary, and apply the results of the last subsection to study the relationships among these $\Db(\shZ_\pm)$.
%Assume $d_+ - \delta \le d_- \le d_+$ throughout this section.

% Lemma: top strata
\subsubsection*{Contributions from top strata}
\begin{lemma} \label{lem:top} If $d_+ \le \delta$ (equivalently, $\ell_+ \ge m$), then for all $\alpha \in B_{\delta-d_+, d_+}$, the functors 
	$$\Phi^{\alpha}(\blank) := \pi_+^*(\blank) \otimes \Sigma^{\alpha^t} \shU_+^\vee \colon \Db(X) \to \Db(\shZ_{+})$$
 are fully faithful, and their images form an $X$-linear admissible semiorthogonal sequence $\{ \Im \Phi^\alpha\}_{\alpha \in B_{\delta-d_+,d_+}^{\preceq}}$, with semiorthogonal order strongly compatible with the partial order of $B_{\delta-d_+, d_+}^{\preceq}$, i.e. $\Im \Phi^{\alpha} \subseteq (\Im \Phi^{\beta})^\perp$ whenever $\alpha \nsucceq \beta$. 
\end{lemma}

\begin{proof} This corresponds to the case $d_-=0$ when diagram (\ref{diag:fact}) degenerates into:
\begin{equation*}
	\begin{tikzcd} \widehat{\shZ} = \shZ_+ \ar[hook]{r}{j_-} \ar{d}[swap]{r_- = \pi_{+}} & X \times \GG_{+} \ar{dl}{pr_-} \\
	\shZ_- = X
	\end{tikzcd}	
	\qquad\text{and}\qquad
	\begin{tikzcd} \widehat{\shZ} \ar[equal]{r}{j_+=\Id} \ar[equal]{d}[swap]{r_+=\Id~} & \shZ_+  \ar[equal]{dl}{pr_+ = \Id} \\
	\shZ_+
	\end{tikzcd}
	\end{equation*}
where $j_-$ is the inclusion of zero locus of a regular section of the vector bundle $W^\vee \boxtimes \shU_+^\vee$. Therefore for any $A,B \in \Db(X)$ and $\alpha, \beta \in B_{\delta-d_+,d_+}$ such that $\alpha \nsucc \beta$ (i.e. $\alpha = \beta$ or $\alpha \nsucceq \beta$),
	\begin{align*}
	& \Hom_{\shZ_+}(\Phi^{\beta}(B), \Phi^{\alpha}(A)) = 
	\Hom_{\shZ_+}(j_-^*( B \boxtimes \Sigma^{\beta^t} \shU_+^\vee), j_-^*( A \boxtimes \Sigma^{\alpha^t} \shU_+^\vee))  \\
	&= \Hom_{X \times \GG_+}(B \boxtimes \Sigma^{\beta^t} \shU_+^\vee, j_{-*}\,j_-^*(A \boxtimes \Sigma^{\alpha^t} \shU_+^\vee)) \\
	&=  \Hom_{X \times \GG_+}(B \boxtimes \Sigma^{\alpha^t} \shU_+, A \boxtimes \Sigma^{\beta^t} \shU_+ \otimes j_{-*} (\sO_{\widehat{\shZ}} ) )
	%& =  \Hom_{X \times \GG_+}(A \boxtimes \Sigma^{\alpha^t} \shU^\vee, (B \boxtimes \Sigma^{\beta^t} \shU^\vee) \otimes \shK^\bullet({j_-} )) 
	%=  \Hom_{X}(A,B) \otimes \Hom_{\GG_+}(\Sigma^{\alpha^t} \shU_+^\vee,  \Sigma^{\beta^t} \shU_+^\vee \otimes \shK^\bullet_{j_-}).
	\end{align*} 
Since $ j_{-*}\, (\sO_{\widehat{\shZ}}) \simeq \shK^\bullet({j_-})$, where $\shK^\bullet({j_-}) = \{\bigwedge^k(W  \boxtimes \, \shU_+) \}_{k=0,\ldots, md_+}$ is the Koszul complex as usual. From Cauchy's formula, $\bigwedge^k(W \boxtimes \shU_+) = \bigoplus_{|\lambda|=k, \lambda \in B_{m,d}} \Sigma^{\lambda}W \otimes \Sigma^{\lambda^t} \shU_+$, therefore every irreducible summand $\Sigma^{\gamma^t} \shU_+ \subseteq \Sigma^{\beta^t} \shU_+ \otimes  \Sigma^{\lambda^t} \shU_+$ satisfies $\gamma \in B_{\delta - d_+ + m, d_+} = B_{n-d_+,d_+}$ and $\gamma \succeq \beta$ hence $\gamma \npreceq \alpha$. Therefore by Kapranov's Thm. \ref{thm:Kap} \eqref{thm:Kap-1}, $\Hom_{\GG_+}(\Sigma^{\alpha^t} \shU_+^\vee, \Sigma^{\gamma^t} \shU_+^\vee) = 0$ expect from the case when $\alpha=\beta=\gamma$ and $k=0$. Therefore above $\Hom$ space reduces to
	$$\Hom_{\shZ_+}(\Phi^{\beta}(B), \Phi^{\alpha}(A)) =  \Hom_{X}(B,A) \otimes \Hom_{\GG_+}(\Sigma^{\beta^t} \shU_+^\vee,  \Sigma^{\alpha^t} \shU_+^\vee)$$
which is zero if $\alpha \nsucceq \beta$, and is $\Hom_{\shZ_+}(\Phi^\alpha(B), \Phi^{\alpha}(A)) =   \Hom_{X}(B,A)$ if $\alpha=\beta$. %when
%	$\Hom_{\shZ_+}(\Phi^\alpha(A), \Phi^{\alpha}(B)) =   \Hom_{X}(A,B)$ holds. 
% Hence the lemma follows.
This shows the fully faithful statements and semiorthogonal relations. The essential images of these functors are X-linear admissible subcategories by Lem. \ref{lem:local:FM}.
\end{proof}

% Lowest strata
\subsubsection*{Contributions from bottom strata}
\begin{lemma} \label{lem:bottom} If $d_+ \ge m$ (equivalently, $\ell_+ \le \delta$). Let $d_- = m$, then for all $\alpha \in B_{\ell_+, \delta - \ell_+}$,
	$$\Psi^{\alpha}(\blank) :=  r_{+\,*} r_{-}^* (\blank) \otimes \Sigma^{\alpha} \shQ_+^\vee \colon \Db(\shZ_{-}) = \Db(\Spec \kk ) \to \Db(\shZ_{+})$$
 are fully faithful, and their images form an $X$-linear admissible semiorthogonal sequence $\{\Im \Psi^{\alpha} \}_{\alpha \in B_{\ell_+, \delta - \ell_+}^{\succeq}}$ with semiorthogonal order compatible with the order of $B_{\ell_+, \delta - \ell_+}^{\succeq}$, i.e. $\Im \Psi^{\alpha} \subseteq (\Im \Psi^{\beta})^\perp$ whenever $\alpha \npreceq \beta$. 
\end{lemma}

\begin{proof} This corresponds to the case $\ell_-=0$, when the diagram of \eqref{diag:fact} degenerates into:
\begin{equation*}
	\begin{tikzcd} \widehat{\shZ} = \GG_{+} \ar[equal]{r}{j_-=\Id} \ar{d}[swap]{r_-} & \shZ_- \times \GG_{+} \ar{dl}{pr_- =r_-} \\
	\shZ_- = \Spec \kk
	\end{tikzcd}	
	\qquad\text{and}\qquad
	\begin{tikzcd} \widehat{\shZ} = \GG_{+} \ar[hook]{r}{j_+} \ar{d}[swap]{r_+ =j_+} & \shZ_+  \ar[equal]{dl}{pr_+ = \Id} \\
	\shZ_+
	\end{tikzcd}
	\end{equation*}
where $j_+$ is the inclusion of zero locus of a regular section of the vector bundle $W^\vee \boxtimes \shQ_+^\vee$. Therefore for any $A,B \in \Db(\shZ_-)$ and $\alpha, \beta \in B_{\ell_+, \delta -\ell_+}$ such that $\alpha \nsucc \beta$, %(i.e. $\alpha = \beta$ or $\alpha \npreceq \beta$), 
	$$\Hom_{\shZ_+}(\Psi^{\beta}(B), \Psi^{\alpha}(A)) = \Hom_{\shZ_-}(B,A) \otimes \Hom_{\GG_+}(j_+^*j_{+*} \Sigma^{\beta} \shQ_+^\vee, \Sigma^{\alpha} \shQ_+^\vee).$$ 
The term $j_+^*j_{+*} \Sigma^{\beta} \shQ_+^\vee$ is an iterated extension of 
	$\Sigma^{\beta} \shQ_+^\vee \otimes \bigwedge^k(W \boxtimes \,\shQ_+)[k]$ 
for $k=0,1,\ldots, m \ell_+.$ A similar computation as Lem. \ref{lem:top} shows that $\Hom_{\shZ_+}(\Psi^{\beta}(B), \Psi^{\alpha}(A)) =0$ except from the case $\alpha = \beta$ when $\Hom_{\shZ_+}(\Psi^{\alpha}(B), \Psi^{\alpha}(A)) = \Hom_{\shZ_-}(B,A)$. This shows the fully faithfulness and semiorthogonality. The X-linearity and admissibility follows from Lem. \ref{lem:local:FM}. 
\end{proof}

%%% SOD of top and bottom
\subsubsection*{Semiorthogonality of top and bottom strata} 

% Lem: sod bottom and top
\begin{lemma}\label{lem:bottom_top} Assume $1 \le m \le d_+ \le \delta$, let $ \Phi^\alpha$ and $\Psi^\beta$ be the functors defined in Lem. \ref{lem:top} and  \ref{lem:bottom}. Then for any $1 \le s \le m$, the following forms a semiorthogonal sequence:
		$$\big \langle  \{\Im \Psi^{\beta + s} \}_{\beta \in B_{\ell_+, \delta - \ell_+}^{\succeq}}, ~\{ \Im \Phi^\alpha\}_{\alpha \in B_{\delta-d_+,d_+}^{\preceq}} \big \rangle \subseteq \Db(\shZ_+),$$
	 i.e. $\Im \Psi^{\beta + s} \subseteq (\Im \Phi^{\alpha})^\perp$ for all $\alpha \in B_{\delta-d_+,d_+} ,\beta \in B_{\ell_+, \delta - \ell_+}$, $1 \le s \le m$.
\end{lemma}

\begin{proof} Notice in this case we have a commutative diagram:
	\begin{equation*}
	\begin{tikzcd}  \GG_{+} \ar[hook]{r}{j_+} \ar{d}[swap]{pr_{\GG_+}} & \shZ_+  \ar{d}{\pi_+} \\
	 \Spec \kk \ar[hook]{r}{j_{\kk}} & X.
	\end{tikzcd}
	\end{equation*}
For any $A \in \Db(X)$ and $B \in \Db(\Spec \kk)$, $\alpha \in B_{\delta-d_+,d_+}$, $\beta \in B_{\ell_+, \delta - \ell_+}$, $1 \le s \le m$, we have
	\begin{align*}
		&\Hom_{\shZ_+} (\Phi^{\alpha}(A), \Psi^{\beta + s}(B)) 
		=  \Hom_{\shZ_+} (p_+^* \Sigma^{\alpha^t} \shU_+^\vee \otimes \pi_+^*(A), p_+^* \Sigma^{\beta + s} \shQ_+^\vee \otimes j_{+\,*} \,pr_{\GG_+}^* (B) ) \\
		&= \Hom_{\GG_+} \big( pr_{\GG_+}^*\,j_{\kk}^*(A), pr_{\GG_+}^* (B) \otimes  \Sigma^{\alpha^t} \shU_+ \otimes  \Sigma^{\beta + s} \shQ_+^\vee \big) \\
		&=  \Hom_{\kk} (j_{\kk}^*(A), B) \otimes_\kk \Hom^\bullet_{\GG_+}(\Sigma^{\alpha^t} \shU_+^\vee, \Sigma^{\beta + s} \shQ_+^\vee).
	\end{align*}
(Notice this also follows from the degenerate case of Lem. \ref{lem:key}.) Now by Kapranov's result \ref{thm:Kap} \eqref{thm:Kap-2}, we have $\Hom^\bullet_{\GG_+}(\Sigma^{\alpha^t} \shU_+^\vee, \Sigma^{\beta + s} \shQ_+^\vee) = 0$ for any $1 \le s \le m$, since $B_{\delta-d_+,d_+} \cap (B_{\ell_+, \delta- \ell_+} + s) = \emptyset$ (here $B_{i,j} + s = \{\lambda + s \mid \lambda \in B_{i,j}\}$ as usual). The lemma is proved. 
\end{proof}

% Virtual flips
\medskip
\subsubsection*{Contributions from intermediate strata and virtual flip phenomenon}
% Lem: fully faithful	
\begin{lemma} \label{lem:local:ff} If $d_+ - \delta \le d_- \le d_+$, then $r_{+\,*} \, r_{-}^* \colon \Db(\shZ_-) \to \Db(\shZ_+)$ is fully faithful.
\end{lemma}

% New proof
\begin{proof} Denote $\Phi := r_{+\,*} \, r_{-}^*$, then it admits a left adjoint functor $\Phi^L = r_{-\,!} \, r_{+}^*$ by Lem. \ref{lem:local:FM}. For any $\alpha \in B_{\ell_-,d_-}$, let $F^\bullet = \Phi (\Sigma^{\alpha} \shQ_{-}^\vee)= r_{+\,*} \, r_{-}^*(\Sigma^{\alpha} \shQ_{-}^\vee)$ be the complex of vector bundles from Lem. \ref{lem:key} \eqref{lem:key-3}, in particular, $F^0 = \Sigma^{\alpha} \shQ_{+}$, and $F^p$'s for $p \ne 0$ are direct sums of the form $K_{\gamma} \otimes_\kk \Sigma^{\gamma}  \shQ_{+}$ for $\gamma \in  B_{\ell_+, d_-} \backslash B_{\ell_-, d_-}$, where $K_\gamma$ are vector spaces. Then by Lem. \ref{lem:key} \eqref{lem:key-2}, 
	$$\Phi^L(F^0) =  \Sigma^{\alpha} \shQ_{-}^\vee \quad \text{and} \quad \Phi^L(F^p) = 0, ~\text{for}~ p \ne 0.$$
Hence by considering the Postnikov system from the ``stupid" truncation (Example \ref{eg:conv:stupid}) and Lem. \ref{lem:conv}, we obtain $\Phi^L(F^\bullet) = \Phi^L(F^0) =\Sigma^{\alpha} \shQ_{-}^\vee$. Therefore for any $\alpha,\beta \in B_{\ell_-,d_-}$,
	$$\Hom_{\shZ_+}(\Phi(\Sigma^{\alpha} \shQ_{+}^\vee), \Phi(\Sigma^{\beta} \shQ_{+}^\vee)) = \Hom_{\shZ_-}(\Phi^L \, \Phi (\Sigma^{\alpha} \shQ_{+}^\vee), \Sigma^{\beta} \shQ_{+}^\vee) = \Hom_{\shZ_-}(\Sigma^{\alpha} \shQ_{+}^\vee, \Sigma^{\beta} \shQ_{+}^\vee).$$
Since $\{\Sigma^{\alpha} \shQ_{-}^\vee\}_{\alpha \in B_{\ell_-,d_-}}$ generates $\Db(\shZ_-)$ (Lem. \ref{lem:Z_i:generator}), we are done by Lem. \ref{lem:span:f.f.}.
\end{proof}

Next we consider the version of this lemma twisted by line bundles.
%%% virtual flip: twisted version
\begin{lemma}\label{lem:sod:O(i)} If $d_+ - \delta \le d_- \le d_+$, for any $i \in \ZZ$,  the functor
	$$\Phi_i : = (r_{+\,*} \, r_{-}^*(\blank)) \otimes \sO_+(i) \colon \Db(\shZ_-) \to \Db(\shZ_+).$$
is fully faithful. If $\min \{d_+ - d_-, \ell_+ - \ell_-\} > 0$, then for any fixed $i$, the essential images 
	$$\{ \Im (\Phi_i) ,  \Im (\Phi_{i+1}), \ldots, \Im (\Phi_{i+\delta-1}) \}$$ 
form an admissible $X$-linear semiorthogonal sequence.
\end{lemma}

\begin{proof} 
The fully faithfulness of $\Phi_i$ follows from previous lemma, and the X-linearity and admissibility follows from Lem. \ref{lem:local:FM} as before. It remains to show the semiorthogonal relations of $\Im (\Phi_i)$s. If $\min \{d_+ - d_-, \ell_+ - \ell_-\} > 0$, then $\delta = d_+ - d_- + \ell_+ - \ell_-  \ge 2$. For for any pair of generators $A = p_{-}^{*} \,\Sigma^{\alpha} \shQ_{-}^{\vee}, B =p_{-}^{*} \,\Sigma^{\beta} \shQ_{-}^{\vee}$ of $\Db(\shZ_-)$, where $\alpha,\beta \in B_{\ell_-,d_-}$, denote $F^\bullet_{A} \simeq \Phi_{0}(A)$ the complex of vector bundles from Lem. \ref{lem:key} \eqref{lem:key-3}, then 
	$$F^\bullet_{A} \in \Big \langle \big \{ p_+^* \Sigma^{\lambda} \shQ_+ \big \}_{\lambda \in B_{\ell_+,d_-}}   \Big \rangle \subseteq \Db(\shZ_+).$$
(Here recall $p_+ \colon \shZ_+ \to \GG_+$ is the natural projection.) Then for any $t \in \ZZ$, 
	\begin{align*}  & \Hom_{\shZ_+}( \Phi_{i+t} (A),\Phi_{i} (B))  = \Hom_{\shZ_-} \big( \Phi_0^L (\Phi_0 (A) \otimes \sO_+(t)), B \big),	
	\end{align*}
where $\Phi_0^L = r_{-\,!} \, r_{+}^*$ as previous lemma. To prove the lemma, it suffices to show that for any $t \in [1, \delta-1]$, $\Phi_0^L (\Phi_0 (A) \otimes \sO_+(t)) = 0$. By considering the Postnikov system from the ``stupid" truncation of $F_A^\bullet$ (see Example \ref{eg:conv:stupid}) and Lem. \ref{lem:conv}, it suffices to prove that
	$$\Phi_0^L (p_+^*( \Sigma^{\gamma} \shQ_+ \otimes \sO_{\GG_+}(t)) ) = r_{-\,!} \, r_{+}^*(p_+^* (\Sigma^{\gamma} \shQ_+ \otimes \sO_{\GG_+}(t))) = 0$$
 for all $\gamma \in B_{\ell_+,d_-}$,  $t \in [1,\delta-1]$. By \ref{lem:key} \eqref{lem:key-2}, this holds if the following holds:
 	$$ \Sigma^{\gamma} \shQ_+ \otimes \sO_{\GG_+}(t) \in  \Big \langle \Sigma^{\lambda} \shQ_+ \Big \rangle_{\lambda \in B_{\ell_+,d_+} \backslash B_{\ell_-, d_+}}   \subseteq \Db(\GG_+).$$
 Since $\Sigma^{\gamma} \shQ_+ \otimes \sO_{\GG_+}(s) \in \{ \Sigma^{\lambda} \shQ_+  \}_{\lambda \in B_{\ell_+,d_+-1}}$ for all $s \in [0, d_+ - d_- - 1]$, therefore above holds by Lem. \ref{lem:G:mut} applied to the case $r= \ell_+ - \ell_-$, as $\delta -1= d_+ - d_- + \ell_+ - \ell_-  -1$.
\end{proof}

\begin{example}
	\begin{enumerate}[leftmargin=*]
		\item If $\ell_- = \ell_+$, then $\shZ_+$ and $\shZ_-$ is related by a {\em flip}, the lemma implies a fully faithful embedding $\Db(\shZ_-) \hookrightarrow \Db(\shZ_+)$.
		\item If $d_- = d_+$, then $\shZ_+$ and $\shZ_-$ is related by a {\em d-critical flip}, the lemma implies a fully faithful embedding $\Db(\shZ_-) \hookrightarrow \Db(\shZ_+)$.
		\item If $\ell_+ \ne \ell_-$, $d_+ \ne d_-$, then the lemma produces $\delta$-many embeddings $\Phi_i \colon \Db(\shZ_-) \hookrightarrow \Db(\shZ_+)$ for $i=0, 1, \ldots, \delta-1$, and the essential images form an admissible $X$-linear subcategory: 
			$$\langle \Im (\Phi_0) ,  \Im (\Phi_{1}), \ldots, \Im (\Phi_{\delta-1}) \rangle \subseteq \Db(\shZ_+).$$
	\end{enumerate}
\end{example}

%% Projectivization
\subsection{The case $d_+ = 1$: projectivization} In this subsection, we consider the  case $d_+ = 1$, $d_- \in \{0,1\}$. We use the notations $\shZ_-, \widehat{\shZ}, r_-,r_+$ to denote the schemes and maps of diagram (\ref{diagram:Corr}) in the case $d_- = 1$. Hence $\ell_+ = n - 1$, $\ell_- = m - 1 = \ell_+ - \delta$. Recall $X = \Hom_\kk(W,V)$, $m = \rank W \le n = \rank V$, and  $\delta = n-m$, and denote $\sigma \colon W \otimes \sO_X \to V \otimes \sO_X$ denotes the tautological morphism. Then $\GG_+ = \PP(V)$, $\GG_- = \PP(W^\vee)$, $\shU_+ = \sO_{\PP(V)}(-1)$, $\shQ_+ = \shT_{\PP(V)}(-1)$, $\shU_- = \sO_{\PP(W^\vee)}(-1)$, $\shQ_- = \shT_{\PP(W^\vee)}(-1)$, and
	\begin{align*}
	p_+ \colon \shZ_+ & = \PP(\Coker \sigma) = \Tot_{\PP(V)}(W^\vee \otimes \Omega_{\PP(V)}(1)) \to \GG_+ = \PP(V), \\
	p_- \colon \shZ_-  &= \PP(\Coker \sigma^\vee) = \Tot_{\PP(W^\vee)}(\Omega_{\PP(W^\vee)}(1) \otimes V) \to \GG_- = \PP(W^\vee).
	\end{align*}
Recall $\pi_\pm \colon \GG_\pm \to X$ denote the natural projections. 
% Theorem: projectivization
\begin{theorem}[projectivization] \label{thm:local:proj} If $\delta: = n-m \ge 1$, then for any $i \in \ZZ$ the functors:
	\begin{align*}
		\Psi_i :=\pi_+^*(\blank) \otimes \sO_+(i) \colon \Db(X) \to \Db(\shZ_+) ,\quad \text{and} \quad
		 \Phi: = r_{+\,*} \, r_{-}^*(\blank) \colon  \Db(\shZ_-) \to \Db(\shZ_+),
	\end{align*}
are fully faithful. Furthermore, the essential images of $\Psi_i$ for $i =1,\ldots, \delta $ and $\Phi$ give rise to an admissible $X$-linear semiorthogonal decomposition of $\Db(\shZ_+)$:
	\begin{align*}
	\Db(\shZ_+) = \big \langle \Im \Phi , ~  \Im \Psi_1, \ldots, \Im \Psi_\delta \big \rangle.
	\end{align*}
\end{theorem}

\begin{proof}
This is proved (in the general situation) by the author and Leung in \cite{JL18}. Here we use the framework of this paper to provide a different proof. First, Lem. \ref{lem:top} (contributions from top stratum) implies $\Psi_i =\pi_+^*(\blank) \otimes \sO_+(i)$ is fully faithful for $i \in [1,\delta]$ and $\{\Im \Psi_1, \ldots, \Im \Psi_\delta \}$ forms a semiorthogonal sequence (notice that we twist the functors by $\sO_+(1)$ to agree with the usual convention), and Lem. \ref{lem:local:ff} implies $\Phi$ is fully faithful.

\medskip \noindent {\em Semiorthogonal relation $\Im \Phi \perp \Im \Psi_i$.}  For any $A \in \Db(X)$, $B \in \Db(\shZ_-)$, $t \in  [1,\delta]$, %we want to show
	$$\Hom_{\shZ_+}(\pi_+^*(A) \otimes \sO_+(t), \Phi(B)) = \Hom_{\shZ_-}(\Phi^L(\sO_+(t)) \otimes \pi_-^*(A), B) $$
where $\Phi^L = r_{-\,!} \, r_{+}^*$ is the left adjoint as before. To show above $\Hom$ is zero, it suffices to show $\Phi^L(\sO_+(t)) = 0$. By Lem. \ref{lem:key} \eqref{lem:key-2}, it suffices to show the following holds for all $t \in  [1,\delta]$:
	$$\sO_+(t) = p_+^* \sO_{\PP(V)}(t) \in  p_+^* \Big \langle \Sigma^{\lambda} \shQ_+ \Big \rangle_{\lambda \in B_{n-1,1} \backslash B_{m-1, 1}}.$$ 
(Here the right hand side denotes the essential image of $p_+^*$ of the source subcategory.) However, Lem. \ref{lem:G:mut} in this case implies nothing but the fact that
	\begin{align} \label{eqn:lem:proj:gen}
	\big\langle \sO_{\PP(V)}(t) \big\rangle_{t\in [1, \delta]} =  \Big \langle \Sigma^{\lambda} \shQ_+ \Big \rangle_{\lambda \in B_{n-1,1} \backslash B_{m-1, 1}} \big(= \big\langle \wedge^j \shT_{\PP(V)} (-j) \big\rangle_{j \in [m,n]} \big).
	\end{align}
Hence $\Phi^L(\sO_+(t)) = 0$ holds for $t \in  [1,\delta]$, and the desired vanishing holds.

\medskip \noindent {\em Generation.} To prove $\{\Im \Psi_i\}_{i\in[1,\delta]}$ and $\Im \Phi$ generate $\Db(\shZ_+)$, by Lem. \ref{lem:span:f.f.} it suffices to show $\langle \Im \Phi, \{\Im \Psi_i\}_{i\in[1,\delta]} \rangle$ contains a set of generators of $\Db(\shZ_+)$. It is clear that $\sO_+(i) = p_+^* \sO_{\PP(V)}(i) \in \Im \Psi_i $ for $i\in[1,\delta]$, hence by (\ref{eqn:lem:proj:gen}), we have the following inclusions: % of subcategories of $\Db(\shZ_+)$:
	$$ \Big \langle p_+^* \Sigma^{\lambda} \shQ_+ \Big \rangle_{\lambda \in B_{m-1,1} \backslash B_{n-1, 1}} \subseteq \Big \langle \{\Im \Psi_i\}_{i\in[1,\delta]} \Big \rangle.$$

By  Lem. \ref{lem:key} \eqref{lem:key-3}, the map $\Phi$ sends each generator $\Sigma^\alpha \shQ_-^\vee$ of $\Db(\shZ_-)$, $\alpha \in B_{m-1,1}$, to a complex of vector bundles $F^\bullet \simeq \Phi(\Sigma^\alpha \shQ_-^\vee)$ with $F^0 = p_+^* \Sigma^{\alpha} \shQ_{+}$, and
 	$$F^p \in  \Big \langle p_+^* \Sigma^{\lambda} \shQ_+ \Big \rangle_{\lambda \in B_{m-1,1} \backslash B_{n-1, 1}} \quad \text{for} \quad p\ne 0.$$
 Hence $F^p \in  \langle \{\Im \Psi_i\}_{i\in[1,\delta]} \rangle$ for $p\ne 0$, therefore $F^0 = p_+^*\Sigma^{\alpha} \shQ_{+} \in \langle \Im \Phi, \{\Im \Psi_i\}_{i\in[1,\delta]} \rangle$ for all $\alpha \in B_{m-1,1}$. Since $B_{n-1,1} = B_{m-1,1}\cup (B_{m-1,1} \backslash B_{n-1, 1})$, and $\{p_+^*\Sigma^{\alpha} \shQ_{+}\}_{\alpha \in B_{n-1,1}}$ generates $\Db(\shZ_+)$, the generation result is proved.
\end{proof}

% Sec: standard flips
\subsection{The case $\ell_+ = 1$: standard flips} In this subsection we consider the  standard flip case $\ell_+ = 1$, $\ell_- \in \{0,1\}$. We reserve the notations $\shZ_-, \widehat{\shZ}, r_-,r_+$ for the schemes and maps of diagram (\ref{diagram:Corr}) in the case $\ell_- = 1$, hence $d_+ = n - 1$, $d_- = m - 1 = d_+ - \delta$. Then $\shZ_+$ and $\shZ_-$ are both resolutions of the degeneracy locus $\Hom^{\le 1}(W,V)$, and $\shZ_+ \dashrightarrow \shZ_-$ is called {\em a standard flip of type $(m,n)$}.  Lem. \ref{lem:local:ff} gives us a fully faithful functor from fiber product:
	$$\Phi: = r_{+\,*} \, r_{-}^*(\blank) \colon  \Db(\shZ_-) \hookrightarrow \Db(\shZ_+).$$
Notice in this case $\GG_+ = \PP(V^\vee)$, $\GG_- = \PP(W)$, $\shQ_+ = \sO_{\PP(V^\vee)}(1)$, $\shQ_- = \sO_{\PP(W)}(1)$, and
	\begin{align*}
	p_+ \colon \shZ_+ & =  \Tot_{\PP(V^\vee)}(W^\vee \otimes \sO_{\PP(V^\vee)}(-1) )\to \GG_+ = \PP(V^\vee), \\
	p_- \colon \shZ_-  & =  \Tot_{\PP(W)} (\sO_{\PP(W)}(-1) \otimes V) \to \GG_- = \PP(W),
	\end{align*}
are the natural projections. The case $\ell_-=0$ corresponds to the bottom strata case, and since $\Sigma^{(i)} \shQ_+^\vee = \sO_+(-i)$, Lem. \ref{lem:bottom} gives us fully faithful functors for $i \in \ZZ$, 
	$$\Psi_{i}(\blank) :=  j_{+\,*} pr_{\PP(V^\vee)}^* (\blank) \otimes p_+^* \sO_{\PP(V^\vee)}(i)  \colon \Db(\Spec \kk ) \hookrightarrow \Db(\shZ_{+}) $$
and the images for $i \in [-\delta, -1]$ form a semiorthogonal sequence $\{\Im \Psi_{-\delta}, \ldots, \Im \Psi_{-2}, \Im \Psi_{-1}\}$. Here $pr_{\PP(V^\vee)} \colon \PP(V^\vee) \to \Spec \kk$ is the natural projection, and $j_+ \colon \PP(V^\vee) \hookrightarrow  \Tot_{\PP(V^\vee)}(W^\vee \otimes \sO_{\PP(V^\vee)}(-1))$ is the inclusion of zero section. In this case we have:

% Theorem: standard flip
\begin{theorem}[Standard flip] \label{thm:local:standardflip} There is an $X$-linear  semiorthogonal decomposition:
	\begin{align*}% \label{eqn:l=1:sod}
	\Db(\shZ_+)  = \big \langle  \Im \Psi_{-\delta}, \ldots, \Im \Psi_{-2}, \Im \Psi_{-1},  ~ \Im \Phi   \big \rangle.
	\end{align*}
\end{theorem}

\begin{proof} It remains to show the semiorthogonal relation $\Im \Psi_i \subseteq (\Im \Phi)^\perp$ and generation. 
%\medskip \noindent {\em Semiorthogonal relation of $\Im \Phi$ and $\Im \Psi_i$.} 
Let $A = p_-^* \sO_{\PP(W)}(-s)$ be a generator of $\Db(\shZ_-)$, $s \in [0,m-1]$, by  Lem. \ref{lem:key} \eqref{lem:key-3} in the case $\ell_-=1$, we have $\Phi (A) = p_+^*\sO_{\PP(V^\vee)}(s)$. Hence for any $B \in \Db(\Spec \kk)$ and $t \in [1,\delta]$, 	$$\Hom_{\shZ_+}(\Phi (A), \Psi_{-s}(B)) = \Hom_{\kk}(\Psi_0^L(p_+^*\sO_{\PP(V^\vee)}(t+s)), B) = 0, $$
where $\Psi_0^L$ is the left adjoint of $\Psi_0$ as usual. In fact, by  Lem. \ref{lem:key} \eqref{lem:key-2} in the case $\ell_-=0$, we see $\Psi_0^L(p_+^*\sO_{\PP(V^\vee)}(j)) = 0$ for all $(j) \in B_{1,n-1} \backslash \{0\}$, i.e. for all $j \in [1, n-1]$. Since $1 \le t + s \le m-1 + \delta = n-1$, hence $\Psi_0^L(p_+^*\sO_{\PP(V^\vee)}(t+s)) = 0$. Therefore the relation $\Im \Psi_i \subseteq (\Im \Phi)^\perp$ is proved.

\medskip \noindent {\em Generation.} 
We already see that $\Im \Phi$ contains $p_+^* \sO_{\PP(V^\vee)}(s)$ for all $s \in [0,m-1]$. On the other hand, by  Lem. \ref{lem:key} \eqref{lem:key-3} in the case $\ell_- = 0$ (tensoring with $p_+^*\sO_{\PP(V^\vee)}(-t)$), we obtain:
	$$\Psi_{-t} (\sO_{\Spec \kk}) \simeq \{ p_+^* \sO_{\PP(V^\vee)}(m-t) \otimes \wedge^m W \to \ldots \to p_+^* \sO_{\PP(V^\vee)}(1-t) \otimes W \to F^0 = p_+^* \sO_{\PP(V^\vee)}(-t)\}.$$
Therefore inductively, starting from $t=1$, we see that $\langle \{\Im \Psi_{-t}\}_{t \in [1,\delta]}, \Im \Phi \rangle$ contains  $p_+^* \sO_{\PP(V^\vee)}(-t)$ for all $t \in [1,\delta]$. Since $\{p_+^*\sO_{\PP(V^\vee)}(j)\}_{j \in [-\delta, m-1]}$ generate $\Db(\shZ_+)$, hence the images $\{\Im \Psi_{-t}\}_{t \in [1,\delta]}$ and $\Im \Phi$ generate the whole category. 
\end{proof}

\begin{remark} \label{rmk:standardflip:local:k}
The theorem holds for a field $\kk$ of arbitrary characteristic. In fact, %Lem. \ref{lem:key} (1) and (2) holds for any $\kk$, and 
the resolution of $\Psi_{-t} (\sO_{\Spec \kk})$ in the proof of the theorem from Lem.  Lem. \ref{lem:key} \eqref{lem:key-3} in this case is a Koszul resolution, hence it holds for any characteristic; On the other hand, all the involved vanishing and mutation results for $\PP(V^\vee)$ and $\PP(W)$ hold over $\kk$ by \S \ref{sec:proj.bundle}.
\end{remark}

\begin{remark} \label{rmk:standardflip:local:blowup} The fiber product $\widetilde{\shZ} = \shZ_+ \times_X \shZ_-$ is the total space of a line bundle:
		$$\widehat{\shZ} = \Tot_{ \PP(W) \times \PP(V^\vee) }(\sO_{\PP(W)}(-1) \boxtimes \sO_{\PP(V^\vee)}(-1)) \to  \PP(W) \times \PP(V^\vee),$$
and the inclusion of zero section $j \colon E = \PP(W) \times \PP(V^\vee) \to \widehat{\shZ}$ is a divisor. Then it is easy to see that $\widehat{\shZ}$ is equal to the blowup of $\shZ_+$ along the zero section $j_+ \colon \PP(V^\vee) \hookrightarrow \shZ_+$, and the blowup of $\shZ_-$ along the zero section $j_- \colon \PP(W) \hookrightarrow \shZ_-$. Moreover, $E=\PP(W) \times \PP(V^\vee)$ is the common exceptional divisor for both blowups. Hence the flip $\shZ_+ \dashrightarrow \shZ_-$ in this subsection is the universal local case for standard flips of \cite[\S 11.3]{Huy}.
\end{remark}

% m=1
\subsection{The case $m=1$: Pirozhkov's theorem}
Assume in this subsection $m=1$, and denote $d_+  = d \ge 1$, $\ell_+ = \ell$, assume $\delta = n - 1 \ge d$. The only possible choices for $d_-$ is $d_- \in \{0, 1\}$. The case $d_- = 0$ corresponds to the situation of Lem. \ref{lem:top}, from which we obtain fully faithful functors
	$$\Phi^{\alpha}(\blank) := \pi_+^*(\blank) \otimes \Sigma^{\alpha^t} \shU_+^\vee \colon \Db(X) \hookrightarrow \Db(\shZ_{+}) \quad \text{for all} \quad \alpha \in B_{\ell-1, d}.$$
The case $d_-=1$ corresponds to Lem. \ref{lem:bottom}, from which we obtain fully faithful functors
	$$\Psi^{\beta}(\blank) :=  j_{+\,*} pr_{\GG_+}^* (\blank) \otimes \Sigma^{\beta} \shQ_+^\vee \colon \Db(\Spec \kk ) \hookrightarrow \Db(\shZ_{+})  \quad \text{for all} \quad\beta \in B_{\ell, d-1},$$
where $pr_{\GG_+} \colon \GG_+ = \Gr_d(V) \to \Spec \kk$ is the natural projection and $j_+ \colon \GG_+ = \widehat{\shZ} \hookrightarrow \shZ_+ = \Tot_{\GG_+}(W^\vee \otimes \shQ_+^\vee)$ is the inclusion of zero section. 

The following theorem is due to Pirozhkov \cite{Pi20}; We provide a slightly different proof from \cite{Pi20} based on the general method of this paper. 

% Thm. m =1
\begin{theorem} \label{thm:local:gen:Cayley} There is an $X$-linear admissible semiorthogonal decomposition:
		$$\Db(\shZ_+) = \big \langle  \{\Im \Psi^{\beta + 1} \}_{\beta \in B_{\ell, d-1}^{\succeq}}, ~\{ \Im \Phi^\alpha\}_{\alpha \in B_{\ell - 1, d}^{\preceq}} \big \rangle.$$
\end{theorem}

\begin{proof} By Lem. \ref{lem:bottom_top} applied to $s=1$, the right hand side already forms a semiorthogonal sequence. We only need to show the generation result, i.e. subcategory $Span$ which they generate contains a set of generators $\{p_+^* \Sigma^{\lambda} \shQ_+^\vee \}_{\lambda \in B_{\ell,d}}$ of $\Db(\shZ_+)$.

From the images of $\Im \Phi^\alpha$, we see $Span$ contains $p_+^* A$, for any 
	$$A \in \big \langle \Sigma^{\alpha^t} \shU_+^\vee\big \rangle_{\alpha \in B_{\ell-1,d}} = \big \langle  \Sigma^{\alpha} \shQ_+^\vee\big \rangle_{\alpha \in B_{\ell-1,d}} \subseteq \Db(\GG_+). $$

We claim that for each $\beta \in B_{\ell, d-1}$, the following holds:
	\begin{align*} %\label{eqn:lem:m=1}
		p_+^* \Sigma^{\beta + 1} \shQ_+^\vee \in Span =  \big \langle  \{\Im \Psi^{\beta + 1} \}_{\beta \in B_{\ell, d-1}^{\succeq}}, ~\{ \Im \Phi^\alpha\}_{\alpha \in B_{\ell - 1, d}^{\preceq}} \big \rangle.
	\end{align*}
Since $B_{\ell,d} = B_{\ell-1,d} \cup (B_{\ell, d-1}+1)$, the generation $Span = \Db(\shZ_+)$ follows from the claim. 

For any $\beta \in  B_{\ell, d-1}$, we assume the claim holds for all $\gamma \in B_{\ell, d-1}$ such that $\gamma \prec \beta$ (this assumption is empty for the base case $\beta=0$). From  Lem. \ref{lem:key} \eqref{lem:key-3} in the case $d_-=1$, $\ell_-=0$ (and tensoring with $p_+^* \Sigma^{\beta + 1} \shQ_+^\vee$), we obtain:
	$$\Psi^{\beta+1}(\sO_{\Spec \kk}) = \{0 \to F^{-\ell} \to \ldots \to F^{-1} \to F^0 = p_+^* \Sigma^{\beta + 1} \shQ_+^\vee\},$$
where each $F^{p}$-term, for $-\ell \le p \le 0$, is given by:
	$$F^{p} = p_+^* (\wedge^{-p} \shQ_+ \otimes \Sigma^{\beta + 1} \shQ_+^\vee ) \otimes S^{-p} W \simeq p_+^* (\wedge^{\ell+p} \shQ_+^\vee \otimes  \Sigma^{\beta} \shQ_+^\vee)  \otimes S^{-p} W.$$
For any $0 \le s \le \ell-1$, consider the summands of $\wedge^{s} \shQ_+^\vee \otimes  \Sigma^{\beta} \shQ_+^\vee$ by Pieri's formula:
%  all possibilities are:
	\begin{enumerate}[leftmargin=*]
		\item If $\beta_1^t< \ell$, the summand of $\wedge^{s} \shQ_+^\vee \otimes  \Sigma^{\beta} \shQ_+^\vee$ either has the form $\Sigma^{\theta} \shQ_+^\vee$ with $\theta \in B_{\ell-1,d}$, or has the form $\Sigma^{\gamma+1} \shQ_+^\vee$ for $\gamma \prec \beta$ (the latter case is empty if $\beta=0$).
		\item If $\beta_1^t = \ell$, all summands of $\wedge^{s} \shQ_+^\vee \otimes  \Sigma^{\beta} \shQ_+^\vee$ have the form $\Sigma^{\gamma+1} \shQ_+^\vee$ for $\gamma \prec \beta$.
	\end{enumerate}
In either case, we see that for $p \ne 0$, $F^p$ is already contained $Span$, therefore $F^0 \in Span$. Hence the claim holds for $\beta$. By induction, the claim,  hence the theorem, is proved.
\end{proof}

% \delta=1
\subsection{The cases $\delta \le 3$}
\subsubsection*{The case $\delta=0$} We begin by remark that the case $\delta = n-m=0$ corresponds to the flop case, i.e. $d_+ = d_- = d$, $\ell_- = \ell_+ =: \ell$, $\shZ_+,\shZ_-$ are both crepant resolutions of $\Hom^{\le \ell}(W,V)$, and $\shZ_+  \dashrightarrow \shZ_-$ is a flop. Then Lem. \ref{lem:key} immediately implies that the fiber product correspondence induced equivalence $r_{+\,*} \, r_{-}^* \colon \Db(\shZ_-) \simeq \Db(\shZ_+)$.

\subsubsection*{The case $\delta = 1$}
Now we assume $\delta = 1$, and denote $d_+ = d$, $\ell_+ = \ell$. Then there are only two possible choices for $d_- \in \{d-1,d\}$. The case $d_- = d-1$ (resp. $d_- = d$) corresponds to the case $\ell_- = \ell$ (resp. $\ell_- = \ell-1$), and $\shZ_{-}$ is a flip (resp. d-critical flip, or say virtual flip) of $\shZ_{+}$, and we denote all the schemes and maps in diagram (\ref{diagram:Corr}) and Lem. \ref{lem:key} by the same notations but all with indices ``${\mathrm{flip}}$" (resp. ``$\mathrm{vf}$"), to distinguish the two cases. Therefore by Lem. \ref{lem:local:ff} we have two fully faithful embeddings:
	\begin{align*}
		\Phi^{\mathrm{flip}}: =  r_{+\,*}^{\mathrm{flip}} \circ r_{-}^{{\mathrm{flip}}\,*} \colon \Db(\shZ_{-}^{\mathrm{flip}}) \hookrightarrow \Db(\shZ_+), \qquad
		 \Phi^{\mathrm{vf}} : =  r_{+\,*}^{\mathrm{vf}} \circ r_{-}^{{\mathrm{vf}}\,*} \colon \Db(\shZ_{-}^{\mathrm{vf}}) \hookrightarrow \Db(\shZ_+).
	\end{align*}

% Thm: rank G = 1
\begin{theorem}  \label{thm:rk=1:local} There is an $X$-linear admissible semiorthogonal decomposition:
		$$\Db(\shZ_+) = \langle \Db(\shZ_{-}^{\mathrm{vf}}), ~\Db(\shZ_{-}^{\mathrm{flip}}) \otimes \sO_+(1)\rangle.$$
\end{theorem}

\begin{proof} To show the semiorthogonal relation, notice for any $A \in  \Db(\shZ_{-}^{\mathrm{flip}})$, $B \in \Db(\shZ_{-}^{\mathrm{vf}})$,
	$$\Hom_{\shZ_+}( \Phi^{\mathrm{flip}}(A) \otimes \sO_+(1),  \Phi^{\mathrm{vf}} (B)) = \Hom_{\shZ_{-}^{\mathrm{vf}}}((\Phi^{\mathrm{vf}})^L( \Phi^{\mathrm{flip}}(A) \otimes \sO_+(1)), B)$$
where $(\Phi^{\mathrm{vf}})^L = r_{-\,!}^{\mathrm{vf}} \circ r_{+}^{{\mathrm{vf}}\,*}$ is the left adjoint as before. Take generators $A = \Sigma^{\alpha} \shQ_-^{\mathrm{flip} \,\vee}$ of $\Db(\shZ_{-}^{\mathrm{flip}})$ for $\alpha \in B_{\ell,d-1}$, then by  Lem. \ref{lem:key} \eqref{lem:key-3} (applied to the case $d_-=d-1$, $\ell_- = \ell$), 
	$$ \Phi^{\mathrm{flip}}(A) \otimes \sO_+(1) = \Sigma^\alpha \shQ_+  \otimes \sO_+(1) = \Sigma^{\alpha+1} \shQ_+.$$
Therefore by  Lem. \ref{lem:key} \eqref{lem:key-2} (applied to the case $d_-=d$, $\ell_- = \ell-1$), we obtain that 
	$$(\Phi^{\mathrm{vf}})^L( \Phi^{\mathrm{flip}}(A) \otimes \sO_+(1)) =(\Phi^{\mathrm{vf}})^L( \Sigma^{\alpha+1} \shQ_+) = 0 \quad \text{since} \quad \alpha + 1 \in B_{\ell,d} \backslash B_{\ell-1,d}.$$

To show generation, notice by Lem. \ref{lem:key} (3) in the flip case, the image $\Im (\Phi^{\mathrm{flip}} (\blank)\otimes \sO_+(1))$ contains $\Sigma^{\gamma} \shQ_+ \in \Db(\shZ_+)$ for all $\gamma (= \alpha + 1) \in B_{\ell,d} \backslash B_{\ell-1,d}$. On the other hand, by Lem. \ref{lem:key} \eqref{lem:key-3} in the virtual flip case $d_-=d$, $\ell_- = \ell-1$, if we take generators $B = \Sigma^{\beta} \shQ_-^{\mathrm{vf} \, \vee} \in \Db(\shZ_-^{\rm vf})$, where $\beta$ is any element of $B_{\ell-1,d}$, then 
$\Phi^{\mathrm{vf}} (B) \simeq F^\bullet$, with $F^0 = \Sigma^{\beta} \shQ_{+}$, and $F^p$'s for $p \ne 0$ are direct sums of the form $K_{\gamma} \otimes_\kk \Sigma^{\gamma}  \shQ_{+}$ for $\gamma \in B_{\ell, d} \backslash B_{\ell-1, d}$, where $K_\gamma$ are vector spaces. Hence for $p \ne 0$, $F^p \in \Im (\Phi^{\mathrm{flip}}(\blank) \otimes \sO_+(1))$. Therefore 
	$$F^0 = \Sigma^{\beta} \shQ_{+} \in \langle  \Im (\Phi^{\mathrm{vf}}),  \Im (\Phi^{\mathrm{flip}} (\blank)\otimes \sO_+(1)) \rangle \quad  \text{for all} \quad \beta \in B_{\ell-1,d}.$$
Hence the right hand side contains a set of generators of $\Db(\shZ_+)$, and we are done.
\end{proof}
\begin{remark} The decomposition of the theorem is mutation-equivalent to
	$$\Db(\shZ_+) = \langle \Db(\shZ_{-}^{\mathrm{flip}}) , ~\Db(\shZ_{-}^{\mathrm{vf}}) \rangle.$$
\end{remark}

% \delta = 2
\subsubsection{The case $\delta = 2$}
Now we assume $\delta = 2$, and denote $d_+ = d$, $\ell_+ = \ell$. There are only three possible choices for $d_- \in \{d-2, d-1,d\}$. The cases $d_- = d-2, d-1,d$ correspond respectively to the case $\ell_- = \ell, \ell-1, \ell-2$ and we use upper indices ``$\mathrm{flip}$", ``$\mathrm{mid}$", and ``$\mathrm{vf}$"  respectively to label the all the schemes and maps in diagram (\ref{diagram:Corr}) and Lem. \ref{lem:key}. As the name suggest, in the case $d_- = d-2$, $\shZ_+ \dashrightarrow \shZ_{-}^{\mathrm{flip}}$
is a flip, and in the case $d_- = d$, $\shZ_{-}^{\mathrm{vf}}$ is a virtual (d-critical) flip of $\shZ_+$. For the remaining case $d_- = d-1$, $\shZ_{-}^{\mathrm{mid}}$ is the middle strata which is responsible for the orthogonal of above two contributions.

\begin{theorem} \label{thm:rk=2:local} For any $i \in \ZZ$, there are fully faithful embeddings:
	\begin{align*}
		&\Phi_i^{\mathrm{flip}}(\blank): =  r_{+\,*}^{\mathrm{flip}} \circ r_{-}^{{\mathrm{flip}}\,*} (\blank) \otimes \sO_+(i) \colon & \Db(\shZ_{-}^{\mathrm{flip}}) \hookrightarrow \Db(\shZ_+), \\
		&\Phi_i^{\mathrm{mid}}(\blank): =  r_{+\,*}^{\mathrm{mid}} \circ r_{-}^{{\mathrm{mid}}\,*} (\blank) \otimes \sO_+(i) \colon & \Db(\shZ_{-}^{\mathrm{mid}}) \hookrightarrow \Db(\shZ_+), \\
		 &\Phi_i^{\mathrm{vf}} (\blank) : =  r_{+\,*}^{\mathrm{vf}} \circ r_{-}^{{\mathrm{vf}}\,*}(\blank)  \otimes \sO_+(i)  \colon & \Db(\shZ_{-}^{\mathrm{vf}}) \hookrightarrow \Db(\shZ_+).
	\end{align*}
The images induce an $X$-linear admissible semiorthogonal decomposition:
			$$\Db(\shZ_+) = \langle \Im \Phi_{i-1}^{\mathrm{vf}}, ~ \Im \Phi_{i}^{\mathrm{mid}}, ~\Im  \Phi_{i+1}^{\mathrm{mid}}, ~ \Im \Phi_{i+2}^{\mathrm{flip}} \rangle.$$
		\end{theorem}
The set of semiorthogonal relations among these images are explicitly given in the proof. The semiorthogonal decomposition could be informatively written as:
	\begin{align*}
	\Db(\shZ_+) = \big \langle  \Db(\shZ_{-}^{\rm vf}), ~\text{2-copies of} ~ \Db(\shZ_{-}^{\rm mid}), ~ \Db(\shZ_-^{\rm flip}) \big \rangle.
	\end{align*}
		
\begin{proof} For simplicity of notations we will omit the functor $p_{\pm}^*$ in the expression of generators $\Sigma^\alpha p_{\pm}^* \shQ_{\pm} = \Sigma^\alpha\shQ_{\pm} \in \Db(\shZ_{\pm})$. The fully faithful embeddings follow directly from \ref{lem:local:ff}, and Lem. \ref {lem:sod:O(i)} shows that $\langle \Im  \Phi_{i}^{\mathrm{mid}}, \Im  \Phi_{i+1}^{\mathrm{mid}}  \rangle$ is a semiorthogonal sequence. 

To show the semiorthogonal relation, as before, for $A \in \Db(\shZ_{-}^{\mathrm{flip}})$, $B \in \Db(\shZ_{-}^{\mathrm{vf}})$, $i,t\in\ZZ$,
	$$\Hom_{\shZ_+}( \Phi_{i+t}^{\mathrm{flip}}(A),  \Phi_i^{\mathrm{vf}} (B)) = \Hom_{\shZ_{-}^{\mathrm{vf}}}((\Phi_0^{\mathrm{vf}})^L( \Phi_0^{\mathrm{flip}}(A) \otimes \sO_+(t)), B)$$
where $(\Phi_0^{\mathrm{vf}})^L = r_{-\,!}^{\mathrm{vf}} \circ r_{+}^{{\mathrm{vf}}\,*}$ is the left adjoint as before. Take generators $A = \Sigma^{\alpha} \shQ_-^{\mathrm{flip} \,\vee}$ of $\Db(\shZ_{-}^{\mathrm{flip}})$ for $\alpha \in B_{\ell,d-2}$, then by Lem. \ref{lem:key} \eqref{lem:key-3} (applied to the case $d_-=d-2$, $\ell_- = \ell$) and 
	$$ \Phi^{\mathrm{flip}}(A) \otimes \sO_+(t) = \Sigma^\alpha \shQ_+  \otimes \sO_+(t) = \Sigma^{\alpha+1} \shQ_+ \otimes \sO_+(t-1).$$
As in $\delta=1$ case, by Lem. \ref{lem:key} \eqref{lem:key-2} (applied to the case $d_-=d$, $\ell_- = \ell-1$) and  Lem. \ref{lem:G:mut} (applied to the case $r= 2$), we obtain that 
	$$(\Phi^{\mathrm{vf}})^L( \Phi^{\mathrm{flip}}(A) \otimes \sO_+(t)) =(\Phi^{\mathrm{vf}})^L( \Sigma^{\alpha+1} \shQ_+ \otimes \sO_+(t-1))  = 0$$
for $t=1,2,3$. Hence $\Im \Phi_i^{\mathrm{vf}} \subseteq (\Im \Phi_{i+t}^{\mathrm{flip}})^\perp$ for $t=1,2,3$.
The same argument (cf. proof of Lem. \ref {lem:sod:O(i)}) shows the rest of semiorthogonal relations among all strata. To summarize:
	\begin{align*}
	\Im \Phi_i^{\mathrm{vf}} \subseteq (\Im \Phi_{i+t}^{\mathrm{flip}})^\perp \quad \text{and} \quad  \Im \Phi_{i}^{\mathrm{flip}} \subseteq (\Im \Phi_{i+2-t}^{\mathrm{vf}})^\perp  \quad & \text{for} \quad  t=1,2,3, \\
	\Im \Phi_i^{\mathrm{mid}} \subseteq (\Im \Phi_{i+t}^{\mathrm{flip}})^\perp \quad \text{and} \quad  \Im \Phi_{i}^{\mathrm{flip}} \subseteq (\Im \Phi_{i+2-t}^{\mathrm{mid}})^\perp  \quad & \text{for} \quad  t=1,2, \\
	\Im \Phi_i^{\mathrm{vf}} \subseteq (\Im \Phi_{i+t}^{\mathrm{mid}})^\perp  \quad \text{and} \quad   \Im \Phi_i^{\mathrm{mid}} \subseteq (\Im \Phi_{i+2-t}^{\mathrm{vf}})^\perp  \quad & \text{for} \quad  t=1,2.
	\end{align*}

To show generation, we only need to prove in the case $i=1$. By Lem. \ref{lem:key} \eqref{lem:key-3} in the flip case $\ell_- = \ell$, the image $\Im (\Phi_{3}^{\mathrm{flip}}) = \Im (\Phi_0^{\mathrm{flip}} (\blank)\otimes \sO_+(3))$ contains exactly $\Sigma^{\gamma+1} \shQ_+ \otimes \sO_+(2) \in \Db(\shZ_+)$ for $\gamma \in B_{\ell,d-2}$. Then by Lem. \ref{lem:key} \eqref{lem:key-3} in the case $d_- = d-1$, $\ell_- = \ell - 1$, for any generator $B=\Sigma^{\beta} \shQ_-^{\mathrm{mid} \, \vee} \in \Db(\shZ_-^{\rm mid})$, $\beta \in B_{\ell-1,d-1}$, the image of $ \Phi_{2}^{\mathrm{mid}}(B) = \Phi_0^{\mathrm{mid}} (B)\otimes \sO_+(2) \simeq F^\bullet$ is a complex with $F^0 = \Sigma^{\beta} \shQ_+ \otimes \sO_+(2)$, and $F^p \in \Im (\Phi_{3}^{\mathrm{flip}})$ for $p \ne 0$. Since $B_{\ell, d-1} = (B_{\ell, d-2} + 1) \cup B_{\ell-1,d-1}$, therefore $\langle \Im (\Phi_{2}^{\mathrm{mid}}), \Im  (\Phi_{3}^{\mathrm{flip}}) \rangle $ contains all 
	$$\Sigma^{\gamma} \shQ_+ \otimes \sO_+(2) \in \Db(\shZ_+), \qquad \forall \gamma \in B_{\ell, d-2}.$$
Next, by Lem. \ref{lem:key} \eqref{lem:key-3} in the case $d_- = d-1$, $\ell_- = \ell - 1$ again, for $\beta \in B_{\ell-1,d-1}$, $\Phi_{1}^{\mathrm{mid}}(\Sigma^{\beta} \shQ_-^{\mathrm{mid} \, \vee} ) \simeq F'^\bullet$ with $F'^0 = \Sigma^{\beta} \shQ_+ \otimes \sO_+(1)$ and $F'^p \in \langle \Im (\Phi_{2}^{\mathrm{mid}}), \Im  (\Phi_{3}^{\mathrm{flip}}) \rangle$. Hence 
	$$\Sigma^{\beta} \shQ_+ \otimes \sO_+(1) \in \langle \Im (\Phi_{1}^{\mathrm{mid}}), \Im (\Phi_{2}^{\mathrm{mid}}), \Im (\Phi_{3}^{\mathrm{flip}}) \rangle, \qquad \forall \beta  \in B_{\ell-1,d-1}.$$
This shows that $\langle \Im (\Phi_{1}^{\mathrm{mid}}), \Im (\Phi_{2}^{\mathrm{mid}}), \Im (\Phi_{3}^{\mathrm{flip}}) \rangle$ contains all
	$$  \{\Sigma^{\gamma} \shQ_+ \otimes \sO_+(1) \}_{\gamma \in B_{\ell,d-1}} \cup \{ \Sigma^{\gamma} \shQ_+ \otimes \sO_+(2) \}_{\gamma \in B_{\ell,d-1}} \subseteq \Db(\shZ_+).$$
By Lem. \ref{lem:G:mut} in the case $r=2$, this shows that
	$$\Sigma^{\gamma} \shQ_+ \in \langle \Im (\Phi_{1}^{\mathrm{mid}}), \Im (\Phi_{2}^{\mathrm{mid}}), \Im (\Phi_{3}^{\mathrm{flip}}) \rangle, \qquad \forall \gamma \in B_{\ell,d} \backslash B_{\ell-2,d}.$$
Finally, by Lem. \ref{lem:key} \eqref{lem:key-3} in the case $\ell_- = \ell-2$, for any $\alpha \in B_{\ell-2,d}$, we have $\Phi_0^{\mathrm{vf}} (\Sigma^{\alpha} \shQ_-^{\mathrm{vf} \, \vee}) \simeq F''^\bullet$ with $F''^0 =  \Sigma^{\alpha} \shQ_+$, and 
	$F''^p \in \langle \Im (\Phi_{1}^{\mathrm{mid}}), \Im (\Phi_{2}^{\mathrm{mid}}), \Im (\Phi_{3}^{\mathrm{flip}}) \rangle$ for $p \ne 0.$
This shows that
	$$F''^0 =  \Sigma^{\alpha} \shQ_+ \in \langle \Im (\Phi_0^{\mathrm{vf}}), \Im (\Phi_{1}^{\mathrm{mid}}), \Im (\Phi_{2}^{\mathrm{mid}}), \Im (\Phi_{3}^{\mathrm{flip}}) \rangle, \qquad \forall \alpha \in B_{\ell-2,d}.$$
Since $B_{\ell,d} = B_{\ell-2,d}\cup (B_{\ell,d} \backslash B_{\ell-2,d})$, hence the right hand side contains a set of generators $\{ \Sigma^\alpha \shQ_+ \}_{\alpha \in B_{\ell, d}}$ of $\Db(\shZ_+)$, hence the generation is proved. 
\end{proof}

\begin{remark} There are many different mutation-equivalent ways to rewrite above semiorthogonal decomposition via the semiorthogonal relationships in the proof, for example:
	\begin{align*} 
	\Db(\shZ_+) & =   \langle \Im \Phi_{i}^{\mathrm{mid}},  \Im \Phi_i^{\mathrm{vf}}, ~\Im  \Phi_{i+1}^{\mathrm{mid}}, ~ \Im \Phi_{i+2}^{\mathrm{flip}} \rangle
	% =   \langle \Im \Phi_{i}^{\mathrm{mid}}, \Im  \Phi_{i+1}^{\mathrm{mid}}, ~ \Im \Phi_{i+1}^{\mathrm{vf}},   ~ \Im \Phi_{i+2}^{\mathrm{flip}} \rangle.
%	&  = \langle  \Im \Phi_{i}^{\mathrm{mid}}, \Im  \Phi_{i+1}^{\mathrm{mid}},   ~ \Im \Phi_{i+2}^{\mathrm{flip}}, ~ \Im \Phi_{i+1}^{\mathrm{vf}}  \rangle \\
	 =     \langle  \Im \Phi_{i}^{\mathrm{flip}}, ~ \Im \Phi_{i}^{\mathrm{mid}}, \Im  \Phi_{i+1}^{\mathrm{mid}},  ~ \Im \Phi_{i+1}^{\mathrm{vf}}  \rangle. % \\
%	& =  \langle   \Im \Phi_{i-1}^{\mathrm{vf}}, ~ \Im \Phi_{i}^{\mathrm{mid}}, ~\Im \Phi_{i+1}^{\mathrm{flip}}, ~\Im  \Phi_{i+1}^{\mathrm{mid}} \rangle. 
	\end{align*}
\end{remark}

% \delta = 3
\subsubsection{The case $\delta = 3$}
Now consider $\delta = 3$, and denote $d_+ = d$, $\ell_+ = \ell$. Then $d_- \in \{d-3,d-2,d-1,d\}$. We may assume $d \ge 3$ and $\ell \ge 3$, as other the situation degenerates into one of the previous cases. We label the schemes and maps of diagram (\ref{diag:fact}) by the upper index $d_-$ as before. But for the case $d_- = d-3$ and resp. $d_- = d$, we also use  
upper indices ``$\mathrm{flip}$" and resp. ``$\mathrm{vf}$". As the names suggest, in the case $d_- = d-3$, $\shZ_+ \dashrightarrow \shZ_-^{(d-3)} = \shZ_{-}^{\mathrm{flip}}$
is a flip, and in the case $d_- = d$, $\shZ_{-}^{(d)} = \shZ_{-}^{\mathrm{vf}}$ is a virtual (d-critical) flip of $\shZ_+$. 

% thm: delta=3
\begin{theorem} \label{thm:rk=3:local} For any $i \in \ZZ$, there are fully faithful embeddings:
	\begin{align*}
		&\Phi_i^{(d-3)}(\blank) \equiv \Phi_i^{\mathrm{flip}}(\blank): =  r_{+\,*}^{\mathrm{flip}} \circ r_{-}^{{\mathrm{flip}}\,*} (\blank) \otimes \sO_+(i) \colon & \Db(\shZ_{-}^{\mathrm{flip}}) \hookrightarrow \Db(\shZ_+), \\
		&\Phi_i^{(d-2)}(\blank): =  r_{+\,*}^{(d-2)} \circ r_{-}^{(d-2)\,*} (\blank) \otimes \sO_+(i) \colon & \Db(\shZ_{-}^{(d-2)}) \hookrightarrow \Db(\shZ_+), \\
		&\Phi_i^{(d-1)}(\blank): =  r_{+\,*}^{(d-1)} \circ r_{-}^{(d-1)\,*} (\blank) \otimes \sO_+(i) \colon & \Db(\shZ_{-}^{(d-1)}) \hookrightarrow \Db(\shZ_+), \\
		 &\Phi_i^{(d)}(\blank) \equiv \Phi_i^{\mathrm{vf}} (\blank) : =  r_{+\,*}^{\mathrm{vf}} \circ r_{-}^{{\mathrm{vf}}\,*}(\blank)  \otimes \sO_+(i)  \colon & \Db(\shZ_{-}^{\mathrm{vf}}) \hookrightarrow \Db(\shZ_+).
	\end{align*}
(The semiorthogonal relations among these images are explicitly given below.) Furthermore, there is an $X$-linear admissible semiorthogonal decomposition:
	$$\Db(\shZ_+) = \langle \Im \Phi_{i-1}^{\mathrm{vf}}, ~ \Im \Phi_{i}^{(d-1)},  \Im \Phi_{i+1}^{(d-2)},  \Im \Phi_{i+1}^{(d-1)},  \Im \Phi_{i+2}^{(d-2)},  \Im \Phi_{i+2}^{(d-1)} , \Im \Phi_{i+3}^{(d-2)}, ~ \Im \Phi_{i+4}^{\mathrm{flip}} \rangle.$$
\end{theorem}

 The semiorthogonal decomposition could be informatively written as:
	\begin{align*}
	\Db(\shZ_+) = \big \langle  \Db(\shZ_-^{\rm vf}) , ~\text{3-copies of} ~ \Db(\shZ_{-}^{(d-1)}), ~\text{3-copies of} ~  \Db(\shZ_{-}^{(d-2)}), ~ \Db(\shZ_{-}^{\rm flip}) \big \rangle.
	\end{align*}

% Proof
\begin{proof} The exact same arguments as the $\delta=2$ case show that:
	\begin{align*}
	\Im \Phi_i^{\mathrm{vf}} \subseteq (\Im \Phi_{i+t}^{\mathrm{flip}})^\perp \quad \text{and} \quad  \Im \Phi_{i}^{\mathrm{flip}} \subseteq (\Im \Phi_{i+3-t}^{\mathrm{vf}})^\perp  \quad & \text{for} \quad  t=1,2,3,4,5; \\
	 \Im \Phi_{i}^{(d-1)} \subseteq (\Im \Phi_{i+t}^{\mathrm{flip}})^\perp \quad \text{and} \quad  \Im \Phi_{i}^{\mathrm{flip}} \subseteq (\Im \Phi_{i+3-t}^{(d-1)})^\perp  \quad & \text{for} \quad  t=1,2,3,4; \\
	 \Im \Phi_i^{\mathrm{vf}} \subseteq (\Im \Phi_{i+t}^{(d-2)})^\perp  \quad \text{and} \quad   \Im \Phi_i^{(d-2)} \subseteq (\Im \Phi_{i+3-t}^{\mathrm{vf}})^\perp  \quad & \text{for} \quad  t=1,2,3,4; \\
	 \Im \Phi_{i}^{(d-2)} \subseteq (\Im \Phi_{i+t}^{\mathrm{flip}})^\perp \quad \text{and} \quad  \Im \Phi_{i}^{\mathrm{flip}} \subseteq (\Im \Phi_{i+3-t}^{(d-2)})^\perp  \quad & \text{for} \quad  t=1,2,3;\\
		\Im \Phi_{i}^{(d-1)}  \subseteq  (\Im \Phi_{i+t}^{(d-2)})^\perp  \quad \text{and} \quad   \Im \Phi_{i+t}^{(d-2)} \subseteq (\Im \Phi_{i+3-t}^{(d-1)})^\perp  \quad & \text{for} \quad  t=1,2,3;\\
	\Im \Phi_i^{\mathrm{vf}} \subseteq (\Im \Phi_{i+t}^{(d-1)})^\perp  \quad \text{and} \quad   \Im \Phi_i^{(d-1)} \subseteq (\Im \Phi_{i+3-t}^{\mathrm{vf}})^\perp  \quad & \text{for} \quad  t=1,2,3.
	\end{align*}

\medskip \noindent \textit{Generation}. The proof of generation is also similar to the case $\delta=2$. We only need to prove in the case $i=1$, i.e. to show the subcategory generated by the images: 
	$$\shS pan : = \big\langle  \Im \Phi_{0}^{\mathrm{vf}}, ~~ \{ \Im \Phi_{t}^{(d-1)}\}_{t = 1,2,3},~~ \{\Im  \Phi_{t}^{(d-2)}\}_{t = 2,3,4},~~ \Im \Phi_{5}^{\mathrm{flip}} \big \rangle$$
 contains a set of generators $\{\Sigma^{\alpha} \shQ_+\}_{\alpha \in B_{\ell,d}}$ of $\Db(\shZ_+)$.
 
 First consider $\Phi_{5}^{\mathrm{flip}}$. By Lem. \ref{lem:key} \eqref{lem:key-3} in the flip case $d_-= d-3$, $\ell_- = \ell$, the image $\Im (\Phi_{5}^{\mathrm{flip}})$ contains $\Phi_{5}(\Sigma^\alpha \shQ_-^{\mathrm{flip} \,\vee}) =  p_+^*\Sigma^{\alpha+1}\shQ_+ \otimes \sO_+(4)$ for all $\alpha \in B_{\ell,d-3}$. 
 
Secondly, consider $\Im \Phi_{t}^{(d-2)}$ for $t=4,3,2$. First let $j=4$, then by Lem. \ref{lem:key} \eqref{lem:key-3} in the case $d_- = d-2$, $\ell_- = \ell - 1$, for any generator $A=\Sigma^{\beta} \shQ_-^{(d-2) \, \vee}$, $\beta \in B_{\ell-1,d-2}$, the image $\Phi_{4}^{(d-2)}(A)$ hits $F^0 = \Sigma^{\beta} \shQ_+ \otimes \sO_+(4)$ modulo elements of $\Im \Phi_{5}^{\mathrm{flip}}$ (since $F^p \in \Phi_{5}^{\mathrm{flip}}$ for $p \ne 0$). Hence combined with $\Im \Phi_{5}^{\mathrm{flip}}$, we see $\langle \Im \Phi_{4}^{(d-2)},  \Im \Phi_{5}^{\mathrm{flip}} \rangle$ contains  $p_+^*\Sigma^{\alpha} \shQ_+ \otimes \sO_+(4)$ for all $\alpha \in B_{\ell,d-2}$. Inductively,  the exact same argument for the case $t=3$ and $t=2$ shows that $\langle \{\Im  \Phi_{t}^{(d-2)}\}_{t = 2,3,4},~ \Im \Phi_{5}^{\mathrm{flip}} \rangle \subseteq \shS pan$ contains all elements of the form:
	$$p_+^*\Sigma^{\alpha} \shQ_+ \otimes \sO_+(t) \quad \text{for} \quad \alpha \in B_{\ell,d-2}, ~ t=2,3,4.$$
 
Thirdly, consider $\Im \Phi_{t}^{(d-1)}$ for $t=3,2,1$. For $t=3$, consider the twisted version Lem. \ref{lem:key:twist} for $d_- = d-1$, $\ell_- = \ell - 2$ and $j=1$. Since the following holds for Young diagrams:
 	$$B_{\ell,d-1}  = B_{\ell, d-2} \cup B_{\ell-2,d-1}^{\{1\}} \cup (B_{\ell-1,d-2}^{\{1\}}+1),$$
and (if twisted by $\sO_+(3)$) we already see that $\shS pan$ contains all the elements $p_+^*\Sigma^{\alpha} \shQ_+ \otimes \sO_+(3)$ for $\alpha \in B_{\ell, d-2} \cup (B_{\ell-1,d-2}^{\{1\}}+1)$, hence by Lem. \ref{lem:key:twist}, let $A=\Sigma^{\alpha} \shQ_-^{(d-1) \, \vee} \otimes \sO_{-}^{(d-1)}(1) \in \Db(\shZ_-^{(d-1)})$, $\alpha \in B_{\ell-2,d-1}$, then $ \Phi_{3}^{(d-1)}(A)$ hits the element $p_+^*\Sigma^{\alpha} \shQ_+ \otimes \sO_+(3)$ for $\alpha \in B_{\ell-2,d-1}^{\{j=1\}}$ modulo elements of $\shS pan$. This shows that $\langle \Im \Phi_{3}^{(d-1)}, \{\Im  \Phi_{t}^{(d-2)}\}_{t = 2,3,4},~ \Im \Phi_{5}^{\mathrm{flip}} \rangle$ contains $p_+^*\Sigma^{\alpha} \shQ_+ \otimes \sO_+(3)$ for all $\alpha \in B_{\ell,d-1}$. Similarly, the same argument works for $t=2$, hence $\shS pan$ contains all elements of the form:
	$$p_+^*\Sigma^{\alpha} \shQ_+ \otimes \sO_+(t) \quad \text{for} \quad \alpha \in B_{\ell,d-1}, ~ t=2,3.$$
Therefore by Lem. \ref{lem:G:mut} in the case $r=2$ (twisted by $\sO_+(1)$),  $\shS pan$ contains:
	$$p_+^*\Sigma^{\alpha} \shQ_+ \otimes \sO_+(1) \quad \text{for} \quad \alpha \in B_{\ell,d} \backslash B_{\ell-2,d}.$$
Now consider $t = 1$, by Lem. \ref{lem:key} in the case of $d_- = d-1$, $\ell_- = \ell - 2$ (twisted by $\sO_+(1)$), $\Im \Phi_{1}^{(d-1)}$ hits all elements $p_+^*\Sigma^{\alpha} \shQ_+ \otimes \sO_+(1)$ for $\alpha \in B_{\ell-2, d-1}$ modulo elements of above form (i.e. for $\alpha \in B_{\ell,d-1} \backslash B_{\ell-2,d-1} \subseteq B_{\ell,d} \backslash B_{\ell-2,d}$). Together, this shows that $\shS pan$ contains:
	$$p_+^*\Sigma^{\alpha} \shQ_+ \otimes \sO_+(1) \quad \text{for} \quad \alpha \in B_{\ell,d-1}.$$
Hence by  Lem. \ref{lem:G:mut} in the case $r=3$, $\shS pan$ contains all elements:
	$$p_+^*\Sigma^{\alpha} \shQ_+ \quad \text{for} \quad \alpha \in B_{\ell,d} \backslash B_{\ell-3,d}.$$	

Finally, consider $\Im \Phi_{0}^{\mathrm{vf}}$. By Lem. \ref{lem:key} \eqref{lem:key-3}in the case $d_- = d$, $\ell_- = \ell - 3$, we see that $\Im \Phi_{0}^{\mathrm{vf}}$ hits all 
$p_+^*\Sigma^{\alpha} \shQ_+$ for $\alpha \in  B_{\ell-3,d}$ if modulo elements that are already in $\shS pan$ by previous step. This shows that $\shS pan$ contains all $p_+^*\Sigma^{\alpha} \shQ_+$ for $\alpha \in B_{\ell,d}$, hence the generation is proved.
\end{proof}

\begin{remark} There are many different mutation-equivalent ways to rewrite above semiorthogonal decomposition as previous cases, for example:
	\begin{align*} 
	\Db(\shZ_+) = \langle  \Im \Phi_{i}^{\mathrm{flip}}, ~~\Im \Phi_{i}^{(d-2)}, \Im \Phi_{i}^{(d-1)}, \Im \Phi_{i+1}^{(d-2)},  \Im \Phi_{i+1}^{(d-1)} , \Im \Phi_{i+2}^{(d-2)}, \Im \Phi_{i+1}^{(d-1)}, ~~ \Im \Phi_{i+2}^{\mathrm{vf}} \rangle.
	\end{align*}
\end{remark}

% Case d=2
\subsection{The case $d_+= 2$: $\Quot_2$-formula}
In this section we focus on the case $d_+=2$, $d_- \in \{0,1,2\}$. Denote the line bundle  $\sO_{+}(1): = p_+^* \sO_{\GG_+}(1) \in \Pic(\shZ_+)$. In order to distinguish the notations for different $d_-$, we label the schemes and maps of diagram (\ref{diag:fact}) by the upper index $d_- \in \{0,1,2\}$, for example the projections from fibered product is written as $r_{-}^{(d_-)} \colon \widehat{\shZ}^{(d_-)}   \to \shZ_{-}^{(d_-)} $ and $r_{+}^{(d_-)} \colon \widehat{\shZ}^{(d_-)}   \to \shZ_{+}$ for $d_- \in \{0,1,2\}$, where $\shZ_{+} = \shZ_{+}^{(d_-)}$ is the same scheme for all $d_-$. Notice that we may assume that $m \ge 2$ and $\delta \ge 2$, as otherwise the situation degenerates to one of previous cases. Then $n = m + \delta \ge 4$, $\ell_+ = m + \delta -2 \ge 2$.

% Thm: local d=2
\begin{theorem} \label{thm:local:d=2} In above situation, for any $k \in \ZZ$, $\alpha \in B_{\delta-2,2}$,  the following functors:
 	\begin{align*}
	& \Omega_k(\blank) : = r_{+ \,*}^{(2)} \circ r_{-}^{(2) \,*}(\blank) \otimes \sO_+(k) \colon  \qquad & \Db(\shZ_{-}^{(2)}) \lhook\joinrel\longrightarrow \Db(\shZ_+), \\
	&\Phi_k(\blank) := r_{+ \,*}^{(1)} \circ r_{-}^{(1)\,*}(\blank) \otimes \sO_+(k)  \colon \qquad &\Db(\shZ_{-}^{(1)})  \lhook\joinrel\longrightarrow \Db(\shZ_+), \\
	&\Psi_{\alpha,k}(\blank) := \pi_+^*(\blank) \otimes \Sigma^{\alpha^t} \shU_+^\vee  \otimes \sO_+(k) \colon  \qquad  &\Db(X) \lhook\joinrel \longrightarrow \Db(\shZ_{+})
	\end{align*}
are fully faithful. Furthermore, for any fixed $k \in \ZZ$, the images $\Im \Omega_{k-1}$, $\{\Im \Phi_i\}_{i \in [k, k+\delta-1]}$ and $\{\Im \Psi_{\alpha,k+1} \}_{\alpha \in B_{\delta-2,2}}$ induce a semiorthogonal decomposition
	\begin{align} \label{eqn:d=2:sod}
	\Db(\shZ_+) = \big\langle \Im \Omega_{k-1}, \{\Im \Phi_{k+i}\}_{i \in [0, \delta-1]}, \{\Im \Psi_{\alpha,k+1} \}_{\alpha \in B_{\delta-2,2}}  \big\rangle, 
	\end{align}
 with semiorthogonal order given by any total order extending the following partial orthogonal order: for any $\alpha, \beta \in B_{\delta-2,2}$, $i \in [0,\delta-1]$, the following holds:
	 \begin{align}
	& \Im \Omega_{k-1} \subseteq (\Im \Phi_{k+i})^\perp \cup (\Im \Psi_{\alpha, k+1})^\perp & \forall &\quad   i, \alpha \text{~above}; \label{eqn:d=2:order1} \\
	& \Im \Phi_{k+i}  \subseteq  (\Im \Phi_{k+j})^\perp \cup (\Im \Psi_{\alpha, k+1})^\perp & \forall &  \quad  i <  j \le  i + \delta -1,  \alpha +1 \npreceq (2^i);   \label{eqn:d=2:order2} \\
	& \Im \Psi_{\alpha, k+1} \subseteq (\Im \Phi_{k+i})^\perp \cup (\Im \Psi_{\beta, k+1})^\perp & \forall &   \quad (2^i)\npreceq \alpha, \beta \npreceq \alpha,  \label{eqn:d=2:order3}
	\end{align}
where $(2^i)= (2,2,\ldots,2)$ is understood as an element of 
$B_{\delta-1,2} \supset B_{\delta-2,2}$. 
\end{theorem}
The semiorthogonal decomposition of the theorem can be informatively written as
	\begin{align*}
	\Db(\shZ_+) = \big \langle \Db(\shZ_{-}^{(2)}), ~\text{$\delta$-copies of} ~ \Db(\shZ_{-}^{(1)}), ~  \text{$\binom{\delta}{2}$-copies of} ~ \Db(X) \big \rangle.
	\end{align*}
A few words about the semiorthogonal order: note that for any $\alpha \in B_{\delta-2,2}$, $\alpha^t = (a+b, b)$ where $a,b \ge 0$ and $a+b \le \delta-2$, then $ \Sigma^{\alpha^t} \shU_+^\vee  = S^a  \shU_+^\vee \otimes \sO_+(b)$ and 
	$$\Psi_{\alpha, k+1}(\blank) := \pi_+^*(\blank) \otimes S^a  \shU_+^\vee \otimes \sO_+(b)  \otimes \sO_+(k+1).$$
Let $i \in [0, \delta-1]$, then the semiorthogonal order between $\Im \Phi_{k+i}$ and $\Im \Psi_{\alpha,k+1}$ is as follows: %(for a fixed $k \in \ZZ$),
	\begin{align*}
	& \Im \Phi_{k+i}  \subseteq (\Im \Psi_{\alpha, k+1})^\perp  \quad \text{if} \quad   \alpha +1 \npreceq (2^i) \iff i \le a+b ; \\ % (\le i+\delta-2); 
	& \Im \Psi_{\alpha, k+1} \subseteq (\Im \Phi_{k+i})^\perp \quad \text{if} \quad \quad (2^i)\npreceq \alpha \iff  b+1 \le i . %(\le b + \delta -1).
	\end{align*}
In particular, $\Im \Phi_{k+i}$ and $\Im \Psi_{\alpha, k+1}$ are totally orthogonal to each other if 
	$$a \ge 1 \quad \text{and} \quad b+1 \le i \le a+b.$$

\begin{proof} For any fixed $k \in \ZZ$, it follows from Lem. \ref{lem:local:ff} and the fact that $(\blank)\otimes \sO_+(k) \colon \Db(\shZ_+) \to \Db(\shZ_+)$ is an autoequivalence that the functors $\Omega_k$, $\Phi_k$ and $\Psi_{\alpha,k}$ are fully faithful, and from Lem. \ref {lem:sod:O(i)} (resp. Lem. \ref{lem:top}) that $\{\Im \Phi_{k+i}\}_{i \in [0, \delta-1]}$ (resp. $\{\Im \Psi_{\alpha,k+1} \}_{\alpha \in B_{\delta-2,2}}$) forms a semiorthogonal sequence. We next show the semiorthogonal relations (\ref{eqn:d=2:order1}, \ref{eqn:d=2:order2}, \ref{eqn:d=2:order3}).

% First SOD relation
\medskip \noindent \textit{Semiorthogonal relation (\ref{eqn:d=2:order1})}. Similarly as before, for any $A \in \Db(\shZ_{-}^{(2)})$, $B =p_{-}^{(1)\,*} \,\Sigma^{\beta} \shQ_{-}^{(1) \,\vee} \in \Db(\shZ_-^{(1)})$, where $\beta \in B_{m-1,1}$, and any $i \in [0,\delta-1]$, we have %then similarly as before:
	$$\Hom_{\shZ_+}(\Phi_{k+i}(B), \Omega_{k-1}(A)) = \Hom_{\shZ_-^{(2)}} \big(\Omega_0^L (\Phi_{0}(B) \otimes \sO_{+}(i+1)) , A \big) = 0,$$
since from Lem. \ref{lem:key} \eqref{lem:key-2}, $\Omega_0^L(E) = r_{+\,!}^{(2)} \, r_{-}^{(2)\,*} (E) = 0$ for any
	$$E \in \shV an^{(2)} : = \Big \langle p_+^*  \big \langle \Sigma^{\lambda} \shQ_+ \big \rangle_{\lambda \in B_{n-2,2} \backslash B_{n-\delta-2,2}}   \Big \rangle \subseteq \Db(\shZ_+),$$
and $\Phi_{0}(B) \otimes \sO_{+}(i+1) \in \shV an^{(2)}$ for all $i+1 \in [1, \delta]$ since
	% G+
	$$\Big \langle \Sigma^{\lambda} \shQ_+ \otimes \sO(i+1)  \Big \rangle_{\lambda \in B_{n-2,1}} \subseteq  \Big \langle \Sigma^{\lambda} \shQ_+ \Big \rangle_{\lambda \in B_{n-2,2} \backslash B_{n-\delta-2,2}}   \subseteq \Db(\GG_+).$$
which holds by Lem. \ref{lem:G_2:mut} in the case $r = \delta$.

For the second relation, for any $A \in \Db(\shZ_{-}^{(2)}), B \in \Db(X)$, by $X$-linearity:
	$$\Hom_{\shZ_+}(\Psi_{\alpha,k+1}(B), \Omega_{k-1}(A)) = \Hom_{\shZ_-^{(2)}} \big(\Omega_0^L (p_+^*\, \Sigma^{\alpha^t} \shU_+^\vee \otimes \sO_+(2)) \otimes B, A \big) = 0,$$
since $\Omega_0^L (p_+^*\, \Sigma^{\alpha^t} \shU_+^\vee \otimes \sO_+(2)) = 0$ for any $\alpha \in B_{\delta-2,2}$. In fact, $p_+^*\, \Sigma^{\alpha^t} \shU_+^\vee \otimes \sO_+(2) \in \shV an^{(2)}$, which follows from Lem. \ref{lem:G_2:mut}) in the case $r = \delta$:
	% G+
	$$ \Sigma^{\alpha^t} \shU_+^\vee \otimes \sO_{\GG_+}(2) = S^a \shU_+ \otimes \sO_{\GG_+}(a+b+2) \in  \Big \langle \Sigma^{\lambda} \shQ_+ \Big \rangle_{\lambda \in B_{n-2,2} \backslash B_{n-\delta-2,2}} \subseteq \Db(\GG_+),$$
where $\alpha^t = (a+b,b)$, $a,b \ge 0$, $a+b \le \delta-2$ (hence $a < \ell$) as before.

% Second SOD relation
\medskip \noindent \textit{Semiorthogonal relation (\ref{eqn:d=2:order2})}.
The first relation is equivalent to $\{\Im \Phi_{k+i}\}_{i \in [0, \delta-1]}$ forms a semiorthogonal sequence, which has been proved. For the second relation, for any $A \in \Db(\shZ_{-}^{(1)}), B \in \Db(X)$, we want find the condition when the following holds:
	\begin{equation}\label{eqn:d=2:SOD2}
	\Hom_{\shZ_+}(\Psi_{\alpha,k+1}(B), \Phi_{k+i}(A)) = \Hom_{\shZ_-^{(1)}} \big( \Phi_0^L (p_+^*\, \Sigma^{\alpha^t} \shU_+^\vee \otimes \sO_+(1-i)) \otimes B, A   \big) = 0.
	\end{equation}
By Lem. \ref{lem:key} \eqref{lem:key-2}, $ \Phi_0^L(E)  = r_{+\,!}^{(1)} \, r_{-}^{(1)\,*}(E) = 0$ for any $E \in \shV an^{(1)} $, where
	\begin{align*}
	 \shV an^{(1)} : = \Big \langle p_+^*  \big \langle \Sigma^{\lambda} \shQ_+ \big \rangle_{\lambda \in  B_{n-2,2} \backslash B_{n-\delta -1,2}}   \Big \rangle \subseteq \Db(\shZ_+).
	\end{align*}
Hence if we write $\alpha^t = (a+b, b)$ as before (then $a,b \ge 0$ and $a+b\le \delta-2$), then $ \Phi_0^L (p_+^*\, \Sigma^{\alpha^t} \shU_+^\vee \otimes \sO_+(1-i))  = 0$ if the following holds:
	% G+
	$$ \Sigma^{\alpha^t} \shU_+^\vee \otimes \sO_{\GG_+}(1-i) = S^a \shU_+  \otimes \sO_{\GG_+}(a+b+1-i)   \in   \Big \langle \Sigma^{\lambda} \shQ_+ \Big \rangle_{\lambda \in B_{n-2,2} \backslash B_{n-\delta-1,2}} \subseteq \Db(\GG_+)$$
By Lem. \ref{lem:G_2:mut} in the case $r = \delta -1$, above holds for all $0 \le a \le \ell_+$ and $1 \le a+ b + 1 -i \le a+ \delta$. Hence we obtain that (\ref{eqn:d=2:SOD2}) holds for any $\alpha^t = (a+b,b) \in B_{2,\delta-2}$ and $i \in [b - (\delta-1), a + b]$. This in particular holds if $0 \le i \le a+b \iff \alpha + 1 \npreceq (2^i)$.

% Third SOD relation
\medskip \noindent \textit{Semiorthogonal relation (\ref{eqn:d=2:order3})}. The second relation follows from Lem. \ref{lem:top}. For the first relation, for any $A \in \Db(\shZ_{-}^{(1)}), B \in \Db(X)$, similarly as before: 
	\begin{align*} 
	&\Hom_{\shZ_+}(\Phi_{k+i}(A), \Psi_{\alpha, \,k+1}(B)) =\Hom_{\widehat{\shZ}^{(1)}} \big( r_{-}^{(1)\,*} A, ~ r_{+}^{(1)\,!} (\Psi_{\alpha, 1-i}(B)) \big)  \\
	= & \Hom_{\widehat{\shZ}^{(1)}} \big( r_{-}^{(1)\,*} A, ~~r_{+}^{(1)\,*}  (p_+^* \Sigma^{\alpha^t} \shU_+^\vee \otimes \sO_+(1-i)\otimes  \pi_+^*B )  \otimes \omega_{r_+^{(1)}} \big) \\
	 = & \Hom_{\shZ_-^{(1)}}  \Big(A, ~~ r_{-\,*}^{(1)} \, r_{+}^{(1)\,*} p_+^* \big(\Sigma^{\alpha^t} \shU_+^\vee \otimes \sO_+(-i) \big) \otimes  B \otimes \sO_{-}^{(1)}(\delta-2)[\dim r_+^{(1)}] \Big),
	\end{align*}
where we use the functors are all $X$-linear and $\omega_{r_+^{(1)}} = \sO_-(\delta-2) \otimes \sO_+(-1)[\dim r_+^{(1)}]$. To show above $\Hom$ space vanishes,  it suffices to show $r_{-\,*}^{(1)} \, r_{+}^{(1)\,*} p_+^* \big(\Sigma^{\alpha^t} \shU_+^\vee \otimes \sO_+(-i) \big) = 0$. By Lem. \ref{lem:key} \eqref{lem:key-1}, it suffices to show for all $i \in [0, \delta-1]$, $(2^i)\npreceq \alpha$, the following holds:
	$$p_+^* (\Sigma^{\alpha^t} \shU_+^\vee \otimes \sO_+(-i)) \in 	\Big \langle p_+^*  \big \langle \Sigma^{\lambda} \shQ_+^\vee \big \rangle_{\lambda \in  B_{n-2,2} \backslash B_{n-\delta -1,2}} \Big \rangle \subseteq \Db(\shZ_+).$$
If we write $\alpha^t = (a+b, b)$ as before, this holds if the following holds: 
	% G+
	$$
	 \Sigma^{\alpha^t} \shU_+^\vee \otimes \sO_+(-i)  = S^a \shU_+^\vee \otimes \sO_+(b-i)   \in \Big \langle \Sigma^{\lambda} \shQ_+^\vee \Big \rangle_{\lambda \in B_{n-2,2} \backslash B_{n-\delta-1,2}} \subseteq \Db(\GG_+)	$$
By Lem. \ref{lem:G_2:mut} in the case $r = \delta-1$, above holds if $0 \le a \le \ell_+$ and $1 \le i-b \le a + \delta$. Hence $\Hom_{\shZ_+}(\Phi_{k+i}(A), \Psi_{\alpha, \,k+1}(B))=0$ for all $\alpha^t = (a+b,b) \in B_{2,\delta-2}$ and $i \in [b+1, a + b+ \delta]$. This in particular holds if $b+ 1 \le i \le \delta-1 \iff (2^i)\npreceq \alpha $.

% Generation
\bigskip \noindent \textit{Generation}. For simplicity from now on we will also use notation $E(i) := E \otimes \sO_+(i)$ for any $E \in \Db(\GG_+)$. Without loss of generality, we may assume $k=0$. To conclude the proof of Thm. \ref{thm:local:d=2}, it remains to show the category generated by the right hand side of (\ref{eqn:d=2:sod}):
	$$\shS pan_0 : = \big\langle \Im \Omega_{-1}, \{\Im \Phi_{i}\}_{i \in [0, \delta-1]}, \{\Im \Psi_{\alpha,1} \}_{\alpha \in B_{\delta-2,2}}  \big\rangle$$
is the whole category $\Db(\shZ_+)$. To show this, we first claim that 
	$$p_+^*( \bigwedge^s \shQ_+ \otimes \sO_+(i)) \in \shS pan_0 \quad \forall \quad i \in [0,\delta-1], s \in [0,\ell_+] = [0,n-2].$$
	To prove this claim, notice $\shS pan_0$ contains $\{ \Im \Psi_{\alpha,1} (\sO_X) = p_+^* (\Sigma^{\alpha^t} \shU_+^\vee(1)) \}_{\alpha \in B_{\delta-2,2}}$, therefore it contains $p_+^* \shC_i$ for all $i \in [0,\delta-1]$, where $\shC_i \subseteq \Db(\GG_+)$ is the staircase set:
		\begin{align*}
		\shC_i : = \big \langle  \langle \underbrace{ S^{i-1} \shU_+^\vee(1), S^{i-2} \shU_+^\vee(2), \ldots, \sO_+(i)}_{\text{$i$ terms}} \rangle , ~ \langle \underbrace{  \sO_+(i+1), \shU_+^\vee(i+1), \ldots, S^{\delta -2-i} \shU_+^\vee (i+1)}_{\text{$(\delta-1-i)$ terms}} \rangle \big \rangle,
		\end{align*}
where the first (resp. second) component is assumed to be empty if $i =0$ (resp. if $i= \delta-1$). 
	% extreme cases
	\begin{comment}
i.e. the two extremal cases are:
		\begin{align*}
		\shC_0 & : = \langle \sO_+(1), \shU_+^\vee(1), S^2 \shU_+^\vee (1), \ldots, S^{\delta-2} \shU^\vee (1) \rangle; \\
		\shC_{\delta-1} & =  \langle S^{\delta-2} \shU_+^\vee(1), S^{\delta-3} \shU_+^\vee(2), \ldots, \shU_+^\vee(\delta-2), \sO_+(\delta-1)\rangle.
		\end{align*}
	\end{comment}

Now notice that by mutation Lem. \ref{lem:mutation}, we have
	\begin{align*}
		\shC_i & = \big \langle \langle S^{i-1} \shU_+(i), S^{i-2} \shU_+(i), \ldots, \sO_+(i) \rangle, ~~ \langle \bigwedge^{\delta-2-i} \shQ_+^\vee (i+1), \ldots, \shQ_+^\vee (i+1), \sO_+(i+1) \rangle \big \rangle, \\
		& = \big \langle \langle \sO_+(i), \bigwedge^1 \shQ_+ (i), \ldots, \bigwedge^{i-1} \shQ_+ (i) \rangle, ~~ \langle \bigwedge^{\ell_+ -\delta+2+i} \shQ_+ (i), \ldots, \bigwedge^{\ell_+ -1} \shQ_+(i), \bigwedge^{\ell_+} \shQ_+ (i)  \rangle \big \rangle.
	\end{align*}
Hence $\shC_i =  \{\Sigma^{\gamma}  \shQ_+ \}_{\gamma \in B_{\ell_+, 1} \backslash B_{\ell_-, 1}^{\{i\}} }  \otimes \sO_+(i) $, 
where $B_{\ell_-, 1}^{\{i\}}$ is defined as in Lem. \ref{lem:key:twist}. Now apply Lem. \ref{lem:key:twist} in the case $\ell_- = \ell_+ - \delta + 1$, $d_-=1$ to $\Phi_i$ and $j=i$, we obtain that $\Im \Phi_i$ hits every element of $\big \{\Sigma^{\gamma}  \shQ_+ \big \}_{\gamma \in B_{\ell_-,1}^{\{i\}}} \otimes \sO_+(i) = \big \{ \bigwedge^s \shQ_+ (i) \big\}_{s \in [i+1, \ell_+ -\delta+1+i]}$ modulo elements of $\shC_i$, therefore 
	$$\big\{ p_+^* \bigwedge^s \shQ_+ (i) \big\}_{s \in [0, \ell_+ -1], i \in [0, \delta-1]} \subseteq \big \langle \Im \Phi_i,  \{ \Im \Psi_{\alpha,1} \}_{\alpha \in B_{\delta-2,2}} \big\rangle  \subseteq \shS pan_0. $$
Hence the claim is proved.  Now it follows from Lem. \ref{lem:G_2:mut} (twisted by $\sO_+(-1)$) that 
		\begin{equation} \label{eqn:gen:mod} \big \{ p_+^* \Sigma^\alpha \shQ_+  \otimes \sO_+(-1)  \big \}_{\alpha \in B_{\ell_+,2} \backslash B_{\ell_+ - \delta, 2}}  \subseteq \shS pan_0 .
		\end{equation}

Finally,  by Lem. \ref{lem:key} \eqref{lem:key-3} for $d_- = 2$ and $\Omega_{-1}$, which states that the image of $\Omega_{-1}$ hits every generator $ \big \{ p_+^* \Sigma^\alpha \shQ_+  \otimes \sO_+(-1)  \big \}_{\alpha \in B_{\ell_+ - \delta, 2}}$ modulo element of above set (\ref{eqn:gen:mod}). Combined with (\ref{eqn:gen:mod}), this shows that $\shS pan_0$ contains the whole set
	$$ \big \{ p_+^* \Sigma^\alpha \shQ_+  \otimes \sO_+(-1)  \big \}_{\alpha \in B_{\ell_+,2}} $$	
which generates the category $\Db(\GG_+)$ by  Lem. \ref{lem:Z_i:generator}. Hence we are done. \end{proof}

% Case \ell=2
\subsection{The case $\ell_+= 2$: flips from resolving $\rank \le 2$ degeneracy loci} The case $\ell_+=2$, $\ell_- \in \{0,1,2\}$ could be viewed as the dual situation of $d_+ = 2$. We label the schemes and maps of diagram (\ref{diag:fact}) by the upper index $d_- \in \{m - 2,m-1,m\}$ as before. We may assume that $m \ge 2$ and $\delta \ge 2$, as otherwise the situation degenerates to one of previous cases. Then $n = m + \delta \ge 4$, $d_+ = m + \delta -2 \ge 2$.

The case $d_- = m-2$ corresponds to flip case, and we use $r_{\pm}^{\rm flip} = r_{\pm}^{(m-2)}$ to denote the maps in this case. More precisely, $\shZ_-^{\rm flip} = \shZ_-^{(m-2)}$ and $\shZ_+$ are both resolutions of $\Hom^{\le 2}(W,V)$, and $\shZ_+ \dashrightarrow \shZ_{-}^{\rm flip}$ is a flip;  The case $d_-=m$ corresponds to bottom stratum, hence $\shZ_{-}^{(m)} = \Spec \kk$, and $r_{+}^{(m)}=j_+ \colon \GG_+ \hookrightarrow \Tot_{\GG_+}(W^\vee \otimes \shQ_+^\vee)$ is the inclusion of zero section, $ r_{-}^{(m)} = pr_{\GG_+} \colon \GG_+ = \Gr_d(V) \to \Spec \kk$ is the natural projection as in Lem. \ref{lem:bottom}. 

%% Thm: l = 2 local
\begin{theorem} \label{thm:local:l=2} In above situation, for any $k \in \ZZ$, $\alpha \in B_{2, \delta-2}$, the following functors:
 	\begin{align*}
	&\Psi^{\alpha}_k(\blank) :=  j_{+\,*} \circ pr_{\GG_+}^* (\blank) \otimes \Sigma^{\alpha} \shQ_+^\vee  \otimes \sO_+(k) \colon  \qquad  & \Db(\Spec \kk ) \lhook\joinrel \longrightarrow \Db(\shZ_{+}), \\
	&\Phi_k(\blank) := r_{+ \,*}^{(m-1)} \circ r_{-}^{(m-1)\,*}(\blank) \otimes \sO_+(k)  \colon \qquad &\Db(\shZ_{-}^{(m-1)})  \lhook\joinrel\longrightarrow \Db(\shZ_+), \\
	& \Omega_k(\blank) : = r_{+ \,*}^{\rm flip} \circ r_{-}^{\mathrm{flip} \,*}(\blank) \otimes \sO_+(k) \colon  \qquad & \Db(\shZ_{-}^{\mathrm{flip}}) \lhook\joinrel\longrightarrow \Db(\shZ_+)
	\end{align*}
are fully faithful. Furthermore, for any fixed $k \in \ZZ$, the images $\{\Im \Psi^{\alpha}_{k-1} \}_{\alpha \in B_{2, \delta-2}}$, $\{\Im \Phi_{k - i} \}_{i \in [0, \delta-1]}$ and  $\Im \Omega_{k+1}$ induce a semiorthogonal decomposition
	\begin{align*} %\label{eqn:l=2:sod}
	\Db(\shZ_+) = \big\langle \{\Im \Psi^{\alpha}_{k-1} \}_{\alpha \in B_{2, \delta-2}}, \{\Im \Phi_{k - i} \}_{i \in [0, \delta-1]},  \Im \Omega_{k+1} \big\rangle, 
	\end{align*}
 with semiorthogonal order given by any total order extending the following partial orthogonal order: for any $\alpha, \beta \in B_{2, \delta-2}$, $i \in [0,\delta-1]$, the following holds:
	 \begin{align*}
	& \Im \Omega_{k+1} \subseteq {}^\perp (\Im \Phi_{k-i}) \cup {}^\perp(\Im \Psi^{\alpha}_{k-1}) & \forall &\quad   i, \alpha \text{~above};  \\  %\label{eqn:l=2:order1} \\
	& \Im \Phi_{k-i}  \subseteq  {}^\perp(\Im \Phi_{k-j}) \cup {}^\perp(\Im \Psi^{\alpha}_{k-1}) & \forall &  \quad  i <  j \le  i + \delta -1,  \alpha +1 \npreceq (i^2);   \\%\label{eqn:l=2:order2} \\
	& \Im \Psi^{\alpha}_{k-1} \subseteq {}^\perp(\Im \Phi_{k-i}) \cup {}^\perp(\Im \Psi^{\beta}_{k-1}) & \forall &   \quad (i^2) \npreceq \alpha, \beta \npreceq \alpha.  %\label{eqn:l=2:order3}
	\end{align*}
where $(i^2)= (i,i)$ is understood as an element of 
$B_{2, \delta-1} \supset B_{2, \delta-2}$. 
\end{theorem}
The semiorthogonal decomposition of the theorem can be informatively written as
	\begin{align*}
	\Db(\shZ_+) = \big \langle  \text{$\binom{\delta}{2}$-copies of} ~ \Db(\Spec \kk),  ~\text{$\delta$-copies of} ~ \Db(\shZ_{-}^{(m-1)}), ~\Db(\shZ_{-}^{\rm flip}) \big \rangle.
	\end{align*}

\begin{proof} The proof of the semiorthogonal relations part is exactly similar to the case of $d_+=2$; The only nontrivial part of the proof is the semiorthogonal relations among $\Im \Phi_{k-i}$ and $\Im \Psi^{\alpha}_{k-1}$. Similar to $d=2$ case, we claim a slightly stronger result: if we write $\alpha = (a+b,b) \in B_{2,\delta-2}$, where $a,b \ge 0$, $a + b \le \delta -2$, then
	\begin{align*}
		& \Hom( \Im \Phi_{k-i}, \Im \Psi^{\alpha}_{k-1}) = 0  & \text{if} \qquad  & b - (\delta-1) \le i \le a+b; \\ 
		& \Hom(\Im \Psi^{\alpha}_{k-1},  \Im \Phi_{k-i}) = 0 &  \text{if} \qquad  & b+1 \le i \le a+b + \delta.
	\end{align*}
The two equalities are Serre-dual equivalent, so we only need to one side of each. For any $A \in \Db(\shZ_-^{(m-1)})$, $B \in \Db(\Spec \kk)$, then for $\alpha = (a+b,b)$
	\begin{align*}
		& \Hom_{\shZ_+}( \Phi_{k-i}(A), \Psi^{\alpha}_{k-1}(B)) = \Hom_{\Spec \kk} \big( (\Psi^{(0)}_0)^L (p_+^* \Sigma^{\alpha} \shQ_+ \otimes \sO_+(1-i)) \otimes \Phi_0(A)), B  \big) \\
		& =  \Hom_{\Spec \kk} \big( (\Psi^{(0)}_0)^L ( p_+^* S^a \shQ_+ \otimes \sO_+(b-i) \otimes \Phi_0(A)), B  \big)
	\end{align*}
Since by Lem. \ref{lem:key} \eqref{lem:key-3}, $\Phi_0(A) \in \langle p_+^* \Sigma^\lambda \shQ_+\rangle_{\lambda \in B_{2,m-1}}$, and by Lem. \ref{lem:key} \eqref{lem:key-2}, $(\Psi^{(0)}_0)^L (E) = 0$ for any $E \in \langle \Sigma^\lambda \shQ_+ \rangle_{\lambda \in B_{2,d} \backslash \{0\}}$. Therefore above $\Hom$ space is zero if $1 \le a+b-i \le \delta -1$.
	\begin{align*} 
	&\Hom_{\shZ_+}( \Psi^{\alpha}_{k-1}(B),  \Phi_{k-i}(A)) \\
	&=  \Hom_{\Spec \kk}  \Big(B, ~ r_{-\,*}^{\rm flip}  (r_{+}^{\mathrm{flip}\,*} (p_+^* \Sigma^{\alpha} \shQ_+ \otimes \sO_+(i) \otimes \Phi_0(A) ) \otimes \omega_{r_+^{\rm flip}} \big) \Big)  \\
	 &=  \Hom_{\Spec \kk}  \Big(B, ~ r_{-\,*}^{\rm flip} \, r_{+}^{\mathrm{flip}\,*} \big(p_+^* S^a \shQ_+^\vee \otimes \sO_+(a+b-i-1) \otimes \Phi_0(A) \otimes \sO_+(1-m) \big) \Big) [\ell_+ m].
	\end{align*}
(Here we use $\omega_{r_+^{\rm flip}}= \sO_+(-m)[\ell_+ m]$.) Since $\Phi_0(A) \otimes \sO_+(1-m) \in \langle p_+^* \Sigma^\lambda \shQ_+^\vee\rangle_{\lambda \in B_{2,m-1}}$, therefore by Lem. \ref{lem:key} \eqref{lem:key-1}, above $\Hom$ space is zero if $1 \le i+1 - b \le \delta -1$, i.e. $b \le i \le b + \delta - 2$. Now combine these two inequalities and relative Serre duality, the claim is proved. 

\medskip \noindent \textit{Generation}. The proof of generation is very similar to $\ell_+ =1$, with a slightly different pattern. Without loss of generality, we may assume $k=\delta-1$. We want to show 
		$$\shS pan : = \big\langle \{\Im \Psi^{\alpha}_{\delta-2} \}_{\alpha \in B_{2, \delta-2}}, ~~\{\Im \Phi_{j} \}_{j \in [0, \delta-1]}, ~~ \Im \Omega_{\delta} \big\rangle$$
contains a set of generators of $\Db(\shZ_+)$. By Lem. \ref{lem:key} \eqref{lem:key-3} in the case $d_- = m-2, \ell_- = 2$, $\Im \Omega_{\delta}$ contains all the elements
	$p_+^* \Sigma^{\alpha} \shQ_+ \otimes \sO_{+}(\delta)$ for all $\alpha \in B_{2,m-2}.$
By Lem. \ref{lem:key} \eqref{lem:key-3} applied to $d_- = m-1$, $\ell_- = 1$, we see $\Im \Phi_{\delta-1}$ hits every element $p_+^* \Sigma^{\alpha} \shQ_+  \otimes \sO_{+}(\delta-1)$ for $\alpha \in B_{1,m-1}$ modulo elements of $\Im \Omega_{\delta}$. Therefore together $\langle \Im \Phi_{\delta-1},  \Im \Omega_{\delta}\rangle$ contains every element 
	$p_+^* \Sigma^{\alpha} \shQ_+ \otimes \sO_{+}(\delta-1)$ for all $\alpha \in B_{2,m-1}.$
Next, consider $\Im \Phi_{\delta-2}$, by Lem. \ref{lem:key} \eqref{lem:key-3} applied to $d_- = m-1$, $\ell_- = 1$ again, we see that $\Im \Phi_{\delta-2}$ hits every element $p_+^* \Sigma^{\alpha} \shQ_+  \otimes \sO_{+}(\delta-2)$ for $\alpha \in B_{1,m-1}$ modulo elements of $\langle \Im \Phi_{\delta-1},  \Im \Omega_{\delta}\rangle$. Hence inductively, we see that $\shS pan \supset \langle \{\Im \Phi_{j} \}_{j \in [0, \delta-1]}, ~~ \Im \Omega_{\delta} \rangle$ contains every elements of the form:
	$$p_+^* \Sigma^{\alpha} \shQ_+ \otimes \sO_{+}(j)  =p_+^* \Sigma^{\alpha + j} \shQ_+ \quad \text{for all} \quad \alpha \in B_{2,m-1}, 0 \le j \le \delta -1.$$
Since $\{p_+^* \Sigma^{\alpha} \shQ_+\}_{\alpha \in B_{2,n-2}}$ is a set of generators of $\Db(\shZ_+)$, $n - 2 = m + \delta - 2$, and
	$$B_{2, n -2} \backslash \bigcup_{j=0}^{\delta-1} (B_{2,m-1} + j) = \big\{(\nu_1 +m ,\nu_2) \in B_{2, n -2} \mid \nu=(\nu_1, \nu_2) \in B_{2,\delta-2} \big\} = : B_{2,\delta-2}^C$$
(Here $B_{2,m-1} + j$ denotes $\{\alpha + j \mid \alpha \in B_{2,m-1}\} =  \{(\alpha_1 + j, \alpha_2+j) \mid \alpha \in B_{2,m-1}\}$ as usual.) Therefore to show generation, it only remains to show that 
	\begin{equation} \label{eqn:l=2:gen}
	p_+^*\,  \Sigma^{(\nu_1 +m, \nu_2)} \shQ_+ \in \shS pan, \quad \text{for all} \quad \nu = (\nu_1,\nu_2) \in B_{2,\delta-2}.
	\end{equation} 
We will prove (\ref{eqn:l=2:gen}) by induction on $k = \nu_1 - \nu_2 \in [0, \delta-2]$. First, notice that from Lem. \ref{lem:key} \eqref{lem:key-3} applied to the case $d_- = m$, $\ell_- = 0$, we obtain that 
	$$\Psi_{0}^{(0)} (\sO_{\Spec \kk}) \simeq \{ 0 \to F^{-2m} \to \ldots \to F^{-2} \to F^{-1} \to F^0 = \sO_+  \to 0\},$$
where $F^{-k} = \bigoplus_{\lambda \in B_{2,m}, |\lambda| = k} \Sigma^{\lambda^t} W \otimes p_+^* \, \Sigma^{\lambda} \shQ_+$. Hence if we take $k=m$, then $F^{-m}$ contains exact one copy of the summand $\wedge^m W \otimes p_+^* \, S^m \shQ_+ \simeq p_+^*\, S^m \shQ_+$, as $\rank W = m$.

Now we assume for some $k \in [0, \delta-2]$, (\ref{eqn:l=2:gen}) holds for all $\nu \in B_{2,\delta-2}$ such that $\nu_1 - \nu_2 < k $. (Notice this condition is trivial if $k=0$.) We want to show (\ref{eqn:l=2:gen}) also holds for $\nu$ with $\nu_1 - \nu_2 = k$. Consider $\mu: = (2-\delta) - \nu = (2-\delta -\nu_2, 2 - \delta - \nu_1) \in B_{2,\delta-2}$, then 
	$$\Psi_{\delta-2}^{\mu} (\sO_{\Spec \kk}) = \Psi_{0}^{(0)} (\sO_{\Spec \kk})  \otimes \Sigma^{\mu} \shQ_+^\vee \otimes \sO_+(\delta-2) \simeq \Psi_{0}^{(0)} (\sO_{\Spec \kk})  \otimes \Sigma^{\nu} \shQ_+.$$
By considering the summand $p_+^*\, S^m \shQ_+$ of $F^{-m}$, we see that $\Psi_{\delta-2}^{\mu} (\sO_{\Spec \kk})$ contains exactly one copy of the summand $p_+^* (S^m \shQ_+ \otimes \Sigma^{\nu} \shQ_+)$, hence by Pieri's rule contains exactly one copy of $p_+^*\,  \Sigma^{(\nu_1 +m, \nu_2)} \shQ_+$. By Pieri's rule all other summands $p_+^* \Sigma^{\gamma} \shQ_+ \subset p_+^* (S^m \shQ_+ \otimes \Sigma^{\nu} \shQ_+)$ satisfies $0 \le \gamma_1 - \gamma_2 < \nu_1 - \nu_2 + m$. This means that either $0 \le \gamma_1 - \gamma_2 < m$ (i.e. $\gamma \in B_{2,n-2} \backslash B_{2,\delta-2}^C$), or $0 \le (\gamma_1 -m) - \gamma_2 < \nu_1 - \nu_2$ (i.e. $\gamma \in B_{2,\delta-2}^C$ and $(\gamma_1 -m) - \gamma_2 < k$). By induction all these summands $p_+^* \Sigma^{\gamma} \shQ_+$ are already contained in $\shS pan$. 

It remains to compute all other summands of $F^{-k} \otimes  \Sigma^{\nu} \shQ_+$ for all $k \in [0, 2m]$ other than the ones of $S^m \shQ_+ \otimes  \Sigma^{\nu} \shQ_+ $. All summands $\Sigma^{\lambda^t} W \otimes p_+^* \, \Sigma^{\lambda} \shQ_+ \subseteq F^{-k}$ except from the already considered case $\lambda = (m, 0) \in B_{2,m}$ satisfy $0 \le \lambda_1 -\lambda_2 < m$. Hence by Littlewood-Richardson rule, any summand $p_+^* \Sigma^{\gamma} \shQ_+ \subset  p_+^*( \Sigma^{\lambda} \shQ_+ \otimes \Sigma^{\nu} \shQ_+)$ satisfies $\gamma = (\gamma_1,\gamma_2) \in B_{2,n-2}$, $0 \le \gamma_1 - \gamma_2  \le \lambda_1 - \lambda_2 + \nu_1 + \nu_2$. This means that either $\gamma \in B_{2,n-2} \backslash B_{2,\delta-2}^C$, or $\gamma \in B_{2,\delta-2}^C$ and $(\gamma_1 -m) - \gamma_2 < k$. By induction all these summands are already contained in $\shS pan$. 

Hence we see that $\Psi_{\delta-2}^{\mu} (\sO_{\Spec \kk})$ hits the element $p_+^*\,  \Sigma^{(\nu_1 +m, \nu_2)} \shQ_+$ (with $\nu_1 - \nu_2 =k$) if modulo the elements which are already in $\shS pan$ by induction hypothesis. Hence by induction, (\ref{eqn:l=2:gen}) holds for all $\nu$ and $k = \nu_1 - \nu_2 \in [0,\delta-2]$, and the generation is proved. \end{proof}

%%%% Part 3: Global situation
\newpage
\part{Global geometry} \label{part:global}
%%%%
\addtocontents{toc}{\vspace{0.5\normalbaselineskip}}	
\section{Global situation}\label{sec:global}
\subsection{Hom spaces} \label{sec:Homspace}
%\subsubsection{Definitions}
Let $S$ be a scheme, and let $\sV$ and $\sW$ be two finite type locally free sheaves on $S$. For any map $f \colon T \to S$, denote $\sV_T = f^* \sV$ and $\sW_T = f^* \sW$ the (underived) base-change of sheaves. Consider the contravariant functor $F$ defined as follows: for any $S$-scheme $f \colon T \to S$, 
	$T \mapsto F(T) = \Hom_{\sO_T}(\sW_T, \sV_T).$
For any $S$-morphism $g \colon T' \to T$, $F(g) \colon \Hom_{\sO_T}(\sW_T, \sV_T) \to \Hom_{\sO_{T'}}(\sW_{T'}, \sV_{T'})$ is the pullback map $g^* \colon u \mapsto g^*u$.

\begin{lemma} \label{lem:Homspace} The functor $F$ is representable by the smooth affine $S$-scheme:
			$$H_S : = |\sHom_S(\sW, \sV)| : = \underline{\Spec}_S(\Sym^\bullet (\sW \otimes \sV^\vee))  \to S.$$
			The {\em tautological morphism} $\tau \colon \sW_{H_S} \to \sV_{H_S}$ is induced from the canonical homomorphism $\sW \otimes \sV^\vee \to \Sym^\bullet (\sW \otimes \sV^\vee)$. Hence for any $T \to S$ and any $\sigma \in \Hom_{\sO_T}(\sW_T, \sV_T)$, there is a unique $S$-morphism $g_{\sigma} \colon T \to H_S$ such that $g_{\sigma}^* \tau = \sigma$. Furthermore:
	\begin{enumerate}[leftmargin=*]
		\item  \label{lem:Homspace-1}
		{\em (The formation of universal Hom spaces commutes with base change.)} For any base-change $\phi \colon S' \to S$, we have $H_{S'} = H_S \times_S S'$, and the tautological morphism $\tau'$ on $H_{S'}$ is the pullback of the tautological morphism $\tau$ of $H_S$: $\tau' = \phi^* \tau$.
		\item  \label{lem:Homspace-2}
		If we set $T=S$, then each $\sigma \in \Hom_{\sO_S}(\sW,\sV)$ corresponds to a unique $S$-morphism $s_{\sigma} \colon S \to H_S = |\sHom_S(\sW, \sV)|$, called the {\em section map}, such that $s_{\sigma}^* \tau = \sigma$. Then for any $\sigma$, the section map $s_{\sigma}$ is a regular closed immersion.
	\end{enumerate}
\end{lemma}
\begin{proof} The lemma follows from applying \cite[Prop. 9.4.9]{EGAI} to $\sE = \sW\otimes \sV^\vee$; Our universal Hom space $|\sHom_S(\sW, \sV)|$ is the scheme $\VV(\sE)$ of \cite[\S 9.4]{EGAI}. The ``furthermore" statement \eqref{lem:Homspace-1} follows from \cite[Prop. 9.4.11(iii)]{EGAI}; For \eqref{lem:Homspace-2}, $s_\sigma$ is a regular closed immersion since it is a section of a smooth separated morphism; see \cite[IV, Thm. 17.12.1]{EGA}.
\end{proof}

\subsection{Tor-independent conditions and general procedures of base-change} \label{sec:univHom}

\subsubsection{The cases of universal Hom spaces}\label{sec:univHom:flat} Assume $S$ is a quasi-compact, quasi-separated scheme over a ring $\kk$, and $\sW, \sV$ are locally free sheaves on a scheme $S$ with rank $m$ and $n$ with $m \le n$, and denote the universal Hom space $H_S : = |\sHom_S(\sW,\sV)|$. Then by Lem. \ref{lem:Homspace} there is is a tautological morphism $\tau_{H_S} \colon \sW \to \sV$. Denote $\sG_{H_S} = \Coker(\tau_{H_S})$ and $\sK_{H_S}= \Coker(\tau_{H_S}^\vee) = \sExt^1_{H_S}(\sG_{H_S}, \sO_{H_S})$. For any pair of integers $(d_+, d_-)$ such that $0 \le d_- \le m$, $0 \le d_+ \le n$ and $d_- \le d_+$, and consider the following schemes:
	$$\shZ_{+, H_S}^{(d_+)}: = \Quot_{H_S, d_+}(\sG_{H_S}), \quad  \shZ_{-, H_S}^{(d_-)}: = \Quot_{H_S, d_-}(\sK_{H_S}), \quad \widehat{\shZ}_{H_S}^{(d_+,d_-)}: = \shZ_{+, H_S}^{(d_+)}\times_{H_S} \shZ_{-, H_S}^{(d_+)}.$$
Now we fix a pair of integers $(d_+, d_-)$, and we write $\shZ_{+, H_S} = \shZ_{+, H_S}^{(d_+)}$, $\shZ_{-, H_S}= \shZ_{-, H_S}^{(d_-)}$ and $\widehat{\shZ}_{H_S}=  \widehat{\shZ}_{H_S}^{(d_+,d_-)}$ without the supscripts for simplicity of notations. Consider the Grassmannian-bundles (see Ex. \ref{ex:Grass}) $\GG_{+,S} = \Gr_{d_+}(\sV^\vee)$ and $\GG_{-,S} = \Gr_{d_,-}(\sW)$ over $S$, where $\shU_{\pm, S}$ and $\shQ_{\pm, S}$ denotes the corresponding universal subbundles of rank $d_\pm$ and universal quotient bundles of rank $\ell_\pm: = n_\pm - d_\pm$, where $n_+ = n$ and $n_- = m$ as in Sect. \ref{sec:local}. 

By passing to Zariski-local subschemes of $S$, we may assume that $\sW$, $\sV$ are {\em free} modules. Hence there is a structural morphism 
	$$h \colon S \to \Spec \kk$$
such that $\sW = h^* W$, $\sV = h^* V$, where $W$ and $V$ are finite free modules over $\kk$ of rank $m$ and $n$. Denote $H_\kk = |\sHom_\kk(W,V)|$ and $\tau_{H_\kk} \colon W \to V$ the tautological map, $\sG_{H_\kk} = \Coker (\tau_{H_\kk})$ and $\sK_{H_\kk} =\Coker (\tau_{H_\kk}^\vee)$. (Notice in the case when $\kk$ is a field, $\shZ_{+, H_\kk} = \Quot_{H_\kk, d_+}(\sG_{H_\kk})$, $\shZ_{-, H_\kk} =\Quot_{H_\kk, d_-}(\sK_{H_\kk})$ and $\widehat{\shZ}_{H_\kk} = \shZ_{+, H_\kk}\times_{H_\kk} \shZ_{-, H_\kk}$ are exactly the schemes $\shZ_-, \shZ_+$ and $\widehat{\shZ}$ of Sect. \ref{sec:local}; Hence our notations of this subsection are compatible with those of Sect. \ref{sec:local}.)

 %hence the results of Sect. \ref{sec:local} hold for these schemes.
By Lem. \ref{lem:Homspace}, $H_S = |\sHom_S(\sW,\sV)| \simeq H_\kk \times_\kk S$, and the tautological morphism $\tau_S$ is the pullback $h^* \tau_\kk$. Hence $\sG_{H_S} = h^* \sG_{H_\kk}$, $\sK_{H_S} = h^* \sK_{H_\kk}$, by right-exactness of pullbacks. Hence by Thm. \ref{thm:Quot} (1), $\shZ_{+, H_S}= \Quot_{H_S, d_+}(\sG_{H_S})$, $\shZ_{-, H_S}=\Quot_{H_S, d_-}(\sK_{H_S})$ and $\widehat{\shZ}_{H_S} = \shZ_{+, H_S} \times_{H_S} \shZ_{-, H_S}$ are obtained by  base-change of $\shZ_{+, H_\kk}$, $\shZ_{-, H_\kk}$ and respectively $\widehat{\shZ}_{H_\kk}$ along $h\colon S \to \Spec \kk$; Similarly $\GG_{\pm, S}$ are the base-change of $\GG_{\pm, \kk}$, and the universal bundles $\shU_{\pm, S}$ and $\shQ_{\pm, S}$ are obtained from pullbacks from the corresponding universal bundles over $\GG_{\pm, \kk}$.

\begin{lemma} \label{lem:univHom:flat} In the above situation (when $\sW, \sV$ are finite {\em free} modules), the base-change $h \colon S \to \Spec \kk$ is Tor-independent with respect to the pair of schemes 
	$$(\shZ_{+, H_\kk} = \Quot_{H_\kk, d_+}(\sG_{H_\kk}), \shZ_{-, H_\kk} =\Quot_{H_\kk, d_-}(\sK_{H_\kk}))$$
 in the sense of Def. \ref{def:bc}.
\end{lemma}

\begin{proof} In the case when $\kk$ is a field, this lemma is trivial since $h$ is {\em flat}. In general, it suffices to notice that from the local expressions (\ref{eqn:Quot:localexpressions}) of \S \ref{sec:local} and Lem. \ref{lem:Sym}, being the spaces of finite vector bundles over smooth $\kk$-schemes $\GG_{\pm,\kk}$ and $\GG_{+,\kk}\times_\kk \GG_{-,\kk}$, the schemes $\shZ_{+, H_\kk}$, $\shZ_{-, H_\kk}$ and $\shZ_{+, H_\kk} \times_{H_\kk} \shZ_{-, H_\kk}$ are themselves smooth over $\kk$, thus flat over $\kk$.
\end{proof}

\subsubsection{Tor-independent condition for general global situation} \label{sec:Tor-ind:Quot:bc} Now we are back to the main situation of this paper. Let $X$ be a quasi-compact, quasi-separated scheme, let $\sG$ be a quasi-coherent $\sO_X$-module of homological dimension $\le 1$, and set $\sK : = \sExt^1(\sG, \sO_X).$ There is no harm to assume $X$ is connected, then $\delta: = \rank \sG \ge 0$ is a non-negative integer. For any pair of integers $(d_+, d_-)$ such that $0\le d_- \le d_+$, we set
	\begin{equation}
	\shZ_+^{(d_+)} : = \Quot_{X,d_+}(\sG), \quad \shZ_-^{(d_-)}: = \Quot_{X,d_-}(\sK), \quad \widehat{\shZ}^{(d_+,d_-)}: = \shZ_+^{(d_+)}  \times_X \shZ_-^{(d_-)}.
	\end{equation}
In particular, we have a commutative diagram, which is a global version of (\ref{diagram:Corr}):
\begin{equation} \label{diagram:Corr:global}
	\begin{tikzcd}[row sep= 2.5 em, column sep = 4 em]
		\widehat{\shZ}^{(d_+,d_-)}: = \shZ_+^{(d_+)}  \times_X \shZ_-^{(d_-)} \ar{rd}{\widehat{\pi}}\ar{r}{r_+}  \ar{d}[swap]{r_-} & \shZ_+^{(d_+)} \ar{d}{\pi_+}\\
		 \shZ_-^{(d_-)} \ar{r}{\pi_-} & X
	\end{tikzcd}
\end{equation}

Let $U \subseteq X$ be an open subscheme such that $\sG$ admits a presentation $0 \to \sW \xrightarrow{\sigma} \sV \to \sG$, where $\sW$ and $\sV$ are finite locally free sheaves. By Lem. \ref{lem:Homspace} (2), the morphism $\sigma$ induces a section map $s_\sigma \colon U \hookrightarrow H_U : = |\sHom_U(\sW,\sV)|$, such that $\sigma = s_\sigma^* \, \tau_{H_U}$, $\sG|_U = s_{\sigma}^* \sG_{H_U}$ and $\sK|_U = s_{\sigma}^* \sK_{H_U}$, where $\tau_{H_U} \colon \sW_{H_U} \to \sV_{H_U}$ is the tautological map, $\sG_{H_U} = \Coker (\tau_{U})$ and $\sK_{H_U} = \Coker (\tau_{U}^\vee)$. If we use the notation of \S \ref{sec:univHom} for $S= U$, then by Thm. \ref{thm:Quot} (1), the restriction of schemes $\shZ_+^{(d_+)}|_U$, $\shZ_-^{(d_-)}|_U$ and $\widehat{\shZ}^{(d_+,d_-)}|_U$ are exactly the base-change of the schemes $\shZ_{+, H_U}^{(d_+)} =  \Quot_{H_U, d_+}(\sG_{H_U})$, $\shZ_{-, H_U}^{(d_-)}= \Quot_{H_U, d_-}(\sK_{H_U})$ and respectively $\widehat{\shZ}^{(d_+,d_-)}_{H_U} = \shZ_{+, H_U}^{(d_+)}  \times_{H_U} \shZ_{-, H_U}^{(d_-)} $ along the section map $s_\sigma \colon U \to H_U$.

\begin{definition}[Tor-independent condition] \label{def:Tor-ind:quot} For a quasi-compact, quasi-separated scheme $X$ and a quasi-coherent $\sO_X$-module $\sG$ of homological dimension $\le 1$ and a pair of integers $(d_+,d_-)$ with $0 \le d_- \le d_+$ as above, we say that {\em Tor-independent condition holds for $(d_+,d_-)$} if there is a Zariski open cover $\{U\}$ of $X$ for which $\sG$ admits a presentation $0 \to \sW \xrightarrow{\sigma} \sV \to \sG$ over each $U$, such that the base-change $s_\sigma \colon U \to H_U= |\sHom_U(\sW,\sV)|$ is Tor-independent with respect to the pair $(\Quot_{H_U, d_+}(\sG_{H_U}),\Quot_{H_U, d_-}(\sK_{H_U}))$ in the sense of Def. \ref{def:bc}.
\end{definition}

\begin{remark} If $X$ is a scheme over a (unital commutative) ring $\kk$, then by Lem. \ref{lem:univHom:flat} and Lem. \ref{lem:3squares}, the above-defined Tor-independent condition for $(d_+,d_-)$ is equivalent to: {\em Zariski locally over $X$, the base-change $X \to \Spec \kk$ from the universal local situation is Tor-independent with respect to the pair $(\shZ_{+, H_\kk} = \Quot_{H_\kk, d_+}(\sG_{H_\kk}), \shZ_{-, H_\kk} =\Quot_{H_\kk, d_-}(\sK_{H_\kk}))$} of \S \ref{sec:univHom:flat}. We choose the above definition as Def. \ref{def:Tor-ind:quot}  only requires the verification of the conditions in a {\em semi-local} situation, i.e. when $\sG$ admits a global presentation; Also this definition leads direction to the critera Lem. \ref{lem:Tor-ind:quot:CM} for Cohen--Macaulay cases.
\end{remark}

First of all, we need to show that the Tor-independent condition of Def. \ref{def:Tor-ind:quot} is {\em independent of the choice of a local presentation $\sW \xrightarrow{\sigma} \sV$ of $\sG$ on $U \subset X$}. %hence it is an intrinsic condition for the sheaf $\sG$ and the scheme $X$. 
To show this, we may assume $X= U$, and fix a pair of integers $(d_+,d_-)$, and keep the notations of \S \ref{sec:univHom} for $S =X$.

\begin{lemma} \label{lem:Tor-ind-ind} Let $0 \to \sW \xrightarrow{\sigma} \sV \to \sG$ and $0 \to \sW' \xrightarrow{\sigma'} \sV' \to \sG$ be two finite free resolutions of $\sG$. Then the section map $s_{\sigma} \colon X \to H:=\sHom_X(\sW, \sV)|$ is Tor-independent with respect to the pair $(\Quot_{H, d_+}(\sG_{H}), \Quot_{H, d_-}(\sK_{H}))$ iff the section map $s_{\sigma'} \colon X \to H'=  |\sHom_X(\sW', \sV')|$ is Tor-independent with respect to the pair $(\Quot_{H', d_+}(\sG_{H'}'), \Quot_{H', d_-}(\sK_{H'}'))$.
\end{lemma}

\begin{proof} We first prove the lemma in the simple case when $\sW' = \sW \oplus \sO^{\oplus r}$, $\sV' = \sV \oplus \sO^{\oplus r}$ and $\sigma' =  \big(\begin{smallmatrix}
\sigma & 0  \\
0 & \Id_r  \\
\end{smallmatrix} \big)$ for some integer $r \ge 0$. 

Denote $\tau_{H} \colon \sW_H \to \sV_H$ and $\tau'_{H'} \colon \sW'_{H'} = \sW_{H'} \oplus \sO^{\oplus r} \to \sV'_{H'}=\sV_{H'} \oplus \sO^{\oplus r}$ the tautological maps. (For a scheme $Y$ and a morphism $T \to Y$, a sheaf $\sE$ on $Y$, we denote $\sE_T$ the base-change of $\sE$ along $T$ as before.) In this case, by functorality of universal Hom space Lem. \ref{lem:Homspace}, there is an regular immersion $\iota \colon H \to H'$ such that $\iota^* \sW'_{H'} = \sW_{H} \oplus \sO_H^{\oplus r}$, $\iota^* \sV'_{H'}= \sV_H \oplus \sO_H^{\oplus r}$ and $\iota^* \tau'_{H'} = \big(\begin{smallmatrix}
\tau_H & 0  \\
0 & \Id  \\
\end{smallmatrix} \big) $, as well as a smooth projection $\rho \colon H' \to H$, such that $\rho^* \sW_H = \sW_{H'}$, $\rho^* \sV_H = \sV_{H'}$ and $\rho^* \tau_H \colon  \sW_{H'} \to \sV_{H'}$ is the component of $\tau'_{H'}\colon  \sW_{H'} \oplus \sO^{\oplus r} \to \sV_{H'} \oplus \sO^{\oplus r}$ on the factor $\sW_{H'} \to \sV_{H'}$. Then $\iota$ is a section of $\rho$, i.e. $\rho \circ \iota = \id$. Denote $\overline{\tau}_H : =  \big(\begin{smallmatrix}
\tau_H & 0  \\
0 & \Id  \\
\end{smallmatrix} \big) = \iota^* \tau'_{H'}  \colon  \sW_{H} \oplus \sO_H^{\oplus r} \to  \sV_H \oplus \sO_H^{\oplus r}$, and $\overline{\sG}_H : = \Coker \overline{\tau}_H $ then it is clear that $\overline{\sG}_H = \Coker \overline{\tau}_H \simeq \Coker (\tau) = \sG_H$ and $\overline{\sG}_H = \iota^* \Coker(\tau'_{H'}) = \iota^* \sG'_{H'}$. Notice
$\sV_{H'} \oplus \sO_{H'}^{\oplus r} = \rho^* (\sV_{H} \oplus \sO_{H}^{\oplus r})$ (however $\rho^* \overline{\sG}_H$ and $\sG'_{H'}$ are in general not isomorphic). 

We claim that the closed immersion $\iota \colon H \to H'$ is Tor-independent with respect to the pair $(\shZ_{+,H'} =\Quot_{H',d_+}(\sG'_{H'}), \shZ_{-,H'}=\Quot_{H',d_-}(\sK'_{H'}))$ for any given pair of integers $(d_+,d_-)$. We first consider the claim for the base-change of $\Quot_{H',d}(\sG'_{H'})$ along $\iota$ for some integer $d$. The base-change of $\Quot_{H',d}(\sG'_{H'})$ along $\iota$ is the $\Quot_{H,d}(\overline{\sG}_H) \simeq \Quot_{H,d}(\sG_H)$ since $\iota^* \sG'_{H'} = \overline{\sG}_H \simeq \sG_H$. By by Thm. \ref{thm:Quot}, for $d= d_+$, we have the following commutative diagrams of Cartesian squares:
\begin{equation} \label{diag:lem:Tor-ind:proof}
\begin{tikzcd}
	& \Quot_{H,d}(\overline{\sG}_H) \ar{r}{\iota''}  \ar{d}{i}&  \Quot_{H',d}(\sG'_{H'})  \ar{d}{i'}  \\
	&  \Quot_{H,d}( \sV_{H} \oplus \sO_{H}^{\oplus r})  \ar{r}{\iota'}  \ar{d}{\pi} & \Quot_{H',d}(\sV_{H'} \oplus \sO_{H'}^{\oplus r} ) \ar{d}{\pi'} \\
	& H  \ar{r}{\iota}  & H' 
\end{tikzcd}
\end{equation}
The bottom square is Tor-independent since $\pi, \pi'$ are smooth. For the top square, the closed immersion $i$ is given by the pullback of a the section of a vector bundle which defines the closed immersion $i'$ \footnote{More precisely, if we denote by $\sQ_d$ and $\sQ_d'$ the universal quotient bundles of $\Quot_{H,d}( \sV_{H} \oplus \sO_{H}^{\oplus r})$ and respectively $\Quot_{H',d}(\sV_{H'} \oplus \sO_{H'}^{\oplus r} )$, then the closed immersions $i$ and $i'$ induced by Thm. \ref{thm:Quot} \eqref{thm:Quot-2} are given by sections $\theta$ and $\theta'$ of the respective locally free sheaves $(\sW_H \oplus \sO^{r})^\vee \otimes \sQ_d$ and $(\sW_{H'} \oplus \sO^{r})^\vee \otimes \sQ_d'$, which are in turn induced by the morphisms $\overline{\tau}_H \colon \sW_H \oplus \sO^{r} \to \sV_H\oplus \sO^{r}  \to \sQ_d$ and respectively $\tau'_{H'} \colon  \sW_{H'} \oplus \sO^{r} \to \sV_{H'} \oplus \sO^{r}  \to \sQ_d'$. Since $\overline{\tau}_H  =  \iota^* \tau'_{H'} $, therefore $\theta = \iota^* \theta'$.}.
By Lem. \ref{lem:3squares}, to show the ambient square of (\ref{diag:lem:Tor-ind:proof}) is Tor-independent, it suffices to show the top square is. Notice first that in the universal local situation $X = \Spec \ZZ$, $\iota'$ is a regular immersion, and a direct computation shows that:
	$$\dim  \Quot_{H',d}(\sG'_{H'})  - \dim  \Quot_{H,d}(\overline{\sG}_H) = \dim H' - \dim H.$$
Hence the top square of (\ref{diag:lem:Tor-ind:proof}) is Tor-independent by Lem. \ref{lem:bc:CM}. By passing to Zariski open subsets as in \S \ref{sec:univHom:flat}, Lem. \ref{lem:univHom:flat} states that the diagram (\ref{diag:lem:Tor-ind:proof}) is a Tor-independent base-change from the same diagram in the universal local situation. Hence the top square, thus the the ambient square, of (\ref{diag:lem:Tor-ind:proof}) is Tor-independent for general $X$. The same argument works for $\shZ_{-,H'}$ and $\widehat{\shZ}_{H'} = \shZ_{+,H'} \times_{H'} \shZ_{-,H'}$, hence the claim is proved.

%For general $\kk$-scheme $X$, by passing to Zariski open subsets as in \S \ref{sec:univHom:flat}, above diagram in this case is a flat base-change from the universal local case $X= |\sHom_\kk(\sW,\sV)|$, hence the top square is Tor-independent for general $X$. The same argument works for the other diagrams, hence our claim is proved. 
 
Back to the situation of the lemma, since $\sigma' =  \big(\begin{smallmatrix}
\sigma & 0  \\
0 & \Id_r  \\
\end{smallmatrix} \big)$, by Lem. \ref{lem:Homspace} the section map $s_{\sigma'}$ factorises through $s_{\sigma'} \colon X \xrightarrow{s_{\sigma}} H \xrightarrow{\iota} H'$, with $s_{\sigma'}^* \sG'_{H'} \simeq \sG$ and $s_{\sigma'}^* \sK'_{H'} \simeq \sK$. Since $\iota^*\sG'_{H'} = \overline{\sG}_H \simeq \sG_H$ and $\iota^*\sK'_{H'}\simeq \sK_H$, by Thm. \ref{thm:Quot} \eqref{thm:Quot-2}, the formations of the corresponding Quot schemes are compatible with the base-change maps $X \xrightarrow{s_{\sigma}} H \xrightarrow{\iota} H'$. Since we have shown $\iota$ is Tor-independent with respect to the pair of consideration, by Lem. \ref{lem:3squares} the lemma (in the simple case) is proved.

Finally, thanks to the next lemma, by possibly shrinking $X$ we can always reduce to the above simple case. (More precisely, locally there is another resolution $0 \to \sW'' \xrightarrow{\sigma''} \sV'' \to \sG$ such that $\sW'' \simeq \sW \oplus \sO^{\oplus r}$, $\sV'' \simeq  \sW \oplus \sO^{\oplus r}$, $\sigma'' \simeq \big(\begin{smallmatrix}
\sigma & 0  \\
0 & \Id_r  \\
\end{smallmatrix} \big)$, and $\sW'' \simeq \sW' \oplus \sO^{\oplus s}$, $\sV'' \simeq  \sW' \oplus \sO^{\oplus s}$, $\sigma'' \simeq \big(\begin{smallmatrix}
\sigma' & 0  \\
0 & \Id_s  \\
\end{smallmatrix} \big)$ for some integers $r,s \ge 0$.) Hence the lemma is proved.
\end{proof}

\begin{lemma} \label{lem:Quot:comm.alg} Let $0 \to \sW \xrightarrow{\sigma} \sV \to \sG$ and $0 \to \sW' \xrightarrow{\sigma'} \sV' \to \sG$ be two short exact sequences of $\sO_X$-modules over an affine scheme $X =\Spec R$, such that $\sV$ and $\sV'$ are locally free. Then
	$$\sW\oplus \sW' \xrightarrow{ \big(\begin{smallmatrix}
\sigma & 0  \\
0 & 1_{\sW'}  \\
\end{smallmatrix} \big)} \sV \oplus \sW'  \quad \text{and} \quad \sW' \oplus \sW \xrightarrow{\big(\begin{smallmatrix}
\sigma' & 0  \\
0 & 1_{\sW}  \\
\end{smallmatrix} \big)} \sV' \oplus \sW$$
 are {\em isomorphic} as two-term complexes, with cokernels both equal to $\sG$.
\end{lemma}
\begin{proof} Since $X=\Spec R$ is affine, $\sV$ and $\sV'$ are {\em projective} objects in $\Qcoh(X)$, therefore there is a lifting $f \colon \sV \to \sV'$ of $\sV' \twoheadrightarrow \sG$ and resp. a lifting $g \colon \sV' \to \sV$ of $\sV \twoheadrightarrow \sG$. Denote $f_{W} \colon \sW \to \sW'$ the restriction of $f$, resp. $g_{W'} \colon \sW' \to \sW$ the restriction of $g$, i.e. $f \circ \sigma = \sigma' \circ f_{W}$ and $g \circ \sigma' = \sigma \circ g_{W'}$. Then $1 - gf \colon \sV \to \sV$ factorises through a map $\theta \colon \sV \to \sW$, i.e. $1 - gf = \sigma \circ \theta$, and then $\theta_W: =1 - g_{W'} f_{W} = \theta \circ \sigma'$. Similarly there is map $\psi \colon \sV' \to \sW'$ such that $1 - fg = \sigma' \circ \psi$ and $\psi_{W'}: = 1 - f_{W} g_{W'} = \psi \circ \sigma'$. Now it is direct to check that the following two chain maps (the vertical arrows in the commutative diagrams):
	\begin{equation*}
	\begin{tikzcd}[column sep= 4 em, row sep=3.5 em, ampersand replacement=\&]
		\sW \oplus \sW'  \ar{r}{\big(\begin{smallmatrix}
\sigma & 0  \\
0 & 1_{W'}  \\
\end{smallmatrix} \big)} \ar{d}[swap]{\big(\begin{smallmatrix}
f_W & 1_{W'}  \\
\theta_W & -g_{W'}  \\
\end{smallmatrix} \big)}  \& \sV \oplus \sW' \ar{d}{\big(\begin{smallmatrix}
f & \sigma'  \\
\theta & -g_{W'}  \\
\end{smallmatrix} \big)} \\
		\sW' \oplus \sW  \ar{r}{\big(\begin{smallmatrix}
\sigma' & 0  \\
0 & 1_{W}  \\
\end{smallmatrix} \big)}  \& \sV' \oplus \sW
	\end{tikzcd}
	\qquad \text{and} \qquad 
	\begin{tikzcd}[column sep= 4 em, row sep=3.5 em, ampersand replacement=\&]
		\sW' \oplus \sW  \ar{r}{\big(\begin{smallmatrix}
\sigma' & 0  \\
0 & 1_{W}  \\
\end{smallmatrix} \big)} \ar{d}[swap]{\big(\begin{smallmatrix}
g_{W'} & 1_{W}  \\
\psi_{W'} & -f_{W}  \\
\end{smallmatrix} \big)}  \& \sV' \oplus \sW \ar{d}{\big(\begin{smallmatrix}
g & \sigma  \\
\psi & -f_{W}  \\
\end{smallmatrix} \big)} \\
		\sW \oplus \sW'  \ar{r}{\big(\begin{smallmatrix}
\sigma & 0  \\
0 & 1_{W'}  \\
\end{smallmatrix} \big)}  \& \sV \oplus \sW'
	\end{tikzcd}
	\end{equation*}
are mutually inverse to each other, hence the claim is proved. 
\end{proof}

%\begin{remark} In the above proof of Lem. \ref {lem:Quot:comm.alg}, we only use the fact that $\sV$ and $\sV'$ are projective. Thus the two complexes of Lem. \ref {lem:Quot:comm.alg} are always {\em isomorphic} as long as $\sV$ and $\sV'$ are projective, i.e. $\sW$ and $\sW'$ could be any $\sO_X$-modules. \end{remark}

The next lemma shows that the Tor-independence condition Def. \ref{def:Tor-ind:quot} is equivalent to certain expected dimension condition if the scheme $X$ is Cohen--Macaulay.

% Criteria of Tor-independence
\begin{lemma}[Criteria of Tor-independences for Cohen--Macaulay schemes]  \label{lem:Tor-ind:quot:CM}
For a connected Cohen--Macaulay scheme $X$ and a quasi-coherent $\sO_X$-module $\sG$ of homological dimension $\le 1$, denote $\delta = \rank \sG$ and $\sK = \sExt^1(\sG,\sO_X)$. For a pair of integers $(d_+,d_-)$ such that $\max\{d_+ - \delta, 0\} \le d_- \le d_+$, then {\em Tor-independent condition Def. \ref{def:Tor-ind:quot} holds for the pair $(d_+,d_-)$} iff the following ``expected dimension condition" holds:
	\begin{equation}\label{eqns:Quotexpdim}
	\left\{
	\begin{split}
   & \dim \Quot_{X,d_+}(\sG) = \dim X + d_+ (\delta - d_+);  \\
  & \dim \Quot_{X,d_-}(\sK) =  \dim X + d_- (-\delta - d_-); \\
  & \dim \Quot_{X,d_+}(\sG) \times_X \Quot_{X,d_-}(\sK) = \dim X + \delta(d_+ - d_-) + d_+ d_- - d_+^2 - d_-^2.
	\end{split}
	\right.
	\end{equation}
In the case when \eqref{eqns:Quotexpdim} holds, $\Quot_{X,d_+}(\sG)$, $\Quot_{X,d_-}(\sK)$ and $ \Quot_{X,d_+}(\sG) \times_X \Quot_{X,d_-}(\sK)$ are also Cohen--Macaulay schemes.
\end{lemma}

\begin{proof} Since the problem is local, we may assume $\sG$ admits a presentation $0 \to \sW \xrightarrow{\sigma} \sV \to \sG$, and the morphism $\sigma$ induces a section map $s_\sigma \colon X \hookrightarrow H_X : = |\sHom_X(\sW,\sV)|$, where the latter is Cohen--Macaulay. By Lem. \ref{lem:Homspace} \eqref{lem:Homspace-2} the closed immersion $s_\sigma$ is regular, hence is Koszul-regular (see e.g.\cite[\href{https://stacks.math.columbia.edu/tag/063K}{Tag 063K}]{stacks-project}). Since the involved Quot schemes over $H_X$ and their fiber prodcut over $H_X$, being the domains of local complete intersection morphisms to $H_X$, are all Cohen--Macaulay, hence the claim follows from Lem. \ref{lem:bc:CM}.
\end{proof}

\subsubsection{The general procedure for globalization} \label{subsec:univHom}
To summarise, in the same situation of \S \ref{sec:Tor-ind:Quot:bc}, and assume $X$ is a scheme over a ring $\kk$, we have the following base-change procedure:

\begin{enumerate}[label= (\roman*), leftmargin=*]
	\item (Semi-local) By passing to Zariski opens and choosing any presentation $\sW \xrightarrow{\sigma} \sV$ of $\sG$, we obtain a section map $s_\sigma \colon X \hookrightarrow H_X =|\sHom_X(\sW,\sV)|$, which is a regular closed immersion by Lem. \ref{lem:Homspace} \eqref{lem:Homspace-2};
	 \item (Universal local) By possibly further shrinking we may assume $\sV$ and $\sW$ are free modules, then we obtain a morphism $H_X \to H_\kk$ as in \S \ref{sec:univHom:flat}, which is Tor-independent with respect to the pairs of the form $(\Quot_{d_+}(\sG), \Quot_{d_-}(\sK))$ by Lem. \ref{lem:univHom:flat}.
\end{enumerate}

If Tor-independent conditions Def. \ref{def:Tor-ind:quot} are verified for a pair of integers $(d_+,d_-)$, then the composition of base-change $X \to H_X \to H_\kk$ is Tor-independent, and by descent theory and base-change theory, we can globalize all results of the local cases $H_\kk$ of \S \ref{sec:local}. 

In particular, the following properties can be globalized from the universal local situation with base $H_\kk$ to the general situation with base a quasi-compact, quasi-separated scheme $X$:

\begin{enumerate}[leftmargin=*]
	\item The following properties of morphisms between these Quot schemes: quasi-perfectness, properness, smoothness, whether a closed immersion is Koszul-regular or not, whether a morphism is a locally compete intersection or not, whether a morphism is a blowup along Koszul-regular centers or not (see Lem. \ref{lem:blowup_bc}), etc -- since these properties are fppf local and survives Tor-independent base-changes;
	\item Relative exceptional sequences and collections of these Quot schemes over $X$ -- since the theory of relative exceptional sequences enjoys fppf descent Cor. \ref{cor:relexc:descent} and Tor-independent base-change theory Cor. \ref{cor:relexc};
	\item Relative Fourier--Mukai transforms among these Quot schemes over $X$,  the strongness and fully-faithfulness of these transforms; The induced $X$-linear semiorthogonal sequences and semiorthogonal decompositions, etc --  since the theory of relative Fourier--Mukai transforms enjoys fppf descent Thm. \ref{thm:fppf} and Tor-independent base-change theory Thm. \ref{thm:bc}.
\end{enumerate}

% dualizing sheaves
\subsubsection{First results in global situation}
By the preceding subsection, we have the following immediate globalization results. First, the global version of Lem. \ref{lem:local:Serre} is:

\begin{lemma} \label{lem:global:Serre} In the same situation as \S \ref{sec:Tor-ind:Quot:bc}, and assume the Tor-independent condition Def. \ref{def:Tor-ind:quot} holds for the pair $(d_+,d_-)$. Then all the maps of diagram \ref{diagram:Corr:global} are projective and local complete intersection morphisms, with invertible dualizing complexes. Moreover, if we denote
	$\sO_+(1) : = \sO_{\Quot_{d_+}(\sG)}(1)$, 
	$\sO_-(1) : = \sO_{\Quot_{d_-}(\sK)}(1)$,
then corresponding dualizing complexes of these maps are given by:
	\begin{align*} 
	 & \omega_{r_+} = (\det \sG)^{\otimes d_-} \otimes \sO_-(-d_++\delta) \otimes \sO_+(-d_-)[-d_- (\delta - d_+ + d_-)], \quad \\
	 & \omega_{r_-}  = (\det \sG)^{\otimes d_+} \otimes \sO_-(-d_+) \otimes \sO_+(-d_--\delta)[d_+  (\delta - d_+ + d_-)],  \\
	& \omega_{\pi_+}  =   (\det \sG)^{\otimes d_+}  \otimes \sO_+(-\delta)[d_+ (\delta-d_+)],  \\% \qquad \text{and} \qquad
	 & \omega_{\pi_-}  =  (\det \sG)^{\otimes d_-} \otimes \sO_-(\delta) [d_- (-\delta-d_-)], 	\\
	 &\omega_{\widehat{\pi} } =  (\det \sG)^{\otimes (d_+ + d_-)} \sO_-(-d_++\delta) \otimes \sO_+(-d_--\delta)[(d_+-d_-)(\delta-d_+ + d_- )- d_+ d_- ]. 
	\end{align*}
(Here $\det \sG$ is the line bundle defined in Def \ref{def:det}; And for the simplicity of expressions, we use the same notations $\sO_{\pm}(1)$ and $\det \sG$ to denote their corresponding pullbacks.)
% to the schemes $\shZ_\pm$ and $\widehat{\shZ}$.)
%Here we denote $\sO_{\GG_\pm}(1) = \det \shQ_\pm \in \Pic(\GG_\pm)$ the ample line bundle on $\GG_\pm$, and  denote $\sO_\pm (1)$ the pull back of $\sO_{\GG_\pm}(1) $ to the corresponding schemes.
\end{lemma}

\begin{proof} Projectivity follows from properties of $\Quot$ Thm. \ref{thm:Quot}; The statements about locally compete intersection morphisms follow from the cases of  local universal situation, since Koszul-regular closed immersion survives Tor-independent base-change Lem. \ref{lem:bc:CM}. The rest of the lemma is exactly the globalization of the local cases Lem. \ref{lem:local:Serre}.
\end{proof}

Next, for simplicity we assume $X$ is a $\kk$-scheme, where $\kk$ is a field of characteristic zero. Denote $\rank \sG = \delta$ and let $d$ be a fixed integer, and consider the Quot scheme $\pi_+ \colon \Quot_{X,d}(\sG) \to X$. Denote $\sQ_d$ the universal quotient bundle, and denote $\sO_+(1) = \sO_{\Quot_{X,d}(\sG)}(1) = \bigwedge^d \sQ_d$. If $\sG$ admits a presentation $\sW \xrightarrow{\sigma} \sV$ of $\sG$, where $\rank \sW = m$, $\rank \sV = n$, then we can consider the Grassmannian bundles $\Gr_{d}(\sV^\vee)$. Denote $\shU_+$ and $\shQ_+$ the universal subbundle of rank $d$ and respectively universal quotient bundle of rank $n-d$. By Thm. \ref{thm:Quot} there is a closed immersion $\Quot_{X,d}(\sG) \subseteq \Gr_d(\sV)$, such that $\shU_+^\vee|_{\Quot_{X,d}(\sG)} \simeq \sQ_d$. 

Recall $B_{\ell,d}^{\preceq}$ denotes the set of Young diagrams inscribed in a $\ell \times d$-rectangle equipped with the natural partial order of inclusions, $B_{\ell, d}^{\succeq}$ denotes the same set with the opposite order. The next result is a globalization of Lem. \ref{lem:top}, Lem. \ref{lem:bottom}, Lem. \ref{lem:bottom_top}.

\begin{proposition} \label{prop:top:bottom} In the above situation, and let $\D$ stand for $\Dqc, \Db$ or $\Perf$.
\begin{enumerate}[leftmargin=*]
	\item  \label{prop:top:bottom-1}
	 If $d \le \delta$, and the Tor-independent condition Def. \ref{def:Tor-ind:quot} holds for the pair $(d,0)$ (if $\sG$ admits a presentation as above, then this is equivalent to the condition that the closed immersion $\Quot_{X,d}(\sG) \subseteq \Gr_d(\sV)$ is Koszul-regular of codimension $m \cdot d$). Then $\{\Sigma^{\alpha^t} \sQ_d \}_{\alpha \in B_{\delta-d,d}^{\preceq}}$ is a relative expectational sequence of vector bundles on $\Quot_{X,d}(\sG)$ over $X$. In particular, for any $\alpha \in B_{\delta-d, d}$, the Fourier--Mukai functors:
	\begin{align*}
	&\Phi^{\alpha}(\blank) := \pi_+^*(\blank) \otimes \Sigma^{\alpha^t} \sQ_d \colon \qquad &\D(X) \to \D(\Quot_{X,d}(\sG))
	\end{align*}
 are fully faithful, and their images form an $X$-linear admissible semiorthogonal sequence $\{ \Im \Phi^\alpha\}_{\alpha \in B_{\delta-d,d}^{\preceq}}$ such that $\Im \Phi^{\alpha} \subseteq (\Im \Phi^{\beta})^\perp$ whenever $\alpha \nsucceq \beta$. 
	\item  \label{prop:top:bottom-2}
	Assume there is a presentation $\sW \xrightarrow{\sigma} \sV$ of $\sG$ as above such that $n \ge d \ge m$, and the rank-$0$ degeneracy locus $Z = D_0(\sigma) \subset X$ is a Koszul-regular closed subscheme of the expected codimension $mn$. Denote $p_+ \colon G_Z \to Z$ the restriction of $\pi_+$ to $Z \subseteq X$, and denote $j_+ \colon G_Z \to \Quot_{X,d}(\sG)$ the inclusion. Assume further that Tor-independent condition Def. \ref{def:Tor-ind:quot} holds for the pair $(d,m)$. Then for any $\alpha \in B_{n-d, d - m}$, the relative Fourier--Mukai functors over $X$:
	\begin{align*}
	&\Psi^{\alpha}(\blank) :=  j_{+\,*} p_{+}^* (\blank) \otimes \Sigma^{\alpha} \shQ_+^\vee|_{\Quot_{X,d}(\sG)} \colon & \D(Z)  \to \D(\Quot_{X,d}(\sG))
	\end{align*}
 are fully faithful, and their images form an $X$-linear admissible semiorthogonal sequence $\{\Im \Psi^{\alpha} \}_{\alpha \in B_{n-d,d-m}^{\succeq}}$ such that $\Im \Psi^{\alpha} \subseteq (\Im \Psi^{\beta})^\perp$ whenever $\alpha \npreceq \beta$. 
	\item  \label{prop:top:bottom-3}
	Assume both the conditions of (1) and (2) hold, then for any $1 \le s \le m$,		$$\big ( \{\Im \Psi^{\beta + s} \}_{\beta \in B_{n-d, d-m}^{\succeq}}, ~\{ \Im \Phi^\alpha\}_{\alpha \in B_{\delta-d,d}^{\preceq}} \big )$$
	forms a $X$-linear admissible semiorthogonal sequence inside $\D(\Quot_{X,d}(\sG))$, i.e. $\Im \Psi^{\beta + s} \subseteq (\Im \Phi^{\alpha})^\perp$ for all $\alpha \in B_{\delta - d, d} ,\beta \in B_{n-d,d-m}$, $1 \le s \le m$.
\end{enumerate}
\end{proposition}

The next result is a globalization of Lem. \ref{lem:sod:O(i)}.

\begin{proposition} In the above situation, and let $\D$ stands for $\Dqc, \Db$ or $\Perf$. If $\max\{d - \delta, 0 \} \le d_- \le d$ holds, and the Tor-independent condition Def. \ref{def:Tor-ind:quot} holds for the pair $(d,d_-)$. Denote $r_\pm$ the projections from $\Quot_{X,d}(\sG) \times_X \Quot_{X,d_-}(\sK)$ to $\Quot_{X,d}(\sG)$ and $\Quot_{X,d_-}(\sK)$ as usual. Then for any $i \in \ZZ$,  the relative Fourier--Mukai functor over $X$:
	\begin{align*}
	&\Phi_i (\blank) : = (r_{+\,*} \, r_{-}^*(\blank)) \otimes \sO_+(i) \colon & \D(\Quot_{X,d_-}(\sK)) \to \D(\Quot_{X,d}(\sG))
	\end{align*}
is fully faithful. If $\min \{d - d_-, \delta - d + d_-\} > 0$, then for any fixed $i \in \ZZ$,
	$$\big( \Im (\Phi_i) ,  \Im (\Phi_{i+1}), \ldots, \Im (\Phi_{i+\delta-1}) \big)$$ 
forms an admissible $X$-linear semiorthogonal sequence inside $ \D(\Quot_{X,d}(\sG))$.
\end{proposition}

In the following sections we will apply the same procedure to various concrete situations.

% Blowup theorem
\subsection{Blowups along Koszul-regularly immersed centers} 
Let $X$ be a quasi-compact and quasi-separated scheme, and let $i \colon Z \hookrightarrow X$ be a Koszul-regularly immersed closed subscheme of constant codimension $r \ge 2$ cut out by an ideal sheaf $\sI_Z$. Denote by 
	$$\pi \colon \Bl_Z X = \underline{\Proj}_{X} \bigoplus_{n \ge 0} \sI^n_Z \to X$$ the blowup of $X$ along $Z$, and $E = \pi^{-1} (Z)$ the exceptional divisor. Let $p \colon E \to Z$ be the natural projection, and $j \colon E \hookrightarrow \Bl_Z X$ the inclusion. Denote by $\sO_{\Bl_Z X}(1) =\sO_{\Bl_Z X}(-E)$ the $\pi$-relative ample line bundle from the Proj construction, and $\sO_{E}(1)$ the restriction of $\sO_{\Bl_Z X}(1)$ to $E$. There is a Cartesian diagram
	\begin{equation*}
	\begin{tikzcd}[row sep= 2.6 em, column sep = 2.6 em]
	E \ar{d}[swap]{p} \ar[hook]{r}{j} & \Bl_Z X \ar{d}{\pi} \\
	Z \ar[hook]{r}{i}         & X 
	\end{tikzcd}	
	\end{equation*}
The following generalizes Orlov \cite[Thm. 4.3]{Orlov92}; see also \cite[Thm. 6.9]{BS} for the stack case.

\begin{theorem}[Blowing up formula] \label{thm:blow-up}  Let $\pi \colon \Bl_Z X \to X$ be the blowup along a Koszul-regularly immersed center $Z$ of codimension $r \ge 2$ as above. Then
\begin{enumerate}[leftmargin=*]
	\item \label{thm:blow-up-1} 
	For each $k \in \ZZ$, the relative Fourier--Mukai functors over $X$:
	$$\Phi_k :=\pi^*(\blank) \otimes  \sO_{\Bl_Z X}(k) \colon  \Dqc(X) \to \Dqc(\Bl_Z X),$$
	$$\Psi_k := j_* p^*(\blank) \otimes \sO_{\Bl_Z X}(k) \colon \Dqc(Z) \to \Dqc(\Bl_Z X)$$
are strong in the sense of Def. \ref{Def:FM:BS} and fully faithful;
	\item \label{thm:blow-up-2}
	The morphism $\pi \colon \Bl_Z X \to X$ is a projective local complete intersection morphism, with relative dualizing complex given by the line bundle:
		$$\omega = \sO_{\Bl_Z X}(1-r) = \sO_{\Bl_Z X}((r-1)E).$$ 
The category $\Perf(\Bl_Z X)$ admits a relative Serre functor over $X$ given by $\SS = \otimes \omega$. Thus for each $k \in \ZZ$, $\SS ( \Im \Phi_{k}) = \Im \Phi_{k+1-r}$ and $\SS (\Im \Psi_{k}) = \Im \Psi_{k+1-r}$.
%	$$\SS ( \Im \Phi_{k}) = \Im \Phi_{k+1-r}, \qquad \SS (\Im \Psi_{k}) = \Im \Psi_{k+1-r};$$
	\item \label{thm:blow-up-4}
	For each integer $0 \le \ell \le r-1$, there is an $X$-linear semiorthogonal decomposition with admissible components:
	\begin{align*}
		\Dqc(\Bl_Z X) = \langle \underbrace{\Im \Psi_{1-r+\ell}, \ldots, \Im \Psi_{-1}}_{(r-1-\ell)\text{-terms}},  \Im \Phi_0,  \underbrace{\Im \Psi_0, \ldots, \Im \Psi_{\ell-1}}_{\ell\text{-terms}} \rangle;
	\end{align*}
%	\begin{align*} 
%	\Dqc(\Bl_Z X)	& = \langle \Im \Phi_0, ~ \Im \Psi_0, \Im \Psi_1, \ldots, \Im \Psi_{r-2} \rangle; \\
%				& = \langle \Im \Psi_{1-r}, \ldots, \Im \Psi_{-2},  \Im \Psi_{-1}, ~  \Im \Phi_0 \rangle. 			
%	\end{align*}
Similar semiorthogonal decomposition hold if $\Dqc$ is replaced by $\Db$ or $\Perf$, and these semiorthogonal decompositions are compatible with the inclusions $\Perf \subseteq \Db \subseteq \Dqc$. 
\end{enumerate} 
\end{theorem}

\begin{proof} This is essentially the globalization of universal local case of \S \ref{sec:local}; To show that the same argument works over $\ZZ$, we provide more details. Similar to the proof of Lem. \ref{lem:blowup:univ}, we consider the the universal local situation when $X = \AA^r$, $Z = \{0\}$, and $\Bl_Z X = \Bl_{\{0  \} } \AA^r = \Quot_{\AA^r, r-1}(\sG_\ZZ)$. %Since $\pi$ is the composition of the regular closed immersion $\iota \colon \Bl_{\{0  \} } \AA^n \hookrightarrow \AA^r \times \PP^{r-1}$ and the projection $q \colon \AA^r \times \PP^{r-1} \to \AA^r$, hence $\pi$ is projective and perfect.
We claim that there is a semiorthogonal decomposition 
	\begin{equation}\label{eqn:proof:Bl_0} \Perf( \Bl_{\{0  \} } \AA^r) =  \langle \Im \Phi_0, ~ \Im \Psi_0, \Im \Psi_1, \ldots, \Im \Psi_{r-2} \rangle.\end{equation}
It is easy to show the right hand side is a semiorthogonal sequence; In fact the same computations of \cite{Orlov92} or \cite{Huy} work for the case $Z = \{0\}$ and $X=\AA^r$. To show the right hand side of (\ref{eqn:proof:Bl_0}) generates the whole category, notice since the composition map $\Bl_{\{0  \} } \AA^r \to \AA^r \times \PP^{r-1} \to \PP^{r-1}$ is an affine bundle, by Lem. \ref{lem:span} $\Perf( \Bl_{\{0  \} } \AA^r)$ is classically generated by $\sO, \sO(1), \ldots, \sO(r-1)$ (where for $k \in \ZZ$, $\sO(k):= \sO_{\Bl_{\{0  \} }  \AA^r}(1)^{\otimes k}$, and $\sO_{\Bl_{\{0  \} }  \AA^r}(1)$ is the $\sO(1)$ from Proj-construction as usual). On the other hand, for each $k$, $\Psi_k (\sO_{\Spec \ZZ}) =\sO_E(k) \simeq [\sO(k+1) \to \sO(k)]$, and $\Phi_0(\sO_{\AA^r}) = \sO$. Inductively we see that the right hand side of  (\ref{eqn:proof:Bl_0}) contains all $\sO(k)$ for $k \in [0, r-1]$, hence the claim is proved.

Hence \eqref{thm:blow-up-1} holds by fppf descent and Tor-independent base-change Prop. \ref{prop:relFM_bc};  \eqref{thm:blow-up-2} holds by Lem. \ref{lem:blowup:univ}, Ex. \ref{eg:lci} and Prop. \ref{prop:Serre}; For \eqref{thm:blow-up-4}, about decomposition for the category $\Perf( \Bl_{\{0  \} } \AA^r)$, the statement for general $\ell$ follows from the $\ell=0$ case (\ref{eqn:proof:Bl_0}) by Serre duality; Finally, the general case of \eqref{thm:blow-up-4} holds by the procedure \S \ref{subsec:univHom}, i.e. follows from the corresponding semiorthogonal decomposition of $\Perf( \Bl_{\{0  \} } \AA^r)$ by Tor-independent base-change Thm. \ref{thm:bc} and fppf descent Thm. \ref{thm:fppf}.
\end{proof}

\begin{remark} In the situation of the theorem, the proof of \cite[Lem. 2.9]{JL18} also works in this case and shows that each $k \in \ZZ$ , there are natural isomorphisms of functors:
		$$\LL_{\Im \Psi_k} \circ \Phi_{k+1} = \Phi_k, \qquad \RR_{\Im \Psi_k} \circ \Phi_k = \Phi_{k+1}.$$
\end{remark}

% Cayley trick
\subsection{Cayley's trick}
Let $X$ be a quasi-compact, quasi-separated scheme, and let $Z \subseteq X$ be a closed subscheme cut out by a Koszul-regular section $s$ of a locally free sheaf $\sE$ of constant rank $n$. Denote $\sG: = \Coker (\sO_X \xrightarrow{s} \sE)$, then $\sG$ has homological dimension $1$, $\rank \sG = n-1$, and $\sK = \sExt^1(\sG, \sO_X) \simeq \sO_Z$ has homological dimension $n$. The inclusion $\iota \colon \PP(\sG) \hookrightarrow \PP(\sE)$ is cut out by the section $\tilde{s} \in \Gamma(\PP(\sE), \sO_{\PP(\sE)}(1))$ which corresponds to $s$ under the canonical adjunction $\Hom_{\PP(\sE)}(\sO_{\PP(\sE)},\sO_{\PP(\sE)}(1)) = \Hom_X(\sO_X, \sE)$.

\begin{lemma} \label{lem:cayley:univ} The section $\tilde{s}$ is a regular section of the line bundle $ \sO_{\PP(\sE)}(1)$, in particular the zero scheme $Z(\tilde{s}) = \PP(\sG)$ is an effective Cartier divisor on $\PP(\sE)$.
\end{lemma}
\begin{proof} The problem being local, we may assume $X = \Spec R$ for a ring $R$, $\sE$ is given by the free module $R^{\oplus n}$, and $Z \subset X$ is cut out by an ideal generated by a regular sequence $I = (r_1, r_2, \ldots, r_n)$, $r_i \in R$. Then $\PP(\sE) = \Proj S$, where $S = R[X_1, \ldots, X_n]$ is the polynomial algebra with grading given by the degrees of polynomials. The section $\tilde{s} \in \Gamma(\sO_{\Proj S_\bullet}(1)) = S_1$ corresponds to the degree-one element $f = r_1 X_1 + r_2 X_2 + \ldots + r_n X_n \in S_1$. By \cite[\S 5.3, Thm. 7]{No}, $f$ is a non-zerodivisor on $S$, therefore $\tilde{s}$ is a regular section of $\sO_{\PP(\sE)}(1)$. \end{proof}

Denote $q \colon \PP(\sE) \to X$ the projection. Let $P_Z: = \pi^{-1} Z$, and $p \colon P_Z  \to Z$ the projection, and $j \colon P_Z \hookrightarrow \PP(\sG)$ the inclusion. Then $p \colon P_Z= \PP(\sE|_Z) \to Z$ is a projective bundle. If we denote $\sO(1):=\sO_{\PP(\sG)}(1)$ the ample line bundle form the projectivization construction, and $\sO(k): = \sO(1)^{\otimes k}$ for $k \in \ZZ$, then $\sO(k)|_{P_Z} = \sO_{\PP(\sE|_Z)}(k)$. There is a commutative diagram:
\begin{equation}\label{diagram:Cayley}
	\begin{tikzcd}[row sep= 2.6 em, column sep = 2.6 em]
	P_Z \ar{d}[swap]{p} \ar[hook]{r}{j} & \PP(\sG) \ar{d}{\pi} \ar[hook]{r}{\iota} & \PP(\sE) \ar{ld}[near start]{q} 
	\\
	Z \ar[hook]{r}{i}         & X  
	\end{tikzcd}	
\end{equation}

The following generalizes Orlov's result {\cite[Prop. 2.10]{Orlov06}}.

\begin{theorem}\label{thm:Cayley} In the above situation (where $n = \rank \sE$, $n-1 = \rank \sG$): %Let $X$ be a scheme, $Z \subseteq X$ be a regularly immersed closed subscheme cut out by a regular section $s$ of a locally free sheaf $\sE$ of constant rank $n$, and $\sG: = \Coker (\sO_X \xrightarrow{s} \sE)$ as above. 

\begin{enumerate}[leftmargin=*]
	\item \label{thm:Cayley-1}
	For each $k \in \ZZ$, the sequence $\sO(k+1), \sO(k+2), \ldots, \sO(k+n-1) \in \Perf(\PP(\sG))$ is a relative exceptional sequence of $\PP(\sG)$ over $X$ of length $n-1 = \rank \sG$;
	\item \label{thm:Cayley-2}
	For each $k$, the relative Fourier--Mukai functors over $X$:
	$$\Phi_k :=j_* p^* \otimes \sO(k) \colon  \Dqc(Z) \to \Dqc(\PP(\sG)),$$
	$$\Psi_k := \pi^*(\blank) \otimes \sO(k) \colon \Dqc(X) \to \Dqc(\PP(\sG))$$
are strong (in the sense of Def. \ref{Def:FM:BS}) and fully faithful.
	\item \label{thm:Cayley-3}
	The category $\Perf(\PP(\sG))$ admits a relative Serre functor $\SS$ over $X$ given by 
		$$\SS = (\blank) \otimes \omega, \quad  \text{where} \quad \omega = \pi^* (\det \sE) \otimes \sO(1-n)[n-2].$$ 
		In particular for each $k \in \ZZ$, 
			$\SS ( \Im \Phi_{k}) = \Im \Phi_{k+1-n}$, $\SS (\Im \Psi_{k}) = \Im \Psi_{k+1-n}.$
	\item \label{thm:Cayley-5}
	For each integer $0 \le \ell \le n-1$, there is an $X$-linear semiorthogonal decomposition with admissible components:
	\begin{align*}
		\Dqc(\PP(\sG)) = \langle \underbrace{\Im \Phi_{2-n+\ell}, \ldots, \Im \Phi_{0}}_{(n-1-\ell)\text{-terms}},  \Im \Psi_0,  \underbrace{\Im \Phi_1, \ldots, \Im \Phi_{\ell}}_{\ell\text{-terms}} \rangle;
	\end{align*}
Similar semiorthogonal decompositions hold if we replace $\Dqc$ by $\Db$ or $\Perf$, and these semiorthogonal decompositions are compatible with the inclusions $\Perf \subseteq \Db \subseteq \Dqc$. 
\end{enumerate} 
\end{theorem}

\begin{proof} This is a special case of projectivization formula Thm. \ref{thm:projectivization}; To show these results hold over $\ZZ$, we provide more details. By Lem. \ref{lem:cayley:univ}, $\pi$ is the composition of the inclusion of $\sO_{\PP(\sE)}(1)$-divisor $\iota \colon \PP(\sG) \to \PP(\sE)$ followed by the projection $q \colon \PP_X(\sE) \to X$, hence \eqref{thm:Cayley-3} holds by Lem. \ref{lem:blowup:univ}, Ex. \ref{eg:lci} and Prop. \ref{prop:Serre}. For \eqref{thm:Cayley-1}, from Lem. \ref{lem:cayley:univ} and the short exact sequence $0 \to \sO_{\PP(\sE)}(-1) \to \sO_{\PP(\sE)} \to \iota_* \sO \to 0$, we obtain for any $t, s \in \ZZ$, $\shom_X(\sO(t), \sO(s)) = q_* \iota_* \sO(s-t) = \cone (q_* \sO(s-t-1) \to q_* \sO(s-t)) \in \Perf(X)$. Hence by Thm. \ref{thm:proj.bundle}\eqref{thm:proj.bundle-1}, for $s \le t \le s+n-1$, $\shom_X(\sO(t), \sO(s))  = \Sym^{s-t} \sE = \delta_{s,t} \cdot \sO_X$. This proves \eqref{thm:Cayley-1}. 

For the rest of the proof, as the blowup case, we first consider the universal local situation when $X= \AA^n$, and $Z = \{0\}$. We first claim that there are $X$-linear semiorthogonal decomposition of admissible components:
	\begin{align} \label{eqn:proof:cayley}
	 \Perf(\PP(\sG))  = \langle \Im \Psi_0, ~\Im \Phi_1, \ldots, \Im \Phi_{n-1} \rangle 
	 =  \langle \Im \Phi_{2-n}, \ldots, \Im \Phi_{0},~ \Im \Psi_0 \rangle.
	\end{align}
Since the relative Serre functor $\SS$ of \eqref{thm:Cayley-3} takes these two semiorthogonal decompositions to one another, we only need to show one of them. For the first decomposition of (\ref{eqn:proof:cayley}), the last $(n-1)$-terms are induced by the relative exceptional sequence $\sO(1), \ldots, \sO(n-1)$ of \eqref{thm:Cayley-1}.  For $\Psi_0 = j_* p^*$, it admits a left adjoint $\Psi_0^L = p_! \, j^*$ and a right adjoint $\Psi_0^R = p_* \, j^!$. From the diagram (\ref{diagram:Cayley}), the inclusion $j \colon P_Z = \PP_\ZZ^{n-1} \to \PP(\sG)$ is Koszul-regular immersion given by a section of the vector bundle $\sV|_{P_Z}  = \Omega_p^1(1)$ (see \cite[Rmk. 2.5]{JL18}; This could also be easily observed in the local situation since $\PP(\sG) = |\Omega_{\PP^{n-1}}^1(1)|$ and $j$ is the inclusion of zero section). Then Koszul complex induces (see also Lem. \ref{lem:key}):
	$$\Psi_0(\sO_Z) \simeq [\bigwedge^{n-1} \sV^\vee \to \cdots \bigwedge^{1} \sV^\vee \to \sO] \in \Perf(\PP(\sG)).$$
On the other hand, by Thm. \ref{thm:proj.bundle} \eqref{thm:proj.bundle-3} we know that for $d \in [0, n-1]$, $\Psi_0^L (\bigwedge^d \sV^\vee) \simeq (p_*  \Omega_p^d(d))^\vee \simeq \delta_{0,k} \sO_X$. Therefore we have a natural isomorphism $\Psi_0^L \Psi_0 (\sO_Z) \simeq \sO_Z$, and hence $\Psi_0$ is fully faithful. Since $\otimes \sO(k)$ is an equivalence, $\Psi_k$ is fully faithful with admissible image for each $k$. Furthermore, for any $s \in [1,n-1]$, then $\Psi_0^L \circ \Phi_s (\sO_X) = p_! (\sO(s)) = 0$ by Lem. \ref{lem:proj.bundle}  \eqref{lem:proj.bundle-2}, this shows the right hand side of the first decomposition of (\ref{eqn:proof:cayley}) is a semiorthogonal sequence. By Serre duality second decomposition is also a semiorthogonal sequence. Finally, we show the second semiorthogonal sequence of (\ref{eqn:proof:cayley}) generates $\Perf(\PP(\sG))$. The first part of images $\langle \Im \Phi_{2-n}, \ldots, \Im \Phi_{0} \rangle$ contains $\sO(k)$ for all $k \in [2-n, 0]$. On the other hand, by Rmk. \ref{rmk:proj:dual}, we know that $\{\sO_{\PP^{n-1}}(k)\}_{k \in [2-n,0]}$ and $\{(\Omega_{\PP^{n-1}}^k(k))^\vee\}_{k \in [0,n-2]}$ are mutation equivalent, hence the category $\langle \Im \Phi_{2-n}, \ldots, \Im \Phi_{0} \rangle$ also contains the generators $\bigwedge^k \sV^\vee$ for $k \in [0, n-2]$. Therefore by the Koszul resolution of $\Psi_0(\sO_Z)$ above, the image $\langle \Im \Phi_{2-n}, \ldots, \Im \Phi_{0}, \Psi_0 \rangle$ also contains $\bigwedge^{n-1} \sV^\vee$. Since the composition map $\PP(\sG) \to \AA^n \times \PP^{n-1} \to \PP^{n-1}$ is affine, by Rmk. \ref{rmk:proj:dual}, $\Perf(\PP(\sG))$ is generated by $\bigwedge^k \sV^\vee$ for $k \in [0, n-1]$. Hence (\ref{eqn:proof:cayley}) is proved. 

For \eqref{thm:Cayley-5}, in the above local case $X = \AA^n$, the statement for $\Perf(\PP(\sG))$ and general $\ell \in [0, n-1]$ follows from (\ref{eqn:proof:cayley}) via mutations and Serre duality. Then the general cases follow from Tor-independent base-change Thm. \ref{thm:bc} and fppf descent Thm. \ref{thm:fppf}.
\end{proof}

\begin{remark} Similar to the blowup case, the proof of \cite[Lem. 2.10]{JL18}  also works in our case and shows that for each $k \in \ZZ$, there are natural isomorphisms of functors:
		$$\LL_{\Im \Phi_k} \circ \Psi_{k} = \Psi_{k-1}[2], \qquad \RR_{\Im \Phi_k} \circ \Psi_{k-1} = \Psi_{k}[-2].$$
\end{remark}

% sec: projectivization
\subsection{Projectivizations}
Let $X$ be a connected quasi-compact, quasi-separated scheme over a field $\kk$ of characteristic zero, and let $\sG$ be a finite type quasi-coherent sheaf of homological dimension $\le 1$ and rank $\delta$. Denote $\sK : = \sExt^1(\sG, \sO_X)$ as usual. Let $\pi \colon \PP(\sG) = \Proj_X \Sym^\bullet \sG \to X$ be the projectivization. Denote the natural projections maps by 
	$$r_+ \colon \PP(\sG) \times_X \PP(\sK) \to \PP(\sG), \qquad r_- \colon \PP(\sG) \times_X \PP(\sK) \to \PP(\sK).$$
Denote $\sO(1) = \sO_{\PP(\sG)}(1)$ the ample line bundle $\sO(1)$ from Proj construction.

The following global version of Thm. \ref{thm:local:proj} generalizes \cite{JL18}:

% Theorem: projectivization
\begin{theorem}[Projectivization {\cite{JL18}}] \label{thm:projectivization} In the above situation, assume that the Tor-independent condition Def. \ref{def:Tor-ind:quot} holds for the pair $(d_+,d_-)=(1,1)$. In particular, if $X$ is Cohen--Macaulay, %by Lem. \ref{lem:Tor-ind:quot:CM}
this assumption is equivalent to the following expected dimension condition:
	\begin{align*}
  	 & \dim \PP(\sG) = \dim X + (\delta - 1);  \qquad \dim \PP(\sK)=  \dim X - (\delta +1); \\
  	& \dim \PP(\sG) \times_X \PP(\sK) = \dim X - 1.
	\end{align*}
%Then $\PP(\sG)$, $\PP(\sK)$ and $\PP(\sG) \times_X \PP(\sK)$ are Cohen--Macaulay. Furthermore:
\begin{enumerate}[leftmargin=*]
	\item \label{thm:projectivization-1}
	For each $k \in \ZZ$, the sequence $\sO(k+1), \sO(k+2), \ldots, \sO(k+\delta)$ is a relative exceptional sequence for $\PP(\sG)$ over $X$;
	\item \label{thm:projectivization-2}
	For each $k$, the relative Fourier--Mukai functors (Def. \ref{Def:FM:BS}) over $X$:
	\begin{align*}
		&\Psi_k :=\pi_+^*(\blank) \otimes \sO_+(k) \colon & \Dqc(X) \to \Dqc(\PP(\sG)) ,\\
		& \Phi: = r_{+\,*} \, r_{-}^*(\blank) \colon & \Dqc(\PP(\sK)) \to \Dqc(\PP(\sG)),
	\end{align*}
are strong in the sense of Def. \ref{Def:FM:BS} and fully faithful. Furthermore, these functors induce $X$-linear semiorthogonal decompositions with admissible components:
	\begin{align*}
	\Perf(\PP(\sG)) &= \big \langle \Im (\Phi|_{\Perf(\PP(\sK))} ), ~  \Im (\Psi_1|_{\Perf(X)}), \ldots, \Im (\Psi_\delta|_{\Perf(X)}) \big \rangle;  \\
	\Db(\PP(\sG)) &= \big \langle \Im (\Phi|_{\Db(\PP(\sK))} ), ~  \Im (\Psi_1|_{\Db(X)}), \ldots, \Im (\Psi_\delta|_{\Db(X)}) \big \rangle;  \\
	\Dqc(\PP(\sG)) &= \big \langle \Im \Phi , ~  \Im \Psi_1, \ldots, \Im \Psi_\delta \big \rangle; 
	\end{align*}
These decompositions are compatible with the inclusions $\Perf \subseteq \Db \subseteq \Dqc$. 
	\item \label{thm:projectivization-3} 
	The morphism $\pi \colon \PP(\sG) \to X$ is a projective locally compete intersection morphism with relative dualizing complex $\omega$ given by 
		$$\omega = \pi^* (\det \sG) \otimes \sO(-\delta)[\delta-1].$$
	The category $\Perf(\PP(\sG))$ admits a relative Serre functor over $X$ given by $\SS = \otimes \omega$. In particular, for each $k \in \ZZ$, $\SS (\Im \Psi_{k}) = \Im \Psi_{k - \delta}$. 
	\item \label{thm:projectivization-4} 
	For each integer $\ell$ such that $0 \le \ell \le \delta$, there is an $X$-linear semiorthogonal decomposition with admissible components:
	\begin{align*}
		\Dqc(\PP(\sG)) = \langle \underbrace{\Im \Psi_{\ell - \delta +1}, \ldots, \Im \Psi_{0}}_{(\delta-\ell)\text{-terms}},  \Im \Phi,  \underbrace{\Im \Psi_1, \ldots, \Im \Psi_{\ell}}_{\ell\text{-terms}} \rangle;
	\end{align*}
Similar semiorthogonal decompositions hold if we replace $\Dqc$ by $\Db$ or $\Perf$, and these semiorthogonal decompositions are compatible with the inclusions $\Perf \subseteq \Db \subseteq \Dqc$. 
\end{enumerate}
\end{theorem}

\begin{proof} By fppf descent Thm. \ref{thm:fppf} and Tor-independent base-change Thm. \ref{thm:bc}, the results of \eqref{thm:projectivization-1} and  \eqref{thm:projectivization-2} follow directly from the local case Thm. \ref{thm:local:proj} via the general procedure \S \ref{subsec:univHom} as before. 
\eqref{thm:projectivization-3} follows from Lem.  \ref{lem:global:Serre} and Prop. \ref{prop:Serre}. Finally for  \eqref{thm:projectivization-4}, for a given $\ell$, the corresponding formula holds for the category $\Perf(\PP(\sG))$ in the local case, by applying relative Serre duality Lem. \ref{lem:Serre} to Thm. \ref{thm:local:proj}; Then the general cases follow from fppf descent Thm. \ref{thm:fppf} and Tor-independent base-change Thm. \ref{thm:bc} as above.
\end{proof}

% Sec: generalized Cayley's trick
\subsection{Generalized Cayley's trick and Pirozhkov's theorem} Let $X$ be a connected quasi-compact, quasi-separated scheme over a field $\kk$ of characteristic zero, let $Z \subseteq X$ be a regularly immersed closed subscheme cut out by a regular section $s$ of a locally free sheaf $\sE$ of constant rank $n$, and let $1 \le d \le n-1$ be a fixed integer. Denote $\sG: = \Coker (\sO_X \xrightarrow{s} \sE)$, and consider the Quot scheme:
	$$\pi \colon \shZ_+: = \Quot_{X, d}(\sG) \to X.$$
Let $\sQ_d$ be the universal quotient bundle of $\sG$ obtained from Quot construction Thm. \ref{thm:Quot}, and denote by $\sO_{\shZ_+}(1) : = \sO_{\Quot_{X, d}(\sG)}(1) = \bigwedge^d \sQ_d$. Denote $q \colon \Gr_d(\sE^\vee) \to X$ the Grassmannian bundle, let $\shU$ and $\shQ$ be the universal subbundle and quotient bundle of $\sE^\vee$ of rank $d$ and $n-d$ of Ex. \ref{ex:Grass}. 
By Thm. \ref{thm:Quot} \eqref{thm:Quot-2} there is an closed immersion $\iota \colon \shZ_+ \hookrightarrow \Gr_d(\sE^\vee)$ induced by $\sE \twoheadrightarrow \sG$ such that $\shU^\vee|_{\shZ_+} \simeq \sQ_d$; The immersion $\iota$ is cut out by the section $\tilde{s} \in \Gamma(\Gr_d(\sE^\vee), \shU^\vee)$ of $\shU^\vee$  which corresponds to $s$ under the canonical adjunction $\Hom_{\Gr_d(\sE^\vee)}(\sO, \shU^\vee) = \Hom_X(\sO_X, \sE)$. Denote $P_Z: = \pi^{-1} Z$, and let $p \colon P_Z  \to Z$ be the projection, and $j \colon P_Z \hookrightarrow \shZ_+$ be the inclusion. There is a commutative diagram:
\begin{equation*}%\label{diagram:gen:Cayley}
	\begin{tikzcd}[row sep= 2.6 em, column sep = 2.6 em]
	P_Z \ar{d}[swap]{p} \ar[hook]{r}{j} & \shZ_+ \ar{d}{\pi} \ar[hook]{r}{\iota} & \Gr_d(\sE^\vee) \ar{ld}[near start]{q} 
	\\
	Z \ar[hook]{r}{i}         & X  
	\end{tikzcd}	
\end{equation*}

Recall $B_{\ell,d}^{\preceq}$ denotes the set of Young diagrams inscribed in a rectangle of height $\ell$ and width $d$, equipped with the natural partial order of inclusions as usual, and $B_{\ell, d}^{\succeq}$ denotes the same set with opposite partial order. The following is a slight generalization of \cite{Pi20}.

% Theorem: gen. Cayley
\begin{theorem}[{Pirozhkov \cite{Pi20}}] \label{thm:gen:Cayley} 
In the above situation, assume furthermore that either $X$ is Cohen--Macaulay, or the closed immersion $\iota \colon \shZ_+ \hookrightarrow \Gr_d(\sE^\vee)$ is Koszul-regular. 
\begin{enumerate}[leftmargin=*]
	\item \label{thm:gen:Cayley-1}
	The closed immersion $j \colon P_Z \hookrightarrow \shZ_+$ is Koszul-regular, with normal bundle $\sN_j = \shQ^\vee|_{P_Z}$; The normal bundle for $\iota$ is given by $\sN_\iota = \shU^\vee|_{\shZ_+} = \sQ_d$;
	\item \label{thm:gen:Cayley-2} 
	Then the sequence $\{\Sigma^{\alpha^t} \sQ_d \}_{\alpha \in B_{n-d-1,d}^{\preceq}}$ is a relative exceptional sequence of $\shZ_+$ over $X$; 
	\item  \label{thm:gen:Cayley-3}
	For each $\alpha \in B_{n-d-1,d}$, the strong relative Fourier--Mukai functor 
	$$\Phi^{\alpha}(\blank) := \pi^*(\blank) \otimes \Sigma^{\alpha^t} \sQ^\vee \colon \Dqc(X) \to \Dqc(\shZ_{+})$$
is fully faithful; For each $\beta \in B_{n-d, d-1}$, the strong relative Fourier--Mukai functor:
	$$\Psi^{\beta}(\blank) :=  j_{*} \, p^*(\blank) \otimes \Sigma^{\beta} \shQ^\vee|_{\shZ_+} \colon \Dqc(Z) \to \Dqc(\shZ_{+})$$
is fully faithful. Furthermore, these functors induce $X$-linear semiorthogonal decompositions with admissible components:
	\begin{align*}
	\Perf(\shZ_+) = &\big \langle  \{\Im (\Psi^{\beta + 1}|_{\Perf(Z)}) \}_{\beta \in B_{\ell, d-1}^{\succeq}}, ~\{ \Im (\Phi^\alpha|_{\Perf(X)})\}_{\alpha \in B_{\ell - 1, d}^{\preceq}} \big \rangle; \\
	\Db(\shZ_+) = & \big \langle  \{\Im (\Psi^{\beta + 1}|_{\Db(Z)}) \}_{\beta \in B_{\ell, d-1}^{\succeq}}, ~\{ \Im (\Phi^\alpha|_{\Db(X)}) \}_{\alpha \in B_{\ell - 1, d}^{\preceq}} \big \rangle;  \\
	\Dqc(\shZ_+) = & \big \langle  \{\Im \Psi^{\beta + 1} \}_{\beta \in B_{\ell, d-1}^{\succeq}}, ~\{ \Im \Phi^\alpha\}_{\alpha \in B_{\ell - 1, d}^{\preceq}} \big \rangle
	\end{align*}
which are compatible with the inclusions $\Perf \subseteq \Db \subseteq \Dqc$;
	\item  \label{thm:gen:Cayley-4}
	The map $\pi \colon \shZ_+ \to X$ is a projective local complete intersection morphism, with relative dualizing complex given by 
		$$\omega = \pi^*((\det \sE)^{\otimes d}) \otimes \sO_{\shZ_+}(1-n)[d(n-1-d)].$$
	The category $\Perf(\shZ_+)$ admits a relative Serre functor over $X$ given by $\SS = \otimes \omega$.
\end{enumerate} 
\end{theorem}

\begin{proof} First notice that if $X$ is Cohen--Macaulay, then since by Cor. \ref{cor:Quot:degloci}, $\shZ_+|_{X \backslash Z} \to X\backslash Z$ is a Grassmannian bundle of relative dimension $d(n-1-d)$, and $P_Z \to Z$ is a Grassmannian bundle of relative dimension $d(n-d)$, $\shZ_+$ has expected dimension, and thus the closed immersion $\iota \colon \shZ_+ \hookrightarrow \Gr_d(\sE^\vee)$ is regular. Hence we need only consider the case when  $\iota$ is Koszul-regular; We claim that under this condition, the Tor-independent conditions of Def. \ref{def:Tor-ind:quot} are satisfied by $(d_+,d_-) = (d,0)$ and $(d_+,d_-) = (d,1)$. In fact, the section morphism $X \hookrightarrow H=|\Hom_X(\sO,\sE)|$ is a closed regular immersion by Lem. \ref{lem:Homspace}, and since $\Gr_d(\sE^\vee) \to X$ and $P_Z \to Z$ are flat, hence the claim follows easily from Lem. \ref{lem:bc:CM} and Lem. \ref{lem:3squares}. Thus \eqref{thm:gen:Cayley-1} follows from local computations by Tor-independent base change $X \hookrightarrow H=|\Hom_X(\sO,\sE)|$; \eqref{thm:gen:Cayley-2} and \eqref{thm:gen:Cayley-3} follow from the local case Thm. \ref{thm:local:gen:Cayley} via the general procedure \S \ref{subsec:univHom} as before, and \eqref{thm:gen:Cayley-4} follows from Lem. \ref{lem:global:Serre} and Prop. \ref{prop:Serre}.
\end{proof}

\begin{remark} In fact, we expect that the Cohen--Macaulay condition of Thm. \ref{thm:gen:Cayley} could be dropped, and analogous statements as Lem. \ref {lem:cayley:univ} hold. However since the proof we have in mind would involve some long computations over open coverings of Grassmannians, we decide not to pursue this direction further in this already long paper.
\end{remark}

% sec: Case d=2
\subsection{$\Quot_2$-formula} Let $X$ be a connected quasi-compact, quasi-separated scheme over a field $\kk$ of characteristic zero, let $\sG$ be a finite type quasi-coherent sheaf of homological dimension $\le 1$, $\rank \sG = \delta$, and denote $\sK : = \sExt^1(\sG, \sO_X)$ as usual. Let 
	$$\pi \colon \shZ_+ : = \Quot_{X,2}(\sG) \to X$$
 be the natural projection as usual, and denote by $\sQ_+$ the universal rank $2$ quotient bundle of $\shZ_+$, $\sO_+(1) = \bigwedge^2 \sQ_+$ the relative ample line bundle of from Quot construction Thm. \ref{thm:Quot}. Moreover, consider the following schemes:
	\begin{align*}
	\pi_-^{(1)} \colon \shZ_-^{(1)} := \PP_{X}(\sK) \to X, \qquad  \pi_-^{(2)} \colon \shZ_-^{(2)} := \Quot_{X, 2}(\sK) \to X.
	\end{align*}
For $d_- \in \{1,2\}$, denote by $ \widehat{\shZ}^{(d_-)}  = \shZ_+ \times_X \shZ_-^{(d_-)}$ the fiber product, and let $r_{+}^{(d_-)} \colon  \widehat{\shZ}^{(d_-)} \to \shZ_+$ and $ r_{-}^{(d_-)} \colon \widehat{\shZ}^{(d_-)} \to  \shZ_-^{(d_-)}$ be the natural projections as usual. 

Recall $B_{\ell,d}$ denotes the set of Young diagrams inscribed in a rectangle of height $\ell$ and width $d$. The following is a global version of Thm. \ref{thm:local:d=2}.

% Thm: local d=2
\begin{theorem}[$\Quot_2$-formula] \label{thm:d=2} In the above situation, assume that the Tor independent condition Def. \ref{def:Tor-ind:quot} holds for the pairs $(2, d_-)$ where $d_- = 0, 1, 2$. In particular, if $X$ is Cohen--Macaulay, by Lem. \ref{lem:Tor-ind:quot:CM} this is equivalent to the following expected dimension conditions:
	\begin{align*}
	 	&\dim \shZ_+ = \dim X + 2 (\delta-2),  \\
		& \dim  \shZ_-^{(1)} = \dim X - (\delta + 1),  \qquad  \dim \widehat{\shZ}^{(1)} =\dim X - \delta - 3, \\
		& \dim  \shZ_-^{(2)} = \dim X - 2(\delta + 2),  \qquad \dim \widehat{\shZ}^{(2)} =\dim X - 4.
	\end{align*}
Then the sequence $\{\Sigma^{\alpha^t} \sQ_+ \}_{\alpha \in B_{\delta-2,2}}$ is a relative exceptional sequence of $\shZ_+$ over $X$. Moreover, for any $k \in \ZZ$ and $\alpha \in B_{\delta-2,2}$, the relative Fourier--Mukai functors over $X$:
	\begin{align*}
	& \Omega_k(\blank) : = r_{+ \,*}^{(2)} \circ r_{-}^{(2) \,*}(\blank) \otimes \sO_+(k) \colon  \qquad & \Dqc(\shZ_{-}^{(2)})  \to \Dqc(\shZ_+), \\
	&\Phi_k(\blank) := r_{+ \,*}^{(1)} \circ r_{-}^{(1)\,*}(\blank) \otimes \sO_+(k)  \colon \qquad &\Dqc(\shZ_{-}^{(1)})  \to \Dqc(\shZ_+), \\
	&\Psi_{\alpha,k}(\blank) := \pi_+^*(\blank) \otimes \Sigma^{\alpha^t} \sQ_+  \otimes \sO_+(k) \colon  \qquad  &\Dqc(X)  \to \Dqc(\shZ_{+})
	\end{align*}
are strong in the sense of Def. \ref{Def:FM:BS} and fully faithful. Furthermore, for any fixed $k \in \ZZ$, the images $\Im \Omega_{k-1}$, $\{\Im \Phi_i\}_{i \in [k, k+\delta-1]}$ and $\{\Im \Psi_{\alpha,k+1} \}_{\alpha \in B_{\delta-2,2}}$ induce a semiorthogonal decomposition with admissible components:
	\begin{align*} 
	\Dqc(\shZ_+) = \big\langle \Im \Omega_{k-1}, \{\Im \Phi_{k+i}\}_{i \in [0, \delta-1]}, \{\Im \Psi_{\alpha,k+1} \}_{\alpha \in B_{\delta-2,2}}  \big\rangle.
	\end{align*}
The semiorthogonal order of the images are given by any total order extending the following partial orthogonal order: for any $\alpha, \beta \in B_{\delta-2,2}$, $i \in [0,\delta-1]$:
	 \begin{align*}
	& \Im \Omega_{k-1} \subseteq (\Im \Phi_{k+i})^\perp \cup (\Im \Psi_{\alpha, k+1})^\perp & \forall &\quad   i, \alpha \text{~above};  \\
	& \Im \Phi_{k+i}  \subseteq  (\Im \Phi_{k+j})^\perp \cup (\Im \Psi_{\alpha, k+1})^\perp & \forall &  \quad  i <  j \le  i + \delta -1,  \alpha +1 \npreceq (2^i);    \\
	& \Im \Psi_{\alpha, k+1} \subseteq (\Im \Phi_{k+i})^\perp \cup (\Im \Psi_{\beta, k+1})^\perp & \forall &   \quad (2^i)\npreceq \alpha, \beta \npreceq \alpha, 
	\end{align*}
where $(2^i)= (2,2,\ldots,2)$ is understood as an element of 
$B_{\delta-1,2} \supset B_{\delta-2,2}$. Similar semiorthogonal decompositions hold if we replace $\Dqc$ by $\Db$ or $\Perf$, and these semiorthogonal decompositions are compatible with the inclusions $\Perf \subseteq \Db \subseteq \Dqc$. Moreover, the category $\Perf(\shZ_+)$ admits a relative Serre functor given by:
	$$\SS_+ = (\blank) \otimes \pi^* ( (\det \sG)^{\otimes 2}) \otimes \sO_+(-\delta)[2(\delta-2)].$$
In particular, for any $k \in \ZZ$, the following holds: 
	$$\SS_+ (\Im \Omega_k) = \Im \Omega_{k-\delta}, \quad \SS_+ (\Im \Phi_k) = \Im \Phi_{k-\delta}, \quad \SS_+ (\Im \Psi_{\alpha,k}) = \Im \Psi_{\alpha,k - \delta}.$$
\end{theorem}
The semiorthogonal decomposition of the theorem can be informatively written as
	\begin{align*}
	\D(\shZ_+) = \big \langle \D(\shZ_{-}^{(2)}), ~\text{$\delta$-copies of} ~ \D(\shZ_{-}^{(1)}), ~  \text{$\binom{\delta}{2}$-copies of} ~ \D(X) \big \rangle.
	\end{align*}
\begin{proof} This is a globalization of Thm. \ref{thm:local:d=2} via the general procedure \S \ref{subsec:univHom} as before; The statement about Serre functor follows from Lem. \ref{lem:global:Serre} and Prop. \ref{prop:Serre}.
 \end{proof}

% sec: flips
%\addtocontents{toc}{\medskip}	
\addtocontents{toc}{\vspace{0.5\normalbaselineskip}}	
\section{Flips, flops and virtual flips} \label{sec:flips}
In this whole section, for simplicity we assume $X$ is a connected Cohen--Macaulay scheme over a field $\kk$ of characteristic zero, unless otherwise stated. Let $\sG$ be a finite type quasi-coherent sheaf of homological dimension $\le 1$, set $\delta: = \rank \sG \ge 0$, $\sK := \sExt^1_{\sO_X}(\sG, \sO_X)$. Therefore by Lem. \ref{lem:degloci} and Lem. \ref{lem:rank}\eqref{lem:rank-2}, there is a sequence of closed subschemes:
	$$X = X^{\ge \delta}(\sG) \supsetneq X^{\ge \delta+1}(\sG) \supset X^{\ge \delta+2}(\sG) \supset \ldots.$$
Notice that for each $i \ge 0$, the expected codimension of $X^{\ge \delta+i}(\sG)$ inside $X$ is $\delta(\delta+i)$. 

\subsection{First results: Grassmannian flips and virtual flips}
\subsubsection{Grassmannian flips} \label{sec:Grass_flips} For a given integer $d \ge \delta$, consider the degeneracy locus $Y: = X^{\ge d}(\sG) \subset X$. Assume that $\mathring{Y} : = X^{\ge d}(\sG)  \backslash X^{\ge d+1}(\sG) \ne \emptyset$, and $Y$ achieves the expected codimension ${\rm codim}_X(Y) = d (d - \delta)$ (hence $Y$ is also Cohen--Macaulay). Consider: 
	$$\pi_+ \colon Y_+ := \Quot_{X, d}(\sG) \to X, \qquad \pi_- \colon Y_- := \Quot_{X, d - \delta}(\sK) \to X.$$
Then by Cor. \ref{cor:Quot:degloci}, the projections $\pi_{\pm} \colon Y_{\pm} \to X$ factorise through $Y_{\pm} \to Y \subseteq X$, and they induce isomorphisms of schemes over $ \mathring{Y}$: $\pi_+^{-1}(\mathring{Y}) \simeq \mathring{Y} \simeq \pi_-^{-1}(\mathring{Y})$. Let $\widehat{Y}: = Y_+ \times_Y Y_-$, and let $r_\pm \colon \widehat{Y} \to Y_{\pm}$ be the natural projections.  Denote by	 
	$\sO_\pm(1)$ the corresponding relative ample line bundles on $Y_\pm$ from Quot construction Thm. \ref{thm:Quot}.

% Thm: Grassmannian flips
\begin{theorem}[Grassmannian flips] \label{thm:Grassflips} In the above situation, and assume further that the following expected dimension condition holds:
	$$\dim Y_+ = \dim Y_- = \dim \widehat{Y} = \dim X - d(d - \delta).$$
Then $Y_\pm$ and $\widehat{Y}$ are also Cohen--Macaulay, and $Y_+$, $Y_-$ are two different (partial) desingularisations of $Y$. (The birational map $Y_+ \dashrightarrow Y_-$ is usually called {\em a Grassmannian flip}.) Moreover, the relative Fourier--Mukai functors over $X$:
		\begin{align*}
		\Phi: = r_{+\,*} \, r_{-}^*(\blank) \colon  & \Dqc(Y_-) \to \Dqc(Y_+)
		\end{align*}
	are strong in the sense of Def. \ref{Def:FM:BS} and fully faithful. The restrictions of $\Phi$ induce fully faithful embeddings $\Phi|_{\Db} \colon \Db(Y_-) \hookrightarrow \Db(Y_+)$ and $\Phi|_{\Perf} \colon \Perf(Y_-) \hookrightarrow \Perf(Y_+)$. The essential images of $\Phi$, $\Phi|_{\Db}$ and $\Phi|_{\Perf}$ are admissible, and they are compatible with the inclusions $\Perf \subseteq \Db \subseteq \Dqc$. Furthermore, the maps $\pi_\pm$ are projective and local complete intersection morphisms, hence proper and perfect, and the relative dualizing complexes are given by 
		\begin{align*}
		\omega_{\pi_+} = \pi_+^* ( (\det \sG)^{\otimes d}) \otimes \sO_+(-\delta)[d(\delta-d)], \quad  \omega_{\pi_-} =  \pi_-^*( (\det \sG)^{\otimes (d-\delta)} ) \otimes \sO_-(\delta)[d(\delta-d)].
		\end{align*}
Hence $\Perf(Y_\pm)$ admit a relative Serre functor over $X$ given by $\SS_\pm = \otimes \omega_{\pm}$.
\end{theorem}

 \begin{proof} This is a globalization of Lem. \ref{lem:local:ff} via the general procedure \S \ref{subsec:univHom} as before; The statement about Serre functor follows from Lem. \ref{lem:global:Serre} and Prop. \ref{prop:Serre}.
 \end{proof}

Similar results of embedding of derived categories for Grassmannian flips are also established in \cite{BLV2, BLV3, BCF+}.

% sec: virtual flips
\subsubsection{Virtual flips}  For a given integer $d \ge 0$,  and consider: 
	$$\pi_+ \colon \shZ_+ := \Quot_{X, d}(\sG) \to X, \qquad \pi_- \colon \shZ_- := \Quot_{X, d}(\sK) \to X.$$
Then if $d < \delta$, $\pi_+ \colon \shZ_+ \to X$ is surjective, and $\pi_+$ is a $\Gr_d(\delta)$-bundle over the open stratum $X \backslash X^{\ge \delta+1}(\sG)$; if $d \ge \delta$, then $\pi_+$ factorises through $X^{\ge d}(\sG) \subset X$ and induces a isomorphism of schemes over the locus $X^{\ge d}(\sG)\backslash X^{\ge d+1}(\sG)$. On the other and, $\pi_- \colon \shZ_- \to X$ always factorises through $X^{\ge d+\delta}(\sG) \subset X$, and induces a isomorphism of schemes over the locus $X^{\ge d+\delta}(\sG)\backslash X^{\ge d+ \delta +1}(\sG)$. Set $\widehat{\shZ}: = \shZ_+ \times_X \shZ_-$, and let $r_\pm \colon \widehat{\shZ} \to \shZ_{\pm}$ be the projections. Denote by	 $\sO_\pm(1)$ the corresponding line bundles on $\shZ_\pm$ from Quot construction Thm. \ref{thm:Quot}.

% Thm: virtual flips
\begin{theorem}[Virtual flips] \label{thm:virtualflips} In the above situation, and assume that the following expected dimension condition holds:
	\begin{align*}
	 	\dim \shZ_+  = \dim X + d(\delta-d), \quad \dim \shZ_-  = \dim X - d(\delta+d), \quad
		\dim \widehat{\shZ} = \dim X - d^2.
	\end{align*}
Then $\shZ_\pm$ and $\widehat{\shZ}$ are also Cohen--Macaulay, and the relative Fourier--Mukai functors over $X$:
		\begin{align*}
		\Phi: = r_{+\,*} \, r_{-}^*(\blank) \colon  & \Dqc(\shZ_-) \to \Dqc(\shZ_+)
		\end{align*}
	are strong in the sense of Def. \ref{Def:FM:BS} and fully faithful. The restrictions of $\Phi$ induce fully faithful embeddings $\Phi|_{\Db} \colon \Db(\shZ_-) \hookrightarrow \Db(\shZ_+)$ and $\Phi|_{\Perf} \colon \Perf(\shZ_-) \hookrightarrow \Perf(\shZ_+)$. The essential images of $\Phi$, $\Phi|_{\Db}$ and $\Phi|_{\Perf}$ are admissible and compatible with the inclusions $\Perf \subseteq \Db \subseteq \Dqc$. Furthermore, the maps $\pi_\pm$ are projective and local complete intersection morphisms, with relative dualizing complexes are given by 
		\begin{align*}
		\omega_{\pi_+} = \pi_+^* ( (\det \sG)^{\otimes d}) \otimes \sO_+(-\delta)[d(\delta-d)], \quad  \omega_{\pi_-} =  \pi_-^*( (\det \sG)^{\otimes d} ) \otimes \sO_-(\delta)[-d(\delta+d)].
		\end{align*}
Hence $\Perf(\shZ_\pm)$ admits a relative Serre functor over $X$ given by $\SS_\pm = \otimes \omega_{\pm}$.
\end{theorem}

 \begin{proof} This is a globalization of Lem. \ref{lem:local:ff} via the general procedure \S \ref{subsec:univHom} as before; The statement about Serre functor follows from Lem. \ref{lem:global:Serre} and Prop. \ref{prop:Serre}.
 \end{proof}

The relationships between $\shZ_+$  and $\shZ_-$ should be regarded as a ``virtual flip $\shZ_+ \dashrightarrow \shZ_-$" and we expect it to be closely related to Toda's ``d-critical flip" \cite{Tod1, Tod2, Tod3, Tod4}.

% Sec: standard flips
\subsection{Standard flips revisited} \label{sec:standardflip} Let $\sW$ and $\sV$ be vector bundles over $X$  of rank $m$ and $n$, and $\sigma \colon \sW \to \sV$ an $\sO_X$-module map. Denote $Z \subseteq X$ the degeneracy locus where the map $\sigma$ is zero and $Y \subseteq X$ be the degeneracy locus where the map $\sigma$ is has rank $\le 1$, i.e. $Z$ is the zero scheme of a section of the vector bundle $\sW^\vee \otimes \sV$, and $Y$ is defined by the quasi-coherent ideal sheaf $I_2(\sigma)$ of $2 \times 2$-minors of $\sigma$. Consider the following Quot schemes: 
	$$\pi_+ \colon Y_+ := \Quot_{X, n-1}(\sG)  \to X, \quad \text{and} \quad \pi_- \colon Y_- := \Quot_{X, m-1}(\sK) \to X,$$
where $\sG = \Coker (\sigma)$ and $\sK = \Coker(\sigma^\vee)$. Then by Cor. \ref{cor:Quot:degloci}, natural projections $\pi_{\pm} \colon Y_{\pm} \to X$ factorises through $Y_{\pm} \to Y \subseteq X$. 

\begin{remark}[Concrete descriptions of $Y_\pm$] \label{rmk:standardflip} Recall for a vector bundle $\sV$, the scheme $\PP_{\rm sub}(\sV) = \PP(\sV^\vee)$ parametrises subbundles of $\sV$ of rank $1$. Then over any closed point $x \in X$, the scheme $Y_+$ parametrises the pair $(L \in \PP_{\rm sub}(\sV|_x), \varphi_+ \colon \sW|_x \to L)$ such that $\sigma|_x$ factorises as the composition $\sW|_x \xrightarrow{\varphi_+} L \subseteq \sV|_x$. Dually,  over a closed point $x \in X$, the scheme $Y_-$ parametrises the pair $(Q \in \PP(\sW|_x), \varphi_- \colon Q \to \sV|_x)$ such that $\sigma|_x$ factorises as the composition $\sW|_x \twoheadrightarrow Q \xrightarrow{\varphi_-} \sV|_x$.
\end{remark}

Then the restriction of $\pi_+ \colon Y_+ \to X$ (resp. $\pi_- \colon Y_- \to X$) to $P_+\subseteq Y_+$ (resp. $P_- \subseteq Y_-$) is the projective bundle $pr_+ \colon P_+= \PP_{Z, \mathrm{sub}}(\sV|_Z) = \PP_{Z}(\sV|_Z^\vee) \to Z$ (resp. $pr_- \colon P_-= \PP_{Z}(\sW|_Z) \to Z$); Denote by $\sO_{P_\pm}(1)$ the corresponding $\sO(1)$-line bundles. Let $\sO_{Y_\pm}(1)$ be the restrictions of line bundles $\sO_{\PP(\sV^\vee)}(1)$ and $\sO_{\PP(\sW)}(1)$ to $Y_\pm$, then $\sO_{Y_\pm}(1)|_{P_\pm} = \sO_{P_\pm}(1)$. Notice similar to the proof of Lem. \ref{lem:blowup:univ},  $\sO_{Y_\pm}(1)$ are related to the $\sO(1)$'s of Quot schemes by:
	$$\sO_{Y_+}(1)  \simeq \sO_{\Quot_{n-1}(\sG)}(1) \otimes (\det \sV)^{-1}, \qquad \sO_{Y_-}(1)  \simeq \sO_{\Quot_{m-1}(\sK)}(1) \otimes \det \sW.$$
Denote by $j_{\pm} \colon P_{\pm} \hookrightarrow Y_{\pm}$ the natural inclusions. Let $\widehat{Y} = Y_+ \times_Y Y_-$ be the fiber product, and denote by $r_\pm \colon \widehat{Y} = Y_+ \times_Y Y_- \to Y_{\pm}$ the natural projections. 

The following generalises Bondal--Orlov's results on standard flips \cite{BO} (see also \cite[\S 3]{Tod2} and \cite[Thm. A]{BLT}) in our setup:

% Theorem: Standard flips
\begin{theorem} \label{thm:standardflip} In the above situation % (and assume furthermore that $X$ is Cohen--Macaulay over a field of characteristic zero), 
and assume further that $Z$ is nonempty, and that the following expected dimension condition holds:
	\begin{align*}
	\dim Z = \dim X - mn, \qquad \dim Y = \dim X - (m-1)(n-1).
	\end{align*}
\begin{enumerate}[leftmargin=*]
	\item \label{thm:standardflip-1}
	The natural projections $Y_\pm \to Y$ induce isomorphism $Y_+ \backslash P_+ \simeq Y\backslash Z \simeq Y_- \backslash P_-$. The inclusions $j_{\pm} \colon P_{\pm} \hookrightarrow Y_{\pm}$ are regular closed immersions, with normal bundles given by 
		$$\sN_{j_+} = \sO_{P_+}(-1) \otimes \sW^\vee|_Z \quad \text{and} \quad \sN_{j_-} = \sO_{P_-}(-1) \otimes \sV|_Z.$$
	The fiber product $\widehat{Y} : = Y_+ \times_Y Y_-$ is the common blowup $\widehat{Y} = \Bl_{P_+} Y_+ = \Bl_{P_-} Y_-$, $E = P_+ \times_Z P_- \subseteq  \widehat{Y}$ is the common exceptional divisor for both blowups, and the following holds: $\sO_{\hat{Y}}(E)|_E \simeq \sO_{P_+}(-1) \boxtimes \sO_{P_-}(-1)$. If $m \le n$, the birational map $Y_+ \dashrightarrow Y_-$ is called {\em a standard flip of type $(m,n)$ over $Z$}.

	\item \label{thm:standardflip-2}
	The projections $\pi_\pm \colon Y_\pm \to X$ are projective local complete intersection morphisms, with the relative dualizing complexes given by 
		\begin{align*}
		& \omega_+ = \pi_+^*((\det \sV)^{\otimes(m-1)} \otimes (\det \sW)^{\otimes(1-n)}) \otimes \sO_{Y_+}(m-n)[-(m-1)(n-1)],  \\
		&  \omega_- = \pi_-^* ((\det \sV)^{\otimes(m-1)} \otimes (\det \sW)^{\otimes(1-n)} ) \otimes \sO_{Y_-}(n-m)[-(m-1)(n-1)].
		\end{align*}
The categories $\Perf(Y_\pm)$ admit relative Serre functors over $X$ given by $\SS_\pm = \otimes \omega_\pm$. In particualr, if $m \le n$, then for each $k \in \ZZ$, $\SS_+ (\Im \Psi_{k}) = \Im \Psi_{k + m-n}$. 

	 \item \label{thm:standardflip-3}
	 If $m=n$, then there is an equivalence of categories $\Phi = r_{+\,*} \, r_{-}^*(\blank) \colon   \Dqc(Y_-) \simeq \Dqc(Y_+)$; If $m < n$, then  for any $k \in \ZZ$, the relative Fourier--Mukai functors over $X$:
		\begin{align*}
		&\Phi: = r_{+\,*} \, r_{-}^*(\blank) \colon  & \Dqc(Y_-) \to \Dqc(Y_+); \\
		& \Psi_{k}(\blank) :=  j_{+\,*} pr_{+}^* (\blank) \otimes \sO_{Y_+}(k) \colon &\Dqc(Z) \to \Dqc(Y_+)
		\end{align*}
	are strong in the sense of Def. \ref{Def:FM:BS} and fully faithful, and for each integer $0 \le \ell \le n-m$, there is an $X$-linear semiorthogonal decomposition with admissible components:
	\begin{align*}
		\Dqc(Y_+) = \langle \underbrace{\Im \Psi_{-(n-m-\ell)}, \ldots, \Im \Psi_{-1}}_{(n-m-\ell)\text{-terms}},  \Im \Phi,  \underbrace{\Im \Psi_0, \ldots, \Im \Psi_{\ell-1}}_{\ell\text{-terms}} \rangle.
	\end{align*}
Similar semiorthogonal decompositions hold if we replace $\Dqc$ by $\Db$ and $\Perf$, and these semiorthogonal decompositions are compatible with the inclusions $\Perf \subseteq \Db \subseteq \Dqc$. 
\end{enumerate}
\end{theorem}

\begin{proof} The isomorphisms $Y_+ \backslash P_+ \simeq Y\backslash Z \simeq Y_- \backslash P_-$ follow from Cor. \ref{cor:Quot:degloci}. We now verify the Tor-independent conditions of Lem. \ref{lem:Tor-ind:quot:CM}. In fact, by Cor. \ref{cor:Quot:degloci}, $p_{\pm} \colon P_{\pm} \to Z$ are projective bundles of relative dimension $n-1$ and $m-1$ respectively, therefore $\dim P_{+} = \dim X - mn + n-1$, $\dim P_{-} = \dim X - mn + m-1$ and $\dim E = \dim X - mn + m + n - 2$. And hence $\dim Y_{\pm} = \dim \widehat{Y} = \dim X - (m-1)(n-1)$. In particular, the conditions of Lem. \ref{lem:Tor-ind:quot:CM} are satisfied for $(d_+,d_-) = (n-1,m-1)$ and $(d_+,d_-) = (n-1,m)$, where $\delta = n - m \ge 0$.

For \eqref{thm:standardflip-1}, all the claims follow from the local case Rmk. \ref{rmk:standardflip:local:blowup}, since by passing to Zariski local open subsets, the situation is a Tor-independent base-change from the universal local situation via the process of \S \ref{sec:univHom}, and the formation of blowup along Koszul-regularly immersed centers commutes with Tor-independent base-changes by Lem. \ref{lem:blowup_bc}.

 \eqref{thm:standardflip-2} follows from Lem. \ref{lem:global:Serre} and Prop. \ref{prop:Serre} as before.

Finally, as before, \eqref{thm:standardflip-3} follows from the local theorem Thm. \ref{thm:local:standardflip} (the case $\ell=0$), relative Serre duality Lem. \ref{lem:Serre} (hence true for all $\ell$), via the process of \S \ref{subsec:univHom} by fppf descent Thm. \ref{thm:fppf} and Tor-independent base-change Thm. \ref{thm:bc}.
\end{proof}

\begin{remark}[Alternative setup for standard flips] Alternatively, let $Y_{\pm}$ and $Z$ be quasi-compact, quasi-separated schemes, let $\sW$, $\sV$ be vector bundles over $Z$ of ranks $m, n$ respectively, and denote $pr_+ : P_+ = \PP(\sV^\vee) \to Z$ and $pr_{-} \colon P_- = \PP(\sW) \to Z$ the projective bundles. Assume that there are Koszul-regular closed immersions $j_{\pm} \colon P_{\pm} \hookrightarrow Y_{\pm}$ such that their normal bundles satisfy $\sN_{j_+} \simeq \sO_{P_+}(-1) \otimes \sW^\vee$ and $\sN_{j_-} \simeq \sO_{P_-}(-1) \otimes \sV$. Furthermore, assume that $\widehat{Y} =\Bl_{P_+} Y_+ = \Bl_{P_-} Y_-$ is the common blowup of $Y_+$ along $P_+$, resp., $Y_-$ along $P_-$, and $E = P_+ \times_Z P_- \subseteq  \widehat{Y}$ is the common exceptional divisor. Then $\sO_{\hat{Y}}(E)|_E \simeq \sO_{P_+}(-1) \boxtimes \sO_{P_-}(-1)$ holds. The birational map $Y_+ \dashrightarrow Y_-$ is also called {\em a standard flip of type $(m,n)$ over $Z$}; see \cite{ADM, Huy, BLT}. In this setup, the blowup formula Thm. \ref{thm:blow-up} implies that for each $0 \le k \le m-1$, $0 \le \ell \le n-1$, there are semiorthogonal decompositions:
	\begin{align*}
		&\Dqc(\widehat{Y}) =  \langle  \underbrace{\Dqc(P_+)_{1-m+k}, \ldots, \Dqc(P_+)_{-1}}_{(m-1-k)\text{-terms}}, ~\Dqc(Y_+),~  \underbrace{\Dqc(P_+)_0, \ldots, \Dqc(P_+)_{k-1}}_{k\text{-terms}}\rangle; \\
		&\Dqc(\widehat{Y}) =  \langle  \underbrace{\Dqc(P_-)_{1-n+\ell}, \ldots, \Dqc(P_-)_{-1}}_{(n-1-\ell)\text{-terms}}, ~\Dqc(Y_-),~  \underbrace{\Dqc(P_-)_0, \ldots, \Dqc(P_-)_{\ell-1}}_{\ell\text{-terms}}\rangle,
	\end{align*}
where for $i,j \in \ZZ$, $\Dqc(P_+)_{i}$ (resp. $\Dqc(P_-)_{j}$) denotes the image of $\Dqc(P_+) \boxtimes_Z \sO_{P_-}(i)$ (resp. $ \sO_{P_+}(j) \boxtimes_Z \Dqc(P_-)$) under the pushforward $i_{E\,*} \colon \Dqc(E) \to \Dqc(\widehat{Y})$. Then the above situation satisfies the so-called ``axioms of chess game" of \cite{RT, JLX17, JL18}, and the techniques of {\em loc. cit.} could be applied to obtain a semiorthogonal decomposition:
	\begin{align*}
	\Dqc(Y_+)  =  \big \langle \Im \Psi_{-\delta}, \ldots, \Im \Psi_{-2}, \Im \Psi_{-1}, ~ \Im \Phi \big \rangle;
	\end{align*} 
similarly for $\Db$ and $\Perf$. In fact, the above semiorthogonal decomposition (for $\Db$) in the smooth case was established by \cite{BLT} using the ``chess game" method developed in \cite{RT, JLX17, JL18}, and the same argument of {\em loc. cit.} works in our stated generality. Moreover, in the case $m=n$, the methods of \cite{JL18} can be applied to show ``flop--flop=twist" results of \cite{ADM} for the flop $Y_+ \dashrightarrow Y_-$.
\end{remark}

\subsection{Flips from partial desingularizations of rank $\le 2$ degeneracy loci} \label{sec:l=2:flip} Let $\sW$ and $\sV$ be two vector bundles over $X$ of rank $m$ and $n$, such that $n \ge m \ge 2$, and let $\sigma \colon \sW \to \sV$ be an $\sO_X$-module map. Consider the following sequences of degeneracy loci \S \ref{sec:deg} of $\sigma$:
	$$Z: = D_0(\sigma) \quad \subseteq \quad Y_1 : = D_1(\sigma) \quad \subseteq \quad Y_2 : = D_2(\sigma)  \quad \subseteq X.$$
Assume $Y_2 \backslash Y_1 \ne \emptyset$, $Y_1 \backslash Z \ne \emptyset$, and the following expected dimension conditions hold:
	\begin{align*}
		&\dim Z = \dim X - mn, \qquad \dim Y_i = \dim X - (m-i)(n-i)~ \text{for}~ i=1,2.
	\end{align*}
 Denote $\sG = \Coker (\sigma)$ and $\sK = \Coker(\sigma^\vee)$, and consider the following Quot schemes: 
	\begin{align*}
		& \pi_2^+ \colon Y_2^+ : = \Quot_{X,n-2}(\sG) \to X, \qquad
		\pi_2^- \colon Y_2^- : = \Quot_{X,m-2}(\sK) \to X, \\
		&  \pi_1^+ \colon Y_1^+ : = \Quot_{X,n-1}(\sG) \to X, \qquad 
		\pi_1^- \colon Y_1^- : = \Quot_{X,m-1}(\sK) \to X.
	\end{align*}
By Cor. \ref{cor:Quot:degloci} the projections $\pi_i^\pm$ factorise through birational morphisms $Y_i^\pm \to Y_i$ for $i=1,2$. The birational map $Y_1^+ \dashrightarrow Y_1^-$ is the {\em standard flip} considered in \S \ref{sec:standardflip}. The schemes $Y_2^\pm$ have similar concrete descriptions as Rmk \ref{sec:standardflip}.

Consider the Grassmannian bundles $\Gr_{n-2}(\sV^\vee) = \Gr_{2}(\sV)$ and $\Gr_{m-2}(\sW) = \Gr_{2}(\sW^\vee)$. Denote $\shU_\pm$ and $\shQ_\pm$ the corresponding universal bundles, where $\rank \shU_+ = n-2$, $\rank \shU_- = m-2$, $\rank \shQ_+ = \rank \shQ_- = 2$. Then there are tautological sequences:
	$$0 \to \shU_+ \to \sV^\vee \to \shQ_+ \to 0, \qquad 0 \to \shU_- \to \sW \to \shQ_- \to 0.$$
By Thm. \ref{thm:Quot} \eqref{thm:Quot-2} there are canonical immersions $Y_2^+ \hookrightarrow \Gr_2(\sV)$ and $Y_2^- \hookrightarrow \Gr_2(\sW^\vee)$. Set 
	$$\sO_{+, \Gr_2(\sV)} (1): = \bigwedge^2 \shQ_+|_{Y_2^+} \in \Pic (Y_2^+), \qquad \sO_{-, \Gr_2(\sW^\vee)} (1): = \bigwedge^2 \shQ_-|_{Y_2^-} \in \Pic (Y_2^-).$$
As before, these line bundles are related to the $\sO(1)$'s of Quot schemes Thm. \ref{thm:Quot} \eqref{thm:Quot-3} by:
	$$\sO_{+, \Gr_2(\sV)} (1)  \simeq \sO_{\Quot_{n-2}(\sG)}(1) \otimes (\det \sV)^{-1}, \qquad \sO_{-, \Gr_2(\sW^\vee)}(1)  \simeq \sO_{\Quot_{m-2}(\sK)}(1) \otimes \det \sW.$$
The restrictions of $\pi_2^\pm \colon Y_2^\pm \to X$ to $Z \subset X$ are Grassmannian bundles $pr_+ \colon  \Gr_{Z,2}(\sV) \to Z$ and $pr_- \colon  \Gr_{Z,2}(\sW^\vee) \to Z$, and we denote the natural inclusions by $j_+ \colon \Gr_{Z,2}(\sV) \hookrightarrow Y_2^+$ and $j_- \colon \Gr_{Z,2}(\sW^\vee) \hookrightarrow Y_2^-$. For $i=1,2$, denote $\widehat{Y}_i = Y_2^+ \times_Y Y_i^-$ fiber product, and denote $r_+ \colon \widehat{Y}_i \to Y_2^+$ and $r_- \colon \widehat{Y}_i \to Y_i^+$ the natural projections. 
 
If $n = m$, then by \S \ref{sec:rk<=3} there is an equivalence of categories $\Omega_0 = r_{2, +\,*} \, r_{2, -}^* \colon \Dqc(Y_2^-) \simeq \Dqc(Y_2^+)$; 
If $n- m =1$, then there is a structural description of the derived category of $Y_2^+$ in terms of these of $Y_2^-$ and $Y_1^-$ given by Thm. \ref{thm:rk=1}. Hence we may assume $\delta = n-m \ge 2$.

% Theorem: rk<=2
\begin{theorem} \label{thm:rk<=2flip} In the above situation, for any $k \in \ZZ$ and $\alpha \in B_{2, \delta-2}$ (where $\delta = n-m$),
 	\begin{align*}
	&\Psi^{\alpha}_k(\blank) :=  j_{+\,*} \circ pr_{+}^* (\blank) \otimes \Sigma^{\alpha} \shQ_+^{\vee}|_{Y_2^+}  \otimes \sO_{+,\Gr_2(\sV)}(k) \colon  \qquad  & \Dqc(Z) \to \Dqc(Y_2^{+}), \\
	&\Phi_k(\blank) := r_{1, + \,*} \circ r_{1,-}^{*}(\blank) \otimes \sO_{+,\Gr_2(\sV)}(k) \colon \qquad &\Dqc(Y_1^-) \to \Dqc(Y_2^+), \\
	& \Omega_k(\blank) : = r_{2, + \,*} \circ r_{2, -}^{\,*}(\blank) \otimes \sO_{+,\Gr_2(\sV)}(k) \colon  \qquad & \Dqc(Y_2^-) \to  \Dqc(Y_2^+)	\end{align*}
are fully faithful strong relative Fourier--Mukai functors over $X$ (in the sense of Def. \ref{Def:FM:BS}).  Furthermore, for any fixed $k \in \ZZ$, the images $\{\Im \Psi^{\alpha}_{k-1} \}_{\alpha \in B_{2, \delta-2}}$, $\{\Im \Phi_{k - i} \}_{i \in [0, \delta-1]}$ and  $\Im \Omega_{k+1}$ induce a semiorthogonal decomposition
	\begin{align*} %\label{eqn:l=2:sod}
	\Dqc(Y_2^+) = \big\langle \{\Im \Psi^{\alpha}_{k-1} \}_{\alpha \in B_{2, \delta-2}}, \{\Im \Phi_{k - i} \}_{i \in [0, \delta-1]},  \Im \Omega_{k+1} \big\rangle, 
	\end{align*}
 with semiorthogonal order given by any total order extending the following partial semiorthogonal order: for any $\alpha, \beta \in B_{2, \delta-2}$, $i \in [0,\delta-1]$, the following holds:
	 \begin{align*}
	& \Im \Omega_{k+1} \subseteq {}^\perp (\Im \Phi_{k-i}) \cup {}^\perp(\Im \Psi^{\alpha}_{k-1}) & \forall &\quad   i, \alpha \text{~above};  \\  %\label{eqn:l=2:order1} \\
	& \Im \Phi_{k-i}  \subseteq  {}^\perp(\Im \Phi_{k-j}) \cup {}^\perp(\Im \Psi^{\alpha}_{k-1}) & \forall &  \quad  i <  j \le  i + \delta -1,  \alpha +1 \npreceq (i^2);   \\%\label{eqn:l=2:order2} \\
	& \Im \Psi^{\alpha}_{k-1} \subseteq {}^\perp(\Im \Phi_{k-i}) \cup {}^\perp(\Im \Psi^{\beta}_{k-1}) & \forall &   \quad (i^2) \npreceq \alpha, \beta \npreceq \alpha.  %\label{eqn:l=2:order3}
	\end{align*}
where $(i^2)= (i,i)$ is understood as an element of $B_{2, \delta-1} \supset B_{2, \delta-2}$. Similar semiorthogonal decompositions hold if we replace $\Dqc$ by $\Db$ or $\Perf$, and these semiorthogonal decompositions are compatible with the inclusions $\Perf \subseteq \Db \subseteq \Dqc$. Moreover, the schemes $Y_2^\pm$ are also Cohen--Macaulay, the morphisms $\pi_2^\pm$ are projective local complete intersection morphisms, with dualizing complexes given by:
	\begin{align*}
		& \omega_+ = \pi_2^{+\,*}((\det \sV)^{\otimes(m-2)} \otimes (\det \sW)^{\otimes(2-n)}) \otimes \sO_{+, \Gr_2(\sV)}(m-n)[-(m-2)(n-2)],  \\
		&  \omega_- = \pi_2^{-\,*} ((\det \sV)^{\otimes(m-2)} \otimes (\det \sW)^{\otimes(2-n)} ) \otimes \sO_{-, \Gr_2(\sW^\vee)}(n-m)[-(m-2)(n-2)].
	\end{align*}
Thus the categories $\Perf(Y_2^\pm)$ admit relative Serre functors over $X$ given by $\SS_\pm = \otimes \omega_\pm$. In particular, for any $k \in \ZZ$, the following holds: 
	$$\SS_+ (\Im \Omega_k) = \Im \Omega_{k-\delta}, \quad \SS_+ (\Im \Phi_k) = \Im \Phi_{k-\delta}, \quad \SS_+ (\Im \Psi_{\alpha,k}) = \Im \Psi^{\alpha}_{k - \delta}.$$
\end{theorem}

As usual, we could regard the semiorthogonal decomposition of the theorem as:
	\begin{align*}
	\D(Y_2^+) = \big \langle  \text{$\binom{\delta}{2}$-copies of} ~ \D(Z),  ~\text{$\delta$-copies of} ~ \D(Y_1^-), ~\D(Y_2^-) \big \rangle.
	\end{align*}

\begin{proof} This is a globalization of Thm. \ref{thm:local:l=2} via the general procedure \S \ref{subsec:univHom} as before; The statement about Serre functor follows from Lem. \ref{lem:global:Serre} and Prop. \ref{prop:Serre}.
 \end{proof}

% Sec: the cases \delta \le 3
\subsection{The cases $\rank \sG \le 3$, and blowups of determinantal ideals of height $\le 4$} \label{sec:rk<=3} 
% sec: case \delta=0: flops
\subsubsection{The case $\rank \sG =0$: Grassmannian flops} If $\rank \sG = 0$, then this is a special case of Thm. \ref{thm:Grassflips}: $Y_\pm$ are two crepant desingularizations of the scheme $Y$ (if we assume $Y$ is Gorenstein and $\QQ$-factorial; this is the case, for example if $X$ is smooth) and the birational map $Y_+ \dashrightarrow Y_-$ is called a {\em a Grassmannian flop}.) Then Thm. \ref{thm:Grassflips} implies that 
	\begin{align*}
		\Phi: = r_{+\,*} \, r_{-}^*(\blank) \colon  & \Dqc(Y_-) \to \Dqc(Y_+)
	\end{align*}
is an {\em equivalence} of categories, and the restrictions of $\Phi$ induce compatible equivalences $\Phi|_{\Db} \colon \Db(Y_-) \simeq \Db(Y_+)$ and $\Phi|_{\Perf} \colon \Perf(Y_-) \simeq \Perf(Y_+)$. The derived equivalence for Grassmannian flops is also studied in \cite{BLV3, DS,  BCF+}.

% Sec: case \delta=1
\subsubsection{The case $\rank \sG = 1$} If $\rank \sG = 1$, then for any $d \ge 1$, consider the following Quot schemes: 
	\begin{align*}
	\shZ_+: = \Quot_{X, d}(\sG), \quad  \shZ_-^{\mathrm{flip}} := \Quot_{X, d-1}(\sK), \quad \shZ_-^{\mathrm{vf}} := \Quot_{X, d}(\sK).
	\end{align*}
We denote $\widehat{\shZ}^{\heartsuit} =  \shZ_+ \times_X \shZ_-^{\heartsuit}$ the fiber products, and $r_\pm^{\heartsuit}$ the natural projections, where $\heartsuit \in \{ \mathrm{flip}, \mathrm{mid}, \mathrm{vf} \}$. 
The relationship ``$\shZ_+ \dashrightarrow \shZ_-^{\rm flip}$" is a Grassmannian flip of Thm. \ref{thm:Grassflips}, and  the relationship ``$\shZ_+ \dashrightarrow \shZ_-^{\rm vf}$" is a virtual flip of Thm. \ref{thm:virtualflips}. Denote by	 
	$\sO_+(1) : = \sO_{\Quot_{d,X}(\sG)}(1)$
the line bundle from Quot construction Thm. \ref{thm:Quot} \eqref{thm:Quot-3} as usual.

% Thm: \rank G = 1
\begin{theorem}[$\rank \sG=1$] \label{thm:rk=1} In the above situation, assume the following holds:
	\begin{align*}
	 	& \dim \shZ_+  = \dim \shZ_-^{\mathrm{flip}}  = \dim \widehat{\shZ}^{\mathrm{flip}} = \dim X - d(d-1), \\
		& \dim \shZ_-^{\mathrm{vf}}  = \dim X - d(1+d),  \qquad \dim \widehat{\shZ}^{\mathrm{vf}}  = \dim X - d^2.
	\end{align*}
Then $\shZ_+$, $\shZ_-^{\mathrm{flip}}$, $\shZ_-^{\mathrm{vf}}$, $ \widehat{\shZ}^{\mathrm{flip}}$ and $\widehat{\shZ}^{\mathrm{vf}}$ are also Cohen--Macaulay schemes. Furthermore, the following relative Fourier--Mukai functors over $X$:
	\begin{align*}
		\Phi^{\mathrm{flip}}: =  r_{+\,*}^{\mathrm{flip}} \circ r_{-}^{{\mathrm{flip}}\,*} \colon \Dqc(\shZ_{-}^{\mathrm{flip}}) \to \Dqc(\shZ_+), \qquad
		 \Phi^{\mathrm{vf}} : =  r_{+\,*}^{\mathrm{vf}} \circ r_{-}^{{\mathrm{vf}}\,*} \colon \Dqc(\shZ_{-}^{\mathrm{vf}}) \to \Dqc(\shZ_+).
		\end{align*}	
	are strong in the sense of Def. \ref{Def:FM:BS} and fully faithful, and their essential images induce semiorthogonal decompositions with admissible components:
	\begin{align*}
	\Dqc(\shZ_+) = \langle \Dqc(\shZ_{-}^{\mathrm{vf}}), ~\Dqc(\shZ_{-}^{\mathrm{flip}}) \otimes \sO_+(1)\rangle =  \langle \Dqc(\shZ_{-}^{\mathrm{flip}}) , ~\Dqc(\shZ_{-}^{\mathrm{vf}}) \rangle.
	\end{align*}
Similar semiorthogonal decompositions hold if we replace $\Dqc$ by $\Db$ and $\Perf$, and these semiorthogonal decompositions are compatible with the inclusions $\Perf \subseteq \Db \subseteq \Dqc$. The category $\Perf(\shZ_+)$ admits a relative Serre functor given by:
	$$\SS_+ = (\blank) \otimes \pi_+^* ( (\det \sG)^{\otimes d}) \otimes \sO_+(-1)[d(1-d)].$$
In particular, $\SS_+ (\Im \Phi^{\mathrm{flip}}) = \Im \Phi^{\mathrm{flip}} \otimes \sO_+(-1)$ and  $\SS_+ (\Im \Phi^{\mathrm{vf}}) = \Im \Phi^{\mathrm{vf}} \otimes \sO_+(-1)$.
\end{theorem}
\begin{proof} This is a globalization of Thm. \ref{thm:rk=1:local} via the general procedure \S \ref{subsec:univHom} as before; The statement about Serre functor follows from Lem. \ref{lem:global:Serre} and Prop. \ref{prop:Serre}.
 \end{proof}

Let $d= \delta = 1$ in above theorem, then by Lem. \ref{lem:Quot=Bldet} $\shZ_+ = \PP(\sG)$ is the blowup of $X$ along a Cohen--Macaulay subscheme $Z$ of codimension $2$, hence we obtain {\cite{JL18}}:

\begin{corollary}[{\cite{JL18}}]\label{cor:rk=1}  Let $X$ be a Cohen--Macaulay scheme, and let $\sG$ be a quasi-coherent $\sO_X$-module of homological dimension $\le 1$ and rank $1$. Denote $Z = X^{\ge 2}(\sG)$ the degeneracy locus of $\sG$. Let $\pi \colon \Bl_Z X \to Z$ be the blowup of $X$ along $Z$. Assume that:
	$$\codim_X (X^{\ge 1+i}(\sG) ) \ge 2i, \qquad \forall i \ge 1.$$
(Notice that the expected codimension is ${\rm exp.codim}_X X^{\ge 1+i}(\sG) = i(1+i) \ge 2i$ for $i \ge 1$.) Then there are $X$-linear semiorthogonal decompositions with admissible components:	
	\begin{align*}\D(\Bl_Z X) = \langle \D(\widetilde{Z}), ~\D(X) \otimes \sO_{\Bl_Z X}(1)\rangle =  \langle \D(X) , ~\D(\widetilde{Z}) \rangle,
	\end{align*}
where $\widetilde{Z} := \PP(\sK) \to Z$ is a partial desingularization of $Z$, $\D$ stands for $\Perf$, $\Db$ or $\Dqc$. Furthermore, $\Perf(\Bl_Z X)$ admits a relative Serre functor over $X$ given by $\SS = \otimes \sO_{\Bl_Z X}(-1)$.
\end{corollary}

As noted in \cite{JL18}, if $X$ is regular, then by Hilbert--Burch theorem, every Cohen--Macaulay subscheme $Z \subset X$ of codimension $2$ arises in this way.

% Sec: \delta = 2
\subsubsection{The case $\rank \sG = 2$} If $\rank \sG = 2$, then for any $d \ge 2$, consider the following Quot schemes: $\shZ_+: = \Quot_{X, d}(\sG)$ as usual, and we set:
	\begin{align*}
	\shZ_-^{\mathrm{flip}} := \Quot_{X, d-2}(\sK), \quad  \shZ_-^{\mathrm{mid}} := \Quot_{X, d-1}(\sK), \quad \shZ_-^{\mathrm{vf}} := \Quot_{X, d}(\sK).
	\end{align*}
We denote $\widehat{\shZ}^{\heartsuit} =  \shZ_+ \times_X \shZ_-^{\heartsuit}$ the fiber products, and $r_\pm^{\heartsuit}$ the natural projections, where $\heartsuit \in \{ \mathrm{flip}, \mathrm{mid}, \mathrm{vf} \}$.  
Denote by	 
	$\sO_+(1) : = \sO_{\Quot_{d,X}(\sG)}(1)$
the line bundle of Thm. \ref{thm:Quot} \eqref{thm:Quot-3}.

\begin{theorem} [$\rank \sG=2$] \label{thm:rk=2} In the above situation, assume the following holds:
	\begin{align*}
	 	&\dim \shZ_+  = \dim \shZ_-^{\mathrm{flip}}  = \dim \widehat{\shZ}^{\mathrm{flip}} = \dim X - d(d-2), \\
		&\dim  \shZ_-^{\mathrm{mid}}  = \dim X - (d-1)(d+1),  \qquad  \dim \widehat{\shZ}^{\mathrm{mid}}  = \dim X - d^2 + d + 1, \\
		&\dim \shZ_-^{\mathrm{vf}}  = \dim X - d(d+2),  \qquad \dim \widehat{\shZ}^{\mathrm{vf}}  = \dim X - d^2.
	\end{align*}
Then $\shZ_+$, $\shZ_-^{\heartsuit}$, $\widehat{\shZ}^{\heartsuit}$ are also Cohen--Macaulay schemes, where $\heartsuit  \in \{ \mathrm{flip}, \mathrm{mid}, \mathrm{vf} \}$. Furthermore, for any $i \in \ZZ$, the following relative Fourier--Mukai functors over $X$:
	\begin{align*}
		&\Phi_i^{\mathrm{flip}}(\blank): =  r_{+\,*}^{\mathrm{flip}} \circ r_{-}^{{\mathrm{flip}}\,*} (\blank) \otimes \sO_+(i) \colon & \Dqc(\shZ_{-}^{\mathrm{flip}}) \to \Dqc(\shZ_+), \\
		&\Phi_i^{\mathrm{mid}}(\blank): =  r_{+\,*}^{\mathrm{mid}} \circ r_{-}^{{\mathrm{mid}}\,*} (\blank) \otimes \sO_+(i) \colon & \Dqc(\shZ_{-}^{\mathrm{mid}}) \to \Dqc(\shZ_+), \\
		 &\Phi_i^{\mathrm{vf}} (\blank) : =  r_{+\,*}^{\mathrm{vf}} \circ r_{-}^{{\mathrm{vf}}\,*}(\blank)  \otimes \sO_+(i)  \colon & \Dqc(\shZ_{-}^{\mathrm{vf}}) \to \Dqc(\shZ_+)
	\end{align*}
are strong in the sense of Def. \ref{Def:FM:BS} and fully faithful, and their essential images induce semiorthogonal decompositions with admissible components:
		$$\Dqc(\shZ_+) = \langle \Im \Phi_{i-1}^{\mathrm{vf}}, ~ \Im \Phi_{i}^{\mathrm{mid}}, ~\Im  \Phi_{i+1}^{\mathrm{mid}}, ~ \Im \Phi_{i+2}^{\mathrm{flip}} \rangle.$$
Similar semiorthogonal decompositions hold if we replace $\Dqc$ by $\Db$ or $\Perf$, and these semiorthogonal decompositions are compatible with the inclusions $\Perf \subseteq \Db \subseteq \Dqc$. The category $\Perf(\shZ_+)$ admits a relative Serre functor given by:
	$$\SS_+ = (\blank) \otimes \pi_+^* ( (\det \sG)^{\otimes d}) \otimes \sO_+(-2)[d(2-d)].$$
In particular, for any $i \in \ZZ$ and $\heartsuit  \in \{ \mathrm{flip}, \mathrm{mid}, \mathrm{vf} \}$, $\SS_+ (\Im \Phi_i^{\heartsuit}) = \Im \Phi^{\heartsuit}_{i-2}$.
		\end{theorem}
As usual, the semiorthogonal decomposition could be informatively written as:
	\begin{align*}
	\D(\shZ_+) = \big \langle  \D(\shZ_{-}^{\rm vf}), ~\text{2-copies of} ~ \D(\shZ_{-}^{\rm mid}), ~ \D(\shZ_-^{\rm flip}) \big \rangle.
	\end{align*}

\begin{proof} This is a globalization of  Thm. \ref{thm:rk=2:local} via the general procedure \S \ref{subsec:univHom} as before; The statement about Serre functor follows from Lem. \ref{lem:global:Serre} and Prop. \ref{prop:Serre}.
 \end{proof}

Set $d= \delta = 2$ in above theorem, then by Lem. \ref{lem:Quot=Bldet}, $\shZ_+ = \Quot_2(\sG)$ is the blowup of $X$ along a determinantal subscheme $Z$ of codimension $3$, hence we obtain:

\begin{corollary} \label{cor:rk=2} Let $X$ be a Cohen--Macaulay scheme, let $\sG$ be a quasi-coherent $\sO_X$-module of homological dimension $\le 1$ and rank $2$, and denote by 
	$$Z_2 := X^{\ge 4}(\sG) \quad \subset \quad Z := X^{\ge 3}(\sG) \quad \subset \quad X = X^{\ge 2}(\sG)$$
the second and the first degeneracy loci of $\sG$. Let $\pi \colon \Bl_Z X \to Z$ be the blowup of $X$ along $Z$. Assume $\codim_X(Z) = 3$, and furthermore that:
	$$\codim_X (X^{\ge 2+i}(\sG) ) \ge 4i, \qquad \forall i \ge 2.$$
(The expected codimension is ${\rm exp.codim}_X X^{\ge 2+i}(\sG) = i(2+i) \ge 4i$ for $i \ge 2$.) 
Then there is a $X$-linear semiorthogonal decompositions with admissible components:	
	\begin{align*}\D(\Bl_Z X) = \langle \D(\widetilde{Z}_2), ~\D(\widetilde{Z})\otimes \sO(1), \D(\widetilde{Z}) \otimes \sO(2), ~ \D(X) \otimes \sO(3) \rangle,
	\end{align*}
where $\widetilde{Z}_2 := \Quot_2(\sK) \to Z_2$ and $\widetilde{Z} := \PP(\sK) \to Z$ are partial desingularizations of $Z_2$ and $Z$; $\D$ stands for $\Perf$, $\Db$ or $\Dqc$; $\sO(1)$ stands for $\sO_{\Bl_Z X}(1) = \sO_{\Bl_Z X}(-E)$. Furthermore, $\Perf(\Bl_Z X)$ admits a relative Serre functor over $X$ given by $\SS = \otimes \sO_{\Bl_Z X}(-2)$.
\end{corollary}

% Sec: \delta = 3
\subsubsection{The case $\rank \sG = 3$} If $\rank \sG = 3$, then for any $d \ge 3$, consider the following Quot schemes: let $\shZ_+: = \Quot_{X,d}(\sG)$, and we set:
	\begin{align*}
	&\shZ_-^{\mathrm{flip}} := \Quot_{X,d-3}(\sK),  &\shZ_-^{(d-2)} := \Quot_{X, d-2}(\sK), \\
	& \shZ_-^{(d-1)} := \Quot_{X, d-1}(\sK), & \shZ_-^{\mathrm{vf}} := \Quot_{X, d}(\sK).
	\end{align*}

We denote by $\widehat{\shZ}^{\heartsuit} =  \shZ_+ \times_X \shZ_-^{\heartsuit}$ the fiber products, and $r_\pm^{\heartsuit}$ the natural projections as usual, where $\heartsuit \in \{ \mathrm{flip}, (d-2), (d-1), \mathrm{vf} \}$.  Denote by	 
	$\sO_+(1) : = \sO_{\Quot_{d,X}(\sG)}(1)$
the line bundle of Thm. \ref{thm:Quot} \eqref{thm:Quot-3} as usual.

\begin{theorem} [$\rank \sG=3$] \label{thm:rk=3} In the above situation, assume the following holds:
	\begin{align*}
	 	&\dim \shZ_+  = \dim \shZ_-^{\mathrm{flip}}  = \dim \widehat{\shZ}^{\mathrm{flip}} = \dim X - d(d-3), \\
		&\dim  \shZ_-^{(d-2)}  = \dim X - (d-2)(d+1),  \qquad  \dim \widehat{\shZ}^{(d-2)} =\dim X - d^2 + 2d + 2, \\
		&\dim  \shZ_-^{(d-1)}  = \dim X - (d-1)(d+2),  \qquad  \dim \widehat{\shZ}^{(d-1)}  = \dim X - d^2 + d + 2, \\
		&\dim \shZ_-^{\mathrm{vf}}  = \dim X - d(d+3),  \qquad \dim \widehat{\shZ}^{\mathrm{vf}}  = \dim X - d^2.
	\end{align*}
Then $\shZ_+$, $\shZ_-^{\heartsuit}$, $\widehat{\shZ}^{\heartsuit}$ are also Cohen--Macaulay schemes, where $\heartsuit \in \{ \mathrm{flip}, (d-2), (d-1), \mathrm{vf} \}$. Furthermore, for any $i \in \ZZ$, the following relative Fourier--Mukai functors over $X$:
	\begin{align*}
		&\Phi_i^{(d-3)}(\blank) \equiv \Phi_i^{\mathrm{flip}}(\blank): =  r_{+\,*}^{\mathrm{flip}} \circ r_{-}^{{\mathrm{flip}}\,*} (\blank) \otimes \sO_+(i) \colon & \Dqc(\shZ_{-}^{\mathrm{flip}}) \to \Dqc(\shZ_+), \\
		&\Phi_i^{(d-2)}(\blank): =  r_{+\,*}^{(d-2)} \circ r_{-}^{(d-2)\,*} (\blank) \otimes \sO_+(i) \colon & \Dqc(\shZ_{-}^{(d-2)}) \to \Dqc(\shZ_+), \\
		&\Phi_i^{(d-1)}(\blank): =  r_{+\,*}^{(d-1)} \circ r_{-}^{(d-1)\,*} (\blank) \otimes \sO_+(i) \colon & \Dqc(\shZ_{-}^{(d-1)}) \to \Dqc(\shZ_+), \\
		 &\Phi_i^{(d)}(\blank) \equiv \Phi_i^{\mathrm{vf}} (\blank) : =  r_{+\,*}^{\mathrm{vf}} \circ r_{-}^{{\mathrm{vf}}\,*}(\blank)  \otimes \sO_+(i)  \colon & \Dqc(\shZ_{-}^{\mathrm{vf}}) \to \Dqc(\shZ_+)
	\end{align*}
are strong in the sense of Def. \ref{Def:FM:BS} and fully faithful, and their essential images induce semiorthogonal decompositions with admissible components:
	$$\Dqc(\shZ_+) = \langle \Im \Phi_{i-1}^{\mathrm{vf}}, ~ \Im \Phi_{i}^{(d-1)},  \Im \Phi_{i+1}^{(d-2)},  \Im \Phi_{i+1}^{(d-1)},  \Im \Phi_{i+2}^{(d-2)},  \Im \Phi_{i+2}^{(d-1)} , \Im \Phi_{i+3}^{(d-2)}, ~ \Im \Phi_{i+4}^{\mathrm{flip}} \rangle.$$
Similar semiorthogonal decompositions hold if we replace $\Dqc$ by $\Db$ or $\Perf$, and these semiorthogonal decompositions are compatible with the inclusions $\Perf \subseteq \Db \subseteq \Dqc$. The category $\Perf(\shZ_+)$ admits a relative Serre functor given by:
	$$\SS_+ = (\blank) \otimes \pi_+^* ( (\det \sG)^{\otimes d}) \otimes \sO_+(-3)[d(3-d)].$$
In particular, for any $i \in \ZZ$, and $\heartsuit \in \{ \mathrm{flip}, (d-2), (d-1), \mathrm{vf} \}$,  $\SS_+ (\Im \Phi_i^{\heartsuit}) = \Im \Phi^{\heartsuit}_{i-3}$.
		\end{theorem}
As usual, there are many different mutation-equivalent ways to rewrite above semiorthogonal decomposition, and we could understand these decompositions as:
	\begin{align*}
	\D(\shZ_+) = \big \langle  \D(\shZ_-^{\rm vf}) , ~\text{3-copies of} ~ \D(\shZ_{-}^{(d-1)}), ~\text{3-copies of} ~  \D(\shZ_{-}^{(d-2)}), ~ \D(\shZ_{-}^{\rm flip}) \big \rangle.
	\end{align*}

\begin{proof} This is a globalization of  Thm. \ref{thm:rk=3:local} via the general procedure \S \ref{subsec:univHom} as before; The statement about Serre functor follows from Lem. \ref{lem:global:Serre} and Prop. \ref{prop:Serre}.
 \end{proof}

Set $d= \delta = 3$ in above theorem, then by Lem. \ref{lem:Quot=Bldet}, $\shZ_+ = \Quot_3(\sG)$ is the blowup of $X$ along a determinantal subscheme $Z$ of codimension $4$, and we obtain:

\begin{corollary}\label{cor:rk=3}  Let $X$ be a Cohen--Macaulay scheme, let $\sG$ be a quasi-coherent $\sO_X$-module of homological dimension $\le 1$ and rank $3$, and denote by 
	$$Z_3: = X^{\ge 6}(\sG) \quad \subset \quad Z_2 := X^{\ge 5}(\sG) \quad \subset \quad Z := X^{\ge 4}(\sG) \quad \subset \quad X = X^{\ge 3}(\sG)$$
the third, second and first degeneracy loci of $\sG$. Let $\pi \colon \Bl_Z X \to Z$ be the blowup of $X$ along $Z$. Assume $\codim_X(Z) = 4$, $\codim_X(Y) = 10$ and furthermore that:
	$$\codim_X (X^{\ge 3+i}(\sG) ) \ge 6i, \qquad \forall i \ge 3.$$
(The expected codimension is ${\rm exp.codim}_X X^{\ge 3+i}(\sG) = i(3+i) \ge 6i$ for $i \ge 3$.) 
Then there is a $X$-linear semiorthogonal decompositions with admissible components:
	\begin{align*}
	\D(\Bl_Z X) = \big \langle \D(\widetilde{Z}_3), ~\text{3-copies of} ~ \D(\widetilde{Z}_2), ~\text{3-copies of} ~ \D(\widetilde{Z}), ~ \D(X) \big \rangle.
	\end{align*}
(We refer the readers to Thm. \ref{thm:rk=3} for the precise functors and semiorthogonal relations.) Here $\widetilde{Z}_3 := \Quot_3(\sK) \to Z_3$, $\widetilde{Z}_2 := \Quot_2(\sK) \to Z_2$ and $\widetilde{Z} := \PP(\sK) \to Z$ are partial desingularizations of $Z_3$, $Z_2$ and $Z$; $\D$ stands for $\Perf$, $\Db$ or $\Dqc$. Furthermore, $\Perf(\Bl_Z X)$ admits a relative Serre functor over $X$ given by $\SS = \otimes \sO_{\Bl_Z X}(-3)$.
\end{corollary}

\newpage
% Appendices
\begin{appendix}
\addtocontents{toc}{\vspace{\normalbaselineskip}}
%\addtocontents{toc}{\vspace{\normalbaselineskip}}	
\section{Relations in the Grothendieck rings of varieties} \label{sec:K0}
Let $\kk$ be a field, and denote by $\mathrm{Var}/\kk$ the category of finite type $\kk$-schemes. 

\begin{definition} The {\em Grothendieck ring of $\kk$-varieties}, denoted by $K_0(\text{Var}/\kk)$, is defined as follows. First, as an abelian group $K_0(\mathrm{Var}/\kk)$ is the quotient of the free abelian group on the classes $[X]$ of finite type $\kk$-schemes $X \in \mathrm{Var}/\kk$ modulo the ``scissor
relations":
	$$[X] = [Y ] + [X \backslash Y],  \qquad \text{if $Y\subseteq X$ is a closed subscheme;}$$
The zero element is $[\emptyset] = 0$. Secondly, $K_0(\mathrm{Var}/\kk)$ carries a unique  ring structure given by:
	$$[X] \cdot [Y ] = [X \times_\kk Y] \in K_0(\mathrm{Var}/\kk),  \qquad \text{for any $X,Y\in \mathrm{Var}/\kk$};$$
The unit element is $1 = [\Spec \kk] \in K_0(\mathrm{Var}/\kk)$.
\end{definition}

\begin{example} Denote the class of an affine line by $\LL = [\AA^1]$. For any $n \ge 0$, the affine space has class $[\AA^{n}] = \LL^n$ (by convention $\LL^{0} = 1$), and the projective space has class
	$$[\PP^{n}] = 1 + \LL + \LL^2 + \cdots + \LL^n = \frac{\LL^n-1}{\LL-1} \in K_0(\mathrm{Var}/\kk).$$ 
\end{example}

\begin{example} Let $\Gr_d(n)$ be the Grassmannian variety of $d$-dimensional subspaces of an $n$-dimensional vector space over $\kk$; see Ex. \ref{ex:Grass}. The Grassmannian $\Gr_d(n)$ has class 
	$$[\Gr_d(n)] = \frac{(\LL^{n}-1)(\LL^{n-1}-1)\cdots (\LL-1)}{(\LL^{d}-1)(\LL^{d-1}-1)\cdots (\LL-1)} = \sum_{i=0}^{d(n-d)} b^{(d,n)}_{i} \, \LL^i \in K_0(\mathrm{Var}/\kk),$$
where the integer $b^{(d,n)}_{i} = b_{2i}(\Gr_d(n))$ is the $2i$-th Betti number of $\Gr_d(n)$. By convention, $\Gr_{0}(n) =\Spec \kk$, and we set $\Gr_{d}(n) = \emptyset$ if either $d > n$ or $d<0$.
\end{example}

\begin{lemma} \label{lem:motif} Let $X \in \text{Var}/\kk$,  $\sG$ be a quasi-coherent $\sO_X$-module of homological dimension $\le 1$, and set $\sK : = \sExt^1_{\sO_X}(\sG, \sO_X)$. Denote $\delta: = \rank \sG$ (which is non-negative by Lem. \ref{lem:rank}), and let $d$ be a positive integer. Then the following holds in $K_0(\mathrm{Var}/\kk)$:
	\begin{align} \label{eq:app:motivic}
	[ \Quot_{X,d}(\sG) ] = \sum_{i=0}^{\min\{d, \delta\}} \LL^{(d-i)(\delta-i)} \cdot [\Gr_{i}(\delta)] \cdot [\Quot_{X,d-i}(\sK)] \in K_0(\mathrm{Var}/\kk).
	\end{align} 
(Recall by convention for any $\sE$, $\Quot_{X,0}(\sE) = X$, and $\Quot_{X,d}(\sE) = \emptyset$ if $d<0$.)
 \end{lemma}
   
 \begin{proof} Since $X$ is noetherian, the sequence of closed subschemes of Lem. \ref{lem:degloci}, \ref{lem:hdK}:
 	$$X = X^{\ge \delta}(\sG) \supset X^{\ge \delta+1}(\sG) \supset X^{\ge \delta+2}(\sG) \supset \ldots$$
 is finite. By ``scissor relations", it suffices to show \eqref{eq:app:motivic} over each strata $X^{\ge \delta+k}(\sG) \backslash X^{\ge \delta +k+1}(\sG)$ for $k=0, 1, \ldots$. Due to Cor. \ref{cor:Quot:degloci} and the well-known fact that the class ring structure of $K_0(\mathrm{Var}/\kk)$ are multiplicative for piecewise trivial (in particular, Zariski-locally trivial) fibrations, this is equivalent to show that for all $k \ge 0$, the following holds:
  	\begin{align*} 
	[\Gr_{d}(\delta+k)] = \sum_{i=0}^{d} \LL^{(\delta-i)(d-i)} \cdot [\Gr_{i}(\delta)] \cdot [\Gr_{d-i}(k)].
	\end{align*}
(This could be regarded as a motivic version of the binomial formula 
	$\binom{\delta+k}{d} = \sum_{i=0}^{d} \binom{\delta}{i} \binom{k}{d-i}.$)
This equality follows from direct computations; We also present a geometric argument as follows: for $\ell \in \NN$, denote $V_\ell$ a $\kk$-vector space of dimension $\ell$, and fix a decomposition $V_{\delta+k} = V_{\delta} \oplus V_{k}$. Then $\Gr_d(\delta+k) = \Gr_d(V_{\delta+k})$
can be stratified by $\Gr_d(V_{\delta+k}) = \bigcup_{j=0}^{d} (G_{j} \backslash G_{j+1})$, where the closed subscheme $G_j$ is defined by $G_j:= \{ L_d \in \Gr_d(\delta+k) \mid \dim (L_d \cap V_{\delta}) \le j \}$, and $\Gr_d(V_{\delta+k}) = G_0 \supset G_1 \supset G_2 \supset \cdots \supset G_d$. Then for any $j$, $\forall L_d \in G_{j} \backslash G_{j+1}$, via the projection $V_{\delta+k} \to V_{k}$ from $V_{\delta}$, we have $L_d \mapsto \overline{L}_{d-j} \in \Gr_{d-j}(V_{\delta})$ with kernel $L_j = L_d \cap V_{\delta} \in \Gr_{j}(V_{\delta})$. For any fixed pair $(L_j, \overline{L}_{d-j})$, the choice of such $L_d$ corresponds to a linear graph of $\overline{L}_{d-j}$ in $\overline{L}_{d-j} \times (V_{\delta}/L_j)$. Hence each strata $G_{j} \backslash G_{j+1}$ admits a piecewise trivial affine space $\AA^{(\delta-j)\times (d-j)}$-fibration over $\Gr_{j}(V_{\delta})\times \Gr_{d-j}(V_{d})$. The formula is proved.
 \end{proof}
 
Inspired by Orlov's conjecture \cite{O05} on the relationships between motives and derived categories, and Kuznetsov--Shinder's philosophy \cite{KS} about the relationships between $D$-equivalences and $L$-equivalences, it is reasonable to expect that there is a categorification of \eqref{eq:app:motivic}. More precisely, we make the following conjecture:

 \begin{conjecture} \label{conj:cat} Let $X$ be a scheme, let $\sG$ be a quasi-coherent $\sO_X$-module of homological dimension $\le 1$, and set $\sK : = \sExt^1_{\sO_X}(\sG, \sO_X)$. Assume the Tor-independent conditions Def. \ref{def:Tor-ind:quot} hold for the pairs of integers $(d,d-i)$, where $i \in [0,\min\{d,\delta\}]$. Then for each $i \in [0,\min\{d,\delta\}]$, there are strong relative Fourier--Mukai transforms over $X$:
 	$$\Phi_{\shE_{i,\alpha}} \colon \D(\Quot_{X,d-i}(\sK)) \to \D(\Quot_{X,d}(\sG))$$
parametrised by the Young diagrams $\alpha \in B_{i ,\delta-i}$, such that $\Phi_{\shE_{i,\alpha}}$ is fully faithful for each $i$ and $\alpha$. Moreover, the essential images of $\Phi_{\shE_{i,\alpha}}$, where $(i, \alpha)$ runs through $i \in [0,\min\{d,\delta\}]$ and $\alpha \in B_{i ,\delta-i}$, induce a semiorthogonal decomposition:
	$$ \D(\Quot_{X,d}(\sG)) = \Big \langle \left\{ \Im \Phi_{\shE_{i,\alpha}}  \right\}_{i \in [0,\min\{d,\delta\}], \, \alpha \in B_{i ,\delta-i}} \Big \rangle.$$
Here the derived category symbol $\D$ stands for $\Perf$, $\Db$, or $\Dqc$ (see \S \ref{sec:generalities:derived}).
  \end{conjecture}

This conjecture \ref{conj:cat} is verified in the case $d=1$ in \cite{JL18}, and in the case $m=1$ in \cite{Pi20}. The current paper provides various evidences for this conjecture, in particular we verify Conj. \ref{conj:cat} in the following cases: (i) $d \le 2$, (ii) $\rank \sG \le 3$, (iii) $\ell \le 2$ (where $\ell = n - d$, $n$ is the number of generators of $\sG$ in a local presentation). 

We also expect that there is a Chow-theoretical version of the formula  \eqref{eq:app:motivic}:
 
 \begin{conjecture} \label{conj:chow} Under the same condition of Lem. \ref{lem:motif} and assume that the base scheme $X$ and the Quot schemes $\Quot_{X,d-i}(\sK)$ and $\Quot_{X,d}(\sG)$ satisfy certain reasonable regularity conditions. Then there is an isomorphism of integral Chow groups:
	\begin{align*} 
	{\rm CH}^*(\Quot_{X,d}(\sG)) \simeq \bigoplus_{i=0}^{\min\{d,\delta\}}  \bigoplus_{j=0}^{i(\delta-i)}  {\rm CH}^{*-(d-i)(\delta-i) -j}(\Quot_{X,d-i}(\sK)) ^{\oplus b_j^{(i,\delta)}}   .
	\end{align*}
Moreover, if Quot schemes $\Quot_{X,d-i}(\sK)$ and $\Quot_{X,d}(\sG)$ are smooth and projective over $\kk$, then there is an isomorphism of integral Chow motives:
	\begin{align*} 
	\foh(\Quot_{X,d}(\sG) )= \bigoplus_{i=0}^{\min\{d,\delta\}}  \bigoplus_{j=0}^{i(\delta-i)}  \left( \foh(\Quot_{X,d-i}(\sK )\otimes L^{(d-i)(\delta-i) + j}  \right) ^{\oplus b_j^{(i,\delta)}},
	\end{align*} 
	where $L$ stands for the Lefschetz motive. 
 \end{conjecture}

Notice that Conj. \ref{conj:chow} is much stronger than what will follow from Conj. \ref{conj:cat} if Orlov's conjecture \cite{O05} is true: First, it is stated over {\em integral} coefficients rather than over rational coefficients; Secondly, it predicts {\em graded} isomorphisms of Chow groups and motives. %rather than an ungraded one. 

We have verified the Conj. \ref{conj:chow} in the case $d=1$ in \cite{J19}, and in the general case in \cite{J20} under the assumption that all degeneracy loci of $\sG$ have expected dimensions.

\newpage
\section{Characteristic-free results for projective bundles} \label{sec:proj.bundle}
The results in this appendix are well-known to experts; However, as it is hard to find a reference in our stated generality, we provide the details here. Let $S$ be a quasi-compact, quasi-separated scheme, and let $\sE$ be a locally free sheaf of rank $r+1$ over $S$, $r \ge 0$ is an integer. We denote by $\pi \colon \PP(\sE) = \Quot_{S,1}(\sE) = \Proj \Sym^\bullet (\sE) \to S$ the projective bundle associated to $\sE$. Let $\sO(1) = \sO_{\PP(\sE)}(1)$ be Grothendieck ample line bundle, and $\Omega_{\PP(\sE)/S}$ the sheaf of relative K\"ahler differentials. For $i,j \in \ZZ$, $i \ge 0$, we set $\Omega^i(j) := \bigwedge^i \Omega_{\PP(\sE)/S} \otimes \sO(j)$.  Then there is a short exact sequence of locally free sheaves over $\PP(\sE)$, called {\em Euler sequence}:
	\begin{align}\label{eqn:proj:Euler}
	0 \to \Omega^1(1) \to \pi^*\sE \to \sO(1) \to 0.
	\end{align}
The morphism $\pi \colon \PP(\sE) \to S$ is perfect and proper, therefore for $A, B \in \Perf(\PP(\sE))$, the $\Dqc(S)$-valued hom object of Def. \ref{def:shom} takes values in $\Perf(S)$:
	$$\shom_S(A,B) = \pi_* \RsHom_{\PP(\sE)}(A,B) \in \Perf(S).$$
Here $\pi_* = \bR \pi_*$ denotes the {\em derived} pushforward as usual.

% Lem: Grothendieck for projective space
\begin{lemma} \label{lem:proj.bundle}
	\begin{enumerate}[leftmargin=*]
	\item \label{lem:proj.bundle-1} For any $d \in \ZZ$, the following holds:
	$$
	 \pi_* (\sO_{\PP(\sE)}(d)) =
	\begin{cases}
	\Sym^d \sE, & \text{if~}  d \ge 0; \\
	(\Sym^{-d-r-1} \sE \otimes \det \sE)^\vee [-r], & \text{if~} d \le -r-1;\\
	0 &  \text{if ~} d \in [-r, -1].\\
	\end{cases}
	$$
	\item  \label{lem:proj.bundle-2} The relative dualizing complex $\omega_{\pi} =\Omega^r [r] =  \det \sE \otimes \sO_{\PP(\sE)}(-r-1)[r]$ is a shift of a line bundle, and the following holds: $\pi_* (\omega_\pi) \xleftarrow{\sim} \sO_S$, and for any $d \in \ZZ$,
	$$
	\pi_*( \omega_{\pi} \otimes \sO(-d))= \pi_! (\sO(-d)) =
	\begin{cases}
	(\Sym^d \sE)^\vee, & \text{if~}  d \ge 0; \\
	\Sym^{-d-r-1} \sE \otimes \det \sE [r], & \text{if~} d \le -r-1;\\
	0 &  \text{if } d \in [-r, -1].\\
	\end{cases}
	$$
\end{enumerate}
\end{lemma}
\begin{proof} This is essentially a reformulation of \cite[III, 2.1.15 \& 2.1.16]{EGA}.
\end{proof}
\begin{remark} \eqref{lem:proj.bundle-1} and \eqref{lem:proj.bundle-2} are connected by Serre duality: for $d \in \ZZ$, the natural pairing
	$$ \pi_* (\sO_{\PP(\sE)}(d) ) \otimes  \pi_* (\omega_{\pi} \otimes \sO(-d) ) \longrightarrow  \pi_* (\omega_{\pi}) \simeq \sO_S,$$
induces an isomorphism $\pi_* (\omega(-d) ) = \pi_! (\sO_{\PP(\sE)}(-d)) \simeq (\pi_* (\sO_{\PP(\sE)}(d) ))^\vee$.
\end{remark}

Now Beilinson's and Orlov's theorem can be formulated as follows:

\begin{theorem}[Projective bundles] \label{thm:proj.bundle} Let $\pi \colon \PP(\sE)  = \Proj \Sym^\bullet (\sE) \to S$ be the projective bundle associated to a locally free sheaf $\sE$ of rank $r+1$ as above. Then the following holds:
\begin{enumerate}[leftmargin=*]
	\item \label{thm:proj.bundle-1} $(\sO, \sO(1), \ldots, \sO(r))$ is a relative full exceptional collection of $\PP(\sE)$ over $S$, with left dual exceptional collection given by $(\Omega^r(r)[r], \ldots, \Omega^1(1)[1], \sO)$. Furthermore, the following holds: for any $i, j \in [0, r]$,
	\begin{align*}
		& \shom_S(\sO(i), \sO(j)) = \Sym^{j-i} (\sE), \qquad
		 \shom_S(\Omega^i(i), \Omega^j(j)) =\bigwedge^{i-j} \nolimits \sE^\vee;\\
		&\shom_S(\sO(i), \Omega^j(j)) = \delta_{i,j} \cdot \sO_S[-i].
	\end{align*}
(By convention $\Sym^d = 0 =\bigwedge^d$ if $d <0$; $\delta_{i,j} = 1$ if $i=j$, $\delta_{i,j}=0$ if $i \ne j$.)
	\item \label{thm:proj.bundle-2} $\Perf(\PP(\sE))$ admits a relative Serre functor over $S$ given by $\SS_{\PP(\sE)/S} = (\blank) \otimes \omega_\pi$, where $\omega_{\pi} =\Omega^r [r] =  \bigwedge^{r+1} \sE \otimes \sO_{\PP(\sE)}(-r-1)[r]$ is given in Lem. \ref{lem:proj.bundle} \eqref{lem:proj.bundle-2}.
	
	\item \label{thm:proj.bundle-3} There are $S$-linear semiorthogonal decompositions with admissible components
 	\begin{align*}
	\Perf(\PP(\sE)) &= \langle \pi^* \Perf(S) \otimes \sO(-r), \ldots,\pi^* \Perf(S)  \otimes \sO(-1), \pi^* \Perf(S)  \rangle;\\
	\Db(\PP(\sE)) &= \langle \pi^* \Db(S) \otimes \sO(-r), \ldots,\pi^* \Db(S)  \otimes \sO(-1), \pi^* \Db(S)  \rangle; \\
	\Dqc(\PP(\sE))& = \langle \pi^* \Dqc(S) \otimes \sO(-r), \ldots,\pi^* \Dqc(S)  \otimes \sO(-1), \pi^* \Dqc(S)  \rangle
	\end{align*}
which are compatible with the natural inclusions $\Perf \subseteq \Db \subseteq \Dqc$. The corresponding projection functors $\pr_i \colon \Dqc(\PP(\sE)) \to \Dqc(S) \otimes \sO(i)$ (resp. $\pr_i \colon \Db(\PP(\sE)) \to \Db(S) \otimes \sO(i)$ and  $\pr_i \colon \Perf(\PP(\sE)) \to \Perf(S) \otimes \sO(i)$) are given by the formula:
	$$\pr_i (\blank) = \pi_* \big( \blank \otimes \Omega^{-i}(-i)[-i] \big) \otimes \sO(i), \quad \text{for} \quad i= -r, \ldots, -1, 0.$$
	
	\item \label{thm:proj.bundle-4} There are $S$-linear semiorthogonal decompositions with admissible components:
	 \begin{align*}
	\Dqc(\PP(\sE)) &= \langle \pi^* \Dqc(S) \otimes \Omega^r(r), \ldots, \pi^* \Dqc(Y)  \otimes \Omega^1(1), \pi^* \Dqc(S)  \rangle; \\
	\Db(\PP(\sE)) &= \langle \pi^* \Db(S) \otimes \Omega^r(r), \ldots, \pi^* \Db(Y)  \otimes \Omega^1(1), \pi^* \Db(S)  \rangle; \\
	\Perf(\PP(\sE))& = \langle \pi^* \Perf(S) \otimes \Omega^r(r), \ldots, \pi^* \Perf(Y)  \otimes \Omega^1(1), \pi^* \Perf(S)  \rangle,
	\end{align*}
which are compatible with the natural inclusions $\Perf \subseteq \Db \subseteq \Dqc$. The corresponding projection functors $\pr_j' \colon \Dqc(\PP(\sE)) \to \Dqc(S) \otimes \Omega^j(j)$ (resp. $\pr_j' \colon \Db(\PP(\sE)) \to \Db(S) \otimes \Omega^j(j)$ and  $\pr_j' \colon \Perf(\PP(\sE)) \to \Perf(S) \otimes \Omega^j(j)$) are given by the formula:
	$$\pr_j' (\blank) = \pi_* \big( \blank \otimes \sO(-j)[j] \big) \otimes \Omega^j(j), \quad \text{for} \quad j=  r, \ldots, 1, 0.$$
\end{enumerate}
\end{theorem}

% proof
\begin{proof} For \eqref{thm:proj.bundle-1}, the formula $\shom_S(\sO(i), \sO(j)) = \Sym^{j-i} \sE$ follows directly from Lem. \ref{lem:proj.bundle} \eqref{lem:proj.bundle-1}. For the other formulae, first we prove that for each $i \in [0, r]$, the following holds: 
	\begin{equation}\label{eqn:proj:induction}
		\shom_S(\sO(j), \Omega^i(i)) =
	\begin{cases}
	\bigwedge^{i+1} \sE & \quad \text{if}  \quad   j=-1; \\
	\sO_S[-i] & \quad \text{if} \quad   j = i; \\
	0 & \quad  \text{if} \quad   0 \le j \le i-1.
	\end{cases}
	\end{equation}
This could proved by descending induction on $i \in [0,r]$. The base case $i = r$ follows from Lem. \ref{lem:proj.bundle} \eqref{lem:proj.bundle-2}, since $\Omega^r(r) = \bigwedge^{r+1} \sE \otimes \sO(-1) =  \omega_{\pi} (r)[-r]$. Assume the claim holds for some $i \in [1,r]$, then the Euler sequence \eqref{eqn:proj:Euler} induces a short exact sequence:
	\begin{equation}\label{eqn:proj:Euler_i}
	0 \to \Omega^i(i) \to \pi^* (\bigwedge^i  \nolimits \sE) \to \Omega^{i-1}(i) \to 0.
	\end{equation}
Hence for any $0 \le j \le i-1$,  
	\begin{align*}
		& \shom_S(\sO(-1), \Omega^{i-1}(i-1)) = \shom_S(\sO, \Omega^{i-1}(i)) \simeq  \bigwedge^i \nolimits \sE; \\
		& \shom_S(\sO(i-1),  \Omega^{i-1}(i-1)) = \shom_S(\sO(i), \Omega^{i-1}(i)) \simeq   \shom_S(\sO(i), \Omega^{i}(i))[1] = \sO_S[-i+1]; \\
		& \shom_S(\sO(j),  \Omega^{i-1}(i-1)) = \shom_S(\sO(j+1), \Omega^{i-1}(i)) \simeq  \shom_S(\sO(j+1), \Omega^{i}(i))[1] = 0.
	\end{align*}
By induction \eqref{eqn:proj:induction} is proved. Next, we claim that for each $i \in [1,r]$ the following holds:
	\begin{equation}\label{eqn:proj:induction2}
		\begin{cases}
	\shom_S(\Omega^i(i), \Omega^j(j)) \simeq \bigwedge^{i-j} \nolimits (\sE^\vee) 		&\qquad \text{for} \qquad j \in [0,r];\\
	\shom_S(\Omega^i(i) ,\sO(-k)) = 0 &\qquad \text{for} \qquad k \in [1, r-i].
	\end{cases}
	\end{equation}
We prove it by ascending induction on $i \in [0,r]$. The base case $i=0$ holds by Lem. \ref{lem:proj.bundle} \eqref{lem:proj.bundle-1}. Assume the claim holds for $i-1$, where $i \in [1,r]$ then it follows from \eqref{eqn:proj:Euler_i} that $\shom_S(\Omega^i(i), \sO) \simeq \bigwedge^i \sE^\vee$, and for $k \in [1,r-i]$, 
	$\shom_S(\Omega^i(i), \sO(-k)) = \shom_S(\Omega^{i-1}(i-1), \sO(-k-1))[1] =0.$ 
Finally by \eqref{eqn:proj:Euler_i}, \eqref{eqn:proj:induction}, and \eqref{eqn:proj:Euler_i} with $i$ substituted by $j \in [1,r]$:
	\begin{align*}
	\shom_S(\Omega^i(i), \Omega^j(j)) \simeq \shom_S(\Omega^{i-1}(i), \Omega^j(j))[1] \simeq  \shom_S(\Omega^{i-1}(i), \Omega^{j-1}(j)) \simeq \bigwedge^{i-j}  \nolimits \sE^\vee.
	\end{align*}
Hence \eqref{eqn:proj:induction2} holds, and all formulae of \eqref{thm:proj.bundle-1} are proved. In particular, $\{\sO(i)\}$ and $\{\Omega^i(i)[i]\}$ are relative dual exceptional sequences. The fullness follows from \eqref{thm:proj.bundle-3} and \eqref{thm:proj.bundle-4} below.

The statement \eqref{thm:proj.bundle-2} follows from Lem. \ref{lem:proj.bundle} \eqref{lem:proj.bundle-2}. 

Finally for \eqref{thm:proj.bundle-3} and \eqref{thm:proj.bundle-4}, from \eqref{thm:proj.bundle-1} we know that $\{\sO(i)\}_{i \in [a, a+r]}$ for any $a \in \ZZ$ forms a relative exceptional sequence, since $\otimes \sO(a)$ is an autoequivalence of $\Dqc(\PP(\sE))$. Then it follows from Lem. \ref{lem:relexc} and Cor. \ref{cor:relexc} that the right hand sides of these formulae form semiorthogonal sequences. Next, observe that Orlov and Kapranov's techniques of resolution of diagonal still works in our setup. In fact, if we denote by $\Delta \colon \PP(\sE) \hookrightarrow \PP(\sE) \times_S \PP(\sE)$ the diagonal embedding, then it is a regular closed immersion with Koszul resolution
	$$K^\bullet: \quad 0 \to \sO(-r)\boxtimes \Omega^{r}(r) \to \cdots \to \sO(-1) \boxtimes \Omega(1) \to \sO \boxtimes \sO \to \sO_\Delta.$$
Regard $K^\bullet$ as a complex over $\Perf(X \times_Y X)$, then the stupid truncation $Y^i = \sigma^{\ge i} K^\bullet$, where $i \in [-r, 0]$, gives rises to a canonical right Postnikov system (see Ex. \ref{eg:conv:lci}, with $a=-r, b=0$) of Fourier--Mukai kernels, whose associated graded objects (see Def. \ref{def:Postnikov}) are given by $\sO(i) \boxtimes \Omega^{-i}(-i)[-i]$ for $i \in [-r, 0]$. Pushing forward this Postnikov system along the two natural projections $\PP(\sE) \times_S \PP(\sE) \to \PP(\sE)$ to the first and second factors, we obtain the fullness and the formulae for projection functors of \eqref{thm:proj.bundle-3} and \eqref{thm:proj.bundle-4}. 
\end{proof}

\begin{remark}[Tilting properties] It follows from the theorem that the two exceptional collections $\{\sO(i)\}_{i \in [0,r]}$ and $\{\Omega^i(i)\}_{i \in [0,r]}$ are {\em strong (or tilting) over $S$} in the sense that: for any two exceptional objects $E_i, E_j \in \{\sO(i)\}_{i \in [0,r]}$ or $E_i, E_j \in \{\Omega^i(i)\}_{i \in [0,r]}$, $\shom_S(E_i, E_j) \in \Coh(S) \cap \Perf(S)$ is a {\em sheaf}, i.e. $R^k \pi_* \RsHom_{\PP(\sE)}(E_i, E_j) = 0$ for $k >0$. Hence $\sT = \bigoplus_{i=0}^r \sO(i)$ and $\sT' = \bigoplus_{i=0}^r \Omega^i(i)$ are two {\em relative tilting bundles of $\PP(\sE)$ over $S$}. 
\end{remark}
 
 \begin{remark}[Dual version] \label{rmk:proj:dual} Since $(\blank)^\vee = \RsHom_{\PP(\sE)}(\blank, \sO_{\PP(\sE)}) \colon \Perf(\PP(\sE))^{\rm op} \to \Perf(\PP(\sE))$ is an anti-autoequivalence, the sequence $\{(\Omega^i(i))^\vee\}_{i \in [0,r]} = \{\sO, (\Omega^1(1))^\vee, \ldots, (\Omega^{n-1}(n-1))^\vee\}$ is also a {\em strong} relative exceptional collection of $\PP(\sE)$ over $S$, and its {\em left} dual exceptional collection is given by $\{\sO(-i)[i]\}_{i \in [0,r]}$. Hence $\sT^\vee = \bigoplus_{i=0}^r \sO(-i)$ and $(\sT')^\vee = \bigoplus_{i=0}^r (\Omega^i(i))^\vee$ are also two relative tilting bundles of $\PP(\sE)$ over $S$. 
 \end{remark}
 
%The majority of results about mutations of this section holds for $\PP(\sG)$ (provided that certain Tor-independent conditions are satisfied), where $\sG$ has homological dimension $\le 1$, once one replaces the dual of $\sE$ by the derived dual of $\sG$. The critical difference is that the relative differentials $\Omega^i(i)$ are generally not vector bundles but sheaves of finite homological dimensions. These sheaves will play a crucial role in the sequel papers.

{\em Updates.} In the sequel paper \cite{J22}, we have generalized the results of this appendix to the cases where $X$ is any prestack and $\sE$ is a perfect complex on $X$ of Tor-amplitude in $[0,1]$. The upshot is that, the results on mutations of this appendix remain true in the latter general situation, as long as we replace the relative cotangent bundle $\Omega = \Omega_{\PP(\sE)/X}^1$ by the relative cotangent complex $L_{\PP(\sE)/X}$, and replace the usual symmetric and exterior products by the {\em derived} symmetric and exterior products in the sense of Lurie.
 \end{appendix}

\newpage	
\addtocontents{toc}{\vspace{\normalbaselineskip}}

\end{document}